\documentclass[12pt, a4paper]{amsart}
\usepackage[centering]{geometry}
\usepackage{amssymb}
\usepackage{mathrsfs}% for \mathscr
\usepackage{hyperref}
\usepackage{graphicx}
\usepackage{enumerate}
\usepackage[all,cmtip]{xy}
\usepackage{amscd}

\numberwithin{equation}{section}
\usepackage{amsthm}
\newtheorem{theorem}{Theorem}[section]
\newtheorem{proposition}{Proposition}[section]
\newtheorem{corollary}{Corollary}[section]
\newtheorem{lemma}{Lemma}[section]
\newtheorem{conj}{Conjecture}[section]
\newtheorem{definition}{Definition}[section]
\newtheorem{remark}{Remark}[section]
\newtheorem{fact}{Fact}[section]

\theoremstyle{definition}
\newtheorem{example}{Example}[section]

\newcommand{\Pd}{\mathcal{P}^d}
\newcommand{\crit}{\operatorname{Crit}}
\newcommand{\Tp}{|T|_\mathrm{p}}
\newcommand{\Tn}{|T|_\mathrm{np}}
\newcommand{\Tfp}{|T(f_0)|_\mathrm{p}}
\newcommand{\Tfn}{|T(f_0)|_\mathrm{np}}
\newcommand{\RZ}{\mathbb{R}/\mathbb{Z}}
\newcommand{\QZ}{\mathbb{Q}/\mathbb{Z}}
\newcommand{\modZ}{\pmod{\mathbb{Z}}}
\newcommand{\Hbar}{\overline{\mathcal{H}}}
\newcommand{\model}{\mathcal{C}\big(\widehat{\mathcal{P}}_0^T\big)}
\newcommand{\modelf}{\mathcal{C}\big(\widehat{\mathcal{P}}_0^{T(f_0)}\big)}
\newcommand{\dPoin}{\operatorname{\delta_{Poin}}}
\newcommand{\dconf}{\operatorname{\delta_{conf}}}

\newcommand{\Hdim}{\operatorname{H{.}dim}}
\newcommand{\hypdim}{\operatorname{hyp-dim}}
\newcommand{\Cdim}{\dim_\mathbb{C}}
\newcommand{\topdim}{\dim_\mathrm{top}}

\newcommand{\diam}{\operatorname{diam}}
\newcommand{\dist}{\operatorname{dist}}
\newcommand{\area}{\operatorname{area}}
\newcommand{\id}{\operatorname{id}}
\newcommand{\Comp}{\operatorname{Comp}}
\newcommand{\Con}{\operatorname{Con}}
\newcommand{\CK}{C_\mathrm{K}}

\title{Boundaries of capture hyperbolic components} 
\author{Jie Cao}
\author{Xiaoguang Wang}
\author{Yongcheng Yin}
\address{School of Mathematical Sciences, Zhejiang University, Hangzhou, 310027, China}
\email{mathcj@foxmail.com}
\email{wxg688@163.com}
\email{yin@zju.edu.cn}
\date{\today}

\makeatletter
\@namedef{subjclassname@2020}{%
  \textup{2020} Mathematics Subject Classification}
\makeatother


\begin{document}

\begin{abstract}
In complex dynamics, the boundaries of higher dimensional hyperbolic components in 
holomorphic families of polynomials or rational maps are mysterious objects, 
whose topological and analytic properties are fundamental problems.

In this paper,
% we studied the hyperbolic components in the families of polynomials arsing from algebraic vaireties defined %by 
we show that in some typical families of polynomials (i.e. algebraic varieties defined by periodic critical relations), the boundary of a capture hyperbolic 
component $\mathcal H$ is homeomorphic to the sphere $S^{2\Cdim(\mathcal{H})-1}$. 
% More strikingly, 
Furthermore, we establish an unexpected identity for 
the Hausdorff dimension of $\partial \mathcal H$: 
$$\Hdim(\partial\mathcal{H}) = 2 \Cdim(\mathcal{H})-2+\max_{f\in\partial\mathcal{H}}
\Hdim(\partial A^J(f)),$$ where $A^J(f)$ is the union of the bounded attracting Fatou components of $f$ 
associated with the free critical points in the Julia set $J(f)$. 

In the proof, some new results with independent interests are discovered.

\end{abstract}

\maketitle 

\setcounter{tocdepth}{1}
% \addtocontents{toc}{\protect\setcounter{tocdepth}{1}}
\tableofcontents

\section{Introduction}
Rational maps $f:  \widehat{\mathbb{C}}\rightarrow \widehat{\mathbb{C}}$, viewed as dynamical systems,
yield complicated behaviors under iterations.
Their bifurcations, both in the dynamical plane and in the parameter space,
are the major attractions in the field of complex dynamical systems in recent forty years.
%There have been a deep studies for a single map in the 

One of the interesting themes in recent years is to study the bifurcation loci in the parameter spaces of typical families of rational maps. It is known from the fundamental work of Shishikura \cite{Shishikura-Hdim}, Tan Lei \cite{Tan}, McMullen \cite{McM-Univ} and Gauthier \cite{Gauthier}, that in the space of polynomials or general holomorphic families of rational maps, the bifurcation loci (if nonempty) have full Hausdorff dimension.
Recently, Astorg, Gauthier, Mihalache and Vigny \cite{AGMV} have shown that in the moduli space of degree $d\geq 2$ rational maps, the support of the bifurcation measure (a notable subset of the bifurcation locus) has positive Lebesgue measure.

The bifurcation loci have  distinguished subsets: the boundaries of the hyperbolic components. These subsets contribute considerable complexity to the bifurcation loci from the topological and analytic viewpoints.
Recall that, a rational map $f$ is \emph{hyperbolic}, if all the critical orbits are attracted by the attracting cycles.
In a holomorphic family of rational maps, it is known that the hyperbolic maps form an open subset, each component is called a \emph{hyperbolic component}. The topology of a hyperbolic component $\mathcal H$ may be easy to understand, however its boundary $\partial \mathcal H$, in general, exhibits quite complicated topological features, making the exploration of $\partial \mathcal H$ intractable.

%A natural question is, what can we say about the boundary of a hyperbolic component?
%How the maps bifurcates along the boundary $\partial \mathcal H$?

%The study of  these part is a fundamental theme in recent years,
% These parts are the frontier of the bifucation loci when hy

%  Particular interest is focused on the boundary of hyperbolic components. 

Some progress %on the study of boundaries of hyperbolic components in some families 
has been made 
when the parameter space has complex dimension one, for example the McMullen maps \cite{QRWY}, the cubic Newton maps \cite{RWY}, and some one-parameter families of polynomials \cite{Roesch,Wang21}. 
For the families considered therein, it is shown that the boundary of a hyperbolic component (except the unbounded one, in the polynomial case) is a Jordan curve. 

These results are obtained based on two reasons. First, since the parameter space has dimension one,  the  computer picture is helpful to formulate a conjectural description of the boundaries of the hyperbolic components. These  descriptions indicate what to prove in mathematics. 
 Secondly, it is feasible to  study the dynamics of all maps on the boundaries of these hyperbolic components, to the extent that it suffices to understand the boundary topology. 
The  basic philosophy behind is, it transfers the boundary regularity of the Fatou components in the dynamical plane to that of the hyperbolic components in the parameter space.
These are the reasons why previous work mainly focused on dimension one.

If a hyperbolic component has higher dimension, things are different.
It is difficult to visualize a hyperbolic component whose complex dimension is at least two. Hence it seems an  impossible task to make an a priori description of its boundary. 

We would like to mention that there is some progress on understanding the boundary of the central hyperbolic component $\mathcal H_d$ (i.e. the hyperbolic component containing the monomial $z^d$) of the monic and centered polynomials of degree $d\geq 3$. 
Petersen and Tan Lei \cite{PT09} studied the tame part of $\partial \mathcal H_3$.  
Blokh, Oversteegen, Ptacek and Timorin \cite{BOPT} studied the combinatorial model of $\partial \mathcal H_3$. 
Luo \cite{Luo} gave a combinatorial classification of the geometrically finite polynomials
on $\partial \mathcal H_d$. 
 
Let $\mathcal F$ be an algebraic family  in the space  ${\rm Rat}_d$ of rational maps of degree $d\geq 2$, and let $\mathcal H$ be a hyperbolic component in $\mathcal F$. Milnor studied the case that all maps in $\mathcal H$ have the
property that the basin of every attracting cycle
contains exactly one grand orbit equivalence class of critical
points. This kind of hyperbolic component is called UGO (i.e. uni-grand orbit) type.
Milnor showed that if $\mathcal H$ is  UGO, then $\partial \mathcal H$ is semi-algebraic, hence locally connected. For the general case, he posed the following conjecture \cite[Conjecture 1]{MilnorPPT}: 
% \ref{} on the boundary topology

\begin{conj} 
\label{nlc}
 %Let $\mathcal H$ be a hyperbolic component in an algebraic familty $\mathcal F$.
If  the maps in $\mathcal H$ have an attracting cycle with
two distinct free critical points in its attracting basin, then $\partial \mathcal H$ is not locally connected.
\end{conj}

Besides this conjectural topological picture, the analytic complexity of $\partial \mathcal H$ makes a fundamental problem. 
Among all kinds of hyperbolic components, there is  an exceptional  case: that is, all maps on $\partial\mathcal H$ have at least one indifferent cycle.  The UGO type hyperbolic components are natural examples of these exceptional ones. 
From the analytic viewpoint, Milnor posed the following conjecture  on the Hausdorff dimension $\Hdim(\partial \mathcal H)$ of $\partial \mathcal H$  \cite[Conjecture 2(b)]{MilnorPPT}:

\begin{conj} \label{hdc}  If $\mathcal H$ is not exceptional, then $\partial \mathcal H$ is a fractal set in the sense that
$$\Hdim(\partial \mathcal H)>\topdim(\partial \mathcal H),$$
where $\topdim(\partial \mathcal H)$ is the topological dimension of $\partial \mathcal H$.
\end{conj}

Exploring these conjectures is the main theme of the current paper.

\subsection{Main results}

In the paper, we study a typical high-dimensional (also including one-dimensional) case: the capture hyperbolic components. 
Let $\mathcal P^d$ be the space of monic and centered polynomials of degree $d\geq 3$. 
%To state our results,  it is convenient to work with the space $\widehat{\mathcal{P}}^d$ consisting of %polynomials $f\in\mathcal{P}^d$ together with an ordered list $\mathbf{c} = (c_1,\dots,c_{d-1})$ of its %critical points in $\mathbb{C}$. 
For the complex dynamics community,  it is convenient to work with the space $\widehat{\mathcal{P}}^d$ of critically marked polynomials $(f, \mathbf{c})$, where $f\in\mathcal{P}^d$ and $\mathbf{c} = (c_1,\dots,c_{d-1})$ is an ordered list of its critical points in $\mathbb{C}$.

We identify $\widehat{\mathcal{P}}^d$ with 
the complex $(d-1)$-dimensional affine space 
$$\big\{(\mathbf{c},a)=(c_1,\dots,c_{d-1},a)\in \mathbb{C}^d{;~} c_1+\cdots+c_{d-1}=0\big\}$$
through $(f,\mathbf{c})\mapsto (\mathbf{c},f(0))$ and $(\mathbf{c},a)\mapsto (f_{\mathbf{c},a}, \mathbf{c})$, where 
$$f_{\mathbf{c},a}(z)=d\int_{0}^{z}(\zeta-c_1)\cdots(\zeta-c_{d-1})d\zeta +a.$$
For simplicity, a critically marked polynomial in $\widehat{\mathcal{P}}^d$ is written as $f$, 
which can be viewed as the pair of the polynomial $f(z)$ and its critical marking 
$\mathbf{c}(f) = (c_1(f),\dots,c_{d-1}(f))$. 

%Each $f \in \widehat{\mathcal{P}}^d$ determines the monic and centered polynomial $f(z)$, 
%and the critical marking $\mathbf{c}(f) = (c_1(f),\dots,c_{d-1}(f))$ of $f(z)$. 

%Let $\mathbf{c}(f) = (c_1(f),\dots,c_{d-1}(f))$ be a critical marking of $f$. Note that a map $f$ may have $(d-1)!$ different critical markings,  we fix one of them in discussion.
%For simplicity, a critically marked map $(f, \mathbf{c}(f))\in \widehat{\mathcal{P}}^d$ is also written as $f\in \widehat{\mathcal{P}}^d$.

%Then we can identify $f\in\widehat{\mathcal{P}}^d$ with $(\mathbf{c}(f),f(0))$. 
% We will take $f = (f,\mathbf{c})\in\widehat{\mathcal{P}}^d$ or $(\mathbf{c},a)\in\widehat{\mathcal{P}}^d$ according to our convenience. 
%Define 
%$$\pi:\begin{cases} \widehat{\mathcal{P}}^d \rightarrow \mathcal{P}^d,\\
%(f,\mathbf{c})\mapsto f, ~\text{or}~(\mathbf{c},a)\mapsto f_{\mathbf{c},a}. \end{cases}$$
%Then $\pi$ is a proper holomorphic map of degree $(d-1)!$. 

Let $n_1,\dots,n_\ell$ be $\ell$ positive integers with $1\leq \ell \leq d-2$. 
Consider an algebraic family $\mathcal F\subset\widehat{\mathcal{P}}^d$ defined by 
$\ell$ periodic critical relations: 
%finitely many superattracting relations: 
\begin{equation}
\label{superattracting relations}
\mathcal{F}=\big\{f\in \widehat{\mathcal{P}}^d{;~} f^{n_j}(c_j(f))=c_j(f), ~\forall~ 1\leq j\leq \ell\big\}.
\end{equation}
The algebraic set $\mathcal F$ may have many irreducible components.
Taking an irreducible one, we may assume that $\mathcal F$ is irreducible, hence is an algebraic variety.
A \emph{capture hyperbolic component} in $\mathcal F$ is a hyperbolic component consisting of polynomials $f$ for which 
all \emph{free critical points} $(c_j(f))_{\ell+1\leq j\leq d-1}$ are contained in the strictly preperiodic Fatou components of $f$.  
It is known from \cite{MilHyperCompo} that each capture hyperbolic component $\mathcal{H}\subset 
\mathcal F$ is a complex manifold with dimension $\Cdim(\mathcal{H}) = d-1-\ell$, 
and a topological cell of dimension $2\Cdim(\mathcal{H})$ as well (see Lemma \ref{basic facts of H}). 

Our first main result characterizes the boundary topology of a capture hyperbolic component:

\begin{theorem} 
\label{top-boundary} 
The boundary $\partial \mathcal H$  of a capture hyperbolic component $\mathcal H\subset \mathcal F$ is homeomorphic to the sphere $S^{2\Cdim(\mathcal H)-1}$.
\end{theorem}

% Theorem \ref{top-boundary} aims to give a better understanding of  Conjecture \ref{nlc}. 
Theorem \ref{top-boundary} implies that the following situation may happen: if the maps in a hyperbolic component  have an attracting cycle with
at least two distinct free critical points in its attracting basin, and if these free critical points are in the strictly preperiodic Fatou components,  the boundary of the hyperbolic component may be locally connected.
To the best of the authors' knowledge, 
Theorem \ref{top-boundary}  is the first complete topological description of a hyperbolic component boundary in the higher dimensional case. %: $\Cdim(\mathcal H)\geq 2$.

% but also may be pretty regular.

%Even though the boundary of a capture hyperbolic component $\mathcal H$ has a pretty regular topology, it's %analitic complexity indeed exists. 

Our second main result studies the  boundary $\partial \mathcal H$ from the analytic viewpoint. 
% We establish a striking identity for the Hausdorff dimension of $\partial \mathcal H$:
We establish an unexpected relationship between the Hausdorff dimension of $\partial \mathcal H$ and that of 
the boundaries of Fatou components:

\begin{theorem} 
\label{hd-boundary} 
The boundary $\partial \mathcal H$ of a capture hyperbolic component $\mathcal H \subset \mathcal F$ has Hausdorff dimension
$${\Hdim}(\partial \mathcal H)= 2 \Cdim(\mathcal H)-2+\max_{f \in \partial \mathcal H}{\Hdim}(\partial A^J(f)),$$
where $A^J(f)$ is the union of the bounded attracting Fatou components of $f$ associated with the free critical points in the Julia set $J(f)$. 
\footnote{
See \S\ref{section the hd-formula} for the precise definition of $A^J(f)$.
}
\end{theorem}

It is worth noting  that $\Hdim(\partial A^J(f))$ is upper semi-continuous with respect to $f\in\partial \mathcal H$, hence the maximum in the identity can always be attained at some polynomial $f \in \partial \mathcal H$ (see  \S\ref{section the hd-formula}). 
 Moreover, for any $f \in \partial \mathcal H$ and any bounded attracting Fatou component $U$ of $f$, 
the boundary $\partial U$ is a quasicircle (see Corollary \ref{semi-hyperbolicity}), implying that ${\Hdim}(\partial U)<2$. On the other hand, by a result of  Przytycki \cite{Przy06}, we have ${\Hdim}(\partial U)>1$.
Then Theorem \ref{hd-boundary} implies the following estimate 
%\begin{corollary} \label{estimate-dim} The boundary $\partial \mathcal H$ of a capture hyperbolic component %$\mathcal H \subset \mathcal F$
$$2\Cdim(\mathcal H)-1 = \topdim(\partial \mathcal H)<\Hdim(\partial \mathcal H)<2\Cdim(\mathcal H).$$
%\end{corollary}
% This confirms Conjecture \ref{hdc} in the capture case.

The Hausdorff dimension formula, even in the one-dimensional case,
%for which $\mathcal F$ is an algebraic curve,
is new:

\begin{corollary} \label{application-dim1} If $\mathcal F$ is an algebraic curve (i.e. $\Cdim(\mathcal F)=1$), then the boundary $\partial \mathcal H$  of a capture hyperbolic component $\mathcal H$ in $\mathcal F$ is a Jordan curve  with Hausdorff dimension
$${\Hdim}(\partial \mathcal H)=\max_{f \in \partial \mathcal H}{\Hdim}(\partial U_{f}), $$
where $U_f$ is the bounded attracting Fatou component of $f$ associated with the only free critical point (in the Julia set). 
\footnote{ In this case, it is shown in \S\ref{continuity-hd} that ${\Hdim}(\partial U_{f})$ is continuous with respect to  $f \in \partial \mathcal H$.}
% is the collection of attracting Fatou components of $f$ whose boundary meets free critical orbits (not rigious, REWRITE!).
\end{corollary}

%In in the one-dimensional case,  

\begin{example}
Corollary \ref{application-dim1} covers some known examples % for the capture hyperbolic components
 of one-dimensional parameter space of polynomials  \cite{Roesch,Wang21}. Among these one-dimensional cases, a notable one is Milnor's \emph{superattracting curve} $\mathcal S_p$ (see \cite{MilCubicPoly}) in the space of cubic polynomials
$$f_{c,a}(z) = z^3 - 3 c^2 z + a, ~ (c,a)\in\mathbb{C}^2.$$
%This critically marked form are widely adopted by the followers.
%Note that $f_{c,a}$ has two critical points $\pm c$.
%To study the parameter space of $f_{c,a}$,
The set $\mathcal S_p$ with $p\geq 1$ is the collection of pairs $(c,a)\in \mathbb C^2$ for which the critical point $c$ has period $p$ under iterations of $f_{c,a}$:
$$\mathcal{S}_p=\big\{(c, a)\in \mathbb{C}^2{;~}
\text{$f_{c,a}^p(c)=c$ and $f^k_{c,a}(c)\neq c, ~\forall~ 1\leq k<p$}\big\}.$$

Applying Corollary \ref{application-dim1} to $\mathcal F=\mathcal{S}_p$, we see that for any
capture hyperbolic component $\mathcal H$ in $\mathcal S_p$,  the boundary $\partial \mathcal H$ is a Jordan curve with Hausdorff dimension
$${\Hdim}(\partial \mathcal H)=\max_{(c,a)\in \partial \mathcal H}{\Hdim}(\partial U_{c,a}) \in (1,2),$$
where $U_{c,a}$ is the Fatou component of $f_{c,a}$ containing the critical point $c$. 
% The curve $\mathcal{S}_1 = \{(c,c+2c^3){;~} c\in\mathbb{C}\}$ is canonically biholomorphic to the complex $c$-plane. 
See Figure \ref{figure S1}. 
It is known before that the boundaries of all bounded hyperbolic components in $\mathcal S_p$ are Jordan curves; see 
\cite{Roesch,Wang21}.

\begin{figure}[ht]
\centering
\includegraphics{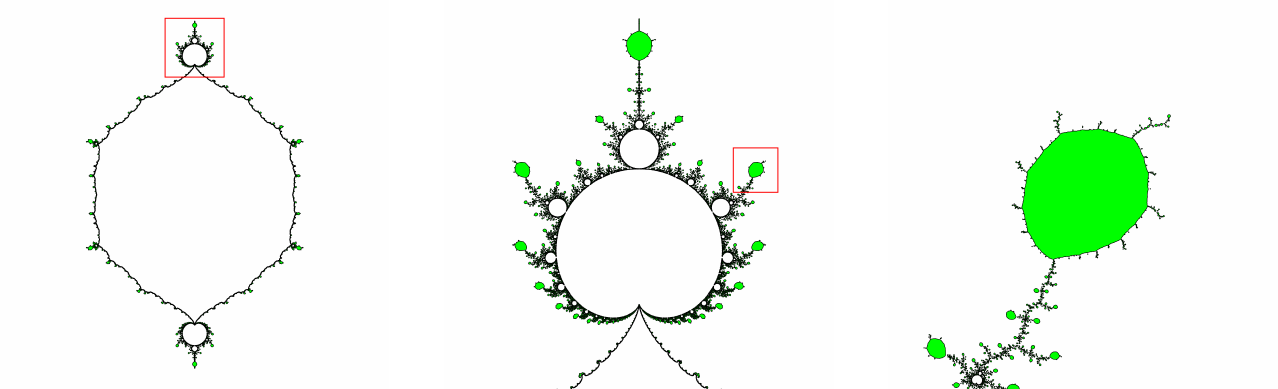}
\caption{The non-hyperbolic locus (colored black) and the capture hyperbolic components (colored green) 
in the curve $\mathcal{S}_1 = \{(c,c+2c^3){;~} c\in\mathbb{C}\}$, projected into the $c$-plane.}
\label{figure S1}
\end{figure}
\end{example}

\begin{example}
\label{example 2plus}
Consider $\mathcal F=\big\{f\in\widehat{\mathcal{P}}^4{;~} f(c_1(f)) = c_1(f)\big\}$. 
Let $f_0\in\mathcal{F}$ satisfy $f_0^2(c_2(f_0)) = f_0^2(c_3(f_0)) = c_1(f_0)$ as in Figure \ref{figure Peanuts1}. 
\footnote{
$(\mathbf{c}(f_0),f_0(0))
\approx (- 0.7232,~ 0.3616 + 0.7379i,~ 0.3616 - 0.7379i,~ 0.2565)$. 

$(\mathbf{c}(f),f(0))
\approx (-0.718700+0.004800i,~   0.362897+0.733624i,~   0.355803-0.738424i,~\\   0.240927-0.018759i)$. 
%(-0.7187+0.0048i,~ 0.3629+0.7337i,~ 0.3558-0.7385i,~ 0.2410-0.0187i).
}
Let $\mathcal{H}\subset\mathcal{F}$ be the hyperbolic component containing $f_0$. 
Then $\mathcal{H}$ is a capture hyperbolic component. 

\begin{figure}[ht]
\centering
\includegraphics{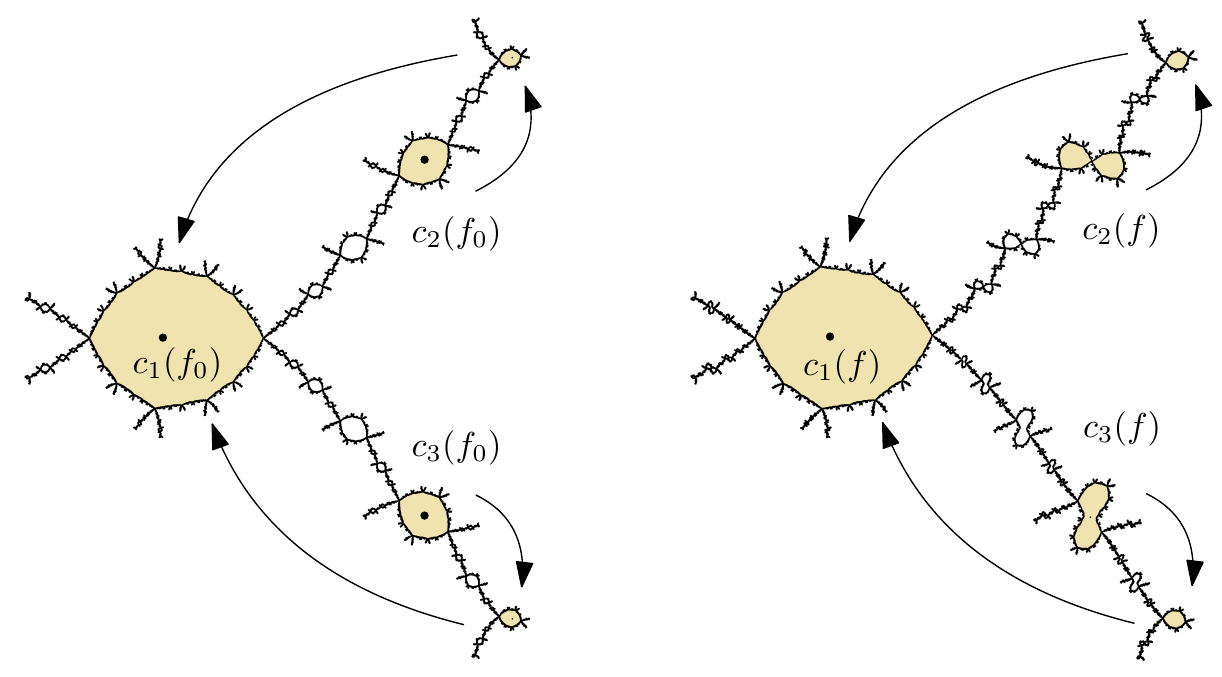}
\caption{Julia sets for $f_0\in\mathcal{H}$ and some  $f\in\partial \mathcal{H}$.
}
\label{figure Peanuts1}
\end{figure}

By Theorem \ref{top-boundary}, the boundary $\partial \mathcal H$ is homeomorphic to the sphere $S^3$. 
By Theorem \ref{hd-boundary}, we have 
$${\Hdim}(\partial \mathcal H)=2+\max_{f\in \partial \mathcal H}{\Hdim}(\partial U_f),$$
where $U_f$ is the Fatou component of $f$ containing the critical point $c_1(f)$. 
Of course these results hold for any capture hyperbolic component in $\mathcal F$. 
\end{example}

\subsection{Outline of the proof and novelty}
The paper consists of two parts.
From \S\ref{section external rays and allowable arcs} to \S\ref{section straightening map}, we prove Theorem \ref{top-boundary}, and from \S\ref{section transfer} to \S\ref{section the hd-formula}, we prove Theorem \ref{hd-boundary}.
%Some minor details are given in appendices.

The proof of the main theorems is very long and uses various ideas and methods from complex dynamics and complex analysis in one or several complex variables. In order to exhibit a panoramic view of the paper, in this part, we outline the proof and highlight several new results with independent interests.
% in complex analysis or dynamical systems.

To make the presentation clear and accessible, we will mention some necessary definitions herein.
% for a precise statement of the results.
More details of definitions and terminologies can be found in the corresponding sections (note that some notations might differ).

Our exploration of $\partial \mathcal H$ starts from the %following
 characterization of the maps on $\partial \mathcal H$: 

\begin{proposition}  
\label{boudary-sh}
All maps on $\partial \mathcal H$ %of the captured hyperbolic component $\mathcal H$
are semi-hyperbolic.
\end{proposition}

To prove Proposition \ref{boudary-sh}, the key point is to show that all periodic points on the boundaries of
the bounded attracting Fatou components of $f\in \partial\mathcal H$ are repelling.
The idea is as follows. We denote the unique postcritically finite map in $\mathcal H$ by $f_0$, 
which is called the \emph{center} of $\mathcal H$. 
Replacing $f_0$ by a suitable iterate $f_0^m$, we may assume that any bounded periodic Fatou component has period one. For any bounded invariant Fatou component $U$ of $f_0$, we consider the minimal sectors $S$ 
bounded by two adjacent invariant external rays landing at the same point on $\partial U$.
We introduce an $f_0$-invariant \emph{allowable set} $X(S)$
associated with $S$ in \S\ref{section external rays and allowable arcs}. The set $X(S)$ is in fact a part of the Hubbard tree, and has two different kinds of relative positions with respect to $S$:  $X(S)\subset \overline{S}$ or  $X(S)\not\subset \overline{S}$ (see Lemma \ref{key lemma}). These two cases indicate  different combinatorial properties of the map $f_0$.
Based on a further discussion of these cases, we will give  the proof of Proposition  \ref{boudary-sh}  in \S\ref{section behavior of crit-orbits}.

Proposition \ref{boudary-sh} paves the way to
%Based on the characterization of the maps on $\partial \mathcal H$, we march forward to
study the topology of $\partial \mathcal H$.
For this purpose, we need to construct a concrete model space for the abstract topological object $\overline{\mathcal{H}}$.
The model space is used to give a natural parameterization of $\overline{\mathcal{H}}$.

A suitable candidate of model space is the space of \emph{generalized polynomials over a mapping scheme}, introduced by Milnor \cite{MilHyperCompo}.
The definitions are given in \S\ref{section generalized polynomials}, and here is a brief introduction.
A \emph{mapping scheme} $T=(|T|,\sigma,\delta)$ associated with $\mathcal H$ is a triple of 
a finite set $|T|=\{f_0^{n}(c) {;~}  c\in\crit(f_0),n\geq0\}$ of critical orbits, 
a self-map $\sigma: |T|\rightarrow |T|$ defined by $\sigma(v)=f_0(v)$, 
and a degree map $\delta: |T|\rightarrow \mathbb N$ defined by $\delta(v)=\deg(f_0, v)$. 
The finite set $|T|$ admits a decomposition $|T|=\Tp \sqcup \Tn $, where
$\Tp $ consists of all $\sigma$-periodic points $v\in |T|$, and $\Tn $ is the residual part.

The mapping scheme $T$ induces the space $\widehat{\mathcal{P}}_0^T$ of generalized polynomials.
Each element $\mathbf{f}=(f_v:\mathbb C_v\rightarrow \mathbb C_{\sigma(v)})_{v\in |T|}\in\widehat{\mathcal{P}}_0^T$ 
is a finite collection of polynomial mappings between two copies of complex planes indexed by $|T|$, 
such that for each $v\in \Tp $, $f_v(z)=z^{\delta(v)}$ and for each $v\in \Tn $, 
$f_v$ is a monic centered critically marked polynomial of degree $\delta(v)$.

The fiberwise connectedness locus $\model$ serves as a model space for $\overline{\mathcal H}$. 
In \S\ref{section generalized polynomials}, we characterize the  boundary  topology of this model space:

\begin{proposition}
The boundary $\partial\model$ of the fiberwise connectedness locus of $\widehat{\mathcal{P}}_0^T$ 
is homeomorphic to the sphere $S^{2n-1}$, where $n=\sum_{v\in\Tn}(\delta(v)-1)=\Cdim(\mathcal H)$.
\end{proposition}

In order to get a natural correspondence between $\overline{\mathcal H}$ and $\model$, for each map
$f\in \overline{\mathcal H}$, we extract the mapping information of $f$ restricted to the critical 
and postcritical Fatou components. This gives finitely many pieces of restricted mappings.
However, the degenerate situation needs further considerations: when $f$ approaches the boundary 
$\partial \mathcal{H}$ from inside, a marked Fatou component might collapse at the boundary, 
therefore deform in a continuous fashion into finitely many Fatou components of the limit map.
To describe this situation, we introduce the notion of \emph{lemniscate maps} in 
\S\ref{section lemniscate maps}. We then prove a straightening theorem for these maps, 
by showing that each lemniscate map with a suitable normalization % (an external marking is given)
can be realized by a unique generalized polynomial mapping in $\model$. 
This gives the well-defined \emph{straightening map}
$$\chi:\overline{\mathcal{H}}\rightarrow\model.$$

% Compare the straightening theorem for generalized polynomial-like maps 
% given by Inou and Kiwi in \cite{Inou-Kiwi}.

%The straightening map is inspired by Inou-Kiwi's work \cite{Inou-Kiwi}. 
%The essential difference is, when there are critical points on the boundary of 
%bounded periodic Fatou components, Inou-Kiwi's straightening procedure does not 
%work. Our straightening procedure can deal with such situation with the help of 
%the straightening theorem for lemniscate maps.   

%The straightening map is inspired by Inou-Kiwi's work \cite{Inou-Kiwi}. 
%Inou-Kiwi's straightening theorem makes essential use of the quasiconformal 
%surgery technique, hence requires the renormalizable part of the polynomial 
%to be a repelling system. When there are critical points on the boundary of 
%a bounded periodic attracting Fatou component $U$, there is no polynomial-like 
%restriction with filled Julia set $\overline{U}$. Our straightening procedure is 
%topological and can deal with such situation with the help of the straightening 
%theorem for lemniscate maps. This makes the essential difference between 
%Inou-Kiwi's straightening theorem and ours.

The straightening map is inspired by Inou-Kiwi's work \cite{Inou-Kiwi}. Inou-Kiwi's 
straightening theorem makes essential use of the quasiconformal surgery technique, 
where the periodic renormalizable parts of the polynomial are exactly polynomial-like 
maps in the sense of Douady and Hubbard. For those $f\in\partial\mathcal{H}$ with 
critical points on the boundary of a bounded periodic attracting Fatou component $U$, 
there is no polynomial-like restriction of $f$ with filled Julia set $\overline{U}$, 
so Inou-Kiwi's straightening map can not be defined; see Figure \ref{figure basilica} 
for such an example. Our straightening procedure is topological and can deal with 
such situation with the help of the straightening theorem for lemniscate maps. 

The next task naturally becomes to show the continuity of $\chi$. To do this, we need to prove 
the convergence of external angles. This is a new ingredient in the proof, and we will state below.

Let $\mathcal{A}$ be the collection of all pairs $(K,\gamma)$ such that
\begin{itemize}
\item $K\subset\mathbb{C}$ is a full nondegenerate continuum;
\item $\gamma$ is a path in $\mathbb{C}\setminus K$ landing at $K$.
\end{itemize}

Let $(K,\gamma)\in\mathcal{A}$.
By the Riemann mapping theorem, there is a unique conformal isomorphism
$\phi_K:\mathbb{C}\setminus\overline{\mathbb{D}}\rightarrow\mathbb{C}\setminus K$
such that $\lim_{z\rightarrow\infty}\phi_K(z)/z>0$.
The \emph{external ray} with angle $\theta\in\RZ$ is defined as 
$R_K(\theta)=\phi_K((1,\infty)e^{2\pi i\theta})$.
An external ray $R_K(\theta)$ \emph{lands} at a point $x\in\partial K$ 
if $\overline{R_K(\theta)}\cap\partial K=\{x\}$. It is known from \cite[Corollary 6.4]{McM} that 
$\phi_K^{-1}(\gamma)$ lands at a point $e^{2\pi i t}\in \partial \mathbb D$. 
Moreover, the external ray $R_K(t)$ and $\gamma$ land at the same point in $\partial K$.
Since the angle $t$ is determined by $K$ and $\gamma$, it will be denoted by $\theta(K,\gamma)$.

\begin{proposition}
[Convergence of external angles, without dynamics] 
\label{cea} 
Let $(K, \gamma),$ $(K_n, \gamma_n)\in \mathcal A$ so that
$K_n\rightarrow K$ and $\gamma_n\rightarrow\gamma$, both in Hausdorff topology.
Then $\theta(K_n, \gamma_n) \rightarrow \theta(K, \gamma)$.
\end{proposition}

%It is worth noting  that Proposition \ref{cea} does not imply the convergence of the external rays %$R_{K_n}%(\theta(K_n, \gamma_n))$ to $R_{K}(\theta(K, \gamma))$  in Hausdorff topology as $n\rightarrow %\infty$ (in fact, parabolic implosion provides many counterexamples).
Proposition \ref{cea} has an independent interest in complex analysis. It is a little bit surprising
%what is more interesting is 
that its proof uses real analysis rather than complex methods.  See Proposition \ref{prop-convergence of external angles}.

By Proposition \ref{cea} and the preparations of 
\S\ref{section lemniscate maps},\S\ref{section straightening map}, 
we will prove that $\chi$ is continuous, open and injective, hence a homeomorphism; 
see \S\ref{section straightening map}.
This completes the proof of Theorem \ref{top-boundary}.
See Figure \ref{figure HomeoThm}. 

\begin{figure}[ht]
\centering
\includegraphics{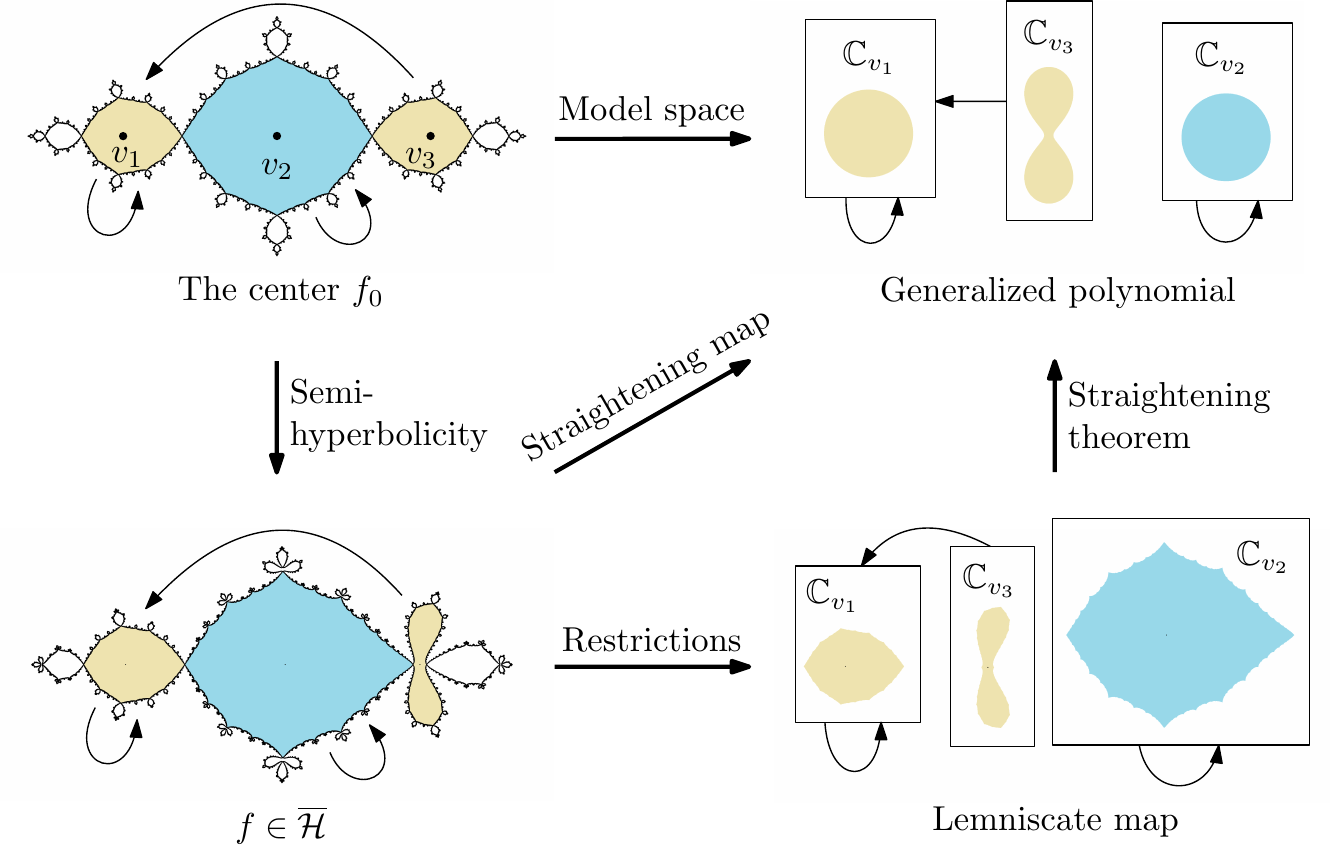}
\caption{Sketch illustrating the proof of Theorem \ref{top-boundary}.}
\label{figure HomeoThm}
\end{figure}

We are now halfway to our final goal. In the following, we will study the boundary $\partial \mathcal H$ 
from the analytic viewpoint, by proving the dimension formula: Theorem \ref{hd-boundary}.  This requires 
more efforts, and the proof mainly consists of three ingredients:
\begin{itemize}
\item the transfer equalities for Hausdorff dimensions;

\item a transversality result for most maps on $\partial \mathcal H$;

\item the continuity of Hausdorff dimensions for the boundaries of marked Fatou components.
\end{itemize}

In \S\ref{section transfer}, we establish the transfer equalities.
% allowing us to  transfer the Hausdorff dimension from the dynamical plane to the parameter space.

\begin{proposition}
[Transfer equalities]
\label{HD-trans-e}
Let $X$ be a subset of $\mathbb C$ with $x_0\in X$, and $\Omega$ be a domain in $\mathbb{C}^n$ with $z_0\in\Omega$.
Let $h: \Omega\times X\rightarrow \mathbb C$ be a holomorphic motion of $X$ with basepoint $z_0$ (i.e. $h(z_0,\cdot)=\id$),  and $v: \Omega\rightarrow\mathbb C$ be a holomorphic map such that $v(z_0)=x_0$.

Let $\mathcal X=\{z\in\Omega{;~} v(z)\in h(z, X)\}$. For $\delta,r>0$, let $\mathcal{X}_\delta=\mathcal{X}\cap\mathbb{D}^n(z_0,\delta)$  and $X_r=X\cap \mathbb{D}(x_0,r)$.
Assume the transversality $dv|_{z=z_0}\neq dh(z,x_0)|_{z=z_0}$ holds.
Then 
\begin{align*}
\lim_{\delta\rightarrow 0}{\Hdim}(\mathcal X_\delta) &= 2(n-1)+\lim_{r\rightarrow 0}{\Hdim}(X_r), \\
\lim_{\delta\rightarrow 0}{\Hdim}(\mathcal X_\delta^0) &= 2(n-1)+\lim_{r\rightarrow 0}{\Hdim}(X_r^0),
\end{align*}
%If $X$ is homogenuous, then
%$$\lim_{r\rightarrow 0^+}{\Hdim}(Q_v^r)= 2(n-1)+{\Hdim}(X),$$
where $\mathcal X_\delta^0$ is the component of $\mathcal X_\delta$ containing $z_0$, and
$X_r^0$ is the component of $X_r$ containing $x_0$.
\end{proposition}

In applications, the set $X_r$ plays a role of a dynamical set (e.g. a subset of the Julia set of 
some reference map), while $\mathcal{X}_\delta$ serves as a parameter set (e.g. a subset of the 
bifurcation locus). Proposition \ref{HD-trans-e} gives a way to transfer the Hausdorff dimension 
in the dynamical plane to that in the parameter space. These equalities have independent interests 
and are probably useful for other situations.

To make the transfer equalities work, we need to verify the transversality assumption in Proposition \ref{HD-trans-e}. 
We remark that the transversality  theory is an important topic with independent interests in complex dynamics, and it plays an important role in the study of the parameter space of rational maps. See \cite{Epstein, vStrien}.

In our situation,   unfortunately, this  transversality property may not hold for all maps on the boundary $\partial \mathcal H$. 
To guarantee the transversality, it is reasonable to consider the maps $f\in \partial \mathcal H$ for 
which a holomorphic motion naturally arises from dynamics, and $f$ is not contained in
the singular locus $\Sigma(\mathcal F)$ of the algebraic family $\mathcal F$.
Besides these requirements imposed on $f\in \partial \mathcal H$, we need to figure out the essential 
obstructions for transversality. We will search for clues in the model space 
$\widehat{\mathcal{P}}_0^{T}$ first.

In the model space $\widehat{\mathcal{P}}_0^{T}$, each element
 $\mathbf{f} = (f_v)_{v\in|T|}\in \widehat{\mathcal{P}}_0^{T}$ is uniquely
 determined by the factors $f_v$ with $v\in \Tn $ and $\delta(v)\geq 2$. %Let $v\in \Tn $.
 If $\delta(v)=2$, then $f_v(z)=z^2+a_v$ and it is determined by the constant term $a_v$; if $\delta(v)\geq 3$, then $f_v$ is determined  by the critical points $c_{v,1},\dots, c_{v, \delta(v)-2}$ and the constant term $a_v$.  Therefore $\mathbf{f}$ can be parameterized by the ordered list of parameters:
$$(a_v)_{v\in \Tn, \ \delta(v)= 2}\ \text{ and } \ (c_{v,1},\dots, c_{v, \delta(v)-2}, a_v)_{v\in \Tn, \ \delta(v)\geq 3}.$$

%$a_v$ with $\delta(v)= 2$, and $c_{v,1},\dots, c_{v, \delta(v)-2}, a_v$ with $\delta(v)\geq 3$,  $v\in \Tn $.

%For each $v\in \Tn $, the map $f_v=f_{\mathbf{c}_v,a_v}$  is a critically marked  monic and centered %polynomial of degree $\delta(v)$, hence  $f_v$ is either the identity map (if $\delta(v)=1$), or a quadratic %polynomial $z^2+a_v$ (if $\delta(v)=2$), or determined  by the critical points $c_{v,1},\dots, c_{v, %\delta(v)-2}$ and the constant term $a_v$ (if $\delta(v)\geq 3$).
%Hence $a_v$ and all possible marked critical points $c_{v,1}, \dots, c_{v,\delta(v)-2}$ with $v\in |T|$,
%give a natural parameterization of $\mathbf{f}$.

For each $v\in \Tn $, let $r_v\geq 1$ be the first integer so that $\sigma^{r_v}(v)\in \Tp $.
Our discussion of transversality is inspired by the following crucial observation:
\begin{proposition}
[Transversality in model space] 
\label{transver-ms} 
Let
\begin{align*}
W:  \begin{cases}
\widehat{\mathcal{P}}_0^{T}\rightarrow \mathbb{C}^{\Cdim (\mathcal H)}, \\
\mathbf{f} \mapsto   (f_{\sigma^{r_v-1}(v)}\circ\cdots\circ f_v(c_{v,k}))_{v\in \Tn,\  1\leq k<\delta(v)}
\end{cases}
\end{align*}
be the first entry map. Then its Jacobian determinant
is a constant multiple of
$$\prod_{v\in \Tn }\bigg[ \prod_{1\leq k_1< k_2<\delta(v)} {\left(c_{v,k_1}-c_{v,k_2}\right)}^2
\prod_{k, k', 1\leq j<r_v} (c_{\sigma^j(v),k'}-f_{\sigma^{j-1}(v)}\circ\cdots\circ f_v(c_{v,k}))\bigg].$$
This implies that $W$ is locally biholomorphic near $\mathbf{f}$ if and only if $\mathbf{f}$ has no critical relation of the form
$c_{v,k_1}=c_{v,k_2}$ or $c_{\sigma^j(v),k'}=f_{\sigma^{j-1}(v)}\circ\cdots\circ f_v(c_{v,k})$ for some $1\leq j<r_v$ and some $v\in \Tn $. 
\end{proposition}

Proposition \ref{transver-ms} is actually the transversality in the model space $\widehat{\mathcal{P}}_0^{T}$; see Proposition \ref{prop-JW}. 
It implies that the maps with critical relations are exactly those for which the transversality fails.
This sheds light on a natural formulation of a transversality result for the maps on $\partial \mathcal H$.

For each  $v\in \Tp $, note that $v$ is a superattracting point of $f_0$. 
This $v$ induces a continuous map $v: \overline{\mathcal H}\rightarrow \mathbb C$ 
so that $v(f)$ is a superattracting point for $f\in \overline{\mathcal H}$, with normalization $v(f_0)=v$; 
see \S\ref{section behavior of crit-orbits}.
% persist throughout $\overline{\mathcal H}$ and become to the superattracting point $v(f)$ of $f\in %\overline{\mathcal H}$.
For $f\in \overline{\mathcal H}$, let $U_{f,v}$ be the Fatou component containing $v(f)$, and $A(f)=\bigcup_{v\in \Tp }U_{f,v}$ be the union of all bounded immediate attracting basins.
% Let $\Sigma(\mathcal{F})$ be the singular locus of $\mathcal F$.
%To highlight the critical marking $\mathbf{c}$ of $f$, in the following, we will use the pair $(f,\mathbf{c})$ to %denote a critically marked polynomial $f$.
For $f\in\partial\mathcal H\setminus \Sigma(\mathcal{F})$ satisfying $\crit(f)\cap \partial A(f)=\emptyset$, 
there arises a holomorphic motion $h: \mathcal N\times \overline{A(f)}\rightarrow \mathbb C$
of $\overline{A(f)}$ compatible with dynamics.

Our transversality result is stated as follows:

\begin{proposition}
[Transversality]
\label{transversality-bd}
Let $f\in\partial\mathcal H\setminus\Sigma(\mathcal{F})$ and assume
$\crit(f)\cap \partial A(f)=\emptyset$.
% the points $f^j(c_{v,k})$ for $(v,k)\in I, 0\leq j<r_v$ are pairwise different.
Let $\mathcal N\subset\mathcal{F}$ be a neighborhood of $f$ so that the holomorphic motion $h: \mathcal N\times \overline{A(f)}\rightarrow \mathbb C$ is well-defined.

If $f$ has no free critical relation, then the map
$G = (G_{v,k})_{(v,k)\in I}:\mathcal N\rightarrow \mathbb{C}^{\Cdim(\mathcal{H})}$ defined by
\begin{align*}
G_{v,k}(g) = g^{r_v}(c_{v,k}(g))-h(g,f^{r_v}(c_{v,k}(f)))
\end{align*}
is locally biholomorphic near $f$, where $I=\{(v,k){;~} v\in \Tn , 1\leq k<\delta(v)\}$.
\end{proposition}

The proof of Proposition \ref{transversality-bd} is given in \S\ref{transver}. 
The strategy of the proof uses an idea of van Strien \cite{vStrien}, based on quasiconformal 
distortion estimates. %The difference is, 
There is a new difficulty arising here: we need to show that the maps with 
prescribed critical orbit behavior are indeed on the boundary $\partial \mathcal H$. This will 
be overcome by Proposition \ref{capture-separation}.

To get the dimension formula, we need to march one step further: prove the continuity of the 
Hausdorff dimension for Fatou component boundaries. This is the main theme of \S\ref{continuity-hd}.

\begin{proposition}
\label{contin-hdim}
For any $v\in \Tp $, the Hausdorff dimension $\Hdim(\partial U_{f,v})$ is continuous with respect to $f\in\overline{\mathcal H}$.
\end{proposition}

Translating this continuity property to the parameter space, this implies that the boundary 
$\partial \mathcal H$ has no singular part.
To prove the continuity, in \S\ref{continuity-hd} we first establish an identity for various 
quantities --- conformal dimension, hyperbolic dimension and Poincar\'e exponent --- for a semi-hyperbolic 
polynomial acting on the boundary of a bounded periodic Fatou component.
Then by various distortion estimates, and also with the help of log-B\"ottcher techniques, 
we prove a uniform expansion property along external rays under perturbation:

\begin{proposition}
\label{expansion-e-r}
Let $d\geq2$ be an integer, and let $\theta\in\QZ$ have preperiod $q\geq1$ and period $p\geq 1$ 
under $m_d: \RZ\rightarrow \RZ$, $t\mapsto dt$. That is, $d^{p+q} \theta\equiv d^q\theta \modZ$.
For $f\in\mathcal{C}(\mathcal{P}^d)$,
let $$g_f=f^{-q}\circ f^p\circ f^q: \overline{R_{f}(\theta)}\rightarrow \overline{R_{f}(\theta)}$$
be the lift of $f^p:\overline{R_{f}(d^q\theta)}\rightarrow \overline{R_{f}(d^q\theta)}$.
Given $h\in\mathcal{C}(\mathcal{P}^d)$ and
suppose $R_h(\theta)$ lands at a pre-repelling point.
Then there exists a neighborhood $\mathcal{N}\subset \mathcal{P}^d$ of $h$ and a constant $n>0$ such that
$$\big|(g_f^n)'(z)\big|\geq 2$$
for every $f\in\mathcal{C}(\mathcal{P}^d)\cap\mathcal{N}$ and every $z\in\overline{R_{f}(\theta)}$.
\end{proposition}

Proposition \ref{expansion-e-r} is new and has an independent interest in polynomial dynamics.
It is the essence in the proof of the continuity of Hausdorff dimension; see \S\ref{distortion-thm} and
\S\ref{continuity-h-d}. Our method is different from that in the previous proofs  \cite{McM2000, Juan, GS09} of
the continuity of the Hausdorff dimension of Julia sets, for which the main ideas are to  control the critical orbits behavior in a uniform way.

We will prove the dimension formula in \S\ref{section the hd-formula}. The proof relies on the key 
observation that the following subset of $\partial \mathcal H$:
$$\big\{f\in \partial \mathcal H{;~} \text{$f\in\Sigma(\mathcal{F})$, 
or $\crit(f)\cap \partial A(f)\neq\emptyset$, or $f$ has a free critical relation}\big\}$$
for which the transversality may fail,
takes a smaller part in $\partial \mathcal H$, in the sense of Hausdorff dimension.
Then the simplicial structure of $\partial \mathcal H$ allows us to give a decomposition of it into 
finitely many open subsets together with a residual part, based on the dynamical behavior of critical 
orbits. By Propositions \ref{HD-trans-e}, \ref{transversality-bd} and \ref{contin-hdim},
we will complete the proof of Theorem \ref{hd-boundary} in \S\ref{proof-hdf}.

In appendices, we provide some supplements for the paper, including a discussion of the dimension 
of the singular locus in \S\ref{appendix smooth points}, a brief introduction of weak convergence 
on conformal measures in \S\ref{appendix weak topology on measures}, and an equivalent condition 
for uniform convergence of univalent maps in \S\ref{appendix uniform-convergence} (compare also 
\cite[Corollary 2.4]{Pom}). In \S\ref{appendix u-qc}, we prove a uniform property for the maps in 
$\overline{\mathcal{H}}$: \emph{the boundaries of all bounded periodic Fatou components 
for all maps in $\overline{\mathcal{H}}$ are uniform quasicircles.} This is a new result. Since it 
is a dynamical result, we prefer to put it in the appendix. The detailed proof is given therein for 
the interested readers.

\section{Preliminaries}
% \section{External rays and allowable arcs}
\label{section external rays and allowable arcs}
We introduce external rays and allowable arcs in this section, as a preparation for further discussions. These materials are standard in polynomial dynamics. 

For a polynomial $f:\mathbb{C}\rightarrow\mathbb{C}$ with $\deg(f)\geq2$, we denote the Julia set 
by $J(f)$, the filled Julia set by $K(f)$ and the set of critical points by $\crit(f)\subset\mathbb{C}$. 
If $z\in K(f)\setminus J(f)$, let $U_f(z)$ denote the (bounded) Fatou component of $f$ containing $z$. 
Let $U_{f,\infty}=\mathbb{C}\setminus K(f)$ be the \emph{basin of infinity} for $f$.

\subsection{External rays}
\label{subsection external rays}
% external rays
Let $\mathcal P^d$ be the space of monic and centered polynomials of degree $d\geq 2$. 
Let $f\in\Pd$. 
There exists a \emph{B\"ottcher coordinate} $B_f$, which is tangent to the identity at $\infty$, 
and satisfies $B_f(f(z))=B_f(z)^d$ whenever $|z|$ is large enough.
If $K(f)$ is connected, $B_f$ can be extended to a conformal map from $U_{f,\infty}$ 
to $\mathbb{C}\setminus \overline{\mathbb{D}}$.
The \emph{Green function} $G_f:\mathbb{C}\rightarrow[0,\infty)$ of $f$ is defined by 
$$G_f(z)=\lim_{n\rightarrow\infty}\frac{1}{d^n}\log^+|f^{n}(z)|,$$
where $\log^+(x)=\log(\max\{1,x\})$ for $x\in\mathbb{R}$.
Then $G_f(f(z))=d\cdot G_f(z)$.
Whenever $|z|$ is large enough, we have $G_f(z)=\log|B_f(z)|$.
Each locus $G_f^{-1}(s)$ with $s>0$ is called an \emph{equipotential curve}.
For an \emph{angle} $\theta\in\mathbb{R/Z}$, the \emph{external ray} $R_f(\theta)\subset U_{f,\infty}$ 
contains $B_f^{-1}((r,\infty)e^{2\pi i\theta})$ for some $r>1$,
and extends along the decreasing gradient line of $G_f$ (possibly bifurcates).
If $\overline{R_f(\theta)}\cap K(f)=\{x\}$,
we say that $R_f(\theta)$ \emph{lands} at $x\in K(f)$, and $\theta$ is an \emph{external angle} of $x$. 

Let $\mathcal{C}(\Pd)$ denote the \emph{connectedness locus} of $\Pd$. 
It is known by \cite[Theorem 18.10]{Mil} that 
if $f\in\mathcal{C}(\Pd)$ and $\theta\in\QZ$, then $R_f(\theta)$ lands at a preperiodic point. 
For more details about external rays, we refer to \cite{Mil}.

% stability of external rays
Let $f\in\Pd$ and $\theta\in\QZ$.
Suppose $R_f(\theta)$ lands. 
Then we can extend $B_{f}$ holomorphically along the ray $R_f(\theta)$. 
Let $\gamma_\theta(f,s) = B_f^{-1}(e^{s+2\pi i\theta})$,  $s\in[0,\infty)$ be a parameterization of $\overline{R_f(\theta)}$, where $B_f^{-1}(e^{2\pi i\theta})$ takes the landing point of $R_f(\theta)$. 
The following lemmas describe the stability of external rays. 

\begin{lemma}
[Stability of external rays {\cite[Proposition 8.5]{Orsay}}]
\label{stability of external rays 0}
Let $f\in\Pd$ and $\theta\in\QZ$.
Suppose $R_f(\theta)$ lands at a pre-repelling point $x\in J(f)$. 
Moreover, suppose $\{f^n(x){;~}n\in\mathbb{N}\}\cap\crit(f)=\emptyset$. 
Then there is a neighborhood $\mathcal{N}\subset\Pd$ of $f$ 
with the following properties. 
\begin{enumerate}
\item For any $g\in\mathcal{N}$, the ray $R_{g}(\theta)$ lands at a pre-repelling point. 
\item The function $\gamma_\theta:\mathcal{N}\times[0,\infty)\rightarrow\mathbb{C}$ is continuous. 
\item For any $s\in[0,\infty)$, the function $\gamma_\theta(\cdot,s):\mathcal{N}\rightarrow\mathbb{C}$ is holomorphic. 
\end{enumerate}
\end{lemma}

\begin{lemma}
[Stability of external rays]
\label{stability of external rays}
Let $f\in\mathcal{C}(\Pd)$ and $\theta\in\QZ$.
Suppose $R_f(\theta)$ lands at a pre-repelling point. 
Then there is a neighborhood $\mathcal{N}\subset\Pd$ of $f$ such that $\gamma_\theta:(\mathcal{N}\cap\mathcal{C}(\Pd))\times[0,\infty)\rightarrow\mathbb{C}$ is continuous. 
\end{lemma}

\begin{proof}
If the landing point of $R_f(\theta)$ is periodic, the result follows from Lemma \ref{stability of external rays 0}. 
Otherwise, choose $m\geq 1$ such that $R_f(d^m\theta)$ lands at a periodic point. 
Then by Lemma \ref{stability of external rays 0}, there is a neighborhood $\mathcal{N}\subset\Pd$ of $f$ such that $\gamma_{d^m\theta}:(\mathcal{N}\cap\mathcal{C}(\Pd))\times[0,\infty)\rightarrow\mathbb{C}$ is continuous. 
It follows by lifting that there is a continuous map $\gamma_{\theta}:(\mathcal{N}\cap\mathcal{C}(\Pd))\times[0,\infty)\rightarrow\mathbb{C}$ satisfying  $g^m(\gamma_\theta(g,s))=\gamma_{d^m\theta}(g,d^m s)$. 
\end{proof}

\begin{lemma}
\label{limit-land-rela}
Let $f\in\mathcal{C}(\Pd)$ and $\theta,\theta'\in\QZ$ with $\theta\neq\theta'$. 
Suppose $R_f(\theta), R_f(\theta')$ land at pre-repelling points $x,x'$ respectively. 

\begin{enumerate}
\item \label{limit-land-rela-1}
Let $\{f_n\}_{n\geq1}\subset \mathcal{C}(\Pd)$ satisfy $\lim_{n\rightarrow\infty} f_n = f$. 
If $R_{f_n}(\theta), R_{f_n}(\theta')$ land at the same point $x_n$ for any $n\geq1$, then $x=x'$. 

\item \label{limit-land-rela-2}
If $x=x'$ and $\{f^n(x){;~}n\in\mathbb{N}\}\cap\crit(f)=\emptyset$, 
then there is a neighborhood $\mathcal{N}\subset\Pd$ of $f$ such that $R_g(\theta)$ and $R_g(\theta')$ land at a common pre-repelling point of $g$ for any $g\in\mathcal{N}$. 
\end{enumerate}
\end{lemma}
% \ref{limit-land-rela}(\ref{limit-land-rela-1})
% \ref{limit-land-rela}(\ref{limit-land-rela-2})

\begin{proof}
By Lemma \ref{stability of external rays}, we have $x = \lim_{n\rightarrow\infty} x_n =x'$, so the first reuslt holds. 
The second result follows from Lemma \ref{stability of external rays 0} or Lemma \ref{stability of external rays}. 
\end{proof}

\begin{remark}
With the same notations and assumptions as Lemma \ref{limit-land-rela}, 
if $x=x'$ but $\{f^n(x){;~} n\in\mathbb{N}\}\cap\crit(f)\neq\emptyset$, 
there may be $g\in\mathcal{C}(\Pd)$ arbitrarily close to $f$ such that $R_g(\theta)$ and $R_g(\theta')$ land at two different pre-repelling points. 
For example, consider $\theta = 1/12$ and $\theta' = 1/6$ in Figure \ref{figure basilica}.
\end{remark}

\subsection{Sectors and limbs}
% sector
Angles $\theta_1,\dots,\theta_n\in\RZ$ are in \emph{positive cyclic order} 
if it is possible to choose representatives $\hat\theta_k\in\mathbb{R}$ 
so that $\hat\theta_1<\cdots<\hat\theta_n<\hat\theta_1+1$.
Let $\theta_1,\theta_2\in\RZ$ with $\theta_1\neq\theta_2$. 
The \emph{open interval} $(\theta_1,\theta_2)$ 
will mean the set of all angles $\theta\in\RZ$ for which $\theta_1,\theta,\theta_2$ are 
in positive cyclic order. 
The \emph{closed interval} is defined to be $[\theta_1,\theta_2]=\overline{(\theta_1,\theta_2)}$. 
Let $\operatorname{frac}:\RZ\rightarrow[0,1)$ be the \emph{fractional part function} that 
maps each point of $\RZ$ to its unique representative in $[0,1)$. 

\begin{definition}
[Sector]
Let $f\in\Pd$.
Suppose $R_f(\theta_1)$ and $R_f(\theta_2)$ are different external rays landing at the same point $x$.
The \emph{sector} ${{S}}_f(\theta_1,\theta_2)$ is defined to be the connected component of 
$\mathbb{C}\setminus (R_f(\theta_1)\cup R_f(\theta_2)\cup\{x\})$
containing the external rays with angles in the interval $(\theta_1,\theta_2)$. 
\end{definition} 

We call $x$ the \emph{root point} (or \emph{vertex}) of $S:={{S}}_f(\theta_1,\theta_2)$, 
and denote it by $r(S)$. 

We say $S$ is \emph{minimal} (resp. \emph{maximal}), 
if $S$ (resp. $S_f(\theta_2,\theta_1)$) includes no external rays landing at its root point. 

We say $S$ \emph{attaches to} the set $A\subset\mathbb{C}$, 
if $\overline{S}\cap\overline{A}=\{{{{r}}}(S)\}$.

\begin{lemma}
\label{sector with no critical point}
Let $f\in\Pd$. 
If $\overline{{{S}}_f(\theta_1,\theta_2)}\cap\crit(f)=\emptyset$, then 
$f:{{S}}_f(\theta_1,\theta_2)\rightarrow{{S}}_f(d\theta_1,d\theta_2)$ is conformal.
\end{lemma}

\begin{proof}
Since $r(S_f(\theta_1,\theta_2))$ is not a critical point, the polynomial $f$ is a local homeomorphism near it, hence $d\theta_1\not\equiv d\theta_2\modZ$. 
Furthermore, the restriction $f:\partial S_f(\theta_1,\theta_2)\rightarrow \partial S_f(d\theta_1,d\theta_2)$ is a homeomorphism. 
Let $D$ be the connected component of $f^{-1}(S_f(d\theta_1,d\theta_2))$ such that $\partial S_f(\theta_1,\theta_2)\subset\partial D$. 
Then $D\subset S_f(\theta_1,\theta_2)$. 
Since $D$ contains no critical points, the restriction $f:D\rightarrow S_f(d\theta_1,d\theta_2)$ is conformal,  
so $D$ has only one boundary component. 
Thus $D=S_f(\theta_1,\theta_2)$, and $f:{{S}}_f(\theta_1,\theta_2)\rightarrow{{S}}_f(d\theta_1,d\theta_2)$ is conformal. 
\end{proof}

\begin{lemma}
\label{sectors attaching to U}
Let $f\in\mathcal{C}(\Pd)$ be a hyperbolic polynomial, and let $U$ be a bounded Fatou component of $f$.
Then the root point of any sector attaching to $U$ is preperiodic.
\end{lemma}

\begin{proof}
Without loss of generality, we assume $f(U)=U$. 
Let $S_f(\theta_1,\theta_2)$ be a sector attaching to $U$.
Assume $r(S_f(\theta_1,\theta_2))$ is not preperiodic; we will find a contradiction.
For any $n\in\mathbb{N}$, the sector $S_f(d^n\theta_1,d^n\theta_2)$ attaches to $U$. 
Since they are pairwise disjoint, there exists an $N>0$ such that  $\overline{S_f(d^n\theta_1,d^n\theta_2)}\cap \crit(f)=\emptyset$ whenever $n>N$.
By Lemma \ref{sector with no critical point},
the restriction $f:S_f(d^n\theta_1,d^n\theta_2)\rightarrow S_f(d^{n+1}\theta_1,d^{n+1}\theta_2)$ is conformal when $n>N$. 
It follows that $\operatorname{frac}(d^{n+1}\theta_2-d^{n+1}\theta_1) = 
d \operatorname{frac}(d^n\theta_2-d^n\theta_1) $ for any $n>N$. 
When $n$ tends to infinity, $\operatorname{frac}(d^n\theta_2-d^n\theta_1)$ will exceed $1$.
This is a contradiction. 
So $r(S_f(\theta_1,\theta_2))$ is preperiodic. 
\end{proof}

For an integer $\delta>0$, define 
$$m_\delta:\begin{cases}
\RZ\rightarrow\RZ,\\
\theta\mapsto \delta\theta. 
\end{cases}$$

Let $f\in\Pd$ and let $U$ be a bounded periodic Fatou component which is not a Siegel disk.
Then $U$ is a Jordan domain, and $f^{p}|_{\partial U}$ is topologically conjugate to $m_\delta$, 
where $p$ is the period of $U$ and $\delta = \deg(f^{p}|_U)$ (see \cite[Theorem A]{RY2022}).
There are $\delta-1$ choices of such topological conjugacy. 

\begin{proposition}
[Limb {\cite[Theorem B]{RY2022}}]
\label{limb}
Let $f\in\mathcal{C}(\Pd)$ and let $U$ be a bounded periodic Fatou component which is not a Siegel disk. 
Then $K(f)$ has a partition 
$$%
%\label{partition of Kf}
K(f)=U\bigsqcup\bigsqcup_{x\in\partial U} L_{U,x}
$$
such that 
\begin{enumerate}
\item $L_{U,x}\cap\overline{U}=\{x\}$; 
\item when $L_{U,x}\neq\{x\}$, there are two different external angles $\theta_U^-(x),\theta_U^+(x)$ of $x$
satisfying $L_{U,x}=\overline{{{S}}_f(\theta_U^-(x),\theta_U^+(x))}\cap K(f)$; and 
\item when $L_{U,x}=\{x\}$, there is only one external angle $\theta_U^-(x)=\theta_U^+(x)$ of $x$.
\end{enumerate}
\end{proposition}

We call $L_{U,x}$ the \emph{limb} for $U$ with root $x$.

When $L_{U,x}\neq\{x\}$, the sector ${{S}}_f(\theta_U^-(x),\theta_U^+(x))$ is a maximal sector attaching to $U$. 
By the partition of $K(f)$, we have a partition of the plane:
\begin{equation}
\label{partition of the plane}
\mathbb{C}=U\bigsqcup\bigsqcup_{L_{U,x}\neq\{x\}} \overline{{{S}}_f(\theta_U^-(x),\theta_U^+(x))}
\bigsqcup\bigsqcup_{L_{U,x}=\{x\}} \overline{R_f(\theta_U^-(x))}.
\end{equation}
Taking counterclockwise as positive orientation on $\partial U$, we have: 
\begin{equation}
\label{limb1}
\begin{split}
\lim_{x\rightarrow x_0^+} \theta_U^+(x)&=\lim_{x\rightarrow x_0^+} \theta_U^-(x) =\theta_U^+(x_0);\\
\lim_{x\rightarrow x_0^-} \theta_U^+(x)&=\lim_{x\rightarrow x_0^-} \theta_U^-(x) =\theta_U^-(x_0).
\end{split}
\end{equation}

\subsection{Rational lamination and real lamination}
Let $f\in\mathcal{C}(\Pd)$. 
The \emph{rational lamination} $\lambda_\mathbb{Q}(f)\subset(\QZ)^2$ of $f$ is defined to be the landing relation of rational angles. 
That is, for $(\theta_1,\theta_2)\in(\QZ)^2$, we have  $(\theta_1,\theta_2)\in\lambda_\mathbb{Q}(f)$ 
if and only if $R_f(\theta_1)$ and $R_f(\theta_2)$ land at the same point. 

\begin{lemma}
[{\cite[\S3.2]{Kiwi01}}]
\label{rat-lam is closed}
Let $f\in\mathcal{C}(\Pd)$. Then $\lambda_\mathbb{Q}(f)$ is closed in $(\QZ)^2$. 
\end{lemma}

Let $f\in\Pd$. 
Assume its Julia set $J(f)$ is connected and locally connected. 
By Carath\'eodory's theorem, the inverse $B_f^{-1}$ of B\"ottcher coordinate can be extended continuously to $\mathbb{C}\setminus\mathbb{D}$. 
Then for any $\theta\in\RZ$, the external ray $R_f(\theta)$ lands at $B_f^{-1}(e^{2\pi i\theta})$. 
The \emph{real lamination} $\lambda_\mathbb{R}(f)\subset(\RZ)^2$ of $f$ is defined to be the landing relation of external angles. 
That is, for $(\theta_1,\theta_2)\in(\RZ)^2$, we have  $(\theta_1,\theta_2)\in\lambda_\mathbb{R}(f)$ 
if and only if $B_f^{-1}(e^{2\pi i\theta_1})=B_f^{-1}(e^{2\pi i\theta_2})$. 
In particular, we have $\lambda_\mathbb{Q}(f)=(\QZ)^2\cap\lambda_\mathbb{R}(f)$. 
By the continuous extension of $B_f^{-1}$ on $\mathbb{C}\setminus\mathbb{D}$, the real lamination $\lambda_\mathbb{R}(f)$ is closed in $(\RZ)^2$. 
For more details of laminations, see, e.g. \cite{DCompa}, \cite{Kiwi01}. 

Let $\lambda$ be a subset of $(\RZ)^2$ (resp. $(\QZ)^2$). 
The smallest closed equivalence relation in $\RZ$ (resp. $\QZ$) 
that contains $\lambda$ is denoted by $\langle\lambda\rangle_\mathbb{R}$ (resp. $\langle\lambda\rangle_\mathbb{Q}$). 

\begin{lemma}
[{\cite[Lemma 4.17]{Kiwi01}}]
\label{connection between rational lamination and real lamination}
Let $f\in\Pd$. 
Assume $J(f)$ is connected and locally connected. 
Then $\langle\lambda_\mathbb{Q}(f)\rangle_\mathbb{R}=\lambda_\mathbb{R}(f)$ if and only if
every critical point on the boundary of any bounded Fatou component is preperiodic. 
\end{lemma}

The inclusion relation of real laminations will induce the following property. 

\begin{lemma}
\label{continuous extension on the closure of infinity basin}
Let $f_1,f_2\in\Pd$. 
Assume their Julia sets are connected and locally connected. 
If $\lambda_\mathbb{R}(f_1)\subset\lambda_\mathbb{R}(f_2)$, 
then $B_{f_2}^{-1}\circ B_{f_1}:U_{f_1,\infty}\rightarrow U_{f_2,\infty}$ 
extends continuously to a map $\varphi:\overline{U_{f_1,\infty}}\rightarrow\overline{U_{f_2,\infty}}$, and $\varphi\circ f_1(z)=f_2\circ\varphi(z)$ for $z\in\overline{U_{f_1,\infty}}$. 
\end{lemma}

\begin{proof}
For $k=1,2$, the conformal map $B_{f_k}^{-1}:\mathbb{C}\setminus\overline{\mathbb{D}}\rightarrow U_{f_k,\infty}$
extends continuously to $\phi_k:\mathbb{C}\setminus\mathbb{D}\rightarrow\overline{U_{f_k,\infty}}$ 
by the local connectivity of $J(f_k)$. 
Since $\lambda_\mathbb{R}(f_1)\subset\lambda_\mathbb{R}(f_2)$, 
we see that $\phi_2(\phi_1^{-1}(z))$ is a singleton for any $z\in\overline{U_{f_1,\infty}}$.
We define $\varphi(z)$ to be this point. 

Let $\{z_n\}_{n\in\mathbb{N}}\subset\overline{U_{f_1,\infty}}$ and $z\in J(f_1)$ satisfy 
$\lim_{n\rightarrow\infty} z_n = z$.  
For each $n\in\mathbb{N}$, choose $w_n\in\phi_1^{-1}(z_n)$. 
Given a subsequence $\{z_{n_k}\}_{k\in\mathbb{N}}$ of $\{z_n\}_{n\in\mathbb{N}}$. 
Then there is a convergent subsequence $\{w_{n_{k_j}}\}_{j\in\mathbb{N}}$ of $\{w_{n_k}\}_{k\in\mathbb{N}}$, whose limit is denoted by $w$. 
By the continuity of $\phi_1$, we have $w\in\phi_1^{-1}(z)$. 
% we have $\dist(\phi_1^{-1}(z_n),\phi_1^{-1}(z))\rightarrow0$ as $n\rightarrow\infty$. 
Then the continuity of $\phi_2$ gives 
$\varphi(z_{n_{k_j}}) - \varphi(z) = \phi_2(w_{n_{k_j}}) - \phi_2(w)\rightarrow0$ as $j\rightarrow\infty$. 
Since $\{z_{n_k}\}_{k\in\mathbb{N}}$ was an arbitrary  subsequence of $\{z_n\}_{n\in\mathbb{N}}$, we have $\varphi(z_{n})-\varphi(z)\rightarrow0$ as $n\rightarrow\infty$. 
This shows $\varphi$ is continuous on $\overline{U_{f_1,\infty}}$.
It follows that $\varphi\circ f_1(z)=f_2\circ\varphi(z)$ for $z\in\overline{U_{f_1,\infty}}$. 
%Let $A$ be a closed subset of $\overline{U_{f_2,\infty}}$. 
%Since $\phi_2$ is continuous, the preimage $\phi_2^{-1}(A)$ is a closed subset of $\mathbb{C}\setminus\mathbb{D}$. 
%Since $\phi_1$ is continuous and $B_{f_1}$ is a homeomorphism, the preimage $\varphi^{-1}(A) = \phi_1(\phi_2^{-1}(A)) = \phi_1(\phi_2^{-1}(A)\cap (2\overline{\mathbb{D}}))\cup B^{-1}_{f_1}(\phi_2^{-1}(A)\setminus (2 \mathbb{D}))$ is closed in $\mathbb{C}$. 
%This shows $\varphi$ is continuous on $\overline{U_{f_1,\infty}}$.
%It follows that $\varphi\circ f_1(z)=f_2\circ\varphi(z)$ for $z\in\overline{U_{f_1,\infty}}$. 
\end{proof}

\subsection{Allowable arcs}
\label{subsection allowable arcs}

Let $f_0$ be a postcritically finite polynomial in $\Pd$. 
Following Douady and Hubbard \cite{Orsay}, we will define allowable arcs in $K(f_0)$. 

A compact set $K\subset\mathbb{C}$ is \emph{full} if $\mathbb{C}\setminus K$ is connected. 
A \emph{continuum} is a nonempty compact connected metric space. 
A continuum that contains more than one point is called \emph{nondegenerate}. 
The following is a basic property of local connectivity. 

\begin{lemma}
[{\cite[Proposition 2.3]{Orsay}}]
\label{local connectivity}
Let $K\subset\mathbb{C}$ be a full nondegenerate continuum. 
Denote by $\{U_j {;~}  j\in I\}$ the family of connected components of $K^\circ$. 
If $K$ is locally connected, then 
\begin{enumerate}
\item \label{local connectivity-1}
for all $j$, the component $U_j$ is a Jordan domain; 

\item \label{local connectivity-2}
for any $\varepsilon>0$, the set of $j$ such that ${\diam}(U_j)>\varepsilon$ is finite. 
\end{enumerate}
\end{lemma}

Now let $\{U_j {;~}  j\in I\}$ denote the family of bounded Fatou components of $f_0$. 
Since $f_0$ is postcritically finite, by Lemma \ref{local connectivity}, each $U_j$ is a Jordan domain. 
By Sullivan's no wandering domains theorem, 
each $U_j$ contains a unique preperiodic point, which will be denoted by $w_j$. 
This determines, up to multiplication by $\lambda\in\partial\mathbb{D}$, a Riemann mapping 
$\varphi_j:U_j\rightarrow\mathbb{D}$ such that $\varphi_j(w_j)=0$. 
By Carath\'eodory's theorem, the map $\varphi_j$ can be extended to a homeomorphism 
between $\overline{U_j}$ and $\overline{\mathbb{D}}$.

\begin{definition}
[Allowable arc]
An arc $\Gamma\subset K(f_0)$ will be called \emph{allowable} if $\varphi_j(\Gamma\cap\overline{U_j})$ is contained in the union of two radial rays of $\overline{\mathbb{D}}$ for every $j\in I$. 
\end{definition}

By \cite[Proposition 2.6]{Orsay}, 
two distinct points $x,y$ in $K(f_0)$ can be joined by a unique allowable arc,
which is denoted by $[x,y]=[x,y]_{K(f_0)}$. 
Let $(x,y]=[x,y]\setminus \{x\}, [x,y)=[x,y]\setminus \{y\}$ be the \emph{allowable half-open arcs},
and $(x,y)=[x,y]\setminus \{x,y\}$ be the \emph{allowable open arc}. 
It is clear that every subarc of an allowable arc is allowable.

\begin{definition}
\label{def-path-K}
Let $K\subset\mathbb{C}$ be a full nondegenerate continuum. 
A \emph{path to $K$} is a continuous map $\gamma:[0,1]\rightarrow\mathbb{C}$ 
such that $\gamma((0,1])\subset\mathbb{C}\setminus K$ and $\gamma(0)\in\partial K$. 
We also say $\gamma$ is a \emph{path to $\gamma(0)$ relative to $K$}. 

We say two paths $\gamma_0,\gamma_1$ to $K$ are \emph{homotopic} ($\gamma_0\simeq_K\gamma_1$) if
there exists a continuous map $h:[0,1]\times[0,1]\rightarrow\mathbb{C}$ so that
\begin{itemize}
\item $t\mapsto h(s,t)$ is a path to $K$ for all $s\in[0,1]$;
\item $h(s,0)=\gamma_0(0)$ for all $s\in[0,1]$;
\item $h(0,t)=\gamma_0(t)$ and $h(1,t)=\gamma_1(t)$ for all $t\in[0,1]$.
\end{itemize}
\end{definition}

The following lemma describes the restriction of $f_0$ to an allowable arc, 
and provides a way to find external rays landing at a common point. 

\begin{lemma}
[{\cite[Lemma 4.2, Proposition 7.1]{Orsay}}]
\label{allowable arcs in K(f0)}
Let $f_0\in\Pd$ be a postcritically finite polynomial, and $[x,y]$ be an allowable arc in $K(f_0)$. 
\begin{enumerate}
\item \label{allowable arcs in K(f0)-1}
If $(x,y)\cap\crit(f_0)=\emptyset$, 
then $f_0|_{[x,y]}$ is injective and $f_0([x,y])=[f_0(x),f_0(y)]$ is an allowable arc. 

\item \label{allowable arcs in K(f0)-2}
For $z\in [x,y]\cap J(f_0)$, 
every homotopy class of paths to $z$ relative to $[x,y]$ can be represented by at least one external ray of $f_0$ landing at $z$.
\end{enumerate}
\end{lemma}

\begin{lemma}
\label{non-repelling point on arc}
Let $f$ be a polynomial of degree $d\geq2$, and let $\Gamma\subset\mathbb{C}$ be an arc with end points $x_0,x_1$. 
Suppose $f:\Gamma\rightarrow\Gamma$ is a homeomorphism with $f(x_0)=x_0$. 
Then there exists a non-repelling fixed point of $f$ in $\Gamma$. 
\footnote{
Similarly, there exists a non-attracting  fixed point of $f$ in $\Gamma$. 
Using the snail lemma (see \cite[Lemma 16.2]{Mil}), we can deduce that every fixed point of $f$ in $\Gamma$ is attracting, parabolic or repelling.
} 
\end{lemma}

\begin{proof}
Note that there are finitely many fixed points of $f$ in $\Gamma$. 
Without loss of generality, we assume that $f$ has no fixed point in $\Gamma\setminus\{x_0,x_1\}$. 
Choose a parameterization $\gamma:[0,1]\rightarrow \Gamma$ of $\Gamma$ with $\gamma(0)=x_0$, and let $F=\gamma^{-1}\circ f|_{\Gamma}\circ \gamma:[0,1]\rightarrow[0,1]$. 
Let $x\in \Gamma\setminus\{x_0,x_1\}$ and let $x'=\gamma(x)$. 
Then $F(x')\neq x'$. 
Without loss of generality, we may assume $F(x')> x'$. 
Inductively, for any $n\in\mathbb{N}$, we have $F^{n+1}(x')>F^n(x')$. 
Because $F$ has no fixed points in $(0,1)$, we have  $\lim_{n\rightarrow\infty}F^n(x')=1$. 
That is $\lim_{n\rightarrow\infty}f^n(x)=x_1$. 
Since $f^n(x)\neq x_1$ for any $n\in\mathbb{N}$, 
the fixed point $x_1$ is non-repelling. 
\end{proof}

\begin{lemma}
\label{total disconnectivity}
Let $f_0\in\Pd$ be a postcritically finite hyperbolic polynomial, and $[x,y]$ be an allowable arc in $K(f_0)$.
Then $[x,y]\cap J(f_0)$ is totally disconnected (or empty).
\end{lemma}

\begin{proof}
Suppose $[x,y]$ has a nontrivial subarc included in $J(f_0)$; we will find a contradiction.
Without loss of generality we may assume $[x,y]\subset J(f_0)$.

For $z\in J(f_0)$, the \emph{incidence} $\nu(z)$ of $K(f_0)$ at $z$ is the number of external rays landing at $z$.
Then $\nu(z)$ is finite (see \cite{Kiwi02}), 
and counts the number of connected components of $K(f_0)\setminus\{z\}$ (see \cite[Corollary 7.1]{McM}). 
By Lemma \ref{allowable arcs in K(f0)}(\ref{allowable arcs in K(f0)-2}),  we have $\nu(z)\geq 2$ for any $z\in(x,y)$. 
We claim that there is no subarc of $[x,y]$ with constant incidence $2$. 

Assume $[z_1,z_2]$ is a subarc of $[x,y]$ with constant incidence $2$. 
For $k=1,2$, let $S_k$ denote the sector with $r(S_k)=z_k$ and $(z_1,z_2)\subset S_k$. 
Since $f_0$ is hyperbolic, the set $K(f_0)\setminus J(f_0)$ is dense in $K(f_0)$. 
Choose $z_3\in (K(f_0)\setminus J(f_0))\cap S_1\cap S_2$. 
Then $[z_1,z_2]\cap[z_1,z_3]=[z_1,z_4]$ for some $z_4\in [z_1,z_2]$. 
Without loss of generality, we assume $z_4\in(z_1,z_2)$. 
Let $S=S_{f_0}(\theta,\theta')$ be a sector with $r(S) = z_4$. 
Then there is at least one connected component of $\mathbb{C}\setminus\partial S$ containing at least two points in $\{z_1,z_2,z_3\}$. 
Without loss of generality, we assume $z_1,z_2\in S$. 
Then $R_{f_0}(\theta)$ and $R_{f_0}(\theta')$ are homotopic paths to $z_4$ relative to $[z_1,z_2]$. 
By Lemma \ref{allowable arcs in K(f0)}(\ref{allowable arcs in K(f0)-2}), 
there exists another external ray of $f_0$ landing at $z_4$ in $S$, which contradicts $\nu(z_4)=2$. 
This shows the claim. 

Now we can choose $z_1\neq z_2$ in $[x,y]$ with incidences at least $3$.
By \cite[Proposition 3.6]{PoirierHubbardTree}, the points $z_1$ and $z_2$ are preperiodic. 
Note that $J(f_0)\cap\crit(f_0)=\emptyset$. 
By Lemma \ref{allowable arcs in K(f0)}(\ref{allowable arcs in K(f0)-1}), the restriction 
$f_0^n|_{[x,y]}$ is injective and its image $f_0^n([x,y])$ is an allowable arc for any $n\in\mathbb{N}$. 
Choose an $n\in\mathbb{N}$ so that $w_1 := f_0^n(z_1)$ and $w_2 := f_0^n(z_2)$ are fixed points of $g := f_0^n$.
Then $g:[w_1,w_2] \rightarrow [w_1,w_2]$ is a homeomorphism with $g(w_1)=w_1$. 
This contradicts Lemma \ref{non-repelling point on arc}, 
or the shrinking lemma (see \cite{LM} or \cite[Lemma 2.9]{CT}). 
\end{proof}

\subsection{Minimal sectors with fixed angles}
Let $f_0\in\Pd$ be a postcritically finite polynomial. 
In this subsection, we will find proper allowable arcs in $K(f_0)$, 
and then consider the restriction of $f_0$ to these arcs. 

\begin{lemma}
\label{common part}
Let $f_0\in\Pd$ be a postcritically finite polynomial, and
$S={{S}}_{f_0}(\theta,\theta')$ be a minimal sector with root $x$. 
Then for any $y_1,y_2\in S\cap K(f_0)$, we have 
$[x,y_1]\cap[x,y_2]=[x,y]$ for some $y\in S$. 
\end{lemma}

\begin{proof}
Assume $[x,y_1]\cap[x,y_2]=\{x\}$; we will find a contradiction.  
Then $R_{f_0}(\theta)$ and $R_{f_0}(\theta')$ are homotopic paths to $x$ relative to $[y_1,y_2]=[x,y_1]\cup[x,y_2]$. 
By Lemma \ref{allowable arcs in K(f0)}(\ref{allowable arcs in K(f0)-2}), 
there exists another external ray of $f_0$ landing at $x$ in $S$.
This contradicts the minimality of $S$. 
So $[x,y_1]\cap[x,y_2]$ is an allowable arc. 
\end{proof}

\begin{lemma}
\label{growing arc}
Let $f_0\in\Pd$ be a postcritically finite polynomial, and
$S={{S}}_{f_0}(\theta,\theta')$ be a minimal sector with root $x$. 
Suppose $d\cdot\theta\equiv\theta\modZ$. 
Then there exists a point $y\in S\cap K(f_0)$ such that $[x,y]\cap\crit(f_0)=\emptyset$, $[x,y]\subset [x,f_0(y))$, 
and $f_0$ has no fixed point in $(x,y]$. 
\end{lemma}

\begin{proof}
Let $y_1\in S\cap K(f_0)$. 
Choose $y_2\in (x,y_1]$ near $x$ such that $[x,y_2]\cap\crit(f_0)=\emptyset$. 
By Lemma \ref{allowable arcs in K(f0)}(\ref{allowable arcs in K(f0)-1}), 
the restriction $f_0:[x,y_2]\rightarrow [x,f_0(y_2)]$ is a homeomorphism.  
It follows from the assumption $d\cdot\theta\equiv\theta\modZ$ that $f_0(y_2)\in S\cap K(f_0)$. 
By Lemma \ref{common part}, there is a point $y\in S$ such that 
$[x,y_2]\cap[x,f_0(y_2)]=[x,y]$.
It follows that $y,f_0(y)\in[x,f_0(y_2)]$ and 
$[x,y]\cap\crit(f_0)=\emptyset$. 
% we have $[x,y_3]\subset [x,f_0(y_3)]$ or $[x,f_0(y_3)]\subset [x,y_3)$. 

% Assume there exists a fixed point $z\in(x,y]$ of $f_0$. 
% Then $[x,z]\cap|T(f_0)|=\emptyset$ and $f_0:[x,z]\rightarrow[x,z]$ is a homeomorphism. 
% This contradicts the shrinking lemma (see \cite{LM} or \cite[Lemma 2.9]{CT}). 
Assume there is a fixed point $z\in(x,y]$. 
Then $f_0:[x,z]\rightarrow[x,z]$ is a homeomorphism with $f_0(x)=x$.  
By Lemma \ref{non-repelling point on arc}, there is a non-repelling fixed point $w\in[x,z]$. 
Since $f_0$ is postcritically finite, we have $w\in\crit(f_0)$, which  contradicts $[x,y]\cap\crit(f_0)=\emptyset$. 
Therefore $f_0$ has no fixed point in $(x,y]$. 
This implies that either $[x,y]\subset[x,f_0(y))$ or $[x,f_0(y)]\subset[x,y)$. 

Assume $[x,f_0(y)]\subset[x,y)$. 
% Choose a parameterization $\alpha:[0,1]\rightarrow [x,y]$ of $[x,y]$ with $\alpha(0)=x$, and consider $\alpha^{-1}\circ f_0|_{[x,y]}\circ \alpha$. 
Since $f_0$ has no fixed point in $(x,y]$, similar to the proof of Lemma \ref{non-repelling point on arc}, we see that $x$ is non-repelling. 
This contradicts that $x$ is repelling. 
Therefore $[x,y]\subset[x,f_0(y))$. 
The proof is completed. 
\end{proof}

With the notations and assumptions as Lemma \ref{growing arc}, 
let $$X(S)=\bigcup_{n\in\mathbb{N}} f_0^{n}([x,y]).$$
Set $g= f_0|_{[x,y]}:[x,y]\rightarrow [x,f_0(y)]$ and $y_n=g^{-n}(y)$ for each $n\geq 1$. 
Since there is no fixed point of $f_0$ in $(x,y]$, we have $y_n\rightarrow x$ as $n\rightarrow\infty$. 
This implies $X(S)$ is independent of $y$. 

A subset $X\subset K(f_0)$ is \emph{allowable} if $x_1,x_2\in X$ implies $[x_1,x_2]\subset X$. 
The following lemma shows that $X(S)$ is allowable. 
We call $X(S)$ the \emph{allowable set induced by $S$}. 

\begin{lemma}
[Allowable set]
\label{X(S) is allowable}
Let $f_0\in\Pd$ be a postcritically finite polynomial, and
$S={{S}}_{f_0}(\theta,\theta')$ be a minimal sector with $d\cdot\theta\equiv\theta\modZ$. 
Then $f_0(X(S))=X(S)$ and $X(S)$ is allowable. 
\end{lemma}

\begin{proof}
Choose $[x,y]$ as in Lemma \ref{growing arc}, and let $X_n=f_0^{n}([x,y])$ for $n\in\mathbb{N}$. 
Because the image of an allowable set is allowable (see \cite[Corollary 2.8]{PoirierHubbardTree}), 
the set $X_n$ is allowable for any $n\in\mathbb{N}$. 
Since $X_0\subset X_1$, we have $X_0\subset X_1\subset X_2\subset\cdots$. 
Therefore $f_0(X(S))=X(S)$ and $X(S)$ is allowable. 
\end{proof}

\begin{example}
\label{example of f0}
Let $(f_0,\mathbf{c})\in\widehat{\mathcal{P}}^5$ satisfy  $f_0(c_1)=c_1$, $f_0^2(c_2)=c_2$, 
$f_0(c_3)=c_3$ and $f_0^2(c_4)=c_3$ as in Figure \ref{figure some minimal sectors of f0}. 
\footnote{
$(\mathbf{c},f_0(0))\approx(0.5559 + 0.4849i,~ -0.4877 + 0.5181i,~ 0.5351 - 0.4268i,~ 
-0.6033 - 0.5762i,\\ -0.1050 - 0.0158i)$.
}
Then $f_0$ is a postcritically finite hyperbolic polynomial. 
% So $\Tfp=\{c_1,c_2,f_0(c_2),c_4\}$ and $\Tfn=\{c_3,f_0(c_3)\}$. 
Let $x$ denote the landing point of $R_{f_0}(0)$. 
Then we have
\begin{align*}
&X({{S}}_{f_0}(0,1/4))=[x,c_1),\\
&X({{S}}_{f_0}(1/4,1/2))=[x,f_0(c_2)],\\
&X({{S}}_{f_0}(1/2,3/4))=[f_0(c_4),c_3],\\
&X({{S}}_{f_0}(3/4,0))=[x,c_3).
\end{align*}

\begin{figure}[ht]
\centering
\includegraphics{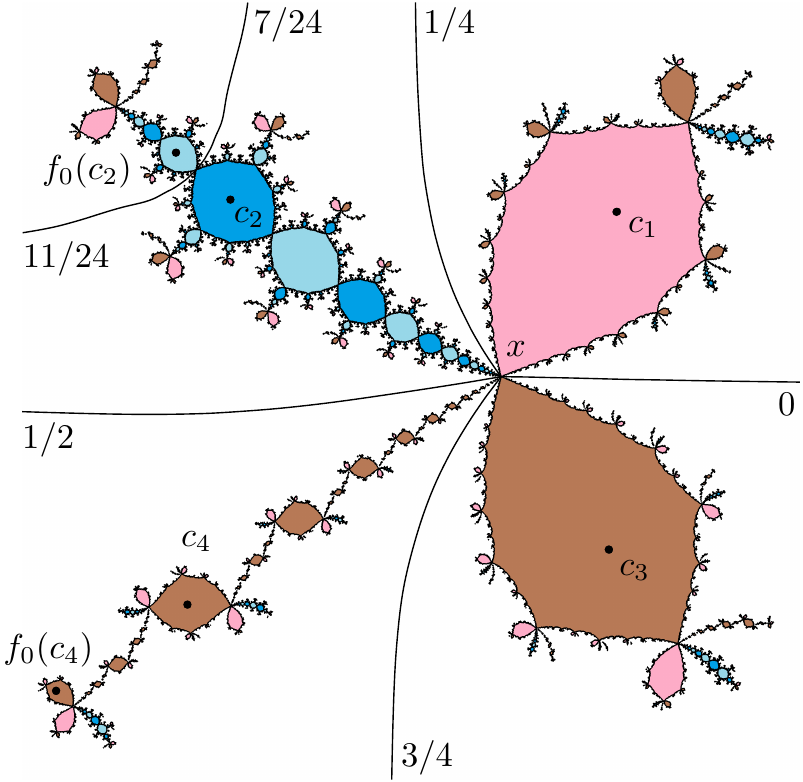}
\caption{Some minimal sectors of $f_0$.}
\label{figure some minimal sectors of f0}
\end{figure}
\end{example}

\begin{proposition}
\label{two sequences of external angles}
Let $f_0\in\Pd$ be a postcritically finite polynomial, and let 
$S={{S}}_{f_0}(\theta,\theta')$ be a minimal sector with $d\cdot\theta\equiv\theta\modZ$. 
Suppose $X(S)\not\subset\overline{S}$. 
Then there exist two sequences $\{\theta_n\}_{n\in\mathbb{N}},\{\theta'_n\}_{n\in\mathbb{N}}$ of angles such that 
\begin{enumerate}
\item $R_{f_0}(\theta_n)$ and $R_{f_0}(\theta'_n)$ land at a common periodic point;

\item $\theta,\theta_{n+1},\theta_{n},\theta_{n}',\theta_{n+1}',\theta'$ are in positive cyclic order;

\item $\theta_n\rightarrow\theta$ and $\theta'_n\rightarrow\theta'$ as $n\rightarrow\infty$. 
\end{enumerate}
\end{proposition}

\begin{proof}
Let $x={{{r}}}(S)$ and choose $[x,y]$ as in Lemma \ref{growing arc}: 
$[x,y]\cap\crit(f_0)=\emptyset$ and $[x,y]\subset [x,f_0(y))$. 

Since $X(S)\not\subset\overline{S}$, we can take a minimal integer $m\geq2$ such that $x\in f_0^{m}((x,y])$. 
Then there is a unique $z_0\in(x,y]$ satisfying $f_0^{m}((x,z_0))\subset S$ and $f_0^{m}(z_0)=x$. 
Let $g= f_0|_{[x,y]}:[x,y]\rightarrow [x,f_0(y)]$ and $z_1=g^{-1}(z_0)$. 
By Lemma \ref{common part}, we can choose $w_0\in(z_1,z_0)$ near $z_0$ such that 
$\crit(f_0)\cap f_0^{k}([w_0,z_0))=\emptyset$ for any $0\leq k\leq m$, and $f_0^{m}(w_0)\in (x,z_0]$. 

Set $z_n=g^{-n}(z_0)$ and $w_n=g^{-n}(w_0)$ for each $n\in\mathbb{N}$. 
Then $\{[w_n,z_n]\}_{n\in\mathbb{N}}$ are pairwise disjoint. 
Because $f_0$ has no fixed point in $(x,y]$ (see Lemma \ref{growing arc}), the points $w_n$ and $z_n$ tend to $x$ as $n\rightarrow\infty$. 
Thus when $n$ is large enough, we have $[w_n,z_n]\subset (x,f_0^{m}(w_0))$. 
By Lemma \ref{allowable arcs in K(f0)}(\ref{allowable arcs in K(f0)-1}), 
the restriction $f_0^{n+m}: [w_n,z_n] \rightarrow [f_0^{m}(w_0),x]$ is a homeomorphism. 
Parameterizing $[x,f_0^{m}(w_0)]$ by $[0,1]$, and according to the intermediate value theorem, 
the restriction $f_0^{n+m}|_{[w_n,z_n]}$ has a unique fixed point $a_n$. 
Now $\{a_n\}_{n>N}$ is a sequence of periodic points in $(x,y)$ tending to $x$. 

When $n$ is large enough, $a_n$ is repelling, hence $a_n\in J(f_0)$.
By Lemma \ref{allowable arcs in K(f0)}(\ref{allowable arcs in K(f0)-2}), 
there exists a pair of external angles $\theta_n,\theta'_n$ of $a_n$, 
corresponding to the two homotopy classes of paths to $a_n$ relative to $[x,y]$, 
such that $\theta,\theta_n,\theta'_n,\theta'$ are in positive cyclic order.
Then $\theta,\theta_{n+1},\theta_{n},\theta_{n}',\theta_{n+1}',\theta'$ are in positive cyclic order.
Because $K(f_0)$ is locally connected and $S$ is minimal, 
we have $\theta_n\rightarrow\theta$ and $\theta'_n\rightarrow\theta'$ as $n\rightarrow\infty$. 
\end{proof}

\begin{definition}
[Mapping scheme]
\label{mapping scheme}
Following Milnor \cite{MilHyperCompo}, a \emph{mapping scheme} $T$ is a triple $(|T|,\sigma,\delta)$ where $|T|$ is a nonempty finite set,
$\sigma$ is a self-mapping of  $|T|$, and $\delta$ is a map from $|T|$ into the positive integers 
such that for any periodic point $v\in|T|$ of $\sigma$, we have
$\prod_{k=0}^{p-1}\delta(\sigma^{k}(v))\geq2$, 
where $p$ is the period of $v$.
% We say $T$ is \emph{reduced} if $\delta(v)\geq 2$ for each $v\in|T|$.
The map $\delta$ is called the \emph{degree function}. 
Let 
\begin{align*}
\Tp &= \{v\in|T|{;~} \text{$v$ is periodic under $\sigma$}\}, \\
\Tn &= \{v\in|T|{;~} \text{$v$ is nonperiodic under $\sigma$}\}. 
\end{align*}

For each $v\in|T|$, let $r_v\geq 0$ be the minimal integer such that $\sigma^{r_v}(v)\in\Tp$. 

Let $f_0\in\Pd$ be a postcritically finite hyperbolic polynomial. 
The \emph{mapping scheme $T(f_0)=(|T(f_0)|,\sigma,\delta)$ associated with $f_0$} can be defined as follows.
Let $|T(f_0)|=\{f_0^{n}(c) {;~}  c\in\crit(f_0),n\geq0\}$. 
For each $v\in|T(f_0)|$, define $\sigma(v)=f_0(v)$ and $\delta(v)=\deg(f_0,v)$. 
\end{definition}

For $f_0$ in Example \ref{example of f0}, we have $\Tfp=\{c_1,c_2,f_0(c_2),c_3\}$ and $\Tfn=\{c_4,f_0(c_4)\}$.

\begin{proposition}
\label{two sequences of external angles 2}
Let $f_0\in\Pd$ be a postcritically finite hyperbolic polynomial, and let 
$S={{S}}_{f_0}(\theta,\theta')$ be a minimal sector with $d\cdot\theta\equiv\theta\modZ$. 
Then $X(S)$ is contained in the closure of a bounded Fatou component of $f_0$ 
if and only if $X(S)\cap\Tfp$ is empty. 

Suppose $X(S)\cap\Tfp\neq \emptyset$. 
Then for each $v\in X(S)\cap\Tfp$, there exist two sequences 
$\{\theta_n\}_{n\in\mathbb{N}},\{\theta'_n\}_{n\in\mathbb{N}}$ of angles  
and some integer $m\in\mathbb{N}$ such that 
\begin{enumerate}
\item $R_{f_0}(\theta_n)$ and $R_{f_0}(\theta'_n)$ land at a common preperiodic point; 

\item $\theta,\theta_{n+1},\theta_{n},\theta_{n}',\theta_{n+1}',\theta'$ are in positive cyclic order; 

\item $\theta_n\rightarrow\theta$ and $\theta'_n\rightarrow\theta'$ as $n\rightarrow\infty$; 

\item ${{S}}_{f_0}(d^{m+n}\theta_n,d^{m+n}\theta'_n)$ is a maximal sector attaching to $U_{f_0}(v)$. 
\end{enumerate}
\end{proposition}

\begin{proof}
Let $x={{{r}}}(S)$ and choose $[x,y]$ as in Lemma \ref{growing arc}: 
$[x,y]\cap\crit(f_0)=\emptyset$ and $[x,y]\subset [x,f_0(y))$. 

If $X(S)$ is contained in the closure of a bounded Fatou component $U$ of $f_0$, 
then $f_0(U)=U$ and there is a fixed point $w\in U$. 
Thus $X(S)=[x,w)$ is disjoint with $\Tfp$. 

On the other hand, if $X(S)$ is not contained in the closure of any bounded Fatou component of $f_0$, 
then $(x,y]\cap J(f_0)$ is nonempty. 
Choose $z\in(x,y]\cap J(f_0)$. 
Since $f_0$ is hyperbolic, by Lemma \ref{total disconnectivity}, there exists a bounded Fatou component $U$ of $f_0$ such that $[x,z]\cap U$ is nonempty. 
Let $w$ be the only preperiodic point in $U$. 
Then we have $w\in[x,z]$ and $\{f_0^{n}(w)\}_{n\geq0}\subset X(S)$. 
This shows $X(S)\cap\Tfp$ is nonempty. 

Now let us prove the second assertion. 
Given $v\in X(S)\cap\Tfp$. 
Choose $w_0\in(x,y]$ and $m\in\mathbb{N}$ such that $f_0^{m}(w_0)=v$. 
For $n\in\mathbb{N}$, define $w_n=g^{-n}(w_0)$, where $g=f_0|_{[x,y]}:[x,y]\rightarrow[x,f_0(y)]$. 
Because $f_0$ has no fixed point in $(x,y]$ (see Lemma \ref{growing arc}), we have $w_n\rightarrow x$ as $n\rightarrow\infty$. 
When $n\geq 1$, the set $[w_n,w_{n-1}]\cap\partial U_{f_0}(w_n)$ is a singleton, which will be denoted by $\{z_n\}$. 
Let $S_n={{S}}_{f_0}(\theta_n,\theta_n')$ be the maximal sector attaching to $U_{f_0}(w_n)$ with root point $z_n$.  
By Lemma \ref{allowable arcs in K(f0)}(\ref{allowable arcs in K(f0)-2}), the sector $S_n$ is well-defined and $w_{n-1}\in S_n$. 
Then $\theta,\theta_{n+1},\theta_{n},\theta_{n}',\theta_{n+1}',\theta'$ are in positive cyclic order. 
By Lemma \ref{sectors attaching to U}, the root ${{{r}}}(S_n)$ is preperiodic.  
Because $K(f_0)$ is locally connected and $S$ is minimal, we have $\theta_n\rightarrow\theta$ and $\theta'_n\rightarrow\theta'$ as $n\rightarrow\infty$. 
Finally, the sector ${{S}}_{f_0}(d^{m+n}\theta_n,d^{m+n}\theta'_n)$ is a maximal sector attaching to 
$U_{f_0}(f_0^{m+n}(w_n))=U_{f_0}(f_0^{m}(w_0))=U_{f_0}(v)$. 
\end{proof}

\section{Behavior of critical orbits}
\label{section behavior of crit-orbits}
This section is the first step to prove Theorem \ref{top-boundary}. Let $\mathcal{H}$ be a capture 
hyperbolic component in $\mathcal F$. Then $\mathcal{H}$ contains a unique postcritically finite 
(critically marked) polynomial $f_0$, which is called the \emph{center} of $\mathcal{H}$ (see 
\cite{MilHyperCompo}). Recall that $T(f_0)=(|T(f_0)|,\sigma,\delta)$ is the mapping scheme associated 
with $f_0$. 

Replacing $\mathcal{F}$ by the irreducible component of $\mathcal{F}$ containing $\mathcal{H}$ 
(see Lemma \ref{basic facts of H}), 
we may assume $\mathcal{F}$ is irreducible. 
Let $v\in\Tfp$. Choose $1\leq j\leq \ell$ and $n\geq 0$ so that $v = f_0^n(c_j(f_0))$ 
(see (\ref{superattracting relations})). 
Define a continuous function $\zeta_v: \mathcal{F} \rightarrow \mathbb{C}$ 
by $\zeta_v(f) = f^n(c_j(f))$. Then $\zeta_v(f)$ is a superattracting periodic point of $f\in \mathcal{F}$ and 
$\zeta_v(f_0)=v$. 
Clearly $\zeta_v$ is independent of the choices of $j$ and $n$. 
For $f\in \mathcal{F}$, let $U_{f,v} = U_f(\zeta_v(f))$ be the Fatou component containing $\zeta_v(f)$, and let $$A(f) = 
\bigcup_{v\in\Tfp} U_{f,v}.$$ 

The aim of this section is to prove the following: 

\begin{proposition}
[Behavior of critical orbits]
\label{behavior of critical orbits}
Let $\mathcal{H}\subset\mathcal{F}$ be a capture hyperbolic component, and let 
$f\in\partial\mathcal{H}$. Then: 
\begin{enumerate}
\item there exists an $n\geq0$ such that $f^{n}(\crit(f))\subset \overline{A(f)}$;  
\item $f$ has no indifferent cycle; 
\item every critical point in $\partial A(f)$ (if any) is preperiodic. 
\end{enumerate}
\end{proposition}

\subsection{Internal rays}
\label{subsection internal rays}

In the following, we will introduce the notations $B_{f,v}$, $R_{f,v}(t)$, 
$\theta^-_{f,v}(t)$ and $\theta^+_{f,v}(t)$ for $f\in\Hbar$, $v\in\Tfp$ and $t\in\RZ$.  

Fix $v\in\Tfp$. 
Without loss of generality, we assume $v$ is a fixed point under $\sigma$; otherwise, we consider 
the iterate $f^p$ for every $f\in\Hbar$, where $p$ is the period of $v$ under $\sigma$. 
For each $f\in\Hbar$, we have $f(\zeta_v(f))=\zeta_v(f)$ and $\deg(f,\zeta_v(f)) = \deg(f|_{U_{f,v}}) = \delta(v)\geq2$.  
A \emph{B\"ottcher coordinate} $B$ of $U_{f,v}$ is a Riemann mapping $B:U_{f,v} \rightarrow 
\mathbb{D}$ satisfying $B(f(z))=B(z)^{\delta(v)}$.
Define $$
\mathcal{X} = \mathcal{X}(v) = {\left\{(f,\lambda)\in \Hbar\times \partial\mathbb{D}{;~} \text{$\lambda\cdot\phi_{f,v}$ is a 
B\"ottcher coordinate of $U_{f,v}$}\right\}},$$
where $\phi_{f,v}:U_{f,v}\rightarrow \mathbb{D}$ 
is the Riemann mapping normalized by $\phi_{f,v}(\zeta_v(f))=0$ and $\phi'_{f,v}(\zeta_v(f))>0$. 

It follows from $\deg(f,\zeta_v(f)) = \delta(v)$ that 
$$f'(\zeta_v(f)) =\cdots= f^{(\delta(v)-1)}(\zeta_v(f))=0, \quad 
f^{(\delta(v))}(\zeta_v(f))\neq 0.$$ 
Let $B$ be a B\"ottcher coordinate of $U_{f,v}$. 
Clearly $B(\zeta_v(f))=0$. 
%Taking $\delta(v)$-th derivatives on both sides of $B(f(z))=B(z)^{\delta(v)}$ 
% and letting $z=\zeta_v(f)$ 
Using the Taylor expansions of $f$ and $B$ at $\zeta_v(f)$ in $B(f(z))=B(z)^{\delta(v)}$, 
we have $$B'(\zeta_v(f))^{\delta(v)-1}=\frac{f^{(\delta(v))}(\zeta_v(f))}{\delta(v)!}.$$
Therefore $$
\mathcal{X} ={\left\{(f,\lambda)\in \Hbar\times \partial\mathbb{D}{;~} 
\lambda^{\delta(v)-1} = \frac{ f^{(\delta(v))}(\zeta_v(f)) } {| f^{(\delta(v))}(\zeta_v(f)) |}
\right\}}.$$
This implies the following fact. 

\begin{fact}
\label{covering map pi}
The projection $\pi:\mathcal{X}\rightarrow\Hbar$, $(f,\lambda)\mapsto f$ is a covering map of degree 
$\delta(v)-1$. 
\end{fact}

%\begin{proof}
%Given $f\in\Hbar$. Let 
%$\pi^{-1}(f) = \{(f,\lambda_j){;~} 1\leq j \leq \delta(v)-1\}$. Now fix $1\leq j \leq \delta(v)-1$. 
%For $g$ in a small neighborhood $\mathcal{N} \subset \Hbar$ of $f$, there is a unique 
%continuous single-valued branch $$\lambda_j(g) = {\left( \frac{ g^{(\delta(v))}(\zeta_v(g)) } {| 
%g^{(\delta(v))}(\zeta_v(g)) |} \right)}^{\frac{1}{\delta(v)-1}}$$ so that $\lambda_j(f) = \lambda_j$. 
%Define $P_j: \mathcal{N} \rightarrow \mathcal{X}$, $g\mapsto (g,\lambda_j(g))$. 
%Shrink $\mathcal{N}$ so that $P_1(\mathcal{N}),\dots,P_{\delta(v)-1}(\mathcal{N})$ are pairwise disjoint. 
%Then $$\pi^{-1}(\mathcal{N}) = \bigsqcup_{1\leq j\leq \delta(v)-1} P_j(\mathcal{N}).$$ 
%This completes the proof. 
%\end{proof}

For $(f,\lambda)\in\mathcal{X}$ and $t\in\RZ$, the \emph{internal ray} of angle $t$ in $U_{f,v}$ is 
defined to be $R_{f,\lambda}(t) = \phi_{f,v}^{-1}( \lambda^{-1} \cdot [0,1)e^{2\pi i t} )$. Since 
$\phi_{f,v}$ can be extended to a homeomorphism from $\overline{U_{f,v}}$ to $\overline{\mathbb{D}}$ 
(see \cite{RY2022}), the internal ray $R_{f,\lambda}(t)$ lands. Define $$\Psi: \begin{cases} \mathcal{X} 
\times \overline{\mathbb{D}} \rightarrow \mathbb{C}, \\ (f,\lambda,z) \mapsto \phi_{f,v}^{-1}(\lambda^{-1}
\cdot z). \end{cases}$$ 

\begin{lemma}
\label{continuity on D-tilde}
The function $\Psi$ is continuous on $\mathcal{X}\times\mathbb{D}$. 
\end{lemma}

\begin{proof}
The pointed Jordan domain $(U_{f,v},\zeta_v(f))$ is continuous with respect to $f\in\Hbar$ in the sense 
of \emph{Carath\'eodory kernel convergence} (see, e.g. \cite[\S5.1]{McM}). By the Carath\'eodory kernel 
convergence theorem, for any $f_n\rightarrow f$ in $\Hbar$, the map $\phi_{f_n,v}^{-1}$ converges to 
$\phi_{f,v}^{-1}$ uniformly on compact subsets of $\mathbb{D}$. By the Arzel\`a-Ascoli theorem, this 
is equivalent to the continuity of the function $\Hbar\times\mathbb{D}\rightarrow\mathbb{C}$, $(f,z)
\mapsto \phi_{f,v}^{-1}(z)$. It follows that $\Psi$ is continuous on $\mathcal{X}\times\mathbb{D}$. 
\end{proof}

According to Lemma \ref{continuity on D-tilde}, similar to the stability of external rays (Lemma 
\ref{stability of external rays}), we have: 

\begin{lemma}
[Stability of internal rays]
\label{stability of internal rays-tilde}
Let $(f,\lambda)\in\mathcal{X}$ and $t\in\QZ$. Suppose $R_{f,\lambda}(t)$ lands at a pre-repelling 
point. Then there is a neighborhood $\mathcal{N}\subset \mathcal{X}$ of $(f,\lambda)$ such that the 
restriction $\Psi: \mathcal{N}\times[0,e^{2\pi i t}]\rightarrow\mathbb{C}$ is continuous. 
\end{lemma}

By Proposition \ref{limb}, for $(f,\lambda)\in\mathcal{X}$ and $t\in\RZ$, we define $$
\theta^-_{f,\lambda}(t)=\theta^-_{U_{f,v}} (x_{f,\lambda}(t)), \quad
\theta^+_{f,\lambda}(t)=\theta^+_{U_{f,v}} (x_{f,\lambda}(t)), 
$$ where $x_{f,\lambda}(t) = \phi_{f,v}^{-1}(\lambda^{-1} \cdot e^{2\pi i t}) \in\partial U_{f,v}$ 
is the landing point of $R_{f,\lambda}(t)$. 

\begin{lemma}
\label{invariant external angles-tilde}
Let $\mathcal{Y}$ be a connected component of $\mathcal{X}$. Then the functions $\theta^-_{f,\lambda}$ 
and $\theta^+_{f,\lambda}$ are independent of $(f,\lambda)\in \mathcal{Y}$. 
\end{lemma}

\begin{proof}
By (\ref{limb1}), we just need to check the assertion for  $\theta^-_{f,\lambda}$. 

Let $\mathcal{Y}_1,\dots,\mathcal{Y}_m$ denote the connected components of $\mathcal{Y}\cap 
\pi^{-1}(\mathcal{H})$. Then $\mathcal{Y} = \bigcup_{j=1}^m \overline{\mathcal{Y}_j}$. Given $1\leq j 
\leq m$. Then $\theta^-_j :=\theta^-_{f,\lambda}$ is independent of $(f,\lambda)\in \mathcal{Y}_j$, 
since $f$ is hyperbolic. Let $\{(f_n,\lambda_n)\}_{n\geq1}\subset \mathcal{Y}_j$ and $(f,\lambda)
\in\partial \mathcal{Y}_j$ satisfy $(f_n,\lambda_n) \rightarrow (f,\lambda)$ as $n\rightarrow\infty$. 
We claim that $\theta^-_{f,\lambda}=\theta^-_j$. 

Given $t\in\RZ$. 
Choose a sequence $\{t_k\}_{k\geq1}\subset \QZ$ so that 
\begin{itemize}
\item $t-1<t_1<t_2<\cdots<t$ and $\lim_{k\rightarrow \infty} t_k=t$, where we choose representatives 
for these angles in $\mathbb{R}$ properly; 
\item $R_{f,\lambda}(t_k)$ lands at a repelling periodic point $x_k$ for any $k\geq 1$. 
\end{itemize} 
By choosing representatives for external angles in $\mathbb{R}$ properly, we have $$\theta_j^+(t)-1 < 
\theta_j^-(t_1)\leq\theta_j^+(t_1) < \theta_j^-(t_2)\leq\theta_j^+(t_2) < \cdots < \theta_j^-(t)\leq
\theta_j^+(t).$$ Now fix $k\geq1$. Note that $R_{f,\lambda}(t_k)$ and $R_f(\theta^-_{f,\lambda}(t_k))$ 
land at the same repelling periodic point $x_k$. By the stability of dynamical rays (Lemmas 
\ref{stability of external rays} and \ref{stability of internal rays-tilde}), when $n$ is large enough, 
the rays $R_{f_n,\lambda_n}(t_k)$ and $R_{f_n}( \theta^-_{f,\lambda}(t_k) )$ land at a common repelling 
periodic point of $f_n$, which is close to $x_k$ (compare Lemma \ref{limit-land-rela}(\ref{limit-land-rela-2})). 
This gives $$\theta^-_j(t_k) \leq \theta^-_{f,\lambda}(t_k) \leq\theta^+_j(t_k),$$ where we choose 
a representative for $\theta^-_{f,\lambda}(t_k)$ in $\mathbb{R}$ properly. Then by (\ref{limb1}), 
$$\theta^-_j(t) = \lim_{k\rightarrow\infty} \theta^-_j(t_k)\leq \lim_{k\rightarrow\infty} 
\theta^-_{f,\lambda}(t_k) = \theta^-_{f,\lambda}(t)\leq \lim_{k\rightarrow\infty} \theta^+_j(t_k) 
= \theta^-_j(t).$$ Therefore $\theta^-_{f,\lambda}(t)=\theta^-_j(t)$. Since $t$ was arbitrary, we 
have $\theta^-_{f,\lambda} = \theta^-_j$. 

The claim implies that $\theta^-_{f,\lambda}$ is independent of $(f,\lambda)\in \overline{\mathcal{Y}_j}$. 
Furthermore, it is independent of $(f,\lambda)\in \mathcal{Y}$ since $\mathcal{Y}$ is connected. The 
proof is completed. 
\end{proof}

Note that the $\delta(v)-1$ functions in $(\theta^-_{f_0,\lambda})_{(f_0,\lambda)\in\pi^{-1}(f_0)}$ 
are pairwise different. By Fact \ref{covering map pi} and Lemma \ref{invariant external angles-tilde}, 
the number of connected components of $\mathcal{X}$ is $\delta(v)-1$, so the restriction of $\pi$ 
on each connected component is a homeomorphism. Choose $\lambda_v\in\partial\mathbb{D}$ so that 
$(f_0,\lambda_v)\in\mathcal{X}$. Then $\pi^{-1}$ has a unique continuous single-valued branch $\Hbar 
\rightarrow \mathcal{X}$, $f\mapsto (f,\lambda_{f,v})$ satisfying $\lambda_{f_0,v} = \lambda_v$. 

For $f\in\Hbar$, $v\in\Tfp$ and $t\in\RZ$, define 
\begin{itemize}
\item the B\"ottcher coordinate $B_{f,v} = \lambda_{f,v}\cdot\phi_{f,v}: \overline{U_{f,v}}\rightarrow 
\overline{\mathbb{D}}$; 

\item the internal ray $R_{f,v}(t) = B_{f,v}^{-1}([0,1)e^{2\pi i t}) \subset U_{f,v}$;

\item the external angles $\theta^-_{f,v}(t) = \theta^-_{U_{f,v}}(B_{f,v}^{-1}( e^{2\pi i t}))$, 
$\theta^+_{f,v}(t) = \theta^+_{U_{f,v}}(B_{f,v}^{-1}( e^{2\pi i t}))$, where $B_{f,v}^{-1}
( e^{2\pi i t})$ is the landing point of $R_{f,v}(t)$. 
\end{itemize}
For each $v\in\Tfp$, we can further require that the diagram $$\begin{CD}
\overline{U_{f,v}} @>f>> \overline{U_{f,\sigma(v)}}\\
@V B_{f,v} VV          @VV B_{f,\sigma(v)} V\\
\overline{\mathbb{D}} @>>z^{\delta(v)}> \overline{\mathbb{D}}
\end{CD}$$ commutes. Then Lemmas \ref{continuity on D-tilde}, \ref{stability of internal rays-tilde} 
and \ref{invariant external angles-tilde} become the following lemmas. 

\begin{lemma}
\label{continuity on D}
For each $v\in\Tfp$, the function $$\Psi_v: \begin{cases}\Hbar\times\overline{\mathbb{D}} \rightarrow 
\mathbb{C}, \\ (f,z)\mapsto B^{-1}_{f,v}(z) \end{cases}$$ is continuous on $\Hbar\times\mathbb{D}$. 
\end{lemma}

\begin{lemma}
[Stability of internal rays]
\label{stability of internal rays}
Let $f\in\Hbar$, $v\in\Tfp$ and $t\in\QZ$. Suppose $R_{f,v}(t)$ lands at a pre-repelling point. 
Then there is a neighborhood $\mathcal{N}\subset \Hbar$ of $f$ such that the restriction $\Psi_v: 
\mathcal{N}\times[0,e^{2\pi i t}] \rightarrow \mathbb{C}$ is continuous. 
\end{lemma}

\begin{lemma}
\label{invariant external angles}
Given $v\in\Tfp$. Then the functions $\theta^-_{f,v}$ and $\theta^+_{f,v}$ are independent of 
$f\in\Hbar$.
\end{lemma}

Because of Lemma \ref{invariant external angles}, we denote $\theta^-_{f,v}, \theta^+_{f,v}$ by $\theta^-_v, \theta^+_v$ respectively.

\subsection{No parabolic cycles}

\begin{lemma}
\label{an iterate of f0}
Let $f_0\in\Pd$ be a postcritically finite hyperbolic polynomial. 
Then there exists an integer $n\geq1$ such that
\begin{enumerate}
\item every point in $\Tfp$ is fixed by $f_0^n$;  
\item for any $v\in\Tfp$ and any sector $S$ attaching to $U_{f_0}(v)$, if ${{{r}}}(S)$ is periodic under $f_0$, 
then it is fixed by $f_0^n$. 
\end{enumerate}
\end{lemma}

\begin{proof}
Choose $m\geq1$ such that for any $v\in\Tfp$, the period of $v$ under $f_0$ divides $m$. 
Then every point in $\Tfp$ is fixed by $f_0^{m}$. 

Given $v\in\Tfp$. 
Let $E_v$ be the set of all periodic points $x$ of $f_0$ 
%in $\partial U_{f_0}(v)$ 
such that there exists a sector attaching to $U_{f_0}(v)$ with root $x$. 
Then $E_v\subset \partial U_{f_0}(v)$. 
We claim that $f_0^{m}(E_v)\subset E_v$. 
To prove this, let $x\in E_v$, which means there is a sector $S_{f_0}(\theta,\theta')$ attaching to $U_{f_0}(v)$ with root $x$. 
Since $f_0^{m}$ is a local homeomorphism near $x$ and $f_0^{m}(U_{f_0}(v))=U_{f_0}(v)$, the sector $S_{f_0}(d^{m}\theta,d^{m}\theta')$ attaches to $U_{f_0}(v)$ with root $f_0^{m}(x)$, so $f_0^{m}(x)\in E_v$. 
This proves the claim.  

Now $E_v$ can be divided into a collection $\{F_j\}_{j\in I}$ of periodic cycles of $f_0^{m}$. 
We will show $I$ is a finite set.  
Let $x\in F_j$ for some $j\in I$, with period $p$ under $f_0^{m}$. 
For $0\leq k< p$, let $S_{j,k}$ denote the maximal sector attaching to $U_{f_0}(v)$ with root $f_0^{m k}(x)$. 
Let $S_j = \bigcup_{0\leq k< p} S_{j,k}$. 
Since $J(f_0) = J(f_0^{m})$, if $S_{f_0}(\theta,\theta')$ is a sector, then $S_{f_0^{m}}(\theta,\theta') = S_{f_0}(\theta,\theta')$. 
By Lemma \ref{sector with no critical point}, for any $j\in I$, we have  $\crit(f_0^{m})\cap S_j \neq\emptyset$. 
Because the sets in $\{S_j\}_{j\in I}$ are pairwise disjoint,  the set $I$ is finite, and so is $E_v$. 

Finally, choose $n\geq1$ such that for any $x\in\Tfp\cup\bigcup_{v\in\Tfp}E_v$, its period under $f_0$ divides $n$. 
Then every point in $\Tfp$ is fixed by $f_0^{n}$. 
For any $v\in\Tfp$ and any sector $S$ attaching to $U_{f_0}(v)$, if ${{{r}}}(S)$ is periodic under $f_0$, then $r(S)\in E_v$, so it is fixed by $f_0^n$. 
This completes the proof. 
\end{proof}

\begin{lemma}
\label{key lemma}
Let $\mathcal{H}\subset\mathcal{F}$ be a capture hyperbolic component. 
Then for any $f\in\partial\mathcal{H}$, every periodic point in $\partial A(f)$ is repelling. 
\end{lemma}

\begin{proof}
{\bf Step 1.} 
Let $f_0$ be the center of $\mathcal{H}$. 
%By Lemma \ref{immediate attracting basins of f}, the collection  $\{U_{f,v}\}_{v\in\Tfp}$ consists of all bounded immediate attracting basins of $f$. 
%We need to show every periodic point in $\bigcup_{v\in\Tfp}\partial U_{f,v}$ is repelling. 
Note that for any integer $n\geq1$, we have $\Tfp = |T(f_0^n)|_{\mathrm{p}}$ and $J(f_0) = J(f_0^n)$. 
According to Lemma \ref{an iterate of f0}, replacing $f_0$ by an iterate of $f_0$ (every polynomial in $\Hbar$ need to be iterated the same number of times), we can make the following assumptions without loss of generality:  
\begin{enumerate}
\item every point in $\Tfp$ is fixed by $f_0$;  
\item for any $v\in\Tfp$ and any sector $S$ attaching to $U_{f_0}(v)$, if ${{{r}}}(S)$ is periodic, 
then it is fixed by $f_0$.
% (see Lemma \ref{sector with no critical point})
\end{enumerate}
%Now let $\Theta(f_0)$ denote the set of all pairs $(\theta,\theta')\in(\RZ)^2$, such that 
%${{S}}_{f_0}(\theta,\theta')$ is a minimal sector with a fixed root point in $\partial A(f_0)$. 
Now let $$\Theta(f_0) ={\left\{ (\theta,\theta')\in(\RZ)^2{;~} 
\begin{array}{l}
\text{${{S}}_{f_0}(\theta,\theta')$ is a minimal sector with}\\
\text{a fixed root point in $\partial A(f_0)$}
\end{array}
\right\}}. $$
Since $f_0$ has finitely many fixed points and each fixed point in $J(f_0)$ is the landing point of finitely many external rays, the set $\Theta(f_0)$ is finite. 
By the first assumption of $f_0$, for any $(\theta,\theta')\in\Theta(f_0)$, we have $d\cdot\theta \equiv \theta$ and $d\cdot\theta'\equiv\theta'\modZ$, where $d=\deg(f_0)$. 
% For $(\theta,\theta')\in\Theta(f_0)$, define $N(\theta,\theta')$ to be the number of points in ${{S}}_{f_0}(\theta,\theta')\cap\Tfp$. 
% Given $f\in\partial\mathcal{H}(f_0)$. 

%By the facts on limbs (see Proposition \ref{limb} and Lemma \ref{limb2}), 
%to prove every periodic point in $\bigcup_{v\in\Tfp}\partial U_{f,v}$ is repelling, 
%we just need to show 
%for any $(\theta,\theta')\in\Theta(f_0)$, the pair $(\theta,\theta')$ has the properties: 
%\begin{enumerate}[~~~~(P1)]
%\item $R_f(\theta)$ and $R_f(\theta')$ land at a common point; 
%\item $f$ has no immediate parabolic basin $U$ such that $U\subset S$ and ${{{r}}}(S)\in\partial U$, 
%where $S={{S}}_f(\theta,\theta')$. 
%\end{enumerate}

\vspace{6 pt}
{\bf Claim.} 
{\it 
Given $f\in\partial\mathcal{H}$. 
Then every pair $(\theta,\theta')\in\Theta(f_0)$ satisfies 
\begin{enumerate}[\indent (P1)]
\item $R_f(\theta)$ and $R_f(\theta')$ land at a common point; 
\item $f$ has no parabolic Fatou component $U$ such that $U\subset S$ and ${{{r}}}(S)\in\partial U$, 
where $S={{S}}_f(\theta,\theta')$. 
\end{enumerate} 
}
% \vspace{6 pt}

{\bf Step 2.} 
Assuming the claim holds, let us prove that every periodic point in $\partial A(f)$ is repelling. 
Let $v\in\Tfp$ and let $x$ be a periodic point of $f$ in $\partial U_{f,v}$. 
Then $x$ is either repelling or parabolic (see \cite{RY2022}). 
Let $t\in\QZ$ such that $R_{f,v}(t)$ lands at $x$. 
% Since $x$ is periodic, we have rational external angles $\theta_v^-(t)$ and $\theta_v^+(t)$ of $x$. 
% Thus $x$ is either repelling or parabolic. 
We have rational external angles $\theta_{f,v}^-(t)$ and $\theta_{f,v}^+(t)$ of $x$. 
% (see Lemma \ref{invariant external angles}). 

First, consider $\theta_{f,v}^-(t)=\theta_{f,v}^+(t)$. 
By Proposition \ref{limb}, we have $L_{U_{f,v},x} = \{x\}$. 
Assume $x$ is parabolic; we will find a contradiction. 
Let $U$ be a parabolic Fatou component such that $x\in\partial U$. 
Then $\overline{U}\subset L_{U_{f,v},x}$, which contradicts $L_{U_{f,v},x} = \{x\}$. 
So $x$ is not parabolic. Hence it is repelling. 

Second, consider $\theta_{f,v}^-(t)\neq\theta_{f,v}^+(t)$. 
By Lemma \ref{invariant external angles}, we have $\theta_{f,v}^-(t)=\theta_{f_0,v}^-(t)$ and $\theta_{f,v}^+(t)=\theta_{f_0,v}^+(t)$. 
Let $x_0$ denote the root point of ${{S}}_{f_0}(\theta_{f_0,v}^-(t),\theta_{f_0,v}^+(t))$. 
Since $x$ is periodic under $f$, the external angle  $\theta_{f_0,v}^-(t)=\theta_{f,v}^-(t)$ is periodic under $m_d$, hence $x_0$ is periodic under $f_0$. 
By the assumptions of $f_0$, we have $f_0(v)=v$ and $f_0(x_0)=x_0$. 
For the polynomial $f_0$, let $\theta_1=\theta_{f_0,v}^-(t), \theta_2,\dots,\theta_n=\theta_{f_0,v}^+(t)$ be all external angles of $x_0$ in positive cyclic order. 
Then $\{(\theta_j,\theta_{j+1})\}_{1\leq j\leq n}\subset\Theta(f_0)$, where we set $\theta_{n+1}=\theta_1$. 
By the property (P1) in the claim, the external rays $R_f(\theta_j)$ $(1\leq j\leq n)$ land at the same point $x$, which is not parabolic by (P2). 
Hence $x$ is repelling. 
This shows the lemma. 

{\bf Step 3.} 
To complete the proof, we show the claim. 
Given $(\theta_0,\theta'_0)\in\Theta(f_0)$. 
We will show that $(\theta_0,\theta'_0)$ satisfies (P1) and (P2). 
Let $S_0={{S}}_{f_0}(\theta_0,\theta'_0)$.  
The proof will be broken into two cases: $X(S_0)\not\subset\overline{S_0}$ and 
$X(S_0)\subset\overline{S_0}$. 

\vspace{6 pt}
{\bf Case 1.} 
{\it $X(S_0)\not\subset\overline{S_0}$.}
\vspace{6 pt}

Choose $\theta_n\rightarrow\theta_0$ and $\theta'_n\rightarrow\theta'_0$ as Proposition \ref{two sequences of external angles}. 
Note that $f$ has at most finitely many parabolic periodic points. 
There exists an $N>0$ such that $R_f(\theta_n)$ and $R_f(\theta'_n)$ land at repelling periodic points of $f$ whenever $n>N$. 
Then by Lemma \ref{limit-land-rela}(\ref{limit-land-rela-1}), for each $n>N$, the external rays $R_f(\theta_n)$ and $R_f(\theta'_n)$ land at a common point. 
% The \emph{rational lamination} $\lambda_\mathbb{Q}(f)\subset(\QZ)^2$ of $f$ is defined to be the landing relation of rational angles. 
By Lemma \ref{rat-lam is closed}, the external rays  $R_f(\theta_0)$ and $R_f(\theta'_0)$ land at a common point, 
i.e. $(\theta_0,\theta_0')$ satisfies (P1). 

Let $S={{S}}_f(\theta_0,\theta'_0)$. 
If $f$ has a parabolic Fatou component $U$ such that $U\subset S$ and ${{{r}}}(S)\in\partial U$, 
then the partition (\ref{partition of the plane}) of the plane
% the partition (\ref{partition of Kf}) of $K(f)$ 
contradicts $(\theta_n,\theta'_n)\in\lambda_\mathbb{Q}(f)$ for $n$ large enough. 
Thus $(\theta_0,\theta_0')$ satisfies (P2). 

% {\bf Case 2:} $X(S_0)\cap\Tfp=\emptyset$. 
\vspace{6 pt}
{\bf Case 2.} 
{\it $X(S_0)\subset\overline{S_0}$.}
\vspace{6 pt}

There are two subcases: 
$X(S_0)\cap\Tfp=\emptyset$ and $X(S_0)\cap\Tfp\neq\emptyset$. 

\vspace{6 pt}
{\bf Case 2.1.} 
{\it $X(S_0)\cap\Tfp=\emptyset$.}
\vspace{6 pt}

By Proposition \ref{two sequences of external angles 2}, 
there exists a unique $v\in S_0\cap\Tfp$ such that $X(S_0)\subset\overline{U_{f_0}(v)}$. 
By the minimality of $S_0$, there is a unique $t\in\RZ$ such that $\theta'_0=\theta_{f_0,v}^-(t)$ and $\theta_0=\theta_{f_0,v}^+(t)$. 
By Lemma \ref{invariant external angles}, we have $\theta'_0=\theta_{f,v}^-(t)$ and $\theta_0=\theta_{f,v}^+(t)$. 
% Thus $R_f(\theta_0)$, $R_f(\theta'_0)$ and $R_{f,v}(t)$ land at the same point $x:=B_{f,v}^{-1}(e^{2\pi it})$, so $(\theta_0,\theta_0')$ satisfies (P1).  
Since $R_f(\theta_0)$, $R_f(\theta'_0)$ and $R_{f,v}(t)$ land at a common point, the pair $(\theta_0,\theta_0')$ satisfies (P1).  
% Since $\overline{{{S}}_f(\theta'_0,\theta_0)}\cap K(f)=L_{U_{f,v},x}$, the pair $(\theta_0,\theta_0')$ satisfies (P2). 
Since ${{{S}}_f(\theta'_0,\theta_0)}$ is a maximal sector attaching to $U_{f,v}$, the pair $(\theta_0,\theta_0')$ satisfies (P2). 

% {\bf Case 3:} $X(S_0)\subset\overline{S_0}$ and $X(S_0)\cap\Tfp\neq\emptyset$. 
\vspace{6 pt}
{\bf Case 2.2.} 
{\it $X(S_0)\cap\Tfp\neq\emptyset$.}
\vspace{6 pt}

For $(\theta,\theta')\in\Theta(f_0)$, 
% define $N(\theta,\theta')$ to be the number of points in ${{S}}_{f_0}(\theta,\theta')\cap\Tfp$. 
define $$N(\theta,\theta') = \#({{S}}_{f_0}(\theta,\theta')\cap\Tfp).$$
Note that $\Theta(f_0)$ is a finite set. 
Assume $(\theta,\theta')$ satisfies (P1) and (P2) if $N(\theta,\theta')<N(\theta_0,\theta'_0)$. 
% Otherwise we consider $(\theta,\theta')$ with $N(\theta,\theta')<N(\theta_0,\theta'_0)$ first. 
Otherwise, we replace $(\theta_0,\theta'_0)$ by some $(\theta,\theta')\in\Theta(f_0)$ with $N(\theta,\theta')<N(\theta_0,\theta'_0)$. 

Let $V_0=X(S_0)\cap\Tfp$ and let $x_0={{{r}}}(S_0)$. 
For each $v\in V_0$, denote the singleton $[x_0,v]\cap\partial U_{f_0}(v)$ by $\{x_v\}$.  
Then $x_v\neq x_0$. 
We will divide Case 2.2 into two subcases 
according to whether there is a point $v\in V_0$ such that $f_0(x_v)\neq x_v$. 

% {\bf Subcase 3.1:} $f_0(x_v)\neq x_v$ for some $v\in V_0$.  
\vspace{6 pt}
{\bf Case 2.2.1.} 
{\it $f_0(x_v)\neq x_v$ for some $v\in V_0$.}
\vspace{6 pt}

%Let $S = {{S}}_{f_0}(\theta,\theta')$ be a minimal sector with a fixed root point in  $\partial U_{f_0}(v)$. 
%If $U_{f_0}(v)\subset S$, then $(\theta,\theta')$ satisfies (P1) and (P2) by Case 2. If $U_{f_0}(v)\cap S=\emptyset$, then $S\subset S_0$, otherwise $r(S) = x_v$, which contradicts $f_0(x_v)\neq x_v$. 
%Thus $N(\theta,\theta')<N(\theta_0,\theta'_0)$ and $(\theta,\theta')$ satisfies (P1) and (P2). 
Let $S = {{S}}_{f_0}(\theta,\theta')$ be a minimal sector with a fixed root point in  $\partial U_{f_0}(v)$. 
Then one of the following holds: 
\begin{itemize}
\item $U_{f_0}(v)\subset S$ (then $X(S)\cap\Tfp=\emptyset$, corresponding to Case 2.1); 
\item $U_{f_0}(v)\cap S=\emptyset$ (then $S\subset S_0$ by $r(S) \neq x_v$, so $N(\theta,\theta')<N(\theta_0,\theta'_0)$). 
\end{itemize}
% Thus $(\theta,\theta')$ satisfies (P1) and (P2). 
By Case 2.1 (the former) or induction (the latter), 
the pair $(\theta,\theta')$ satisfies (P1) and (P2). 
% Similar to the next paragraph of the claim, 
Similar to Step 2, by Proposition \ref{limb} and Lemma \ref{invariant external angles}, 
every periodic point of $f$ in $\partial U_{f,v}$ is repelling. 

Choose $\theta_n\rightarrow\theta_0$, $\theta'_n\rightarrow\theta'_0$ and $m\in\mathbb{N}$ for $v$
as Proposition \ref{two sequences of external angles 2}. 
Fix $n\geq1$. 
There exists a unique $t_n\in\QZ$ such that 
$d^{m+n}\theta_n\equiv\theta^-_{f_0,v}(t_n)$ and $d^{m+n}\theta'_n\equiv\theta^+_{f_0,v}(t_n)\modZ$. 
% (see \S\ref{subsection internal rays} for $\theta^-_v(t_n),\theta^+_v(t_n)$). 
By Lemma \ref{invariant external angles}, we have $d^{m+n}\theta_n\equiv\theta^-_{f,v}(t_n)$ and $d^{m+n}\theta'_n\equiv\theta^+_{f,v}(t_n)\modZ$. 
Since every periodic point of $f$ in $\partial U_{f,v}$ is repelling, 
the external rays $R_f(\theta_n)$ and $R_f(\theta'_n)$ land at pre-repelling points. 
By Lemma \ref{limit-land-rela}(\ref{limit-land-rela-1}), they land at the same point. 
% By an argument similar to the proof of Case 1, the pair $(\theta_0,\theta_0')$ satisfies (P1) and (P2). 
As the proof of Case 1, 
the sequence $\{(R_f(\theta_n),R_f(\theta'_n))\}_{n\geq1}$ 
of pairs of external rays ensures that $(\theta_0,\theta_0')$ satisfies (P1) and (P2).

% {\bf Subcase 3.2:} $f_0(x_v)= x_v$ for all $v\in V_0$. 
\vspace{6 pt}
{\bf Case 2.2.2.} 
{\it $f_0(x_v)= x_v$ for all $v\in V_0$.}
\vspace{6 pt}

Let $\{x_k\}_{k=1}^{{m}}$ be all different points in $\{x_v\}_{v\in V_0}$ such that $[x_0,x_v)\cap \{x_w {;~}  w\in V_0\}=\emptyset$. 
Fix $1\leq k\leq {{m}}$. 
Let $S_k={{S}}_{f_0}(\theta_k,\theta_k')$ be the minimal sector with root point $x_k$ such that $[x_0,x_k)\subset S_k$. 
By Lemma \ref{allowable arcs in K(f0)}(\ref{allowable arcs in K(f0)-2}), 
the sector $S_k$ is well-defined and $\theta_0,\theta'_k,\theta_k,\theta'_0$ are in positive cyclic order. 
By $f_0(x_k)=x_k$ and the first assumption of $f_0$, we have $d\cdot\theta_k\equiv\theta_k\modZ$. 
Thus $X(S_k)$ makes sense (see Lemma \ref{X(S) is allowable}) and $X(S_k)\subset X(S_0)$ according to  
$$[x_0,x_k]\subset \bigcup_{v\in V_0}[x_0,x_v] \subset \bigcup_{v\in V_0}[x_0,v]\subset X(S_0).$$
Let $V_k=X(S_k)\cap\Tfp$. Then $V_k\subset V_0$. 
By Proposition \ref{two sequences of external angles 2}, the set $V_k$ is nonempty. 

By the choice of $\{x_k\}_{k=1}^{{m}}$, 
the plane has a partition $$\mathbb{C}={\left(\bigcap_{k=0}^{{m}} S_k\right)} \bigsqcup{{\left(\bigsqcup_{k=0}^{{m}}(\mathbb{C}\setminus S_k)\right)}}.$$
For each $v\in V_0$, there exists a unique $1\leq k\leq {{m}}$ such that $x_k\in(x_0,x_v]$, so $v\in\mathbb{C}\setminus S_k$. 
This implies $V_0\subset\bigsqcup_{k=1}^{{m}}(\mathbb{C}\setminus S_k)$. 

% It is worth noting that $(\theta_k,\theta_k')\in\Theta(f_0)$ for $1\leq k\leq{{m}}$, 
% but $N(\theta_k,\theta_k')<N(\theta_0,\theta_0')$ may not hold. 
We claim that there is at least one integer $1\leq k\leq{{m}}$ such that $X(S_k)\not\subset\overline{S_k}$ (corresponding to Case 1). 
Suppose $X(S_k)\subset\overline{S_k}$ for all $1\leq k\leq{{m}}$; we will find a contradiction.  
By $V_1\subset X(S_1)\subset \overline{S_1}$ and $V_1\subset V_0\subset\bigsqcup_{k=1}^{{m}}(\mathbb{C}\setminus S_k) $, we have $V_1\subset\bigsqcup_{k=2}^{{m}}(\mathbb{C}\setminus S_k)$. 
Without loss of generality, we assume $V_1\cap(\mathbb{C}\setminus S_2)\neq\emptyset$. 
Choose $v_1\in V_1\cap(\mathbb{C}\setminus S_2)$. 
By Lemma \ref{common part}, we can choose $y_2\in(x_2,x_1]$ such that $X(S_2)=\bigcup_{n\geq0}f_0^{n}([x_2,y_2])$. 
Then $X(S_2)\subset\bigcup_{n\geq0}f_0^{n}([x_1,v_1])\subset X(S_1)$ because $X(S_1)$ is allowable (Lemma \ref{X(S) is allowable}). 
Thus $V_2\subset V_1$. 
By $V_2\subset X(S_2)\subset \overline{S_2}$ and $V_2\subset V_1\subset\bigsqcup_{k=2}^{{m}}(\mathbb{C}\setminus S_k) $, we have $V_2\subset\bigsqcup_{k=3}^{{m}}(\mathbb{C}\setminus S_k)$. 
Repeating the above procedure several times, we have $V_{{m}}\subset\emptyset$, a contradiction.  

Now choose $1\leq k\leq{{m}}$ such that $X(S_k)\not\subset\overline{S_k}$ and choose $v\in V_0$ such that $x_v=x_k$. 
For any minimal sector $S = {{S}}_{f_0}(\theta,\theta')$ with a fixed root point in  $\partial U_{f_0}(v)$, 
at least one of the following holds: 
\begin{itemize}
\item $S=S_k$ (corresponding to Case 1); 
\item $U_{f_0}(v)\subset S$ (then $X(S)\cap\Tfp=\emptyset$, corresponding to Case 2.1); 
\item $S\subset S_0$ (then $N(\theta,\theta')<N(\theta_0,\theta'_0)$).  
\end{itemize}
% Thus $(\theta,\theta')$ satisfies (P1) and (P2). 
By Case 1, Case 2.1 or induction, the pair $(\theta,\theta')$ satisfies (P1) and (P2). 
By the same discussion in Case 2.2.1, 
we can first deduce that every periodic point of $f$ in $\partial U_{f,v}$ is repelling, and then prove that $(\theta_0,\theta_0')$ satisfies (P1) and (P2). 
\end{proof}

\begin{example}
Consider $(f_0,\mathbf{c})$ in Example \ref{example of f0}. 
Let $g_0=f_0^2$. 
Then $g_0$ has the properties in Lemma \ref{an iterate of f0}. 
We have $|T(g_0)|_{\mathrm{p}} = \{c_1,c_2,f_0(c_2),c_3\}$, illustrated in Figure \ref{figure some minimal sectors of f0}.  
Recall that $x$ is the landing point of $R_{f_0}(0)$. 

Case 1. Let $S_0 = S_{g_0}(1/2,3/4)$. 
Then $X(S_0) = [f_0(c_4),c_3]$. 
% It follows that $X(S_0)\not\subset \overline{S_0}$. 

Case 2.1. Let $S_0 = S_{g_0}(0,1/4)$. 
Then $X(S_0) = [x, c_1)$. 
% It follows that $X(S_0)\subset \overline{S_0}$ and $X(S_0)\cap |T(g_0)|_{\mathrm{p}}=\emptyset$. 

Case 2.2.1. Let $S_0 = S_{g_0}(1/4,1/2)$. 
Then $X(S_0) = [x, f_0(c_2)]$. 
It follows that $V_0 = X(S_0)\cap|T(g_0)|_{\mathrm{p}} = \{c_2,f_0(c_2)\}$ and $x_0 = r(S_0) = x $. 
Furthermore, for $c_2\in V_0$, we have $g_0(x_{c_2})\neq x_{c_2}$. 
\end{example}

\begin{example}
Let $(f_0,\mathbf{c})\in\widehat{\mathcal{P}}^5$ satisfy  $f_0(c_1)=c_1$, $f_0(c_2)=c_2$, 
$f_0^2(c_3)=c_1$ and $c_4=c_3$ as in Figure \ref{figure allowable set}. 
\footnote{
$(\mathbf{c},f_0(0))\approx(-0.7723-0.9019i,~ 
0.9038+0.7137i,~ -0.0658+0.0941i,~ -0.0658+0.0941i,\\~  -0.4673+0.7881i)$.
}
Then $f_0$ is a postcritically finite hyperbolic polynomial. 
We have $\Tfp = \{c_1,c_2\}$. 

Case 2.2.2. Let $S_0 = S_{f_0}(1/4,0)$, and let $x_0 = r(S_0)$. 
Then $X(S_0) = [x_0,c_1]\cup [c_3,f_0(c_3)]$, illustrated in Figure \ref{figure allowable set}.  
It follows that $V_0 = X(S_0)\cap\Tfp = \{c_1\}$. 
For $c_1\in V_0$, we see that $x_{c_1}$ is the landing point of $R_{f_0}(1/2)$. 
Therefore $f_0(x_{c_1})= x_{c_1}$. 

\begin{figure}[ht]
\centering
\includegraphics{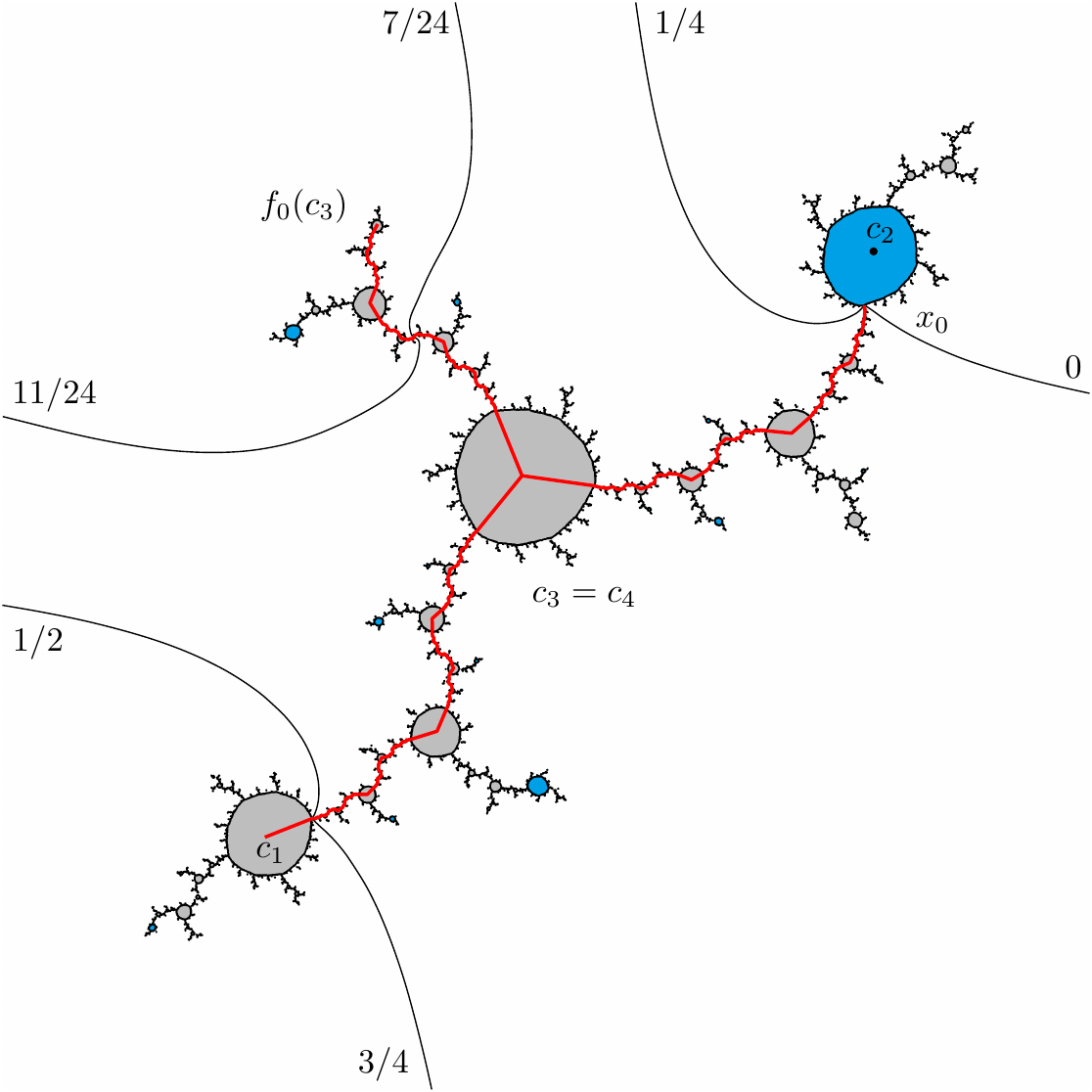}
\caption{The allowable set $X({{S}}_{f_0}(1/4,0))$.}
\label{figure allowable set}
\end{figure}
\end{example}

\subsection{Proof of Proposition \ref{behavior of critical orbits}}

Recall that for $f\in\Hbar$ and $v\in\Tfp$, we define the  B\"ottcher coordinate  $B_{f,v}:\overline{U_{f,v}}\rightarrow \overline{\mathbb{D}}$ in \S\ref{subsection internal rays}. 
Recall also that for each $v\in\Tfp$, we define $\Psi_v:\Hbar\times\overline{\mathbb{D}}\rightarrow \mathbb{C}$ by $\Psi_v(f,z) = B^{-1}_{f,v}(z)$. 

\begin{proposition}
\label{continuity of Ufv}
For each $v\in\Tfp$, the function $\Psi_v$ is continuous. 
\end{proposition}

\begin{proof}
By Lemma \ref{continuity on D}, the function $\Psi_v$ is continuous on $\Hbar\times\mathbb{D}$. 
% By the stability of internal rays (Lemmas  \ref{key lemma} and \ref{stability of internal rays}), for any $t\in\QZ$, the restriction $\Psi: \piHbar\times[0,e^{2\pi i t}]\rightarrow\mathbb{C}$ is continuous. 
Now let $(f,z)\in \Hbar\times\partial\mathbb{D}$ 
and $\{(f_n,z_n)\}_{n\geq1}\subset \Hbar\times\overline{\mathbb{D}}$. 
Suppose $(f_n,z_n)\rightarrow(f,z)$ and $\Psi_v(f_n,z_n)\rightarrow w$ as $n\rightarrow\infty$. 
We will show $w = \Psi_v(f,z)$, which implies $\Psi_v$ is continuous at $(f,z)$.  

Choose $t\in\RZ$ so that $z=e^{2\pi i t}$. 
Given $\varepsilon>0$. Choose $t_1,t_2\in\QZ$ such that $t_1,t,t_2$ are in positive cyclic order and $\operatorname{frac}(t_2-t_1)<\varepsilon$. 
Let $X_\varepsilon$ be the closure of the connected component of 
$$\mathbb{C}\setminus\Big(  B^{-1}_{f,v}((1-\varepsilon)\partial\mathbb{D})
% \Psi\big(f,(1-\varepsilon)\partial\mathbb{D}\big)
\cup \overline{R_{f,v}(t_1)\cup R_{f,v}(t_2)\cup R_f(\theta_v^-(t_1)) \cup R_f(\theta_v^-(t_2))}
\Big)$$ 
containing $\Psi_v(f,z)$;  
see Lemma \ref{invariant external angles} for $\theta^-_v$ and $\theta^+_v$. 
By the continuity of $\Psi_v$ on $\Hbar\times ((1-\varepsilon)\partial\mathbb{D})$ 
and the stability of dynamical rays (Lemmas \ref{key lemma}, \ref{stability of external rays} and \ref{stability of internal rays}), 
we have $\Psi_v(f_n,z_n)\in X_\varepsilon$ for $n$ large enough, so $w\in X_\varepsilon$. 
By the upper semi-continuity of filled Julia sets, we have $w\in K(f)$. 
Then (\ref{limb1}) gives 
$$w\in K(f)\cap\bigcap_{\varepsilon>0} X_\varepsilon = L_{U_{f,v},\Psi_v(f,z)} =:L.$$

If $L = \{\Psi_v(f,z)\}$, then $w = \Psi_v(f,z)$. 

If $L \neq \{\Psi_v(f,z)\}$, then $\theta^-_v(t)\neq\theta^+_v(t)$. 
Note that $S_{f_0}(\theta^-_v(t),\theta^+_v(t))$ is a sector attaching to $U_{f_0}(v)$. 
By Lemma \ref{sectors attaching to U}, the angles $t$, $\theta^-_v(t)$ and $\theta^+_v(t)$ are all rational. 
Since $\Psi_v(f_n,z_n)\in \overline{S_{f_n}(\theta^+_v(t),\theta^-_v(t))}$, 
by the stability of external rays (Lemmas  \ref{key lemma} and \ref{stability of external rays}), 
we have $w\in \overline{S_f(\theta^+_v(t),\theta^-_v(t))}$. 
Then $L\cap \overline{S_f(\theta^+_v(t),\theta^-_v(t))} = \{\Psi_v(f,z)\}$ gives $w = \Psi_v(f,z)$. 
\end{proof}

By Proposition \ref{continuity of Ufv}, 
we can describe the behavior of the critical orbits for polynomials in  $\partial\mathcal{H}$. 
%the boundary of a capture hyperbolic component. 

\begin{proof}
[Proof of Proposition \ref{behavior of critical orbits}]
Let $f_0$ be the center of $\mathcal{H}$. 
Take $n$ to be an integer such that $f_0^{n}(\Tfn)\subset\Tfp$. 
Then $g^{n}(\crit(g))\subset A(g)$ for any $g\in\mathcal{H}$. 
The first claim follows from Proposition \ref{continuity of Ufv} and the continuity of critical sets.  

By the first claim, the polynomial $f$ has no parabolic cycle and no Siegel cycle in $\mathbb{C}$, 
and no Cremer point outside $\partial A(f)$. 
% By the snail lemma (see \cite[Lemma 16.2]{Mil}), 
% there is no Cremer point of $f$ in $\bigcup_{v\in\Tfp}\partial U_{f,v}$. 
Combining this with Lemma \ref{key lemma}, the second claim holds. 

Let $c$ be a critical point of $f$ in $\partial U_{f,v}$ for some $v\in\Tfp$. 
Since the limb $L_{U_{f,v},c}$ is nontrivial, we have $\theta_{f,v}^-(t)\neq\theta_{f,v}^+(t)$, where $t$ is the internal angle of $c$ with respect to $U_{f,v}$. 
By Lemma \ref{invariant external angles}, we have $\theta_{f,v}^-(t)=\theta_{f_0,v}^-(t)$ and $\theta_{f,v}^+(t)=\theta_{f_0,v}^+(t)$. 
Note that $S_{f_0}(\theta^-_{f_0,v}(t),\theta^+_{f_0,v}(t))$ is a sector attaching to $U_{f_0}(v)$. 
By Lemma \ref{sectors attaching to U}, the angles $\theta^-_{f_0,v}(t)$ and $\theta^+_{f_0,v}(t)$ are rational. 
It follows that $c$ is preperiodic. 
This shows the third claim. 
\end{proof}

The following example shows there may be critical points in $\partial A(f)$ for some $f\in\partial\mathcal{H}$. 
%%% example
\begin{example}
\label{example F2}
% Consider $\mathcal{F}=\{(f,\mathbf{c})\in\widehat{\mathcal{P}}^4{;~}f(c_1)=c_1, f(c_2)=c_2\}$. 
Let $$\mathcal{F}=\big\{(\mathbf{c},a)\in \mathbb{C}^4{;~}c_1+c_2+c_3=0, f_{\mathbf{c},a}(c_1)=c_1, f_{\mathbf{c},a}(c_2)=c_2\big\}.$$
Then $\mathcal{F}$ has two components: 
\begin{align*}
\mathcal{F}_1 & = \big\{(\mathbf{c},a)\in\mathcal{F}{;~} c_1=c_2\big\} \\
& = {\left\{{\left(-\frac{c}{2}, -\frac{c}{2}, c, -\frac{8c + 3 c^4}{16}\right)}{;~} c\in\mathbb{C}\right\}}\cong\mathbb{C}, \\
\mathcal{F}_2 & = \big\{(\mathbf{c}.a)\in\mathcal{F}{;~} c_1\neq c_2\big\} \\
& = {\left\{{\left(\frac{-1-b^3}{2b},\frac{1-b^3}{2b},b^2, \frac{(1+3b^6)(1-b^6)}{16b^4} \right)}{;~} b\in\mathbb{C}\setminus\{0\}\right\}}\cong  \mathbb{C}\setminus\{0\}.
\end{align*}
See Figure \ref{figure F2} for $\mathcal{F}_2$. 

\begin{figure}[ht]
\centering
\includegraphics{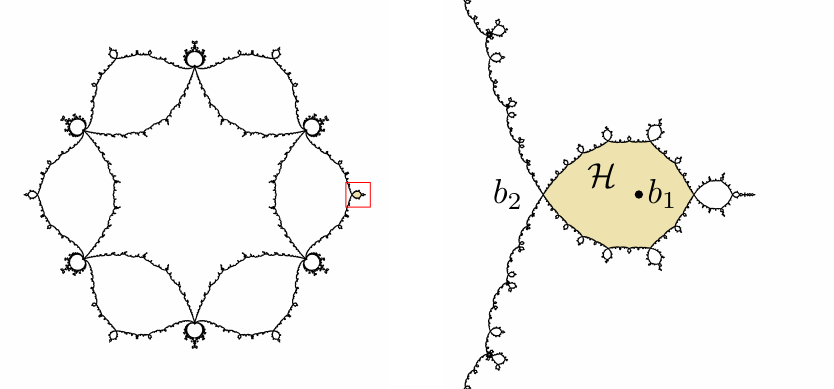}
\caption{The non-hyperbolic locus and a capture hyperbolic component $\mathcal{H}$
in $\mathcal{F}_2$, projected into the $b$-plane.}
\label{figure F2}
\end{figure}

Let $f_1(z)=(z^2-1)^2-1=z^4-2z^2$ and $\mathbf{c}_1=(-1,0,1)$. 
Then $J(f_1)=J(z^2-1)$ is the basilica, and $(f_1,\mathbf{c}_1)$ corresponds to $b_1=1$. 
Let $\mathcal{H}\subset\mathcal{F}_2$ be the capture hyperbolic component with center $(f_1,\mathbf{c}_1)$. 
For any $(f,\mathbf{c})\in\mathcal{H}$, we have $f(c_3)\in U_f(c_1)$. 
Let $(f_2,\mathbf{c}_2)\in\partial\mathcal{H}$ correspond to $b_2 = \sqrt[3]{2+\sqrt{7}} / \sqrt{3}$. 
Then $f_2(c_{2,3})$ is a repelling fixed point of $f_2$, and $c_{2,3}\in\partial U_{f_2}(c_{2,2})$. 
See Figure \ref{figure basilica}. 
% The closed Jordan disks $\overline{U_{f_0}(-1)}$, $\overline{U_{f_0}(0)}$ and $\overline{U_{f_0}(1)}$ become $K_{f,-1} = \overline{U_f(c_1)}$, $K_{f,0}=\overline{U_f(c_2)}$ and a compact subset $K_{f,1}$ of $K(f)$ with figure eight shape respectively, where $(K_{f,v})_{v\in|T(f_0)|}$ will be considered in \S\ref{section straightening map}. 

\begin{figure}[ht]
\centering
\includegraphics{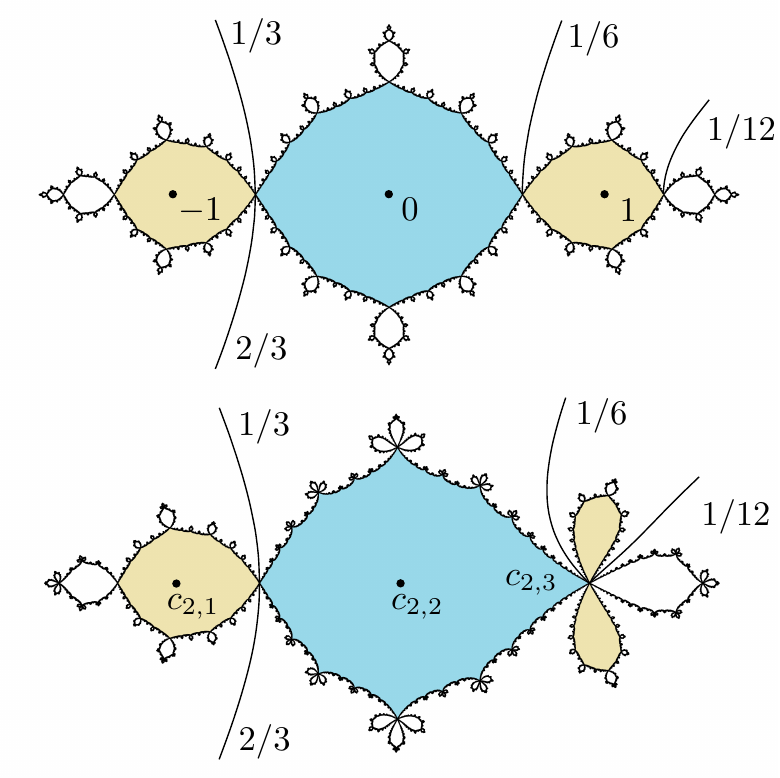}
\caption{Julia sets for $(f_1,\mathbf{c}_1)\in\mathcal{H}$ and $(f_2,\mathbf{c}_2)\in\partial\mathcal{H}$.} 
\label{figure basilica}
\end{figure}
\end{example}

\subsection{Semi-hyperbolicity}
In this subsection, we will investigate more properties for $f\in\Hbar$ based on the behavior of critical orbits. 

For $z\in\mathbb{C}$ and $r>0$, let $\mathbb{D}(z,r) = \{w\in\mathbb{C}{;~}|w-z|<r\}$. 
Let ${\diam}(A)=\sup\{|x-y| {;~}  x,y\in A\}$ be the diameter of $A\subset\mathbb{C}$. 

\begin{proposition}
[Characterizations of semi-hyperbolicity \cite{CJY}]
\label{equiv-condi semi-hyper}
Let $f$ be a polynomial of degree $d\geq2$. 
Then the following conditions are equivalent. 
\begin{enumerate}
\item $f$ has no parabolic cycle and $c\notin\overline{\{f^{n}(c) {;~}  n\geq 1\}}$ whenever $c\in\crit(f)\cap J(f)$. 
\item $U_{f,\infty}$ is a John domain. 
\footnote{
For the definition of John domain, we refer to \cite{CJY} or \cite[\S5.2]{Pom}. 
}
\item There is an $r>0$ and a $D<\infty$ such that for every $x\in J(f)$, every $n\in\mathbb{N}$  
and every connected component $V$ of $f^{-n}(\mathbb{D}(x,r))$, we have $\deg(f^n|_V)\leq D$. 
\end{enumerate}
\end{proposition}

We say $f$ is \emph{semi-hyperbolic} if any of the equivalent conditions in Proposition \ref{equiv-condi semi-hyper} are satisfied.

\begin{lemma}
[{\cite{CJY}}]
\label{properties semi-hyper}
Let $f$ be a semi-hyperbolic polynomial of degree $d\geq2$. 
Then $f$ has the following properties. 
\begin{enumerate}
\item If $K(f)\neq J(f)$, then every bounded Fatou components of $f$ is a quasidisk. 
\item There exists an $r>0$, a $\lambda>1$ and an $M>0$ such that for every $x\in J(f)$, every $n\in\mathbb{N}$ and every connected component $V$ of $f^{-n}(\mathbb{D}(x,r))$, we have  ${\diam}(V)\leq M\lambda^{-n}$.
\end{enumerate}
\end{lemma}

We say that a compact set $K\subset\mathbb{C}$ is \emph{conformally removable}, 
if any homeomorphism of $\mathbb{C}$, which is conformal on $\mathbb{C}\setminus K$, 
is conformal on $\mathbb{C}$.

\begin{corollary}
[Semi-hyperbolicity]
\label{semi-hyperbolicity}
Let $\mathcal{H}\subset\mathcal{F}$ be a capture hyperbolic component, and let $f\in\Hbar$. 
Then $f$ is semi-hyperbolic. 
Moreover, the Julia set $J(f)$ is locally connected and conformally removable. 
\end{corollary}

\begin{proof}
By Proposition \ref{behavior of critical orbits}, 
the polynomial $f$ has no parabolic cycle and $c\notin\overline{\{f^{n}(c) {;~}  n\geq 1\}}$ whenever $c\in\crit(f)\cap J(f)$, hence $f$ is semi-hyperbolic. 
By Proposition \ref{equiv-condi semi-hyper}, the basin $U_{f,\infty}$ is a John domain. 
By \cite[Corollary 1]{Jones} or \cite[Corollary 1]{JS}, the boundary of a John domain is conformally removable. 
So $J(f)=\partial U_{f,\infty}$ is conformally removable. 
Since $J(f)$ is connected, it is locally connected.
\end{proof}

\begin{corollary}
\label{inclusion relation of real laminations}
Let $\mathcal{H}\subset\mathcal{F}$ be a capture hyperbolic component with center $f_0$, and let $f\in\Hbar$. 
Then $\lambda_\mathbb{R}(f_0)\subset\lambda_\mathbb{R}(f)$. 
% Furthermore, let $(\theta,\theta')\in\lambda_\mathbb{R}(f)\setminus\lambda_\mathbb{R}(f_0)$ and let $x$ denote the landing point of $R_f(\theta)$ and $R_f(\theta')$. 
% Then we have $\{f^n(x){;~}n\in\mathbb{N}\}\cap\crit(f)\neq\emptyset$. 
%
%
%Let $(\theta,\theta')\in\lambda_\mathbb{R}(f)$, and let $x$ denote the landing point of $R_f(\theta)$ and $R_f(\theta')$. 
%Suppose $x$ is not on the boundary of any bounded Fatou component of $f$. 
%Then we have $(\theta,\theta')\in\lambda_\mathbb{R}(f_0)$. 
\end{corollary}

\begin{proof}
By Corollary \ref{semi-hyperbolicity}, 
the Julia set $J(f)$ is locally connected, hence the real lamination $\lambda_\mathbb{R}(f)$ of $f$ makes sense. 
By Proposition \ref{behavior of critical orbits}, the polynomial $f$ has no indifferent cycle. 
By Lemma \ref{limit-land-rela}(\ref{limit-land-rela-1}), we have $\lambda_\mathbb{Q}(f_0)\subset\lambda_\mathbb{Q}(f)$. 
It follows that $\lambda_\mathbb{Q}(f_0)\subset\lambda_\mathbb{R}(f)$. 
By Lemma \ref{connection between rational lamination and real lamination}, taking $\langle\cdot\rangle_\mathbb{R}$ on both sides of $\lambda_\mathbb{Q}(f_0)\subset\lambda_\mathbb{R}(f)$ gives $\lambda_\mathbb{R}(f_0)\subset\lambda_\mathbb{R}(f)$. 
\end{proof}

\section{Generalized polynomials}
\label{section generalized polynomials}
% \subsection{Generalized polynomials}
% \label{subsection generalized polynomials}
This section is the second step to prove Theorem \ref{top-boundary}. 
We will consider a model space for the closure of a capture hyperbolic component in the language of generalized polynomials, and show that the boundary of the model space is a topological sphere.

\begin{definition}
[Generalized polynomial]
Let $T = (|T|,\sigma,\delta)$ be a mapping scheme (see Definition \ref{mapping scheme}). 
A \emph{generalized polynomial} $\mathbf{f}$ over $T$ is a map $\mathbf{f}:|T|\times\mathbb{C}\rightarrow|T|\times\mathbb{C}$ 
such that $\mathbf{f}(v, z) = (\sigma(v), f_v(z))$, where $f_v\in\mathcal{P}^{\delta(v)}$, 
i.e. $f_v$ is a monic centered polynomial of degree $\delta(v)$. 
We may also write it as a collection $\mathbf{f}=(f_v)_{v\in |T|}$. 
\end{definition}

Let $\mathcal{P}^T$ be the family of all generalized polynomials over $T$.
Given $\mathbf{f}\in\mathcal{P}^T$.
The \emph{filled Julia set} $K(\mathbf{f})$ of $\mathbf{f}$ is the set of points in $|T|\times\mathbb{C}$ with
precompact forward orbits. The boundary $\partial K(\mathbf{f})$ is called the \emph{Julia set} $J(\mathbf{f})$ of $\mathbf{f}$. 
We say that a subset $A$ of $|T|\times\mathbb{C}$ is \emph{fiberwise connected} 
if $A\cap(\{v\}\times\mathbb{C})$ is connected for any $v\in|T|$.
The \emph{fiberwise connectedness locus} $\mathcal{C}(\mathcal{P}^T)$ of $\mathcal{P}^T$ consists of all $\mathbf{f}\in\mathcal{P}^T$ 
with fiberwise connected filled Julia sets.
% Let $K(\mathbf{f},v)=K(\mathbf{f})\cap(\{v\}\times\mathbb{C})$ denote the $v$-fiber of the filled Julia set $K(\mathbf{f})$.

The critical marked space $\widehat{\mathcal{P}}^T$ consists of 
$(f_v)_{v\in|T|}\in\mathcal{P}^T$ together with a collection $(\mathbf{c}_v = (c_{v,1},\dots,c_{v,\delta(v)-1}))_{v\in|T|,\ \delta(v)\geq2}$, where $\mathbf{c}_v$ is an ordered list of the critical points of $f_v$. 
Then we can identify $\widehat{\mathcal{P}}^T$ with the complex affine space 
%$$\prod_{v\in|T|}{\left\{
%(\mathbf{c}_v,a_v)\in\mathbb{C}^{\delta(v)}{;~}
%\begin{split}
%&\sum_{k=1}^{\delta(v)-1}c_{v,k}=0 \text{~~if~~} \delta(v)\geq2,\\
%& a_v=0             \text{~~if~~} \delta(v)=1
%\end{split}\right\}}
%$$
$$\prod_{v\in|T|,\ \delta(v)\geq2}\big\{
(\mathbf{c}_v,a_v)\in\mathbb{C}^{\delta(v)}{;~}
c_{v,1}+\cdots+c_{v,\delta(v)-1}=0 \big\}. 
$$
% by fixing an order of vertices in $|T|$. 
Each $\mathbf{x} = (\mathbf{c}_v,a_v)_{v\in|T|}\in\widehat{\mathcal{P}}^T$  determines a generalized polynomial 
$\mathbf{f}_\mathbf{x}=(f_{\mathbf{c}_v,a_v})_{v\in|T|}$ in $\mathcal{P}^T$, where 
$$f_{\mathbf{c}_v,a_v}(z)=\delta(v)\int_{0}^{z}
(\zeta-c_{v,1})\cdots(\zeta-c_{v,\delta(v)-1})d\zeta +a_v.$$
If $\delta(v)=1$, then $\mathbf{c}_v$ disappears, $a_v=0$, and $f_{\mathbf{c}_v,a_v}(z) = z$. 
Define $$\widehat{\mathcal{P}}_0^T = 
\big\{(\mathbf{c}_v,a_v)_{v\in|T|}\in\widehat{\mathcal{P}}^T{;~} 
\text{$f_{\mathbf{c}_v,a_v}(z) = z^{\delta(v)}$ for any $v\in\Tp$}\big\}.$$ 
Each $\mathbf{x} = (\mathbf{c}_v,a_v)_{v\in|T|}\in\widehat{\mathcal{P}}_0^T$ can be written as $\mathbf{x} = (\mathbf{c}_v,a_v)_{v\in\Tn}$. 
Let $\model = \widehat{\mathcal{P}}_0^T \cap \mathcal{C}\big(\widehat{\mathcal{P}}^T\big)$ denote the \emph{fiberwise connectedness locus} of $\widehat{\mathcal{P}}_0^T$. 
% Let $\mathcal{H}_0\big(\widehat{\mathcal{P}}^{T}\big)$ (resp. $\mathcal{H}_0\big(\widehat{\mathcal{P}}_0^{T}\big)$) denote the \emph{central hyperbolic component} of  $\widehat{\mathcal{P}}^{T}$ (resp. $\widehat{\mathcal{P}}_0^{T}$), 
% i.e. the hyperbolic component containing $0$. 

By the Riemann-Hurwitz formula, 
we have 
\begin{lemma}
\label{connectedness-crit}
Let $f$ be a polynomial with $\deg(f)\geq2$, and let $K\subset\mathbb{C}$ be a full nondegenerate continuum. 
Then $f^{-1}(K)$ is connected if and only if $\crit(f)\subset f^{-1}(K)$. 
\end{lemma}

\begin{proposition}
\label{fiberwise connectedness locus C(P0T)}
Let $T=(|T|,\sigma,\delta)$ be a mapping scheme with 
$n:=\sum_{v\in\Tn}(\delta(v)$ $-$ $1)>0$. 
Then $\partial\model$ is homeomorphic to the sphere $S^{2n-1}$.  
\end{proposition}

\begin{proof}
Let $$I = \big\{(v,k){;~} v\in\Tn, 1\leq k<\delta(v)\big\}.$$
Then $\#(I)=n>0$. 
Recall that for each $v\in|T|$,  $r_v\geq0$ is the minimal integer such that $\sigma^{r_v}(v)\in\Tp$. 
For $(v,k)\in I$ and $\mathbf{x} = (\mathbf{c}_v,a_v)_{v\in\Tn}\in \widehat{\mathcal{P}}_0^T$, let $$w_{v,k}(\mathbf{x})=f_{\mathbf{x}_{\sigma^{r_v-1}(v)}}\circ\cdots\circ f_{\mathbf{x}_{\sigma(v)}}\circ f_{\mathbf{x}_v}(c_{v,k}).$$ 
Define
$${{M}}:  \begin{cases} 
\widehat{\mathcal{P}}_0^T \rightarrow  [0,\infty), \\
\mathbf{x} \mapsto   \max_{(v,k)\in I}|w_{v,k}(\mathbf{x})|. 
\end{cases}$$
By Lemma \ref{connectedness-crit}, we have 
$$\model=\big\{\mathbf{x}\in\widehat{\mathcal{P}}_0^T {;~} 
{{M}}(\mathbf{x})\leq 1\big\}.$$

Let $\lambda(v)=\prod_{j=0}^{r_v-1}\frac{1}{\delta(\sigma^j(v))}$ 
for $v\in\Tn$ and $\lambda(v)=1$ for $v\in\Tp$.  
Define 
% \begin{equation}
% \label{Lambda}
$$
\Lambda: \begin{cases}  
[0,\infty)\times\widehat{\mathcal{P}}_0^T\rightarrow 
 \widehat{\mathcal{P}}_0^T, \\
(t,(\mathbf{c}_v,a_v)_{v\in\Tn})  \mapsto
(t^{\lambda(v)}\mathbf{c}_v, t^{\delta(v)\lambda(v)}a_v)_{v\in\Tn}. 
\end{cases}
$$
%\end{equation}
For each $v\in\Tn$, we have $\lambda(\sigma(v)) = \delta(v)\lambda(v)$, hence 
$$f_{\Lambda(t,\mathbf{x})_v}(t^{\lambda(v)}z) 
= t^{\delta(v)\lambda(v)} f_{\mathbf{x}_v}(z)
= t^{\lambda(\sigma(v))} f_{\mathbf{x}_v}(z).$$
Then for each $(v,k)\in I$, we have 
\begin{align*}
w_{v,k}(\Lambda(t,\mathbf{x})) 
&= f_{\Lambda(t,\mathbf{x})_{\sigma^{r_v-1}(v)}}\circ\cdots\circ f_{\Lambda(t,\mathbf{x})_{\sigma(v)}}\circ f_{\Lambda(t,\mathbf{x})_v}(t^{\delta(v)}c_{v,k})\\
&= f_{\Lambda(t,\mathbf{x})_{\sigma^{r_v-1}(v)}}\circ\cdots\circ f_{\Lambda(t,\mathbf{x})_{\sigma(v)}}(t^{\lambda(\sigma(v))} f_{\mathbf{x}_v}(c_{v,k})) \\
&= \cdots\\
&= t^{\lambda(\sigma^{r_v}(v))} f_{\mathbf{x}_{\sigma^{r_v-1}(v)}}\circ\cdots\circ f_{\mathbf{x}_{\sigma(v)}}\circ f_{\mathbf{x}_v}(c_{v,k})\\
&= t\cdot w_{v,k}(\mathbf{x}). 
\end{align*}
It follows that 
\begin{equation}
\label{M and Lambda}
{{M}}(\Lambda(t,\mathbf{x}))=t\cdot{{M}}(\mathbf{x}).
\end{equation} 

For $\mathbf{x} = (\mathbf{c}_v,a_v)_{v\in\Tn}\in \widehat{\mathcal{P}}_0^T$, 
define $\|\mathbf{x}\|= \max_{v\in\Tn} \|\mathbf{x}\|_v$, where 
$$\|\mathbf{x}\|_v=\begin{cases} 
0, & \delta(v)=1, \\
|a_v|, &\delta(v)=2, \\
\max\{|c_{v,k}|,|a_v|{;~} 1\leq k\leq \delta(v)-2\}, & \delta(v)\geq3. 
\end{cases}$$
For each $v\in\Tn$ with $\delta(v)\geq3$, 
since $c_{v,\delta(v)-1} = -(c_{v,1}+\cdots+c_{v,\delta(v)-2})$, it is convenient to omit $c_{v,\delta(v)-1}$ in the definition of $\|\mathbf{x}\|_v$. 
Let $$B=\big\{\mathbf{x}\in\widehat{\mathcal{P}}_0^T {;~}  
\|\mathbf{x}\|\leq 1\big\}$$ 
be the unit (closed) ball of the normed vector space $(\widehat{\mathcal{P}}_0^T,\|\cdot\|)$. 
Then $B$ is homeomorphic to $\overline{\mathbb{D}}^n$. 
Let $\mathbf{\hat x}={\mathbf{x}}/{\|\mathbf{x}\|}$ be the normalized vector for $\mathbf{x}\neq 0$.
Define
$$\Phi: \begin{cases} 
\widehat{\mathcal{P}}_0^T   \rightarrow \widehat{\mathcal{P}}_0^T, \\
      0               \mapsto     0,\\
      \mathbf{x}        \mapsto     
\Lambda{{\left(\frac{\|\mathbf{x}\|}{{{M}}(\mathbf{\hat x})}, \mathbf{\hat x}\right)}}.
\end{cases}$$
It follows from (\ref{M and Lambda}) that ${{M}}(\Phi(\mathbf{x}))=\|\mathbf{x}\|$, hence $\model=\Phi(B)$. 
We claim that $\Phi$ is a homeomorphism. 
% Then $\model=\Phi(B)$ is homeomorphic to $\overline{\mathbb{D}}^n$. 

% To complete the proof, we show the claim. 
Since $\widehat{\mathcal{P}}_0^T$ is homeomorphic to $\mathbb{C}^n$, 
we just need to show that $\Phi$ is one-to-one and continuous.
First, we have $\Phi^{-1}(0)=\{0\}$.
Given $\mathbf{y}\in\widehat{\mathcal{P}}_0^T\setminus\{0\}$.
Since $\|\Lambda(\cdot,\mathbf{y})\|:[0,\infty)\rightarrow[0,\infty)$ is a homeomorphism, 
there exists a unique $t\in(0,\infty)$ such that $\|\Lambda(t,\mathbf{y})\|=1$. 
By (\ref{M and Lambda}) and the facts
$\Lambda(t_1t_2,\cdot)=\Lambda(t_1,\Lambda(t_2,\cdot))$, 
$\Lambda(1,\cdot)=\id$, we have 
\begin{align*}
\Phi({{M}}(\mathbf{y})\cdot\Lambda(t,\mathbf{y}))
=\Lambda{\left(\frac{{{M}}(\mathbf{y})}{{{M}}(\Lambda(t,\mathbf{y}))}, \Lambda(t,\mathbf{y})\right)}
=\Lambda{\left(\frac{1}{t},\Lambda(t,\mathbf{y})\right)}
=\Lambda(1,\mathbf{y})=\mathbf{y}.    
\end{align*}
So $\Phi$ is surjective.
On the other hand, suppose $\Phi(\mathbf{x})=\mathbf{y}$. 
We will check that $\mathbf{x}={{M}}(\mathbf{y})\cdot\Lambda(t,\mathbf{y})$. 
By $\Lambda(t_1t_2,\cdot)=\Lambda(t_1,\Lambda(t_2,\cdot))$ and 
$\Lambda(1,\cdot)=\id$ again, we have 
$$
\mathbf{\hat x}=\Lambda{\left(\frac{{{M}}(\mathbf{\hat x})}{\|\mathbf{x}\|}, 
\Lambda{\left(\frac{\|\mathbf{x}\|}{{{M}}(\mathbf{\hat x})},\mathbf{\hat x}\right)}\right)}
=\Lambda{{\left(\frac{{{M}}(\mathbf{\hat x})}{\|\mathbf{x}\|}, \mathbf{y}\right)}}, 
$$
so $\|\Lambda({{M}}(\mathbf{\hat x}) / \|\mathbf{x}\|,\mathbf{y} ) \|
=\|\mathbf{\hat x}\|=1$. 
By the uniqueness of $t$ above, we have 
${{M}}(\mathbf{\hat x}) / \|\mathbf{x}\|$ $=$ $t$.
Then $\mathbf{\hat x}=\Lambda(t,\mathbf{y})$ 
and $\|\mathbf{x}\|={{M}}(\mathbf{\hat x})/t={{M}}(\Lambda(t,\mathbf{y}))/t={{M}}(\mathbf{y})$ 
give $\mathbf{x}={{M}}(\mathbf{y})\cdot\Lambda(t,\mathbf{y})$. 
This shows $\Phi$ is injective.

Because ${{M}}$ and $\Lambda$ are continuous, 
the map $\Phi$ is continuous in $\widehat{\mathcal{P}}_0^T\setminus\{0\}$, 
hence we just need to show that $\Phi$ is continuous at $0$.
Let $m=\min\{{{M}}(\mathbf{x}) {;~}  \|\mathbf{x}\|=1\}$. 
By induction on the integers in $\{r_v{;~}v\in\Tn\}$, we have ${{M}}^{-1}(0)=0$, so $m>0$. 
Let $\alpha=\min\{\lambda(v) {;~} v\in|T|\}\in (0,1]$. 
If $t\in[0,1]$, then $\|\Lambda(t,\mathbf{x})\|\leq t^\alpha\|\mathbf{x}\|$.
For any $\varepsilon>0$, 
whenever $0<\|\mathbf{x}\|\leq \min\{m,\varepsilon\}$, we have 
\begin{align*}
\|\Phi(\mathbf{x})\|&\leq 
\left\|\Lambda{\left( \frac{\|\mathbf{x}\|}{m}, \mathbf{\hat x} \right)}\right\|
\leq {\left(\frac{\|\mathbf{x}\|}{m}\right)}^\alpha \|\mathbf{\hat x}\|
\leq \bigg(\frac{\varepsilon}{m}\bigg)^\alpha. 
\end{align*}
Thus $\Phi$ is continuous at $0$. 
This shows the claim. 

Now $\model=\Phi(B)$ is homeomorphic to $\overline{\mathbb{D}}^n$. 
To complete the proof, we just need to construct a homeomorphism from $\partial(\mathbb{D}^n)$ to $S^{2n-1}$. 
% Let $$h(x_1+iy_1,\dots,x_n+iy_n)=\frac{(x_1,y_1,\dots,x_n,y_n)}{\sqrt{\sum_{k=1}^n(x_k^2+y_k^2)}}.$$
For $z\in\partial(\mathbb{D}^n)$, let $h(z)=z/\|z\|_2$, 
where $\|z\|_2=\sqrt{\sum_{k=1}^n|z_k|^2}$. 
Note that $\|tz\|_\infty=t\|z\|_\infty$ for $z\in\mathbb{C}^n$ and $t>0$, 
where $\|z\|_\infty=\max\{|z_k|{;~} 1\leq k\leq n\}$. 
We can verify that $h$ is a homeomorphism from $\partial(\mathbb{D}^n)$ to $S^{2n-1}$. 
\end{proof}

%\begin{corollary}
%\label{the interior of the model}
%We have $\model^\circ = \mathcal{H}_0\big(\widehat{\mathcal{P}}_0^T\big)
%={\left\{\mathbf{x}\in\widehat{\mathcal{P}}_0^T {;~} 
%{{M}}(\mathbf{x})<1\right\}}$. 
%\end{corollary}
%
%\begin{proof}
%It is clear that $\mathcal{H}_0\big(\widehat{\mathcal{P}}_0^T\big) \subset \model^\circ$.
%Since $\Phi$ is a homeomorphism, 
%we have $\model^\circ=\Phi(B)^\circ = \Phi(B^\circ)$.  
%By ${{M}}(\Phi(\mathbf{x}))=\|\mathbf{x}\|$, 
%we have 
%$$\model^\circ= \Phi(B^\circ)
%=\Phi{\left({\left\{\mathbf{x}\in\widehat{\mathcal{P}}_0^T {;~}  
%\|\mathbf{x}\|< 1\right\}}\right)}
%={{\left\{\mathbf{x}\in\widehat{\mathcal{P}}_0^T {;~} 
%{{M}}(\mathbf{x})<1\right\}}}.$$
%Since ${\left\{\mathbf{x}\in\widehat{\mathcal{P}}_0^T {;~} 
%{{M}}(\mathbf{x})<1\right\}}$ is connected, 
%it is contained in $\mathcal{H}_0\big(\widehat{\mathcal{P}}_0^T\big)$. 
%The proof is completed. 
%\end{proof}

\begin{remark}
\label{remark on C(P0T)}
Proposition \ref{fiberwise connectedness locus C(P0T)} is also valid for the space 
\begin{equation}
\label{P0T}
\mathcal{P}_0^T=\big\{\mathbf{f}\in\mathcal{P}^T {;~} 
\text{$f_{v}(z)=z^{\delta(v)}$ for any $v\in\Tp$}\big\},
\end{equation}
i.e. the boundary $\partial\mathcal{C}(\mathcal{P}_0^T)$ of its fiberwise connectedness locus 
is homeomorphic to the sphere $S^{2n-1}$. 
To prove this, let $\mathbf{f}\in\mathcal{P}_0^T$. 
Then it takes the form 
$f_v(z)=z^{\delta(v)}+\sum_{j=0}^{\delta(v)-2}a_{v,j}z^j$
for each $v\in\Tn$. 
% By fixing an order of these coefficients,  
Since $n=\sum_{v\in\Tn}(\delta(v)-1)$, 
we can identify $\mathcal{P}_0^T$ with $\mathbb{C}^n$. 
For each $v\in\Tn$ with $\delta(v)\geq2$, let $[c_{v,k}(\mathbf{f})]_{k=1}^{\delta(v)-1}$ denote the unordered list of critical points of $f_v$. 
Through the collection $([c_{v,k}(\mathbf{f})]_{k=1}^{\delta(v)-1})_{v\in\Tn,\delta(v)\geq2}$,   
we can define ${{M}}(\mathbf{f})$ as before. 
Multiplying each coefficient $a_{v,j}$ with a proper constant 
depending on $t\in[0,\infty)$, we have a 
map $\Lambda:[0,\infty)\times\mathcal{P}_0^T\rightarrow \mathcal{P}_0^T$ 
satisfying ${{M}}(\Lambda(t,\mathbf{f}))=t\cdot{{M}}(\mathbf{f})$ (compare (\ref{M and Lambda})). 
The rest is the same as the proof of Proposition \ref{fiberwise connectedness locus C(P0T)}. 
% The unordered $m$-tuples of $a_1,\dots,a_m\in\mathbb{C}$ will be denoted by $[a_1,\dots,a_m]$. 
\end{remark}

In the end of this section, we give more properties, which will be used later. 
With the same notations as in the proof of 
Proposition \ref{fiberwise connectedness locus C(P0T)}, define 
$$W: \begin{cases} 
\widehat{\mathcal{P}}_0^T \rightarrow  \mathbb{C}^n, \\
\mathbf{x} \mapsto   (w_{v,k}(\mathbf{x}))_{(v,k)\in I}. 
\end{cases}$$

\begin{lemma}
\label{proper-holo-W}
$W$ is a proper holomorphic map. 
\end{lemma}

\begin{proof}
% Note that $\widehat{\mathcal{P}}_0^T$ is a complex affine space with dimension $n$. 
% Let $\mathbf{x} = (\mathbf{c}_v,a_v)_{v\in|T|}\in \widehat{\mathcal{P}}_0^T$. 
It is clear that $W$ is holomorphic. 
Since $w_{v,k}(\Lambda(t,\mathbf{x}))=t\cdot w_{v,k}(\mathbf{x})$, we have 
$W(\Lambda(t,\mathbf{x})) = t\cdot W(\mathbf{x})$. 
Then by $W^{-1}\big(\overline{\mathbb{D}}^n\big) = \model$, for every $t>0$, we have 
$W^{-1}\big(t\cdot\overline{\mathbb{D}}^n\big) = \Lambda\big(t,\model\big)$, so $W^{-1}\big(t\cdot\overline{\mathbb{D}}^n\big)$ is compact. 
It follows that $W$ is a proper holomorphic map. 
\end{proof}

For convenience, let $\mathcal{C} = \model$. 
For each $\mathbf{x}\in \partial\mathcal{C}$, define 
$$I^J(\mathbf{x}) = \big\{(v,k)\in I{;~} |w_{v,k}(\mathbf{x})|=1\big\} \subset I.$$
For each $(v,k)\in I$, define 
$$\partial_{v,k}\mathcal{C}=\big\{\mathbf{x}\in \partial\mathcal{C}{;~} I^J(\mathbf{x})=\{(v,k)\}\big\}.$$

\begin{lemma}
\label{partial-vk-model}
For any $\mathbf{x}\in\partial\mathcal{C}$, we have  $\mathbf{x}\in \bigcap_{(v,k)\in I^J(\mathbf{x})}\overline{\partial_{v,k} \mathcal{C}}$. 
\end{lemma}

\begin{proof}
Given any $\mathbf{x}_0\in\partial\mathcal{C}$ and any  $(v_0,k_0)\in I^J(\mathbf{x}_0)$. 
We need to show $\mathbf{x}_0\in \overline{\partial_{v_0,k_0} \mathcal{C}}$. 
Because $W$ is a proper holomorphic map (see Lemma \ref{proper-holo-W}), the set $W^{-1}(W(\mathbf{x}_0))$ is finite (see, e.g. \cite[\S15.1]{Rudin}). 
% Hence $\mathbf{x}_0$ is an isolated zero of the holomorphic map $W(\mathbf{x}) - \mathbf{w}_0$. 
Choose an open neighborhood $\mathcal{U}\subset \widehat{\mathcal{P}}_0^T$ of $\mathbf{x}_0$ so that  $W(\mathbf{x})\neq W(\mathbf{x}_0)$ for any $\mathbf{x}\in \overline{\mathcal{U}}\setminus\{\mathbf{x}_0\}$. 

For $j\geq 1$, define $\mathbf{w}_j=(w_{j,v,k})_{(v,k)\in I}$ by $w_{j,v_0,k_0} = w_{v_0,k_0}(\mathbf{x}_0)$ and $w_{j,v,k} = (1-1/j)w_{v,k}(\mathbf{x}_0)$ for $(v,k)\in I\setminus\{(v_0,k_0)\}$. 
It follows that $\lim_{j\rightarrow\infty}\mathbf{w}_j = W(\mathbf{x}_0)$. 
Because $W$ is a proper holomorphic map (see Lemma \ref{proper-holo-W}), it is an open map (see, e.g. \cite[\S15.1]{Rudin}). 
Hence there is an $N>0$ so that $\mathbf{w}_j\in W(\mathcal{U})$ for any $j>N$. 
For each $j>N$, choose $\mathbf{x}_j\in W^{-1}(\mathbf{w}_j)\cap\mathcal{U}$.  
Then $\mathbf{x}_j\in \partial_{v_0,k_0} \mathcal{C}$. 
By the choice of $\mathcal{U}$, we see that $\lim_{j\rightarrow\infty}\mathbf{x}_j = \mathbf{x}_0$. 
Therefore $\mathbf{x}_0\in \overline{\partial_{v_0,k_0} \mathcal{C}}$. 
This completes the proof. 
\end{proof}

\section{Lemniscate maps}
\label{section lemniscate maps}
This section is the third step to prove Theorem \ref{top-boundary}. 
Let $T=(|T|,\sigma,\delta)$ be a mapping scheme with $\Tn\neq\emptyset$. 
%Extracting the generalized polynomials in $\mathcal{P}_0^T$ (see (\ref{P0T})), we will introduce the lemniscate maps over $T$, 
%which is an analogy of generalized polynomial-like maps over $T$ (see \cite{Inou-Kiwi}). 
In \cite{Inou-Kiwi}, Inou and Kiwi proved the straightening theorem for generalized polynomial-like maps over $T$. 
In this section, we will first define the lemniscate maps over $T$, and then establish the straightening theorem for these maps. 
Finally, we consider the convergence of external angles and marked lemniscate maps. 

%In \cite{Inou-Kiwi}, Inou and Kiwi proved the straightening theorem for generalized polynomial-like maps over $T$. 
%Recall from (\ref{P0T}) that 
%$$\mathcal{P}_0^T=\big\{\mathbf{f}\in\mathcal{P}^T {;~} 
%\text{$f_{v}(z)=z^{\delta(v)}$ for any $v\in\Tp$}\big\}.$$
%Extracting the generalized polynomials in $\mathcal{P}_0^T$, we will first introduce the lemniscate maps over $T$, and then establish the straightening theorem for these maps. 
%Finally, we consider the convergence of external angles and marked lemniscate maps. 

\subsection{Straightening of lemniscate maps}
\label{subsection straightening of lemniscate maps}
%Let $X$ and $Y$ be two topological spaces. 
%A continuous map $f:X\rightarrow Y$ is called \emph{proper} if $f^{-1}(K)$ is compact whenever $K\subset Y$ is compact. 
Let $S^2$ be the topological sphere determined by the Riemann sphere $\widehat{\mathbb{C}}$. 
An (orientation-preserving) \emph{branched covering} $f:S^2\rightarrow S^2$ is a continuous map whose behavior is locally modeled on that of $z\mapsto z^k$. 
A \emph{topological polynomial} is the restriction $f:\mathbb{C}\rightarrow\mathbb{C}$ of a branched covering $f:S^2\rightarrow S^2$ with $f^{-1}(\infty)=\{\infty\}$. 
Let $\crit(f)$ denote the set of critical points of a topological polynomial $f:\mathbb{C}\rightarrow\mathbb{C}$. 

\begin{lemma}
\label{pullback of complex structure}
Let $f:S^2\rightarrow S^2$ be a branched covering. 
Then there is a homeomorphism $\psi:S^2\rightarrow S^2$ such that $f \circ \psi^{-1}$ is a rational map. 
% Furthermore, if $f^{-1}(\infty)=\infty$, we can require that $\psi(\infty) = \infty$. 
Moreover, if $f$ is a topological polynomial, we can choose $\psi$ so that $f \circ \psi^{-1}:\mathbb{C}\rightarrow\mathbb{C}$ is a monic and centered polynomial. 
\end{lemma}

\begin{proof}
Consider the pullback of the complex structure of the Riemann sphere $\widehat{\mathbb{C}}$ by $f:S^2\rightarrow \widehat{\mathbb{C}}$ as in \cite[\S10.6]{Teich-2-Hubbard}. 
For each $x\in S^2$, we first choose a M\"obius
transformation $M_x:\widehat{\mathbb{C}}\rightarrow \widehat{\mathbb{C}}$ so that $M_x\circ f(x)=0$, 
and then choose a continuous single-valued branch of $(M_x\circ f)^{1/\deg(f,x)}$ 
that is a local homeomorphism. 
With this complex structure, the sphere $S^2$ is a Riemann surface and $f:S^2\rightarrow \widehat{\mathbb{C}}$ is holomorphic. 
By the uniformization theorem for Riemann surfaces, there is a conformal map $\psi:S^2 \rightarrow \widehat{\mathbb{C}}$. 
Then $f \circ \psi^{-1}$ is holomorphic, hence a rational map. 

If $f^{-1}(\infty)=\{\infty\}$, there is a M\"obius
transformation $M:\widehat{\mathbb{C}}\rightarrow \widehat{\mathbb{C}}$ so that $f\circ\psi^{-1}\circ M|_\mathbb{C}\in\mathcal{P}^{\deg(f)}$. 
Let $\widetilde{\psi} = M^{-1}\circ\psi$. 
Then $f\circ{\widetilde\psi}^{-1}|_\mathbb{C}\in\mathcal{P}^{\deg(f)}$. 
\end{proof}

\begin{definition}
[Lemniscate map]
\label{definition lemniscate map}
% Let $T=(|T|,\sigma,\delta)$ be a mapping scheme with $\Tn\neq\emptyset$. 
We call $(\mathbf{g}, K) = ((g_v)_{v\in\Tn},(K_v)_{v\in|T|})$ a \emph{lemniscate map} over $T$ if 
\begin{itemize}
\item for each $v\in\Tp$, the set $K_v$ is a closed Jordan domain;
\item for each $v\in\Tn$, the map $g_v:\mathbb{C}\rightarrow\mathbb{C}$ is a topological polynomial of degree $\delta(v)$ so that $K_v = g_v^{-1}(K_{\sigma(v)})$ and the restriction $g_v|_{K_v^{\circ}}$ is holomorphic. 
\end{itemize}
We call $K$ the \emph{filled Julia set} of $(\mathbf{g}, K)$. 
\end{definition}

A (topological) \emph{lemniscate} is the preimage of $\overline{\mathbb{D}}$ under some topological polynomial. 
By Lemma \ref{pullback of complex structure}, if $K$ is the preimage of a closed Jordan domain under some topological polynomial, then it is a lemniscate. 
Furthermore, the preimage of a lemniscate under some topological polynomial is also a lemniscate. 
It follows that if $(\mathbf{g}, K)$ is a lemniscate map over $T$, then $K_v$ is a lemniscate for each $v\in|T|$. 
% Concretely, a lemniscate can be characterized as follows: 

Recall that $\mathcal{P}_0^T$ is defined by (\ref{P0T}), 
and $K(\mathbf{f})\subset|T|\times\mathbb{C}$ is the filled Julia set of $\mathbf{f} \in \mathcal{P}_0^T$. 
For each $v\in|T|$, we let $K_{\mathbf{f},v}\subset \mathbb{C}$ so that $$\{v\}\times K_{\mathbf{f},v} = K(\mathbf{f})\cap(\{v\}\times\mathbb{C}).$$
Each $\mathbf{f} = (f_v)_{v\in|T|}\in \mathcal{P}_0^T$ can also be written as $\mathbf{f} = (f_v)_{v\in\Tn}$. 
Then $(\mathbf{f},K(\mathbf{f})) = ((f_v)_{v\in\Tn},(K_{\mathbf{f},v})_{v\in|T|})$ is a lemniscate map over $T$. 

For $j=1,2$, let 
$$(\mathbf{g}_j, K_j) = ((g_{j,v})_{v\in\Tn},(K_{j,v})_{v\in|T|})$$  
be a lemniscate map over $T$. 
We say $(\mathbf{g}_1, K_1)$ is \emph{conjugate} (resp. \emph{strongly conjugate}) to $(\mathbf{g}_2, K_2)$ if there is a collection $(\psi_v: \mathbb{C}\rightarrow \mathbb{C})_{v\in|T|}$ 
of homeomorphisms
\footnote{
By the conditions of a conjugacy, each $\psi_v$ must be an orientation-preserving homeomorphism. 
} 
such that 
\begin{itemize}
\item 
for each $v\in\Tp$, the restriction $\psi_v:K_{1,v}^\circ\rightarrow K_{2,v}^\circ$ is biholomorphic; 
% we have $\psi_v(K_{1,v}) = K_{2,v}$ and that the restriction $\psi_v|_{K_{1,v}^\circ}$ is holomorphic; 

\item 
for each $v\in\Tn$, we have $\psi_{\sigma(v)}\circ g_{1,v}|_{K_{1,v}} = g_{2,v}\circ \psi_v|_{K_{1,v}}$ (resp. $\psi_{\sigma(v)}\circ g_{1,v} = g_{2,v}\circ \psi_v$). 
\end{itemize}
% See Figure \ref{conjugacy and weak conjugacy}. 
%\begin{figure}[ht]
%$\xymatrix{
%  K_{1,v}          \ar[d]_{\psi_v} \ar[r]^{g_{1,v}}
%& K_{1,\sigma(v)}  \ar[d]^{\psi_{\sigma(v)}} 
%& 
%& \mathbb{C}       \ar[d]_{\psi_v} \ar[r]^{g_{1,v}}
%& \mathbb{C}       \ar[d]^{\psi_{\sigma(v)}}   \\
%  K_{2,v}  \ar[r]_{g_{2,v}}
%& K_{2,\sigma(v)} 
%& 
%& \mathbb{C}       \ar[r]_{g_{2,v}}
%& \mathbb{C}
%}$
%\caption{Left: conjugacy; right: strongly conjugacy.}
%\label{conjugacy and weak conjugacy}
%\end{figure}
Lemma \ref{conjugacy of LMs} gives more information of the conjugacy between two lemniscate maps,  
which will be used in the definition of the conjugacy between two marked lemniscate maps in \S\ref{subsection straightening of marked lemniscate maps}. 

\begin{lemma}
\label{conjugacy of LMs}
% Let $(\mathbf{g}_j, K_j) = (\{g_{j,v}\}_{v\in\Tn},\{K_{j,v}\}_{v\in|T|})$ be lemniscate maps over $T$ for $j=1,2$. 
Let $(\psi_v)_{v\in|T|}$ be a conjugacy from $(\mathbf{g}_1, K_1)$ to $(\mathbf{g}_2, K_2)$. 
Then for each $v\in\Tn$, the restriction $\psi_v:K_{1,v}^\circ\rightarrow K_{2,v}^\circ$ is biholomorphic. 
In particular, we have the commutative diagram  
$$\xymatrix{
  K_{1,v}          \ar@<-1ex>[d]_{\psi_v} \ar[rr]^-{g_{1,v}}
& 
& K_{1,\sigma(v)}~  \ar@<-1ex>[d]_{\psi_{\sigma(v)}}  \\
  K_{2,v}  \ar[rr]_-{g_{2,v}}  \ar@<-1ex>[u]_{\psi_v^{-1}}
& 
& K_{2,\sigma(v)}. \ar@<-1ex>[u]_{\psi_{\sigma(v)}^{-1}}
}$$
\end{lemma}

\begin{proof}
By induction, we can assume $\sigma(v)\in \Tp$ without loss of generality. 
Since $g_{2,v}(\psi_{v}(K_{1,v})) = \psi_{\sigma(v)}(g_{1,v}(K_{1,v})) = K_{2,\sigma(v)}$, 
we have $\psi_{v}(K_{1,v}) \subset g_{2,v}^{-1}(K_{2,\sigma(v)}) =  K_{2,v}$. 
To prove $\psi_{v}(K_{1,v}) = K_{2,v}$, let $w\in K_{2,v} \setminus g_{2,v}^{-1}(\psi_{\sigma(v)}(g_{1,v}(\crit(g_{1,v}))))$. 
Then $\psi_{\sigma(v)}^{-1}(g_{2,v}(w))\in K_{1,\sigma(v)}\setminus g_{1,v}(\crit(g_{1,v}))$. 
It follows that $A:=g_{1,v}^{-1}(\psi_{\sigma(v)}^{-1}(g_{2,v}(w)))$ consists of $\delta(v)$ different points in $K_{1,v}$.  
Then $g_{2,v}(\psi_v(A)) = \psi_{\sigma(v)}(g_{1,v}(A))
= \{g_{2,v}(w)\}$, 
and $\psi_v(A)$ consists of $\delta(v)$ different points in $K_{2,v}$. 
Therefore $w\in\psi_v(A)\subset \psi_v(K_{1,v})$. 
Since such $w$ are dense in $K_{2,v}$, we have $\psi_{v}(K_{1,v}) = K_{2,v}$. 

By the commutative diagram above, the restriction $\psi_v:K_{1,v}^\circ\setminus \crit(g_{1,v}) \rightarrow K_{2,v}^\circ\setminus \crit(g_{2,v})$ is biholomorphic. 
By the removable singularities theorem, the restriction $\psi_v:K_{1,v}^\circ\rightarrow K_{2,v}^\circ$ is biholomorphic. 
This completes the proof. 
\end{proof}

\begin{lemma}
\label{straightening of LM}
Every lemniscate map $(\mathbf{g}, K)$ over $T$ is strongly conjugate to $(\mathbf{f},K(\mathbf{f}))$ for some $\mathbf{f}\in\mathcal{P}_0^T$. 
\end{lemma}

\begin{proof}
Let $(\mathbf{g}, K) = ((g_v)_{v\in\Tn},(K_v)_{v\in|T|})$. 
For $v\in\Tp$, by the Riemann mapping theorem, there is a homeomorphism $\psi_v:\mathbb{C}\rightarrow\mathbb{C}$ so that 
its restriction $\psi_v:K_v^\circ\rightarrow\mathbb{D}$ is conformal. 
For $v\in\Tn$, we can construct $\psi_v$ and $f_v$ inductively as follows. 
Assume we have defined the homeomorphism $\psi_{\sigma(v)}:\mathbb{C}\rightarrow\mathbb{C}$. 
Then $\psi_{\sigma(v)}\circ g_v$ is a topological polynomial of degree $\deg(g_v)=\delta(v)$. 
By Lemma \ref{pullback of complex structure}, there is a homeomorphism $\psi_v:\mathbb{C}\rightarrow\mathbb{C}$ so that $f_v:=\psi_{\sigma(v)}\circ g_v\circ \psi_v^{-1}\in\mathcal{P}^{\delta(v)}$. 
Let $\mathbf{f}=(f_v)_{v\in\Tn}$. 
Then $\mathbf{f}\in\mathcal{P}_0^T$, and $(\mathbf{g},K)$ is strongly conjugate to $(\mathbf{f},K(\mathbf{f}))$ by $(\psi_v)_{v\in|T|}$. 
\end{proof}

\subsection{Straightening of marked lemniscate maps}
\label{subsection straightening of marked lemniscate maps}

Let $\mathcal{A}$ be the collection of all pairs $(K,\gamma)$ such that 
\begin{itemize}
\item $K\subset\mathbb{C}$ is a full nondegenerate continuum; 
\item $\gamma$ is a path to $K$ (see Definition \ref{def-path-K}). 
\end{itemize}

Now fix $(K,\gamma)\in\mathcal{A}$. 
By the Riemann mapping theorem, there is a unique conformal isomorphism 
$\phi_K:\mathbb{C}\setminus \overline{\mathbb{D}} \rightarrow \mathbb{C}\setminus K$ 
such that $\lim_{z\rightarrow\infty}\phi_K(z)/z>0$. 
The \emph{external ray} with angle $\theta\in\RZ$ is defined to be  
$R_K(\theta)=\phi_K((1,\infty)e^{2\pi i\theta})$. 
An external ray $R_K(\theta)$ \emph{lands} at a point $x\in\partial K$ if $\overline{R_K(\theta)}\cap\partial K=\{x\}$. 
%By Lindel\"of's theorem \cite[\S3.2]{Ahl}, 
%By the F. and M. Riesz theorem \cite[Appendix A]{Mil}, 
%the lifted path $\phi_K^{-1}\circ\gamma$ converges to a point $e^{2\pi i \eta}\in\partial\mathbb{D}$. 
It is known from \cite[Corollary 6.4]{McM} that 
the lifted path $\phi_K^{-1}\circ\gamma$ converges to a point $e^{2\pi i t}\in\partial\mathbb{D}$, 
and the external ray $R_K(t)$ lands at the same point as $\gamma$. 
% (More precisely, we need to choose a parameterization of the tail of $\overline{R_K(t)}$.)
The angle $t$ will be denoted by $\theta(K,\gamma)$. 
Furthermore, $\theta(K,\gamma_0)=\theta(K,\gamma_1)$ if and only if $\gamma_0\simeq_K\gamma_1$ for any $(K,\gamma_0),(K,\gamma_1)\in\mathcal{A}$ (see Definition \ref{def-path-K} for the homotopy relation $\simeq_K$). 

Recall that a lemniscate is the preimage of $\overline{\mathbb{D}}$ under some topological polynomial. 

\begin{lemma}
\label{the same extension on the ideal boundary}
Let $K,L\subset\mathbb{C}$ be connected lemniscate. 
Let $f_1,f_2$ be topological polynomials with $\deg(f_1)=\deg(f_2)$. 
Suppose $K=f_1^{-1}(L)=f_2^{-1}(L)$ and $f_1|_K=f_2|_K$. 
Then $f_1\circ \gamma\simeq_L f_2\circ\gamma$ for any path $\gamma$ to $K$. 
%\footnote{
%If we require that $K,L\subset\mathbb{C}$ are full nondegenerate continua, 
%the result is still true, using prime ends (see, e.g. \cite[Proposition 5.2(a)]{McM1986}). 
%}
\end{lemma}

\begin{proof}
For $j=1,2$, we define $f_j^*:\RZ\rightarrow\RZ$ by $f_j^* (t)= \theta(L, f_j(R_K(t)))$. 
For $t\in \RZ$, let $x(t)$ denote the landing point of $R_K(t)$, and let $y(t) = f_1(x(t))=f_2(x(t))$. 
Suppose $K\setminus x(t)$ is connected.
Then $L\setminus y(t)$ is connected, and there is only one external ray $R_L(t')$ landing at $y(t)$. 
It follows that $f_1^*(t) =f_2^*(t) = t'$. 
Note that there are only finitely many $t\in\RZ$ so that $K\setminus x(t)$ is not connected. 
By the continuity of $f_1^*$ and $f_2^*$, we have $f_1^* = f_2^*$. 
That is, $f_1(R_K(t)) \simeq_L f_2(R_K(t))$ for any $t\in\RZ$. 
Then $$f_1\circ \gamma\simeq_L f_1(R_K(\theta(K,\gamma))) \simeq_L f_2(R_K(\theta(K,\gamma))) \simeq_L f_2\circ\gamma$$
for any path $\gamma$ to $K$. 
\end{proof}

%The \emph{ideal boundary} $I(\widehat{\mathbb{C}}\setminus K)$ is the space of prime ends of $\widehat{\mathbb{C}}\setminus K$. 
%Then $I(\widehat{\mathbb{C}}\setminus K)$ can be identified with $\partial\mathbb{D}$ via $\phi_K$. 
%For the details of prime ends, we refer to \cite[\S2.4]{Pom}. 
%By the same argument as \cite[Proposition 5.2(a)]{McM1986}, we have the following result. 

%\begin{lemma}
%\label{the same extension on the ideal boundary}
%Let $K,L\subset\mathbb{C}$ be full nondegenerate continua. 
%Let $f_1,f_2$ be topological polynomials with $\deg(f_1)=\deg(f_2)$. 
%Suppose $K=f_1^{-1}(L)=f_2^{-1}(L)$ and $f_1|_K=f_2|_K$. 
%Then $f_1|_{\mathbb{C}\setminus K}$ and $f_2|_{\mathbb{C}\setminus K}$ induce the same map from $I(\widehat{\mathbb{C}}\setminus K)$ to $I(\widehat{\mathbb{C}}\setminus L)$. 
%In particular, for any path $\gamma$ to $K$, we have $f_1\circ \gamma\simeq_L f_2\circ\gamma$. 
%\end{lemma}

\begin{definition}
[Marked lemniscate map]
\label{definition marked LM}
A \emph{marked lemniscate map} $(\mathbf{g},K,Z,\Gamma)$ over $T$ consists of 
\begin{itemize}
\item a lemniscate map $(\mathbf{g}, K) = ((g_v)_{v\in\Tn},(K_v)_{v\in|T|})$ over $T$ with fiberwise connected filled Julia set; 

\item a \emph{central marking} $Z = (\zeta_v)_{v\in\Tp}$, where $\zeta_v\in K_v^\circ$; 

\item an \emph{external marking} $\Gamma = (\gamma_v)_{v\in|T|}$, 
where $\gamma_v$ is a path to $K_v$, and $g_v\circ\gamma_v\simeq_{K_{\sigma(v)}}\gamma_{\sigma(v)}$ if $v\in\Tn$.  
\end{itemize}
\end{definition}
% i.e. $\crit(g_v)\subset K_v$ for any $v\in\Tn$ (compare Lemma \ref{connectedness-crit}).  

Let $\mathbf{f}\in\mathcal{C}(\mathcal{P}_0^T)$. 
There exists a unique collection $(B_{\mathbf{f},v}:\mathbb{C}\setminus K_{\mathbf{f},v}\rightarrow 
\mathbb{C}\setminus \overline{\mathbb{D}})_{v\in|T|}$ of \emph{B\"ottcher coordinates} so that 
$B_{\mathbf{f},v}$ is tangent to the identity as $z\rightarrow\infty$, 
with form 
\begin{equation}
\label{Bfv}
B_{\mathbf{f},\sigma(v)}(f_v(z))=B_{\mathbf{f},v}(z)^{\delta(v)}.
\end{equation}
We remark that $B_{\mathbf{f},v}=\id$ for each $v\in\Tp$. 
The \emph{external ray} for $v\in|T|$ and $\theta\in\RZ$ of $\mathbf{f}$ is defined to be 
$$
R_{\mathbf{f},v}(\theta)=B_{\mathbf{f},v}^{-1}((1,\infty)e^{2\pi i\theta}).
$$
For each $v\in\Tp$, we choose $\zeta_v\in K_{\mathbf{f},v}^\circ =\mathbb{D}$. 
Let $(\theta_v)_{v\in|T|}$ be a collection of angles in $\RZ$ satisfying $\delta(v)\cdot\theta_v \equiv \theta_{\sigma(v)}\modZ$ for each $v\in\Tn$. 
Then $(Z,\Gamma) = ((\zeta_v)_{v\in\Tp},(R_{\mathbf{f},v}(\theta_v))_{v\in|T|})$ is a marking of $(\mathbf{f},K(\mathbf{f}))$. 
We call $(Z,\Gamma)$ the \emph{standard marking} of $(\mathbf{f},K(\mathbf{f}))$ if $\zeta_v=0$ for any $v\in\Tp$ and $\theta_v=0$ for any $v\in|T|$.

For $j=1,2$, let $$(\mathbf{g}_j, K_j, Z_j, \Gamma_j) = ((g_{j,v})_{v\in\Tn},(K_{j,v})_{v\in|T|}, (\zeta_{j,v})_{v\in\Tp}, (\gamma_{j,v})_{v\in|T|} )$$ 
be a marked lemniscate map over $T$. 
We say $(\mathbf{g}_1, K_1, Z_1, \Gamma_1)$ is \emph{conjugate} (resp. \emph{strongly conjugate}) to $(\mathbf{g}_2, K_2, Z_2, \Gamma_2)$ if there exists a conjugacy (resp. strong conjugacy) $(\psi_v)_{v\in|T|}$ from $(\mathbf{g}_1, K_1)$ to $(\mathbf{g}_2, K_2)$ such that 
\begin{itemize}
\item 
for each $v\in\Tp$, we have $\psi_v(\zeta_{1,v})= \zeta_{2,v}$; 

\item 
for each $v\in|T|$, 
we have $\psi_v\circ\gamma_{1,v} \simeq_{K_{2,v}} \gamma_{2,v}$ (it follows from Lemma \ref{conjugacy of LMs} that $\psi_v(K_{1,v})=K_{2,v}$). 
\end{itemize}

\begin{lemma}
\label{uniqueness-f-CP0T}
Let $\mathbf{f}_1,\mathbf{f}_2\in \mathcal{C}(\mathcal{P}_0^T)$. 
% Suppose $\{\phi_v\}_{v\in|T|}$ is a conjugacy from $(\mathbf{f}_1,K(\mathbf{f}_1),Z'_1,\Gamma'_1)$ to $(\mathbf{f}_2,K(\mathbf{f}_2),Z'_2,\Gamma'_2)$, 
Let $(Z_j,\Gamma_j)$ be the standard marking of $(\mathbf{f}_j,K(\mathbf{f}_j))$ for $j=1,2$. 
If $(\mathbf{f}_1,K(\mathbf{f}_1),Z_1,\Gamma_1)$ is conjugate to $(\mathbf{f}_2,K(\mathbf{f}_2),Z_2,\Gamma_2)$, 
then $\mathbf{f}_1 = \mathbf{f}_2$. 

Furthermore, if $(\phi_v)_{v\in|T|}$ is a self-conjugacy of $(\mathbf{f}_1,K(\mathbf{f}_1),Z_1,\Gamma_1)$, 
then $\phi_v|_{K_{\mathbf{f}_1,v}}=\id$ for any $v\in|T|$. 
\end{lemma}

\begin{proof}
Suppose $(\phi_v)_{v\in|T|}$ is a conjugacy from $(\mathbf{f}_1,K(\mathbf{f}_1),Z_1,\Gamma_1)$ to $(\mathbf{f}_2,K(\mathbf{f}_2),Z_2,\Gamma_2)$. 
%Let $v\in\Tp$. 
%Then the restriction $\psi_v:\mathbb{D}\rightarrow\mathbb{D}$ is conformal. 
%Since $\psi_v(0)=0$ and $\phi_v(1)=1$, we have $\psi_v(z)=z$ for any $z\in\overline{\mathbb{D}}$. 
Let $v\in|T|$.
Define $\phi_v^*:\RZ\rightarrow\RZ$ by 
$\phi_v^*(\eta)=\theta(K_{\mathbf{f}_2,v}, \phi_{v}(R_{\mathbf{f}_1,v}(\eta)))$, 
i.e. $R_{\mathbf{f}_2,v}(\phi_v^*(\eta)) \simeq_{K_{\mathbf{f}_2,v}} \phi_{v}(R_{\mathbf{f}_1,v}(\eta))$.

\vspace{6pt}
{\bf Claim.}
{\it $\phi_v^*=\id$.}
\vspace{6pt}

If $v\in\Tp$, then $\phi_v|_{\overline{\mathbb{D}}}=\id$, the claim follows from Lemma \ref{the same extension on the ideal boundary}. 
Now assume $v\in\Tn$.  
By induction, we can assume $\phi_{\sigma(v)}^*=\id$. 
For any $\eta\in\RZ$, we have 
\begin{align*}
&\ \ \ \   R_{\mathbf{f}_2,\sigma(v)}(\delta(v)\cdot\phi_v^*(\eta))\\
&= f_{2,v}(R_{\mathbf{f}_2,v}(\phi_v^*(\eta)))
\quad(\text{by (\ref{Bfv})})\\
&\simeq_{K_{\mathbf{f}_2,\sigma(v)}} f_{2,v}\circ\phi_{v}(R_{\mathbf{f}_1,v}(\eta))
\quad(\text{by $R_{\mathbf{f}_2,v}(\phi_v^*(\eta)) \simeq_{K_{\mathbf{f}_2,v}} \phi_{v}(R_{\mathbf{f}_1,v}(\eta))$})\\
&\simeq_{K_{\mathbf{f}_2,\sigma(v)}}  
\phi_{\sigma(v)}\circ f_{1,v}(R_{\mathbf{f}_1,v}(\eta))
\quad(\text{by Lemma \ref{the same extension on the ideal boundary}})\\
&=
\phi_{\sigma(v)}(R_{\mathbf{f}_1,\sigma(v)}(\delta(v)\cdot\eta))
\quad(\text{by (\ref{Bfv}) again})\\
&\simeq_{K_{\mathbf{f}_2,\sigma(v)}} 
R_{\mathbf{f}_2,\sigma(v)}(\delta(v)\cdot\eta)
\quad(\text{by the assumption  $\phi_{\sigma(v)}^*=\id$}). 
\end{align*}
% \quad(\text{by })
Thus $m_{\delta(v)}\circ \phi_v^* = m_{\delta(v)}$. 
Note that for any $\eta_1,\eta_2,\eta_3\in\RZ$ in positive cyclic order, the images $\phi_v^*(\eta_1),\phi_v^*(\eta_2),\phi_v^*(\eta_3)$ are in positive cyclic order. 
It follows from $\phi_v^*(0)=0$ that  $\phi_v^*=\id$. 
This shows the claim. 
\vspace{6pt}

Define $\Phi_v:\mathbb{C}\rightarrow \mathbb{C}$ by 
$$\Phi_v(z)=\begin{cases}
\phi_v(z), &z\in K_{\mathbf{f}_1,v}, \\
B_{\mathbf{f}_2,v}^{-1}\circ B_{\mathbf{f}_1,v}(z), &z\in \mathbb{C}\setminus K_{\mathbf{f}_1,v}. 
\end{cases}$$ 
Since $\phi_v^*=\id$, for any $\eta\in\RZ$, if  $R_{\mathbf{f}_1,v}(\eta)$ lands at $x\in\partial K_{\mathbf{f}_1,v}$, then $R_{\mathbf{f}_2,v}(\eta)$ lands at $\phi_v(x)\in\partial K_{\mathbf{f}_2,v}$. 
Therefore $B_{\mathbf{f}_2,v}^{-1}\circ B_{\mathbf{f}_1,v}$ can be extended to a homeomorphism $\varphi_v:\overline{\mathbb{C}\setminus K_{\mathbf{f}_1,v}}\rightarrow\overline{\mathbb{C}\setminus K_{\mathbf{f}_2,v}}$ (compare Lemma \ref{continuous extension on the closure of infinity basin}).  
Furthermore, we have $\varphi_v|_{\partial K_{\mathbf{f}_1,v}} = \phi_v|_{\partial K_{\mathbf{f}_1,v}}$, hence $\Phi_v$ is a homeomorphism.  

%Since $\Phi_v$ is holomorphic on $\mathbb{C}\setminus\partial K_{\mathbf{f}_1,v}$, the map $\Phi_v$ is conformal by Morera's theorem (or by the conformal removability of $\partial K_{\mathbf{f}_1,v}$). 
%Since $\Phi_v$ is tangent to the identity as $z\rightarrow\infty$, it takes the form $\Phi_v(z)=z+b_v$ for some $b_v\in\mathbb{C}$. 
%When $v\in\Tp$, it is clear that $b_v=0$. 
%When $v\in\Tn$, we can inductively deduce $b_v=0$ by $\Phi_{\sigma(v)}\circ f_{1,v} = f_{2,v}\circ \Phi_v$ and $f_{1,v},f_{2,v}\in\mathcal{P}^{\delta(v)}$ (monic and centered). 
%It follows that $\mathbf{f}_1=\mathbf{f}_2$, 
%and that $\phi_v|_{K_{\mathbf{f}_1,v}}=\id$ for each $v\in|\Tn|$. 

If $v\in\Tp$, it is clear that $\Phi_v = \id$. 
Now consider $v\in\Tn$ with $\sigma(v)\in\Tp$. 
It follows from $\Phi_{\sigma(v)}\circ f_{1,v} = f_{2,v}\circ \Phi_v$ that $f_{1,v} = f_{2,v}\circ \Phi_v$. 
This implies $\Phi_v$ is holomorphic. 
Since $\Phi_v$ is tangent to the identity as $z\rightarrow\infty$, it takes the form $\Phi_v(z)=z+b_v$ for some $b_v\in\mathbb{C}$. 
By $f_{1,v} = f_{2,v}\circ \Phi_v$ and $f_{1,v},f_{2,v}\in\mathcal{P}^{\delta(v)}$ (monic and centered), we have $b_v=0$, so $\Phi_v = \id$. 
Inductively, we have $\Phi_v = \id$ for any $v\in\Tn$. 
It follows that $\mathbf{f}_1=\mathbf{f}_2$,   
and that $\phi_v|_{K_{\mathbf{f}_1,v}}=\id$ for each $v\in|T|$. 
\end{proof}

\begin{proposition}
[Straightening theorem]
\label{straightening of marked LM}
Let 
$(\mathbf{g}, K, Z, \Gamma) $ 
be a marked lemniscate map over $T$. Then there exists a unique $\mathbf{f}\in\mathcal{C}(\mathcal{P}_0^T)$ such that $(\mathbf{g}, K,Z,\Gamma)$ is conjugate to $(\mathbf{f},K(\mathbf{f}),Z',\Gamma')$, where $(Z',\Gamma')$ is the standard marking of $(\mathbf{f},K(\mathbf{f}))$. 

Furthermore, if $(\psi_v)_{v\in|T|}$ is a conjugacy from $(\mathbf{g}, K, Z, \Gamma)$ to $(\mathbf{f},K(\mathbf{f}),Z',\Gamma')$, then $(\psi_v|_{K_v})_{v\in|T|}$ is unique. 
% For each $v\in\Tp$, let $\phi_{v}:\mathbb{C}\rightarrow\mathbb{C}$ be a homeomorphism so that $\phi_v|_{K_v}=\psi_v|_{K_v}$. 
% Then $\{\phi_{v}\}_{v\in\Tp}$ determines a strong conjugacy $\{\phi_v\}_{v\in|T|}$ from $(\mathbf{g}, K, Z, \Gamma)$ to $(\mathbf{f},K(\mathbf{f}),Z',\Gamma')$. 
\end{proposition}

\begin{proof}

By Lemma \ref{straightening of LM}, there exist $\mathbf{h}\in\mathcal{P}_0^T$ and $(\phi_v)_{v\in|T|}$ so that $(\mathbf{g}, K)$ is strongly conjugate to $(\mathbf{h},K(\mathbf{h}))$ by $(\phi_v)_{v\in|T|}$. 
Furthermore, we can require $\phi_v(\zeta_v)=0$ and $\phi_v(\gamma_v(0))=1$ for each $v\in\Tp$. 
By Lemma \ref{conjugacy of LMs}, we have  $\mathbf{h}\in\mathcal{C}(\mathcal{P}_0^T)$.  
For each $v\in|T|$,  let $\theta_v = \theta(K_{\mathbf{h},v},\phi_v\circ\gamma_v)$, 
$\varphi_v(z)=e^{-2\pi i \theta_v}z$ and $\psi_v = \varphi_v\circ\phi_v$. 
When $v\in\Tn$, we have 
\begin{align*}
    R_{\mathbf{h},\sigma(v)}(\delta(v)\cdot\theta_v)
&= h_v(R_{\mathbf{h},v}(\theta_v)) 
\quad(\text{by (\ref{Bfv})})\\
&\simeq_{K_{\mathbf{h},\sigma(v)}} h_v\circ\phi_v\circ\gamma_v
\quad(\text{by $R_{\mathbf{h},v}(\theta_v) \simeq_{K_{\mathbf{h},v}} \phi_v\circ\gamma_v)$}\\
&= \phi_{\sigma(v)}\circ g_v\circ \gamma_v \\
&\simeq_{K_{\mathbf{h},\sigma(v)}} \phi_{\sigma(v)}\circ \gamma_{\sigma(v)}
\quad(\text{by $g_v\circ\gamma_v \simeq_{K_{\sigma(v)}} \gamma_{\sigma(v)}$})\\
&\simeq_{K_{\mathbf{h},\sigma(v)}} R_{\mathbf{h},\sigma(v)}(\theta_{\sigma(v)}), 
\end{align*}
hence $\delta(v)\cdot\theta_v \equiv \theta_{\sigma(v)}\modZ$ and $f_v:=\varphi_{\sigma(v)}\circ h_v\circ \varphi_v^{-1}\in\mathcal{P}^{\delta(v)}$ (monic and centered). 
Let $\mathbf{f}=(f_v)_{v\in\Tn}\in\mathcal{C}(\mathcal{P}_0^T)$. 
When $v\in|T|$, we have 
\begin{align*}
\psi_v\circ\gamma_v
=\varphi_v\circ\phi_v\circ\gamma_v
\simeq_{K_{\mathbf{f},v}}\varphi_v(R_{\mathbf{h},v}(\theta_v))
=R_{\mathbf{f},v}(0). 
\end{align*}
It follows that $(\mathbf{g},K,Z,\Gamma)$ is (strongly) conjugate to $(\mathbf{f},K(\mathbf{f}),Z',\Gamma')$ by $(\psi_v)_{v\in|T|}$. 

The uniquenesses of $\mathbf{f}$ and $(\psi_v|_{K_v})_{v\in|T|}$ follow from Lemma \ref{uniqueness-f-CP0T}. 
\end{proof}

\begin{corollary}
\label{construction of strong conjugacy}
Suppose $(\mathbf{g}_1, K_1, Z_1, \Gamma_1)$ is conjugate to $(\mathbf{g}_2, K_2, Z_2, \Gamma_2)$. 
For each $v\in\Tp$, let $\phi_{v}:\mathbb{C}\rightarrow\mathbb{C}$ be a homeomorphism so that $\phi_v:K_{1,v}^\circ\rightarrow K_{2,v}^\circ$ is conformal, 
normalized by $\phi_v(\zeta_{1,v})=\zeta_{2,v}$ and $\phi_v(\gamma_{1,v}(0))=\gamma_{2,v}(0)$. 
Then $(\phi_v)_{v\in\Tp}$ determines a strong conjugacy $(\phi_v)_{v\in|T|}$ from $(\mathbf{g}_1, K_1, Z_1, \Gamma_1)$ to $(\mathbf{g}_2, K_2, Z_2, \Gamma_2)$. 
\end{corollary}

\begin{proof}
Let $(\psi_v)_{v\in|T|}$ be a conjugacy from $(\mathbf{g}_1, K_1, Z_1, \Gamma_1)$ to $(\mathbf{g}_2, K_2, Z_2, \Gamma_2)$. 
Then $\phi_v|_{K_{1,v}}=\psi_v|_{K_{1,v}}$ for each $v\in\Tp$. 
Now let $v\in\Tn$. 
We will construct $\phi_v$ inductively. 
Assume we have defined the homeomorphism $\phi_{\sigma(v)}:\mathbb{C}\rightarrow\mathbb{C}$ so that $\phi_{\sigma(v)}|_{K_{1,\sigma(v)}}=\psi_{\sigma(v)}|_{K_{1,\sigma(v)}}$. 
Note that 
\begin{align*}
\phi_{\sigma(v)}\circ g_{1,v}\circ \gamma_{1,v}
&\simeq_{K_{2,\sigma(v)}} 
\phi_{\sigma(v)}\circ \gamma_{1,\sigma(v)}
\quad(\text{by $g_{1,v}\circ \gamma_{1,v}\simeq_{K_{1,\sigma(v)}} \gamma_{1,\sigma(v)}$})\\
&\simeq_{K_{2,\sigma(v)}} 
\psi_{\sigma(v)}\circ \gamma_{1,\sigma(v)}
\quad(\text{by Lemma \ref{the same extension on the ideal boundary}})\\
&\simeq_{K_{2,\sigma(v)}} 
\gamma_{2,\sigma(v)}\\
&\simeq_{K_{2,\sigma(v)}} 
g_{2,v}\circ\gamma_{2,v}.
\end{align*}
By lifting, 
there is a unique homeomorphism $\varphi_v: \mathbb{C}\setminus K_{1,v}\rightarrow \mathbb{C}\setminus K_{2,v}$ so that
the diagram 
$$
\xymatrix{
 & &\mathbb{C}\setminus K_{2,v}  \ar[d]^{g_{2,v}} \\
\mathbb{C}\setminus K_{1,v} \ar[rr]_-{\phi_{\sigma(v)}\circ g_{1,v}} \ar[rru]^-{\varphi_v} & &\mathbb{C}\setminus K_{2,\sigma(v)}
}$$
commutes and $\varphi_v\circ\gamma_{1,v}\simeq_{K_{2,v}}\gamma_{2,v}$, 
where we set $\varphi_v\circ\gamma_{1,v}(0) = \lim_{t\rightarrow 0}\varphi_v\circ\gamma_{1,v}(t)$ (the same below). 
For any $\eta\in\RZ$, we have 
\begin{align*}
g_{2,v}\circ\varphi_v(R_{K_{1,v}}(\eta))
&\simeq_{K_{2,\sigma(v)}}
\phi_{\sigma(v)} \circ g_{1,v}(R_{K_{1,v}}(\eta))\\
&\simeq_{K_{2,\sigma(v)}}
\psi_{\sigma(v)} \circ g_{1,v}(R_{K_{1,v}}(\eta)) 
\quad(\text{by Lemma \ref{the same extension on the ideal boundary}})\\
&\simeq_{K_{2,\sigma(v)}}
g_{2,v}\circ \psi_v(R_{K_{1,v}}(\eta)) 
\quad(\text{by Lemma \ref{the same extension on the ideal boundary} again}). 
\end{align*}
% Note that for any $\eta_1,\eta_2,\eta_3\in\RZ$ in positive cyclic order, the images $\phi_v^*(\eta_1),\phi_v^*(\eta_2),\phi_v^*(\eta_3)$ are in positive cyclic order. 
Combining this with $\varphi_v\circ\gamma_{1,v}\simeq_{K_{2,v}} \gamma_{2,v}\simeq_{K_{2,v}} \psi_v\circ\gamma_{1,v}$,   
for any $\eta\in\RZ$, 
we have $$\varphi_v(R_{K_{1,v}}(\eta))\simeq_{K_{2,v}} \psi_v(R_{K_{1,v}}(\eta)).$$
Then for any $x\in\partial K_{1,v}$, we have $\lim_{z \rightarrow x } \varphi_v(z)=\psi_v(x)$. 
Define $\phi_v:\mathbb{C}\rightarrow \mathbb{C}$ by 
$$\phi_v(z)=\begin{cases}
\psi_v(z), &z\in K_{1,v}, \\
\varphi_v(z), &z\in \mathbb{C}\setminus K_{1,v}. 
\end{cases}$$ 
Then $\phi_v$ is a homeomorphism with $\phi_v|_{K_{1,v}}=\psi_v|_{K_{1,v}}$ and $\phi_{\sigma(v)} \circ g_{1,v}=g_{2,v}\circ \phi_v$. 
It follows that $(\mathbf{g}_1,K_1,Z_1,\Gamma_1)$ is strongly conjugate to $(\mathbf{g}_2,K_2,Z_2,\Gamma_2)$ by $(\phi_v)_{v\in|T|}$. 

To prove the uniqueness, let $(\widetilde\phi_v)_{v\in|T|}$ be another strong conjugacy from 
$(\mathbf{g}_1,K_1,$ $Z_1,\Gamma_1)$ to $(\mathbf{g}_2,K_2,Z_2,\Gamma_2)$, with $\widetilde\phi_v=\phi_v$ for each $v\in\Tp$. 
By Proposition \ref{straightening of marked LM}, 
we have $\widetilde\phi_v|_{K_{1,v}} = \phi_v|_{K_{1,v}}$ for each $v\in|T|$. 
Now let $v\in\Tn$.
By induction, we can assume $\widetilde\phi_{\sigma(v)}=\phi_{\sigma(v)}$. 
Note that $$g_{2,v}\circ (\widetilde\phi_v|_{\mathbb{C}\setminus K_{1,v}})
= \widetilde\phi_{\sigma(v)}\circ g_{1,v} |_{\mathbb{C}\setminus K_{1,v}}
= \phi_{\sigma(v)}\circ g_{1,v} |_{\mathbb{C}\setminus K_{1,v}}$$
and $\widetilde\phi_v\circ\gamma_{1,v} \simeq_{K_{2,v}} \gamma_{2,v}$. 
By the uniqueness of $\varphi_v$, we have $\widetilde\phi_v|_{\mathbb{C}\setminus K_{1,v}} =\varphi_v$. 
It follows that $\widetilde\phi_v=\phi_v$. 
This shows the uniqueness of the collection $(\phi_v)_{v\in\Tn}$. 
\end{proof}

\begin{example}
Consider the mapping scheme $T=(|T|,\sigma,\delta)$ defined by $|T|=\{v_0,v_1\}$, 
$\sigma(v_1)=\sigma(v_0)=v_0$, $\delta(v_0)=2$ and $\delta(v_1)=3$. 
We define a marked lemniscate map $(\mathbf{g}, K, Z, \Gamma) = (g_{v_1},(K_{v_0}, K_{v_1}), \zeta_{v_0}, (\gamma_{v_0},\gamma_{v_1}) )$ over $T$ as follows. 
Let $K_{v_0}=\overline{\mathbb{D}}$ and $K_{v_1} = \overline{\mathbb{D}} \cup (2+\overline{\mathbb{D}})$. 
Define $g_{v_1}$ by $g_{v_1}(z)=z$ for $z\in\overline{\mathbb{D}}$ 
and $g_{v_1}(z)=(z-2)^2$ for $z\in2+\overline{\mathbb{D}}$. 
Then we extend $g_{v_1}$ to a topological polynomial of degree $\delta(v_1)=3$. 
Set $\zeta_{v_0}=0$. 
Let $\gamma_{v_0}(s)=s+1$ and $\gamma_{v_1}(s)=s+3$ for $s\in[0,1]$. 
See Figure \ref{figure lemniscate maps}. 
% By Lemma \ref{independence of extension}, the conjugacy class of $(\mathbf{g}, K, Z, \Gamma)$ is independent of the choice of  $g_{v_1}|_{\mathbb{C}\setminus K_{v_1}}$. 

Let $\mathbf{f}\in\mathcal{C}(\mathcal{P}_0^T)$ satisfy  $f_{v_1}(z)=z^3-\frac{3}{4}\sqrt[3]{4}z+\frac{1}{2}$. 
Then $\crit(f_{v_1}) = \{-\sqrt[3]{2}/2,\sqrt[3]{2}/2\}$,  $f_{v_1}^{-1}(0)=\{-\sqrt[3]{2},\sqrt[3]{2}/2\}$ and $f(-\sqrt[3]{2}/2)=1$. 
Moreover, $(\mathbf{g}, K, Z, \Gamma)$ is conjugate to $(\mathbf{f},K(\mathbf{f}),Z',\Gamma')$, where $(Z',\Gamma')$ is the standard marking of $(\mathbf{f},K(\mathbf{f}))$. 

\begin{figure}[ht]
\centering
% \frame{\includegraphics{images/LemniscateMaps.pdf}}
\includegraphics{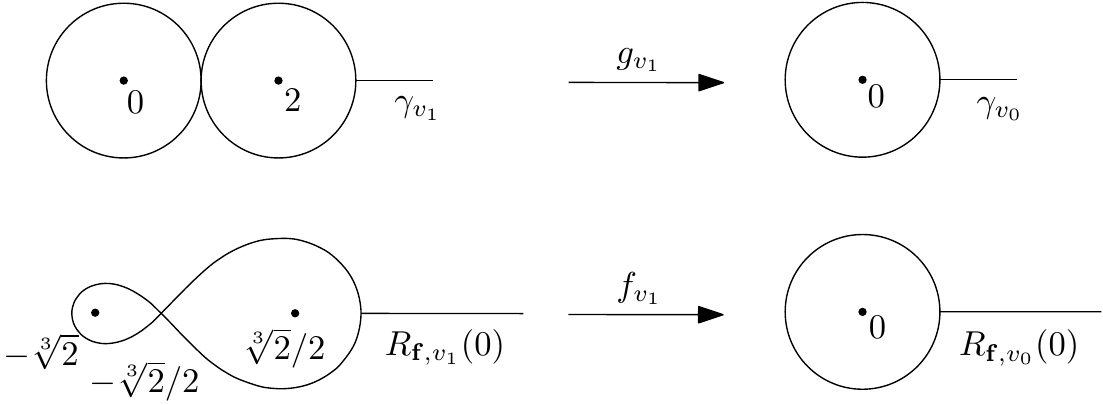}
\caption{Lemniscate maps.}
\label{figure lemniscate maps}
\end{figure}
\end{example}

\subsection{Convergence of external angles}
Recall from \S\ref{subsection straightening of marked lemniscate maps} that $\theta(K,\gamma)$ is the angle in $\RZ$ so that $\phi_K^{-1}\circ\gamma$ converges to the point $e^{2\pi i \theta(K,\gamma)}\in\partial\mathbb{D}$ for each $(K,\gamma)\in\mathcal{A}$. 
The aim of this subsection is to prove the following: 

\begin{proposition}
\label{prop-convergence of external angles}
Let $(K_n,\gamma_n),(K,\gamma)\in\mathcal{A}$ so that  $K_n\rightarrow K$ and  $\gamma_n([0,1])\rightarrow\gamma([0,1])$, both in Hausdorff topology. 
Then $\theta(K_n,\gamma_n)\rightarrow \theta(K,\gamma)$. 
\end{proposition}

% length function
Let $\phi:\mathbb{C}\setminus\overline{\mathbb{D}} \rightarrow \mathbb{C}$ be a univalent function
with $\phi(z) \rightarrow \infty$ as $z\rightarrow\infty$. 
For $\theta\in[0,1)$ (or $\theta\in\RZ$),  
let $\ell_\phi(\theta)=\int_1^2|\phi'(re^{2\pi i\theta})|dr$ be the length of $\phi((1,2]e^{2\pi i\theta})$. 
For any $r>1$, let $A_r=\{z\in\mathbb{C}{;~} 1<|z|<r\}$. 
By the Cauchy-Schwarz inequality, 
\begin{align*}
{\left(\int_0^1\ell_\phi(\theta)d\theta\right)}^2 &={\left(\int_{A_2}\frac{|\phi'(z)|}{2\pi|z|}dS\right)}^2\\
&\leq{\left(\int_{A_2}|\phi'(z)|^2 dS\right)} {\left(\int_{A_2}\frac{1}{(2\pi|z|)^2}dS\right)}\\
&=\area(\phi(A_2))\frac{\log 2}{2\pi}<\infty.
\end{align*}
Therefore $\ell_\phi(\theta)$ is finite for almost every $\theta$ (see \cite[Theorem 6.1]{McM}). 
For $\rho\in(1,2)$, we define $\ell_\phi^\rho:[0,1)\rightarrow\mathbb{R}$ by 
$\ell_\phi^\rho(\theta)=\int_\rho^2|\phi'(re^{2\pi i\theta})|dr$. 

% converge in measure of length function
\begin{lemma}
\label{converge in measure of length function}
Let $\{\phi_n\}_{n\in\mathbb{N}}$ be a sequence of univalent functions on $\mathbb{C}\setminus\overline{\mathbb{D}}$ 
with $\phi_n(z)\rightarrow\infty$ as $z\rightarrow\infty$. 
Assume $\phi_n$ converges to a nonconstant function $\phi$ 
uniformly on compact subsets of $\mathbb{C}\setminus\overline{\mathbb{D}}$. 
Then $\ell_{\phi_n}$ converges in measure to $\ell_\phi$.
That is, for all $\varepsilon>0$, we have 
$$\lim_{n\rightarrow\infty}m\{\theta\in[0,1){;~} |\ell_{\phi_n}(\theta)-\ell_\phi(\theta)|\geq\varepsilon\}=0,$$
where $m$ is the Lebesgue measure on $\mathbb{R}$.
\end{lemma}

\begin{proof}
Fix $\varepsilon>0$. 
Let $E_n=\{\theta\in[0,1){;~} |\ell_{\phi_n}(\theta)-\ell_\phi(\theta)|\geq 3\varepsilon\}$ for each $n\in\mathbb{N}$. 
We need to show $m(E_n)\rightarrow0$ as $n\rightarrow\infty$. 
Let $\varphi\in\{\phi, \phi_n{;~} n\in\mathbb{N}\}$, and let   $\rho\in(1,2)$.  
Define $$E_{\varphi,\rho}=\{\theta\in[0,1){;~} |\ell_\varphi(\theta)-\ell_\varphi^\rho(\theta)|\geq\varepsilon\}.$$ 
By the Cauchy-Schwarz inequality, 
\begin{align*}
(m(E_{\varphi,\rho})\cdot\varepsilon )^2
&\leq {\left(\int_0^1( \ell_\varphi(\theta)-\ell_\varphi^\rho(\theta) )d\theta\right)}^2 \\
&\leq{\left(\int_{A_\rho}|\varphi'(z)|^2 dS\right)} {\left(\int_{A_\rho}\frac{1}{(2\pi|z|)^2}dS\right)}\\
&=\area(\varphi(A_\rho))\frac{\log \rho}{2\pi}\leq C\log\rho, 
\end{align*}
where $C$ is a constant independent of $\varphi$ and $\rho$. 
Since $\ell_{\phi_n}^\rho$ converges uniformly to $\ell_\phi^\rho$, there exists an $N_{\rho}>0$ such that 
$\max_{\theta\in[0,1)} |\ell_{\phi_n}^\rho(\theta)-\ell_\phi^\rho(\theta)|\leq \varepsilon$ whenever $n>N_{\rho}$. 
For any $n>N_{\rho}$, it follows from 
$$|\ell_{\phi_n}(\theta)-\ell_\phi(\theta)|\leq
|\ell_{\phi_n}(\theta)-\ell_{\phi_n}^\rho(\theta)|+
|\ell_{\phi_n}^\rho(\theta)-\ell_\phi^\rho(\theta)|+
|\ell_\phi^\rho(\theta)-\ell_\phi(\theta)|$$
that $E_{n}\subset E_{\phi_n,\rho}\cup E_{\phi,\rho}$, hence 
$m(E_{n})\leq m(E_{\phi_n,\rho})+m(E_{\phi,\rho})
\leq \frac{2\sqrt{C\log\rho}}{\varepsilon}$. 
Letting $\rho\rightarrow1$ gives $m(E_n)\rightarrow 0$ as $n\rightarrow\infty$.
This completes the proof. 
\end{proof}

% converge uniformly by convergence of length
\begin{lemma}
\label{converge uniformly by convergence of length}
Let $\{\phi_n\}_{n\in\mathbb{N}}\rightarrow \phi$ as in Lemma \ref{converge in measure of length function}, and let $\theta\in[0,1)$. 
Suppose $\ell_\phi(\theta)<\infty$ and $\ell_{\phi_n}(\theta)\rightarrow\ell_\phi(\theta)$. 
Then $\phi_n$ converges to $\phi$ uniformly on $[1,2]e^{2\pi i\theta}$, 
where they take the radial limits at $e^{2\pi i\theta}$.
\end{lemma}

\begin{proof}
Given $\varepsilon>0$.
There is a $\rho\in(1,2)$ such that $\ell_{\phi}(\theta)-\ell_\phi^\rho(\theta)<\varepsilon$. 
Since $\phi_n$ converges to $\phi$ uniformly on $[\rho,2]e^{2\pi i \theta}$, 
there exists an $N_1>0$ such that for any $n>N_1$ and any $r\in[\rho,2]$, we have 
$|\phi_n(r e^{2\pi i\theta})-\phi(r e^{2\pi i\theta})|<\varepsilon$.
Since $\ell_{\phi_n}(\theta)-\ell_{\phi_n}^\rho(\theta)
\rightarrow\ell_{\phi}(\theta)-\ell_\phi^\rho(\theta)$, 
there exists an $N_2>0$ such that for any $n>N_2$, we have 
$$\ell_{\phi_n}(\theta)-\ell_{\phi_n}^\rho(\theta)  
<\ell_{\phi}(\theta)-\ell_\phi^\rho(\theta)+\varepsilon<2\varepsilon.$$
Let $N=\max\{N_1,N_2\}$.
Then for any $n>N$ and any $r\in[1,\rho]$, 
\begin{align*}
      &|\phi_n(r e^{2\pi i\theta})-\phi(r e^{2\pi i\theta})| \\
\leq~ &|\phi_n(r e^{2\pi i\theta})-\phi_n(\rho e^{2\pi i\theta})|+
|\phi_n(\rho e^{2\pi i\theta})-\phi(\rho e^{2\pi i\theta})|+
|\phi(\rho e^{2\pi i\theta})-\phi(r e^{2\pi i\theta})|\\
\leq~ &(\ell_{\phi_n}(\theta)-\ell_{\phi_n}^\rho(\theta))
      + \varepsilon + (\ell_{\phi}(\theta)-\ell_\phi^\rho(\theta))\\
<~    & 2\varepsilon +\varepsilon +\varepsilon=4\varepsilon.
\end{align*}
Therefore, for any $n>N$ and any $r\in[1,2]$, we have 
$|\phi_n(r e^{2\pi i\theta})-\phi(r e^{2\pi i\theta})|<4\varepsilon$.
This shows $\phi_n$ converges to $\phi$ uniformly on $[1,2]e^{2\pi i\theta}$.
\end{proof}

% converge uniformly on some external rays
\begin{lemma}
\label{converge uniformly on some external rays}
Let $\{\phi_n\}_{n\in\mathbb{N}}\rightarrow \phi$ as in Lemma \ref{converge in measure of length function}.
Then there is a full-measure subset $E$ of $[0,1)$ (or $\RZ$) and a subsequence 
$\{\phi_{n_k}\}_{k\in\mathbb{N}}$ of $\{\phi_n\}_{n\in\mathbb{N}}$ with the following properties. 
\begin{enumerate}
\item For any $\varphi\in\{\phi, \phi_n{;~} n\in\mathbb{N}\}$ and any $\theta\in E$, 
the map $\varphi$ has a radial limit at $e^{2\pi i\theta}$.  
\item For any $\theta\in E$, the subsequence $\{\phi_{n_k}\}_{k\in\mathbb{N}}$ 
converges to $\phi$ uniformly on $[1,2]e^{2\pi i\theta}$.
\end{enumerate}
\end{lemma}

\begin{proof}
For each $\varphi\in\{\phi, \phi_n{;~} n\in\mathbb{N}\}$, 
let $E_\varphi = \{\theta\in[0,1){;~}\ell_\varphi(\theta)<\infty\}$. 
% Then for any $\theta\in E_\varphi$, the function $\varphi$ has a radial limit at $e^{2\pi i\theta}$. 
Let $E_0=\bigcap_{\varphi\in\{\phi, \phi_n{;~} n\in\mathbb{N}\}} E_\varphi $.
Then $E_0$ is a full-measure subset of $[0,1)$ with the first property. 

By Lemma \ref{converge in measure of length function}, the function $\ell_{\phi_n}$ converges in measure to $\ell_\phi$.
According to the relation between convergence a.e. and convergence in measure (see, e.g. \cite[Theorem 2.30]{Folland}), 
there is a subsequence $\{\ell_{\phi_{n_k}}\}_{k\in\mathbb{N}}$ of $\{\ell_{\phi_n}\}_{n\in\mathbb{N}}$ 
such that $\ell_{\phi_{n_k}}$ converges a.e. to $\ell_\phi$. 
Choose a full-measure subset $E$ of $E_0$ such that 
$\lim_{k\rightarrow\infty} \ell_{\phi_{n_k}}(\theta) = \ell_\phi(\theta)$ for every $\theta\in E$.
Then the second property follows from Lemma \ref{converge uniformly by convergence of length}. 
\end{proof}

% proof of continuity of Theta
\begin{proof}
[Proof of Proposition \ref{prop-convergence of external angles}]
By the Carath\'eodory kernel convergence theorem, the map 
$\phi_n:=\phi_{K_n}$ converges to $\phi:=\phi_K$ uniformly on compact subsets of $\mathbb{C}\setminus\overline{\mathbb{D}}$. 
We need to show $\theta_n:=\theta(K_n,\gamma_n)\rightarrow \theta:=\theta(K,\gamma)$ as $n\rightarrow\infty$. 
Choose a full-measure subset $E\subset\RZ$ and a subsequence 
$\{\phi_{n_k}\}_{k\in\mathbb{N}}$ of $\{\phi_n\}_{n\in\mathbb{N}}$ 
as in Lemma \ref{converge uniformly on some external rays}. 

Given $\varepsilon\in(0,\frac{1}{2})$. 
By the F. and M. Riesz theorem \cite[Appendix A]{Mil}, there exist 
$\eta_1\in E\cap(\theta-\varepsilon,\theta)$ and $\eta_2\in E\cap(\theta,\theta+\varepsilon)$ such that 
$R_K(\theta)$, $R_K(\eta_1)$ and $R_K(\eta_2)$ land at three different points in $\partial K$. 
Choose $r>1$ so that $\phi(V)\cap \gamma([0,1])=\emptyset$, 
where $$V= [1,r] e^{2\pi i\eta_2} \cup re^{2\pi i[\eta_2,\eta_1]}\cup [1,r] e^{2\pi i\eta_1}.$$ 
Since $\phi_{n_k}\rightarrow\phi$ uniformly on $V$ 
and $\gamma_{n_k}([0,1])\rightarrow\gamma([0,1])$ in Hausdorff topology, we have 
$\phi_{n_k}(V)\cap\gamma_{n_k}([0,1]) =\emptyset$ for $k$ large enough. 
Let $D$ denote the unbounded connected component of $\mathbb{C}\setminus(\overline{\mathbb{D}}\cup V)$. 
Note that $\gamma((0,1])\subset \phi(D)$.  
When $k$ is large enough, we have $\gamma_{n_k}((0,1])\subset \phi_{n_k}(D)$, 
hence $\theta_{n_k}\in[\eta_1,\eta_2]\subset(\theta-\varepsilon,\theta+\varepsilon)$. 
This shows $\theta_{n_k}\rightarrow\theta$ as $k\rightarrow\infty$. 

The above discussion shows any subsequence of  $\{\theta_n\}_{n\in\mathbb{N}}$ has a subsequence tending to $\theta$, hence $\theta_n\rightarrow\theta$ as $n\rightarrow\infty$. 
This completes the proof. 
\end{proof}

\subsection{Convergence of marked lemniscate maps}
\label{subsection convergence of marked LM}
Let $T=(|T|,\sigma,\delta)$ be a mapping scheme with $\Tn\neq\emptyset$. 
For each $n\in\mathbb{N}$, let $$(\mathbf{g}_n, K_n, Z_n, \Gamma_n) = ((g_{n,v})_{v\in\Tn},(K_{n,v})_{v\in|T|}, (\zeta_{n,v})_{v\in\Tp}, (\gamma_{n,v})_{v\in|T|} )$$ be a marked lemniscate map over $T$. 

By Proposition \ref{straightening of marked LM}, 
there exists a unique $\mathbf{f}_n\in\mathcal{C}(\mathcal{P}_0^T)$ such that $(\mathbf{g}_n, K_n,Z_n,\Gamma_n)$ 
is conjugate to $(\mathbf{f}_n,K(\mathbf{f}_n),Z'_n,\Gamma'_n)$ by $(\psi_{n,v})_{v\in|T|}$, where $(Z'_n,\Gamma'_n)$ is the standard marking of $(\mathbf{f}_n,K(\mathbf{f}_n))$. 
For each $v\in\Tp$, the restriction $\psi_{n,v}: K_{n,v}\rightarrow\overline{\mathbb{D}}$ is 
the continuous extension of the Riemann mapping normalized by $\psi_{n,v}(\zeta_{n,v})=0$ 
and $\psi_{n,v}(\gamma_{n,v}(0)) =1$. 
% For the corresponding objects of $(\mathbf{g}_n, K_n, Z_n, \Gamma_n)$, we add a subscript $n$. 

\begin{definition}
\label{definition convergence of lemniscate maps}
We say $(\mathbf{g}_n, K_n, Z_n, \Gamma_n)$ converges to $(\mathbf{g}_0, K_0, Z_0, \Gamma_0)$  if 
\begin{itemize}
\item 
for each $v\in\Tp$, we have $\psi_{n,v}^{-1}\rightarrow\psi_{0,v}^{-1}$ uniformly on $\overline{\mathbb{D}}$; 

\item 
for each $v\in\Tn$, we have $g_{n,v}\rightarrow g_{0,v}$ uniformly on compact subsets of $\mathbb{C}$;

\item
for each $v\in|T|$, we have $\gamma_{n,v}([0,1])\rightarrow \gamma_{0,v}([0,1])$ in Hausdorff topology. 
\end{itemize}
\end{definition}

The aim of this subsection is to prove the following: 

\begin{proposition}
\label{continuity-straightening-LM}
If $(\mathbf{g}_n, K_n, Z_n, \Gamma_n)\rightarrow (\mathbf{g}_0, K_0, Z_0, \Gamma_0)$, then $\mathbf{f}_n\rightarrow\mathbf{f}_0$. 
\end{proposition}

Recall from \S\ref{subsection straightening of marked lemniscate maps} 
that for a full nondegenerate continuum $K\subset\mathbb{C}$, 
we let $\phi_K: \mathbb{C}\setminus \overline{\mathbb{D}}\rightarrow  \mathbb{C}\setminus K$ 
be the conformal map normalized by $\lim_{z\rightarrow\infty} \phi_K(z)/z>0$. 

\begin{lemma}
\label{uniform convergence of extension}
Let $U$ and $U_n$ ($n\in\mathbb{N}$) be Jordan domains in $\mathbb{C}$. 
Let $\phi$ and $\phi_n$ denote the continuous extensions of the conformal maps 
$\phi_{\overline{U}}$ and $\phi_{\overline{U_n}}$ on $\mathbb{C}\setminus\mathbb{D}$ respectively. 
Let $\alpha:\partial \mathbb{D}\rightarrow \partial U$ and $\alpha_n:\partial \mathbb{D}\rightarrow \partial U_n$ be orientation-preserving homeomorphisms. 
Suppose $\alpha_n\rightarrow\alpha$ uniformly. 
Then: 
\begin{enumerate}
\item 
$\phi_n\rightarrow \phi$ uniformly on compact subsets of $\mathbb{C}\setminus\mathbb{D}$; 

\item 
$\alpha$ and $\alpha_n$ can be extended to homeomorphisms $\varphi:\mathbb{C}\setminus\mathbb{D}\rightarrow\mathbb{C}\setminus U$ and $\varphi_n:\mathbb{C}\setminus\mathbb{D}\rightarrow\mathbb{C}\setminus U_n$ respectively so that $\varphi_n\rightarrow\varphi$ uniformly on compact subsets of $\mathbb{C}\setminus\mathbb{D}$. 
\end{enumerate}
\end{lemma}

See \cite{Rado1923} or \cite[Theorem 2.11]{Pom} for Lemma \ref{uniform convergence of extension}(1). 
Here we give an alternative proof based on Lemma \ref{converge uniformly on some external rays}.  

\begin{proof}
[Proof of Lemma \ref{uniform convergence of extension}]
(1) 
By the Carath\'eodory kernel convergence theorem, the function 
$\phi_n$ converges to $\phi$ uniformly on compact subsets of $\mathbb{C}\setminus\overline{\mathbb{D}}$. 
So we just need to show $\phi_n$ converges to $\phi$ uniformly on $\overline{A}$, where $A = \{z\in\mathbb{C}{;~} 1<|z|<2\}$. 
Choose a full-measure subset $E$ of $\RZ$ and a subsequence 
$\{\phi_{n_k}\}_{k\in\mathbb{N}}$ of $\{\phi_n\}_{n\in\mathbb{N}}$ 
as in Lemma \ref{converge uniformly on some external rays}. 

Given $\varepsilon>0$. 
Since $\phi|_{\partial\mathbb{D}}$ is uniformly continuous, there is a $\rho\in(0,1)$ so that for any $t_1,t_2\in \mathbb{R}$ with $|t_1-t_2|\leq\rho$, we have $|\phi(e^{2\pi i t_1})-\phi(e^{2\pi i t_2})|<\varepsilon$. 
Fix an integer $m>\frac{4}{\rho}$. 
Choose $\theta_j\in E\cap(\frac{j}{m},\frac{j+1}{m})$ for each $0\leq j\leq m-1$. 
If $j'\equiv j\pmod m$, we set $\theta_{j'} = \theta_j$. 
Let $V_j=\{re^{2\pi i t}{;~} r\in(1,2), t\in(\theta_j,\theta_{j+1})\}$. 
It follows from $\partial V_j\cap \partial\mathbb{D} = e^{2\pi i [\theta_j,\theta_{j+1}]}\subset e^{2\pi i [\frac{j}{m},\frac{j+2}{m}]}$ that ${\diam}(\phi(\partial V_j\cap \partial\mathbb{D}))<\varepsilon$.

Fix $0\leq j\leq m-1$. 
Choose $\eta_j\in E\cap (\frac{j+2}{m},\frac{j}{m})$. 
Since $\alpha_{n_k}\rightarrow\alpha$ uniformly and $\theta_j, \theta_{j+1}, \eta_j \in E$, 
there is an $N>0$ independent of $j$ so that for any $k>N$, 
\begin{align*}
&\phi_{n_k}(e^{2\pi i \theta_j}),
\phi_{n_k}(e^{2\pi i \theta_{j+1}})\notin\alpha_{n_k}\circ \alpha^{-1}\circ\phi(e^{2\pi i [\frac{j+2}{m},\frac{j}{m}]}),  \\
&\phi_{n_k}(e^{2\pi i \eta_j})\notin\alpha_{n_k}\circ \alpha^{-1}\circ\phi(e^{2\pi i [\frac{j}{m},\frac{j+2}{m}]}),  \\
&{\diam}(\alpha_{n_k}\circ \alpha^{-1}\circ\phi(e^{2\pi i [\frac{j}{m},\frac{j+2}{m}]}))
<{\diam}(\phi(e^{2\pi i [\frac{j}{m},\frac{j+2}{m}]}))+\varepsilon
<2\varepsilon. 
\end{align*}
It follows that $\phi_{n_k}(\partial V_j\cap \partial\mathbb{D})\subset \alpha_{n_k}\circ \alpha^{-1}\circ\phi(e^{2\pi i [\frac{j}{m},\frac{j+2}{m}]})$ and ${\diam}(\phi_{n_k}(\partial V_j\cap \partial\mathbb{D}))<2\varepsilon$ whenever $k>N$. 
Since $\theta_j,\theta_{j+1}\in E$, there is an $N'>0$ independent of $j$ so that for any $k>N'$ and any $z\in \overline{\partial V_j\setminus \partial\mathbb{D}}$, we have $|\phi_{n_k}(z)-\phi(z)|<\varepsilon$. 
For any $k\geq\max\{N,N'\}$ and $z\in \partial V_j\cap \partial\mathbb{D}$, 
we have 
\begin{align*}
&|\phi_{n_k}(z)-\phi(z)|\\
\leq~& |\phi_{n_k}(z)-\phi_{n_k}(e^{2\pi i \theta_j})|+|\phi_{n_k}(e^{2\pi i \theta_j})-\phi(e^{2\pi i \theta_j})|+|\phi(e^{2\pi i \theta_j})-\phi(z)|\\
<~& {\diam}(\phi_{n_k}(\partial V_j\cap \partial\mathbb{D})) + \varepsilon + {\diam}(\phi(\partial V_j\cap \partial\mathbb{D}))\\
<~& 2\varepsilon+\varepsilon+\varepsilon=4\varepsilon. 
\end{align*}
Then $|\phi_{n_k}(z)-\phi(z)|<4\varepsilon$ for any $k>\max\{N,N'\}$ and any $z\in \partial V_j$, hence for any $z\in\overline{V_j}$ by the maximum modulus principle.  
This implies $\phi_{n_k}\rightarrow \phi$ uniformly on $\overline{A}=\bigcup_{j=0}^{m-1} \overline{V_j}$. 

The above discussion shows any subsequence of $\{\phi_n\}_{n\in\mathbb{N}}$ has a subsequence converging to $\phi$ uniformly on $\overline{A}$, hence 
$\phi_n\rightarrow\phi$ uniformly on $\overline{A}$. 

(2) 
Define $\varphi(z)=\phi(|z|\cdot\phi^{-1}\circ\alpha(z/|z|))$ and $\varphi_n(z)=\phi_n(|z|\cdot\phi_n^{-1}\circ\alpha_n(z/|z|))$ 
(compare the Alexander trick \cite[Appendix C2]{Teich-2-Hubbard}). 
To prove $\varphi_n\rightarrow\varphi$ uniformly on compact subsets of $\mathbb{C}\setminus\mathbb{D}$, 
we just need to check that $\phi_n^{-1}\circ\alpha_n:\partial \mathbb{D}\rightarrow\partial \mathbb{D}$ converges to $\phi^{-1}\circ\alpha:\partial \mathbb{D}\rightarrow\partial \mathbb{D}$ uniformly. 

Given $\varepsilon>0$.  
Since $\phi^{-1}\circ\alpha$ is uniformly continuous, 
we can choose $m$ large enough so that ${\diam}(\phi^{-1}\circ\alpha(a_j))<\varepsilon$ for every $j\in\{0,\dots,m-1\}$, where $a_j = e^{2\pi i [\frac{j}{m},\frac{j+1}{m}]}$. 
Now fix $j\in\{0,\dots,m-1\}$. 
Since $\phi_n|_{\partial \mathbb{D}}\rightarrow \phi|_{\partial \mathbb{D}}$ uniformly and $\alpha_n\rightarrow\alpha$, 
there is an $N_j>0$ so that for any $n>N_j$, we have 
\begin{align*}
&\alpha_n(e^{2\pi i \frac{j}{m}}),
\alpha_n(e^{2\pi i \frac{j+1}{m}})\notin\phi_n\circ \phi^{-1}\circ\alpha(e^{2\pi i [\frac{j+2}{m},\frac{j-1}{m}]}),  \\
&\alpha_n(e^{2\pi i (\frac{j}{m}+\frac{1}{2})})\notin\phi_n\circ \phi^{-1}\circ\alpha(a_{j-1}\cup a_j\cup a_{j+1}), 
% &{\diam}(\phi_n\circ \phi^{-1}\circ\alpha(a_{j-1}\cup a_j\cup a_{j+1})) <{\diam}(\alpha(a_{j-1}\cup a_j\cup a_{j+1}))+\varepsilon <4\varepsilon. 
\end{align*}
hence $\alpha_n(a_j)\subset \phi_n\circ \phi^{-1}\circ\alpha(a_{j-1}\cup a_j\cup a_{j+1})$, i.e.  $\phi_n^{-1}\circ \alpha_n(a_j)\subset \phi^{-1}\circ\alpha(a_{j-1}\cup a_j\cup a_{j+1})$. 
For any $n>N_j$ and any $z\in a_j$, we have 
$|\phi_n^{-1}\circ \alpha_n(z)-\phi^{-1}\circ \alpha(z)|\leq {\diam}(\phi^{-1}\circ\alpha(a_{j-1}\cup a_j\cup a_{j+1}))<3\varepsilon$. 
For any $n>\max\{N_0, \dots, N_{m-1}\}$ and any $z\in\partial\mathbb{D}$, we have $|\phi_n^{-1}\circ \alpha_n(z)-\phi^{-1}\circ \alpha(z)|<3\varepsilon$.  
Therefore $\phi_n^{-1}\circ\alpha_n\rightarrow\phi^{-1}\circ\alpha$ uniformly. 
This completes the proof. 
\end{proof}

\begin{lemma}
\label{uniform convergence of inverse}
Let $U,\Omega,\Omega_n$ be nonempty open subsets of $\mathbb{C}$.  
Suppose $f:U\rightarrow \Omega$ and $f_n:U\rightarrow \Omega_n$ are orientation-preserving homeomorphisms. 
If $f_n\rightarrow f$ uniformly on compact subsets of $U$, then
$f_n^{-1}\rightarrow f^{-1}$ uniformly on compact subsets of $\Omega$. 
\end{lemma}

\begin{proof}
Let $K$ be a compact subset of $\Omega$. 
When $n$ is large enough, we have $K\subset \Omega_n$, hence $f_n^{-1}$ makes sense on $K$. 
Without loss of generality, we assume $K=[0,1]^2$. 
Given $\varepsilon>0$.  
Since $f^{-1}|_K$ is uniformly continuous, 
there is a positive integer $m$ so that 
$\left[-\frac{1}{m},1+\frac{1}{m}\right]^2\subset\Omega$ and 
${\diam}(f^{-1}(S_{j,k}))<\varepsilon$ for every $(j,k)\in\{-1,0,\dots,m\}^2$, where $S_{j,k}= \left[\frac{j}{m},\frac{j+1}{m}\right]\times \left[\frac{k}{m},\frac{k+1}{m}\right]$. 
Fix $(j,k)\in\{0,\dots,m-1\}^2$. 
Let $\widehat S_{j,k} = \left[\frac{j-1}{m},\frac{j+2}{m}\right]\times \left[\frac{k-1}{m},\frac{k+2}{m}\right]$. 
Since $f_n\rightarrow f$ uniformly on compact subsets of $U$, 
there is an $N_{j,k}>0$ so that for any $n>N_{j,k}$, we have 
$S_{j,k}\subset f_n(f^{-1}(\widehat S_{j,k}))$, hence $f_n^{-1}(S_{j,k})\subset f^{-1}(\widehat S_{j,k})$. 
For any $n>N_{j,k}$ and any $z\in S_{j,k}$, we have 
$|f_n^{-1}(z)-f^{-1}(z)|\leq {\diam}(f^{-1}(\widehat S_{j,k}))<3\varepsilon$. 
For any $n>\max_{(j,k)\in\{0,\dots,m-1\}^2}N_{j,k}$ and any $z\in K$, we have $|f_n^{-1}(z)-f^{-1}(z)|<3\varepsilon$.  
Therefore $f_n^{-1}\rightarrow f^{-1}$ uniformly on $K$. 
\end{proof}

\begin{lemma}
\label{uniform-injectivity}
Let $g$ and $g_n$ ($n\in\mathbb{N}$) be topological polynomials of degree $d\geq1$. 
Suppose $g_n\rightarrow g$ uniformly on compact subsets of $\mathbb{C}$. 
Let $z_0\in\mathbb{C}\setminus\crit(g)$. 
Then there is an $r>0$ and an $N>0$ so that $g$ and $g_n$ are injective on $\mathbb{D}(z_0,r) = \{z\in\mathbb{C}{;~}|z-z_0|<r\}$ for $n>N$. 
% Then $\crit(g_n)\rightarrow\crit(g)$ in Hausdorff topology. 
\end{lemma}

\begin{proof} 
By Lemma \ref{pullback of complex structure}, the argument principle holds for topological polynomials. 
For a path $\gamma:[0,1]\rightarrow\mathbb{C}$ and a point $z\notin\gamma([0,1])$, let $\Delta(\gamma,z)$ denote the  angular increment along the path $\gamma$ around $z$. 

Choose $r>0$ so that $g$ is injective on $\overline{\mathbb{D}(z_0,2r)}$. 
Let $\gamma(t)=z_0+2re^{2\pi i t}$ for $t\in[0,1]$. 
Since $g_n\rightarrow g$ uniformly on $\overline{\mathbb{D}(z_0,2r)}$, 
there is an $N>0$ so that for any $n>N$ and any $z\in \mathbb{D}(z_0,r)$, 
we have $\Delta(g_n\circ \gamma, g_n(z))=\Delta(g\circ\gamma, g(z))=2\pi$, hence $g_n^{-1}(g_n(z))\cap \mathbb{D}(z_0,2r)=\{z\}$ by the argument principle. 
Thus $g_n|_{\mathbb{D}(z_0,r)}$ is injective for any $n>N$. 
This completes the proof. 
% This implies $\limsup_{n\rightarrow\infty}\crit(g_n)\subset \crit(g)$. 
%Given $\varepsilon>0$. 
%Let $c\in\crit(g)$. 
%For $r>0$, let $D(c,r)$ be the connected component of $g^{-1}(\mathbb{D}(g(c),r))$. 
%Fix an $r>0$ small enough so that $\deg(g|_{D(c,r)})=\deg(g,c)$ and ${\diam}(D(c,2r))<\varepsilon$. 
\end{proof}

Let $\Omega$ be a domain (nonempty, connected and open set) in $\mathbb{C}$, 
and let $\{h_n:\Omega\rightarrow\mathbb{C}\}_{n\in\mathbb{N}}$ be a sequence of continuous functions. 
We say $\{h_n\}_{n\in\mathbb{N}}$ is \emph{normal} if any subsequence of $\{h_n\}_{n\in\mathbb{N}}$ has a subsequence that  converges to a limit function $h:\Omega\rightarrow\mathbb{C}$ uniformly on compact subsets of $\Omega$. 
The sequence $\{h_n\}_{n\in\mathbb{N}}$ is said to be \emph{normal at a point} $z_0\in\Omega$ if it is normal in some neighborhood of $z_0$. 
Clearly, the sequence $\{h_n\}_{n\in\mathbb{N}}$ is normal if and only if it is normal at every point in $\Omega$. 

\begin{proof}
[Proof of Proposition \ref{continuity-straightening-LM}]
Assume $(\mathbf{g}_n, K_n, Z_n, \Gamma_n)\rightarrow (\mathbf{g}_0, K_0, Z_0, \Gamma_0)$. 
Let $v\in\Tp$. 
By Lemma \ref{uniform convergence of extension}, each $\psi_{n,v}^{-1}|_{\overline{\mathbb{D}}}$ can be extended to a homeomorphism $\varphi_{n,v}:\mathbb{C}\rightarrow\mathbb{C}$  so that $\varphi_{n,v}\rightarrow\varphi_{0,v}$ uniformly on compact subsets of $\mathbb{C}$. 
Let $\phi_{n,v} = \varphi_{n,v}^{-1}$. 
By Lemma \ref{uniform convergence of inverse}, we have 
$\phi_{n,v}\rightarrow\phi_{0,v}$ uniformly on compact subsets of $\mathbb{C}$. 
By Corollary \ref{construction of strong conjugacy}, 
% By Lemma \ref{independence of extension}, 
the collection $(\phi_{n,v})_{v\in\Tp}$ determines a strong  conjugacy $(\phi_{n,v})_{v\in|T|}$ from $(\mathbf{g}_n, K_n, Z_n, \Gamma_n)$ to $(\mathbf{f}_n,K(\mathbf{f}_n),Z'_n,\Gamma'_n)$. 
After passing to a subsequence, we can assume $\mathbf{f}_n\rightarrow\mathbf{f}$. 
% Let $v\in\Tn$. 
% By induction, we assume $\phi_{n,\sigma(v)}\rightarrow\phi_{0,\sigma(v)}$ uniformly on compact subsets of $\mathbb{C}$. 

\vspace{6pt}
{\bf Claim.}
{\it 
After passing to a subsequence, for each $v\in|T|$, the homeomorphism $\phi_{n,v}$ converges to a homeomorphism $\phi_{v}:\mathbb{C}\rightarrow\mathbb{C}$ uniformly on compact subsets of $\mathbb{C}$. }
\vspace{6pt}

For each $v\in\Tp$, we let $\phi_v = \phi_{0,v}$. 
Now let $v\in\Tn$.
By induction, we can suppose $\phi_{\sigma(v)}$ has been constructed. 
Let $z_0\in \mathbb{C}\setminus\crit(g_{0,v})$. 
By Lemma \ref{uniform-injectivity}, 
there is an $r_1>0$ and an $N_1>0$ so that $g_{n,v}$ is  injective on $\mathbb{D}(z_0,r_1)$ for any $n>N_1$. 
Let $w_0 = \phi_{\sigma(v)}\circ g_{0,v}(z_0)$. 
Since $\phi_{n,\sigma(v)}\circ g_{n,v}\rightarrow \phi_{\sigma(v)}\circ g_{0,v}$ uniformly on $\mathbb{D}(z_0,r_1)$, 
we can choose $r_2>0$ and $N_2>N_1$ so that 
$\mathbb{D}(w_0,r_2)\subset \phi_{n,\sigma(v)}\circ g_{n,v}(\mathbb{D}(z_0,r_1))$ for any $n>N_2$. 
Note that $\phi_{n,v} = f_{n,v}^{-1}\circ \phi_{n,\sigma(v)}\circ g_{n,v}$, where we take $f_{n,v}^{-1}(z) = \phi_{n,v}\circ (g_{n,v}|_{\mathbb{D}(z_0,r_1)})^{-1}\circ \phi_{n,\sigma(v)}^{-1}(z)$ for $z\in \mathbb{D}(w_0,r_2)$. 
Since $\{f_{n,v}^{-1}\}_{n>N_2}$ is normal at $w_0$, the sequence $\{\phi_{n,v}\}_{n>N_2}$ is normal at $z_0$. 
Since $z_0$ was arbitrary, the sequence $\{\phi_{n,v}\}_{n\geq1}$ is normal on $\mathbb{C}\setminus\crit(g_{0,v})$. 
After passing to a subsequence, we assume $\phi_{n,v}\rightarrow \phi_v$ uniformly on compact subsets of $\mathbb{C}\setminus\crit(g_{0,v})$. 
It follows that $\phi_{\sigma(v)}\circ g_{0,v}(z) = f_v\circ \phi_v(z)$ for $z\in \mathbb{C}\setminus\crit(g_{0,v})$. 
$$
\xymatrix{
  \mathbb{C}  \ar[d]_{\phi_{n,v}}  \ar[r]^{g_{n,v}}   
& \mathbb{C}  \ar[d]^{\phi_{n,\sigma(v)}}  
& \ar @{}[dr]|{\ \ \ \ \xrightarrow{n\rightarrow\infty}}
& 
& \mathbb{C}  \ar @{.>}[d]_{\phi_{v}}  \ar[r]^{g_{0,v}} 
& \mathbb{C}  \ar[d]^{\phi_{\sigma(v)}}
\\
  \mathbb{C}  \ar[r]_{f_{n,v}}
& \mathbb{C}  
& 
& 
& \mathbb{C}  \ar[r]_{f_{v}}
& \mathbb{C}  
}$$

Let $c\in\crit(g_{0,v})$, and let $w_c = \phi_{\sigma(v)}\circ g_{0,v}(c)$. 
For $r>0$, let $D(c,r)$ denote the connected component of $(\phi_{\sigma(v)}\circ g_{0,v})^{-1}(\mathbb{D}(w_c,r))$ containing $c$, and let $E(c,r)$ denote the connected component of $f_v^{-1}(\mathbb{D}(w_c,r))$ containing $\phi_v(D(c,r)\setminus\{c\})$. 
We remark that $D(c,r)$ and $E(c,r)$ are Jordan domains. 
Fix an $r>0$ small enough so that $\deg(\phi_{\sigma(v)}\circ g_{0,v}|_{D(c,r)}) = \deg(\phi_{\sigma(v)}\circ g_{0,v},c)$ and $f_v^{-1}(w_c)\cap E(c,r)$ is a singleton, which will be denoted by $\{z_c\}$. 
It follows that $\deg(f_v|_{E(c,r)})=\deg(f_v,z_c)$. 
As the lifting in the commutative diagram 
$$
\xymatrix{
 & &E(c,r)\setminus\{z_c\}  \ar[d]^{f_v} \\
D(c,r)\setminus\{c\} \ar[rr]_-{\phi_{\sigma(v)}\circ g_{0,v}} \ar[rru]^{\phi_v} & &\mathbb{D}(w_c,r)\setminus\{w_c\}, 
}$$
$\phi_v: D(c,r)\setminus\{c\}\rightarrow E(c,r)\setminus\{z_c\}$ is a covering map with $\lim_{z\rightarrow c}\phi_v(z) = z_c$. 
% of degree $\deg(\phi_{\sigma(v)}\circ g_{0,v},c)-\deg(f_v,z_c)$. 
Let $\phi_v(c) = z_c$. 
Then $\phi_v:\mathbb{C}\rightarrow\mathbb{C}$ is a topological polynomial. 
Because $\phi_{\sigma(v)}\circ g_{0,v}(z) = f_v\circ \phi_v(z)$ for $z$ near $\infty$, we have $\deg(\phi_v)=1$, hence $\phi_v:\mathbb{C}\rightarrow\mathbb{C}$ is a homeomorphism. 

For each $c\in\crit(g_{0,v})$, it follows from $\lim_{r\rightarrow0}{\diam}(E(c,r))=0$ 
% Since the above $r$ can be chosen arbitrarily small, we have  $\phi_{n,v}\rightarrow \phi_v$ uniformly near every critical point . 
that $\phi_{n,v}\rightarrow \phi_v$ uniformly on compact subsets of $\mathbb{C}$. 
This shows the claim. 

\vspace{6pt}
By the claim above, we see that $(\mathbf{g}_0, K_0)$ is strongly conjugate to $(\mathbf{f},K(\mathbf{f}))$ by $(\phi_v)_{v\in|T|}$. 
For each $v\in\Tp$, since $\phi_v = \phi_{0,v}$, we have  $\phi_v(\zeta_{0,v})=0$. 
For each $v\in|T|$, 
since $K_{\mathbf{f}_n,v}\rightarrow K_{\mathbf{f},v}$ and $\phi_{n,v}\circ \gamma_{n,v}([0,1])\rightarrow \phi_{v}\circ \gamma_{0,v}([0,1])$ in Hausdorff topology, 
we have 
$$\theta(K_{\mathbf{f},v},\phi_v\circ\gamma_{0,v}) = \lim_{n\rightarrow\infty} \theta(K_{\mathbf{f}_n,v},\phi_{n,v} \circ \gamma_{n,v}) =0$$
by Proposition \ref{prop-convergence of external angles}.  
This shows $(\mathbf{g}_0,K_0,Z_0,\Gamma_0)$ is strongly conjugate to $(\mathbf{f},K(\mathbf{f}),$ $Z',\Gamma')$ by $(\phi_v)_{v\in|T|}$, where $(Z',\Gamma')$ is the standard marking of $(\mathbf{f},K(\mathbf{f}))$. 
By the uniqueness in Proposition \ref{straightening of marked LM}, we have $\mathbf{f}=\mathbf{f}_0$. 

The above discussion shows any subsequence of  $\{\mathbf{f}_n\}_{n\geq1}$ has a subsequence tending to $\mathbf{f}_0$, hence $\mathbf{f}_n\rightarrow\mathbf{f}_0$ as $n\rightarrow\infty$. 
This completes the proof. 
\end{proof}

\begin{remark}
\label{continuity-conjugacy}
In the proof of Proposition \ref{continuity-straightening-LM},  
since $(\phi_v)_{v\in\Tp}=(\phi_{0,v})_{v\in\Tp}$, we have $(\phi_v)_{v\in|T|}=(\phi_{0,v})_{v\in|T|}$ 
% by Lemma \ref{independence of extension}. 
by Corollary \ref{construction of strong conjugacy}. 
For each $v\in|T|$, the above discussion shows any subsequence of $\{\phi_{n,v}\}_{n\in\mathbb{N}}$ 
has a subsequence converging to $\phi_{0,v}$ uniformly on compact subsets of $\mathbb{C}$, hence $\phi_{n,v}\rightarrow \phi_{0,v}$ uniformly on compact subsets of $\mathbb{C}$. 
\end{remark}

\section{Straightening map}
\label{section straightening map}
Let $\mathcal{H}\subset\mathcal{F}$ be a capture hyperbolic component with center $f_0$. 
We will associate each polynomial in $\Hbar$ with a marked lemniscate map over $T(f_0)$,  
and define a straightening map $\chi:\Hbar\rightarrow \modelf$ 
by the straightening theorem of marked lemniscate maps. 
Finally, we show $\chi$ is a homeomorphism. 
This implies Theorem \ref{top-boundary}: $\partial \mathcal H$ is homeomorphic to the sphere $S^{2\Cdim(\mathcal H)-1}$. 

%\begin{lemma}
%\label{continuity of Kfv}
%For each $v\in|T(f_0)|$, the compact set $K_{f,v}$ is continuous with respect to $f\in\piHbar$ in the sense of Hausdorff topology. 
%\end{lemma}

\subsection{Definition of the straightening map}
\label{subsection def-straightening-map}
\begin{definition}
[External marking for $f_0$]
By an \emph{external marking} for $f_0$, we will mean a collection $(\theta_v)_{v\in|T(f_0)|}$ of 
angles in $\RZ$ so that $R_{f_0}(\theta_v)$ lands at a point in $\partial U_{f_0}(v)$ and 
$\theta_{\sigma(v)} \equiv d\cdot\theta_v \modZ$ for each $v\in|T(f_0)|$, where $d=\deg(f_0)$. 
\end{definition}

We will construct an external marking for $f_0$ as follows. 
Let $v\in\Tfp$ have period $p$ under $\sigma$, and let $x$ be a fixed point of $f_0^p$ in $\partial U_{f_0}(v)$. 
Choose $\theta_v\in\RZ$ so that $R_{f_0}(\theta_v)$ lands at $x$. 
Since $v$ and $x$ are fixed points of $f_0^p$, we have $d^p\cdot\theta_v \equiv \theta_v \modZ$.  
If $p\geq2$, for each $1\leq j<p$, we let $\theta_{\sigma^j(v)} = d^j\cdot\theta_v$. 
Then $R_{f_0}(\theta_{\sigma^j(v)})$ lands at $f_0^j(x) \in \partial U_{f_0}(\sigma^j(v))$. 
Now let $v\in\Tfn$ with $\sigma(v)\in\Tfp$. 
Since $f_0(U_{f_0}(v)) = U_{f_0}(\sigma(v))$, there is an angle  $\theta_v\in m_d^{-1}(\theta_{\sigma(v)})$ so that $R_{f_0}(\theta_v)$ lands at a point in $\partial U_{f_0}(v)$. 
Inductively, we get an external marking $(\theta_v)_{v\in|T(f_0)|}$ for $f_0$. 
We remark that there are only finitely many choices of external markings for $f_0$.

% Note that $\theta_v$ may be the same for different $v\in|T(f_0)|$ (see Examples \ref{example of f0} and \ref{example of gf}). 
In the following of this section, we always fix an external marking $(\theta_v)_{v\in|T(f_0)|}$ for 
$f_0$, with $\theta_v = \theta^-_{f_0,v}(0)$ for each $v\in\Tfp$ (see \S\ref{subsection internal rays} for $\theta^-_{f_0,v}(0)$). 

Recall from Lemma \ref{invariant external angles} that we have defined $\theta_v^-$ and $\theta_v^+$ for each $v\in\Tfp$. 
We will associate each $v\in|T(f_0)|$ with $\theta_v^-$ and $\theta_v^+$ as follows. 
Let $\phi_v:\overline{U_{f_0}(v)}\rightarrow \overline{\mathbb{D}}$ be the continuous extension of the Riemann mapping so that $\phi_v(v)=0$ and $\phi_v^{-1}(1)$ is the landing point of $R_{f_0}(\theta_v)$. 
Let $t\in \RZ$ and let $x = \phi_v^{-1}(e^{2\pi i t})$. 
If there is only one external ray $R_{f_0}(\theta)$ landing at $x$, we define $\theta^-_v(t)=\theta^+_v(t)=\theta$. 
If there are at least two external rays landing at $x$, 
we let $\theta_v^-(t)$ and $\theta_v^+(t)$ be two external angles of $x$ such that $S_{f_0}(\theta_v^-(t),\theta_v^+(t))$ is a maximal sector attaching to $U_{f_0}(v)$. 
It follows from $\phi_{\sigma(v)}(f_0(z))=\phi_{v}(z)^{\delta(v)}$ that 
\begin{equation}
\label{d-delta}
d\cdot\theta_v^-(t) \equiv \theta_{\sigma(v)}^-(\delta(v)\cdot t), \quad
d\cdot\theta_v^+(t) \equiv \theta_{\sigma(v)}^+(\delta(v)\cdot t)\modZ.
\end{equation}
Let $f\in\Hbar$. 
Then $\lambda_\mathbb{R}(f_0)\subset\lambda_\mathbb{R}(f)$ by Corollary \ref{inclusion relation of real laminations}. 
For each $v\in|T(f_0)|$, we define 
\begin{equation}
\label{Kfv}
K_{f,v} = K(f) \setminus \bigsqcup_{
\begin{subarray}{c}
t\in\RZ,\\
\theta_v^-(t) \neq \theta_v^+(t)
\end{subarray}
} S_{f}(\theta_v^-(t),\theta_v^+(t)). 
\end{equation}
Since the sectors appearing in (\ref{Kfv}) are pairwise disjoint, the set $K_{f,v}$ is connected. 
By Lemma \ref{invariant external angles}, if $v\in\Tfp$, then $K_{f,v}=\overline{U_{f,v}}$ is a closed Jordan domain. 

\begin{example}
Consider $f_0$ and $f$ as in Figure \ref{figure Peanuts1}. 
Let $v_1 = c_1(f_0)$, $v_2 = c_2(f_0)$, $v_3 = c_3(f_0)$, $v_4 = f_0(c_2(f_0))$ and $v_5 = f_0(c_3(f_0))$. 
See Figure \ref{figure peanuts} for $(K_{f,v})_{v\in|T(f_0)|}$. 

%Consider $\mathcal{F} = \big\{f\in\widehat{\mathcal{P}}^4{;~} f(c_1(f)) = c_1(f)\big\}$. 
%Let $f_0\in\mathcal{F}$ satisfy $f_0^2(c_2(f_0)) = f_0^2(c_3(f_0)) = c_1(f_0)$ with $(\mathbf{c}(f_0),f_0(0))
%\approx (- 0.7232,~ 0.3616 + 0.7379i,~ 0.3616 - 0.7379i,~ 0.2565)$. 
%Let $v_1 = c_1(f_0)$, $v_2 = c_2(f_0)$, $v_3 = c_3(f_0)$, $v_4 = f_0(c_2(f_0))$ and $v_5 = f_0(c_3(f_0))$. 
%Let $\mathcal{H}\subset\mathcal{F}$ denote the capture hyperbolic component with center $f_0$. 
%Choose $f\in\Hbar$ with $(\mathbf{c}(f),f(0))
%\approx (-0.7187+0.0048i,~ 0.3629+0.7337i,~ 0.3558-0.7385i,~ 0.2410-0.0187i)$. 
%See Figure \ref{figure peanuts} for $(K_{f,v})_{v\in|T(f_0)|}$. 

\begin{figure}[ht]
\centering
\includegraphics{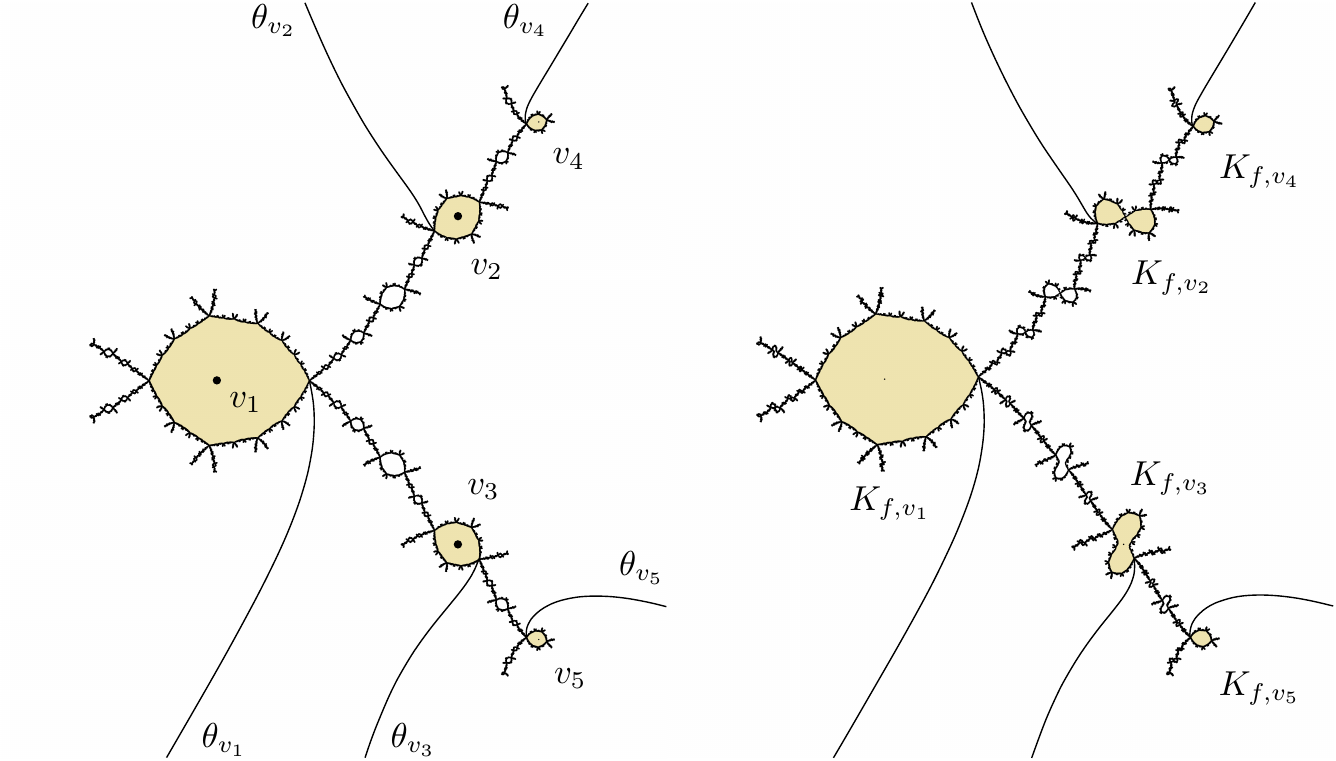}
\caption{The small filled Julia sets $(K_{f,v})_{v\in|T(f_0)|}$.} 
\label{figure peanuts}
\end{figure}
\end{example}

\begin{lemma}
\label{hfvt}
Let $f\in\Hbar$, $v\in\Tfn$, and let $t\in\RZ$ with $\theta_v^-(t)\neq\theta_v^+(t)$. 
Then there is a homeomorphism $$h_{f,v,t}: \overline{S_{f}(\theta_v^-(t),\theta_v^+(t))}\rightarrow \overline{S_{f}(d\cdot\theta_v^-(t),d\cdot\theta_v^+(t))}$$
such that $h_{f,v,t}(z)=f(z)$ for any $z\in\partial S_{f}(\theta_v^-(t),\theta_v^+(t))$. 
\end{lemma}

\begin{proof}
Let $\mathbb{H}=\{z\in\mathbb{C}{;~}\operatorname{Im}(z)>0\}$ denote the upper halfplane. 
Choose $\zeta_{f,v,t}\in S_{f,v,t} := S_{f}(\theta_v^-(t),\theta_v^+(t))$. 
By the Riemann mapping theorem, there is a unique homeomorphism $\phi_{f,v,t}:\overline{\mathbb{H}}\rightarrow \overline{S_{f,v,t}}$ 
that is holomorphic on $\mathbb{H}$, normalized by $\phi_{f,v,t}(i)=\zeta_{f,v,t}$ and $\lim_{z\rightarrow\infty} \phi_{f,v,t}(z) = \infty$. 
Choose $\widetilde\zeta_{f,v,t}\in \widetilde S_{f,v,t} := S_{f}(d\cdot\theta_v^-(t),d\cdot\theta_v^+(t))$ and define $\widetilde\phi_{f,v,t}: \overline{\mathbb{H}}\rightarrow\overline{\widetilde S_{f,v,t}}$ in the same way. 
Define $H_{f,v,t}:\overline{\mathbb{H}}\rightarrow \overline{\mathbb{H}}$ by 
$H_{f,v,t}(z) = \widetilde\phi^{-1}_{f,v,t}\circ f\circ \phi_{f,v,t}(\operatorname{Re}(z))+i\operatorname{Im}(z)$. 
Letting $h_{f,v,t}=\widetilde\phi_{f,v,t}\circ H_{f,v,t}\circ \phi^{-1}_{f,v,t}$ completes the proof. 
\end{proof}

% $$\begin{CD}
%\overline{\mathbb{H}} @> H_{f,v,t} >>   \overline{\mathbb{H}} \\
%@V \phi_{f,v,t} VV            @VV \widetilde\phi_{f,v,t} V\\
%\overline{S_{f,v,t}}  @>> h_{f,v,t} >   \overline{\widetilde S_{f,v,t}}
%\end{CD}$$

Let $f\in\Hbar$.  
For each $v\in\Tfn$, let $$T_v = \{t\in\RZ{;~}
\text{$\theta_v^-(t)\neq\theta_v^+(t)$ and 
$\crit(f_0)\cap S_{f_0}(\theta_v^-(t),\theta_v^+(t))\neq\emptyset$}\},$$
and define $g_{f,v}:\mathbb{C}\rightarrow\mathbb{C}$ by 
\begin{equation*}
g_{f,v}(z) = \begin{cases}
h_{f,v,t}(z),& \text{$z\in S_{f}(\theta_v^-(t),\theta_v^+(t))$ for some $t\in T_v$}, \\
f(z),& z\in \mathbb{C}\setminus \bigsqcup_{t\in T_v} S_{f}(\theta_v^-(t),\theta_v^+(t)),
\end{cases}
\end{equation*}
where $h_{f,v,t}$ is given by Lemma \ref{hfvt}. 
By (\ref{d-delta}), the map $g_{f,v}$ is a topological polynomial of degree $\delta(v)$ with $K_{f,v} = g_{f,v}^{-1}(K_{f,\sigma(v)})$. 
Clearly $g_{f,v}|_{K_{f,v}^\circ} = f|_{K_{f,v}^\circ}$ is holomorphic. 
Recall from \S\ref{subsection internal rays} that $\zeta_v(f)$ is the superattracting periodic point in $U_{f,v}$ for each $v\in\Tfp$. 
Let 
\begin{equation}
\label{(gf,Kf,Zf,Gammaf)}
\begin{split}
(\mathbf{g}_f,K_f,Z_f,\Gamma_f)=
(&(g_{f,v})_{v\in\Tfn},(K_{f,v})_{v\in|T(f_0)|}, \\
&(\zeta_v(f))_{v\in\Tfp}, (R_f(\theta_v))_{v\in|T(f_0)|} ).
\end{split}
\end{equation}
Then it is a marked lemniscate map over $T(f_0)$. 
By Proposition \ref{straightening of marked LM}, 
there exists a unique $\mathbf{P}_f\in
\mathcal{C}\big(\mathcal{P}_0^{T(f_0)}\big)$ 
such that $(\mathbf{g}_f,K_f,Z_f,\Gamma_f)$ is conjugate to $(\mathbf{P}_f,$ $K(\mathbf{P}_f),$ $Z',\Gamma')$, where $(Z',\Gamma')$ is the standard marking of $(\mathbf{P}_f,K(\mathbf{P}_f))$. 
Furthermore, we fix a conjugacy $(\psi_{f,v})_{v\in|T(f_0)|}$ from $(\mathbf{g}_f,K_f,Z_f,\Gamma_f)$ to $(\mathbf{P}_f,K(\mathbf{P}_f),Z',\Gamma')$. 
% When $v\in\Tfp$, since we require $\theta_v = \theta^-_{f_0,v}(0)$, 
% the conjugacy $\psi_{f,v}: K_{f,v}\rightarrow\overline{\mathbb{D}}$ is exactly the  B\"ottcher coordinate $B_{f,v}:\overline{U_{f,v}}\rightarrow \overline{\mathbb{D}}$ defined in \S\ref{subsection internal rays}. 

Let $$I_0 = \{(v,k){;~} v\in|T(f_0)|, 1\leq k<\delta(v)\}.$$
For $f \in\widehat{\mathcal{P}}^d$, we can use the representative 
$$\mathbf{c}(f) = (c_{v,k}(f))_{(v,k)\in I_0}$$ so that $c_{v,k}(f_0) = v$ for every $(v,k)\in I_0$. 
% Let $$I = \{(v,k){;~} v\in\Tfn, 1\leq k<\delta(v)\}.$$
%For $f\in\mathcal{F}$, the free critical points $(c_j(f))_{\ell+1\leq j\leq d-1}$ can be represented by $(c_{v,k}(f))_{(v,k)\in I}$ so that 
% $c_{v,k}(f_0) = v$ for every $(v,k)\in I$. 
Define the \emph{straightening map} 
\begin{equation}
\label{straightening map}
\chi:\begin{cases}\overline{\mathcal{H}}\rightarrow \modelf,\\
f \mapsto 
(\mathbf{P}_f, (\psi_{f,v}(c_{v,k}(f)))_{(v,k)\in I_0}).  
\end{cases}
\end{equation}
By Proposition \ref{straightening of marked LM}, 
the map $\mathbf{P}_f$ and the restriction $(\psi_{f,v}|_{K_{f,v}})_{v\in |T(f_0)|}$ 
are determined by $((g_{f,v}|_{K_{f,v}})_{v\in\Tfn},K_f,Z_f,\Gamma_f)$, hence by $f$ and $(\theta_v)_{v\in|T(f_0)|}$.  
Let $v\in\Tfn$ with $\delta(v)\geq2$. 
Similar to Proposition \ref{continuity of Ufv}, 
the compact set $K_{f,v}$ is continuous with respect to $f\in\Hbar$ in Hausdorff topology, 
hence $c_{v,1}(f), \dots, c_{v,\delta(v)-1}(f)\in K_{f,v}$. 
Moreover, the ordered list $(c_{v,k}(f))_{1\leq k<\delta(v)}$ is a critical marking of $g_{f,v}$. 
\footnote{
For $f\in\Hbar$, $v\in\Tfn$ and $x\in K_{f,v}$, 
since \begin{align*}
\deg(g_{f,v},x) &= 1+\#\{(v',k')\in I_0{;~}\text{$c_{v',k'}(f)=x$ and $v'=v$}\},\\
\deg(f,x)       &= 1+\#\{(v',k')\in I_0{;~}c_{v',k'}(f)=x\}, 
\end{align*}  
it can happen that $\deg(g_{f,v},x)<\deg(f,x)$.
} 
It follows that $(\psi_{f,v}(c_{v,k}(f)))_{(v,k)\in I_0}$ is a critical marking of $\mathbf{P}_f$. 
Therefore $\chi$ is well-defined.

\subsection{Proof of Theorem \ref{top-boundary}}
\begin{lemma}
\label{corresponding of external rays}
Let $f\in\Hbar$, $v\in|T(f_0)|$ and $t\in\RZ$. 
Then $\psi_{f,v}(R_f(\theta_v^-(t)))\simeq_{K_{\mathbf{P}_f,v}}R_{\mathbf{P}_f,v}(t)$. 
In particular, the external ray $R_{\mathbf{P}_f,v}(t)$ lands at $\psi_{f,v}(x)$, where $x$ is the landing point of $R_f(\theta_v^-(t))$. 
\end{lemma}

\begin{proof}
Define $\psi_{f,v}^*:\RZ\rightarrow\RZ$ by $\psi_{f,v}^*(t)=\theta(K_{\mathbf{P}_f,v}, \psi_{f,v}(R_f(\theta_v^-(t))))$. 
Similar to the claim in the proof of Lemma \ref{uniqueness-f-CP0T},  
we will check that $\psi_{f,v}^*=\id$ for each $v\in|T(f_0)|$. 

By Lemma \ref{invariant external angles}, we have $\psi_{f,v}^*=\id$ if $v\in\Tfp$. 
Now assume $v\in\Tfn$. 
By induction, we can assume $\psi_{f,\sigma(v)}^*=\id$. 
For any $t\in\RZ$, we have 
\begin{align*}
&\ \ \ \ R_{\mathbf{P}_f,{\sigma(v)}}(\delta(v)\cdot\psi_{f,v}^*(t))\\
&=   P_{f,v}(R_{\mathbf{P}_f,v}(\psi_{f,v}^*(t)))
\quad(\text{by (\ref{Bfv})})\\
&\simeq_{K_{\mathbf{P}_f,\sigma(v)}}   
     P_{f,v}\circ\psi_{f,v}(R_f(\theta_v^-(t)))
\quad(\text{by $R_{\mathbf{P}_f,v}(\psi_{f,v}^*(t))\simeq_{K_{\mathbf{P}_f,v}}   
    \psi_{f,v}(R_f(\theta_v^-(t)))$})\\
&\simeq_{K_{\mathbf{P}_f,\sigma(v)}}   
   \psi_{f,\sigma(v)}\circ g_{f,v}(R_f(\theta_v^-(t)))
\quad(\text{by Lemma \ref{the same extension on the ideal boundary}})\\
&=   \psi_{f,\sigma(v)}\circ f(R_f(\theta_v^-(t)))
 =   \psi_{f,\sigma(v)}(R_f(d\cdot\theta_v^-(t)))\\
&=   \psi_{f,\sigma(v)}(R_f(\theta_{\sigma(v)}^-(\delta(v)\cdot t)))
\quad(\text{by (\ref{d-delta})})\\
&\simeq_{K_{\mathbf{P}_f,\sigma(v)}}  
   R_{\mathbf{P}_f,{\sigma(v)}}(\delta(v)\cdot t)
\quad(\text{by the assumption $\psi_{f,\sigma(v)}^*=\id$}). 
\end{align*}
Thus $m_{\delta(v)}\circ \psi_{f,v}^* = m_{\delta(v)}$. 
Note that for any $t_1,t_2,t_3\in\RZ$ in positive cyclic order, the images $\psi_{f,v}^*(t_1),\psi_{f,v}^*(t_2),\psi_{f,v}^*(t_3)$ are in positive cyclic order. 
It follows from $\psi_{f,v}^*(0)=0$ that  $\psi_{f,v}^*=\id$. 
This completes the proof. 
\end{proof}

Let $|T^\infty(f_0)|$ denote the grand orbit of $\crit(f_0)$ under $f_0$. 
We extend $\sigma$ and $\delta$ to $|T^\infty(f_0)|$ by $\sigma(v)=f_0(v)$ and $\delta(v)=\deg(f_0,v)$. 
Then we have an ``infinite mapping scheme" $T^\infty(f_0)=(|T^\infty(f_0)|,\sigma,\delta)$. 
% If $u\in|T^\infty(f_0)|\setminus|T(f_0)|$, let $|T^u(f_0)|=|T(f_0)|\cup\{\sigma^n(u) {;~}  n\in\mathbb{N}\}$. 
% Replacing $T(f_0)$ by the mapping scheme $T^u(f_0)=(|T^u(f_0)|,\sigma,\delta)$, 
All notations and properties for $v\in|T(f_0)|$ are still valid for $v\in|T^\infty(f_0)|$. 

\begin{lemma}
\label{chi is injective}
$\chi$ is injective. 
\end{lemma}

\begin{proof}
Given $f_1,f_2\in\overline{\mathcal{H}}$ with $\chi(f_1)=\chi(f_2)$. 
We need to show $f_1 = f_2$. 
By Corollary \ref{inclusion relation of real laminations}, we have $\lambda_\mathbb{R}(f_0)\subset\lambda_\mathbb{R}(f_1)$ and $\lambda_\mathbb{R}(f_0)\subset\lambda_\mathbb{R}(f_2)$. 

\vspace{6pt}
{\bf Claim.}
{\it $\lambda_\mathbb{R}(f_1)=\lambda_\mathbb{R}(f_2)$.}
\vspace{6pt}

Let $S_{f_1}(\theta,\theta')$ be a minimal sector with root point $x_1$. 
To prove $(\theta,\theta')\in\lambda_{\mathbb{R}}(f_2)$, let $x_0,x'_0,x_2,x'_2$ be the landing points of $R_{f_0}(\theta), R_{f_0}(\theta'), R_{f_2}(\theta),R_{f_2}(\theta')$ respectively. 
If $x_0=x'_0$, then $(\theta,\theta')\in\lambda_{\mathbb{R}}(f_0)\subset\lambda_{\mathbb{R}}(f_2)$. 
In the following, we assume $x_0\neq x'_0$. 
Consider the allowable arc $[x_0,x'_0]\subset K(f_0)$. 
Assume $(x_0,x'_0)\cap J(f_0)\neq \emptyset$; we will find a contradiction. 
Choose $y\in(x_0,x'_0)\cap J(f_0)$. 
By Lemma \ref{allowable arcs in K(f0)}(\ref{allowable arcs in K(f0)-2}), 
there exists a pair of external angles $\theta_1,\theta_2$ of $y$, 
corresponding to the two homotopy classes of paths to $y$ relative to $[x_0,x'_0]$, 
such that $\theta,\theta_1,\theta',\theta_2$ are in positive cyclic order. 
Since $(\theta_1,\theta_2)\in\lambda_{\mathbb{R}}(f_0)\subset\lambda_{\mathbb{R}}(f_1)$ and $(\theta,\theta')\in\lambda_{\mathbb{R}}(f_1)$, the external rays $R_{f_1}(\theta)$, $R_{f_1}(\theta_1)$, $R_{f_1}(\theta')$ and $R_{f_1}(\theta_2)$ land at a common point. 
Then $R_{f_1}(\theta_1)$ lands at $x_1$, which contradicts the minimality of $S_{f_1}(\theta,\theta')$. 
Therefore $(x_0,x'_0)\cap J(f_0)=\emptyset$. 

Now there exists a unique $v\in|T^\infty(f_0)|$ so that $x_0,x'_0\in\partial U_{f_0}(v)$. 
There is a unique $t\in\RZ$ so that $R_{f_0}(\theta_v^-(t))$ lands at $x_0$. 
Since $\lambda_\mathbb{R}(f_0)\subset\lambda_\mathbb{R}(f_1)$,  the external rays $R_{f_1}(\theta_v^-(t))$ and $R_{f_1}(\theta)$ land at the same point $x_1$. 
By Lemma \ref{corresponding of external rays}, 
the external ray $R_{\mathbf{P}_{f_1},v}(t)$ lands at $\psi_{f_1,v}(x_1)$. 
Similarly, the external ray $R_{\mathbf{P}_{f_2},v}(t)$ lands at $\psi_{f_2,v}(x_2)$.  
Since $\mathbf{P}_{f_1}=\mathbf{P}_{f_2}$, we have $\psi_{f_1,v}(x_1)=\psi_{f_2,v}(x_2)$. 
Similarly, we have $\psi_{f_1,v}(x_1)=\psi_{f_2,v}(x'_2)$. 
It follows that $x_2=x'_2$, hence $(\theta,\theta')\in\lambda_{\mathbb{R}}(f_2)$. 

Let $x\in J(f_1)$, and let $\theta_1,\dots,\theta_n$ be all external angles of $x$ with respect to $f_1$ in positive cyclic order. 
Suppose $n\geq2$. 
For any $1\leq k\leq n-1$, the above discussion shows $(\theta_k,\theta_{k+1})\in \lambda_{\mathbb{R}}(f_2)$. 
Therefore $\lambda_\mathbb{R}(f_1)\subset\lambda_\mathbb{R}(f_2)$. 
The same reasoning gives $\lambda_\mathbb{R}(f_2)\subset\lambda_\mathbb{R}(f_1)$. 
This shows the claim. 
\vspace{6pt}

Because $\lambda_\mathbb{R}(f_1)=\lambda_\mathbb{R}(f_2)$, 
the map $B_{f_2}^{-1}\circ B_{f_1}:U_{f_1,\infty}\rightarrow U_{f_2,\infty}$ extends to 
a homeomorphism $\varphi:\overline{U_{f_1,\infty}}\rightarrow\overline{U_{f_2,\infty}}$ by Lemma \ref{continuous extension on the closure of infinity basin}. 
Furthermore, we extend $\varphi$ to $\Phi:\mathbb{C}\rightarrow\mathbb{C}$ by 
$$\Phi(z) = \begin{cases}
\varphi(z), &z\in \overline{U_{f_1,\infty}}, \\
\psi_{f_2,v}^{-1}\circ\psi_{f_1,v}(z), &\text{$z\in K_{f_1,v}$ for some $v\in|T^\infty(f_0)|$.}
\end{cases}$$
Let $v\in|T^\infty(f_0)|$, and let $x\in \partial K_{f_1,v}$. 
Choose $t\in\RZ$ so that $R_{f_1}(\theta_v^-(t))$ lands at $x$. 
Then $R_{f_2}(\theta_v^-(t))$ lands at $\varphi(x)\in \partial K_{f_2,v}$. 
By Lemma \ref{corresponding of external rays}, 
the external ray $R_{\mathbf{P}_{f_1},v}(t)$ lands at $\psi_{f_1,v}(x)$ and the external ray $R_{\mathbf{P}_{f_2},v}(t)$ lands at $\psi_{f_2,v}(\varphi(x))$. 
Since $\mathbf{P}_{f_1}=\mathbf{P}_{f_2}$, we have $\psi_{f_1,v}(x) = \psi_{f_2,v}(\varphi(x))$. 
It follows that $\varphi(x) = \psi_{f_2,v}^{-1}\circ\psi_{f_1,v}(x)$. 
Now let $x\in K_{f_1,u}\cap K_{f_1,v}$ for some $u,v\in|T^\infty(f_0)|$ with $u\neq v$. 
Then $\psi_{f_2,u}^{-1}\circ\psi_{f_1,u}(x) = \varphi(x) = \psi_{f_2,v}^{-1}\circ\psi_{f_1,v}(x)$. 
Therefore $\Phi$ is well-defined. 

Clearly $\Phi$ is bijective and $\Phi\circ f_1 = f_2\circ\Phi$. 
For any Fatou component $U$ of $f_1$, the restriction  $\Phi:\overline{U}\rightarrow\Phi(\overline{U})$ is a homeomorphism 
and the restriction  $\Phi:U\rightarrow\Phi(U)$ is conformal. 
By Corollary \ref{semi-hyperbolicity}, the Julia sets $J(f_1)$ and $J(f_2)$ are locally connected. 
Then $\Phi$ is a homeomorphism by Lemma \ref{local connectivity}(\ref{local connectivity-2}). 
By Corollary \ref{semi-hyperbolicity} again, the Julia set $J(f_1)$ is conformally removable, 
hence $\Phi$ is conformal. 
Thus $\Phi$ has the form $\Phi(z)=az+b$ for some  $a,b\in\mathbb{C}$ with $a\neq0$. 
Since $B_{f_1}$ and $B_{f_2}$ are tangent to the identity at $\infty$, we have $a=1$. 
Because $\Phi\circ f_1 = f_2\circ\Phi$ and  $f_1,f_2\in\widehat{\mathcal{P}}^d$ (monic and centered), we have $b=0$. 
Therefore $\Phi=\id$. 
It follows that $(\mathbf{c}(f_1),f_1(0)) =  (\mathbf{c}(f_2),f_2(0))$, i.e. $f_1=f_2$. 
The proof is completed. 
\end{proof}

\begin{lemma}
\label{chi is continuous}
$\chi$ is continuous. 
\end{lemma}

\begin{proof}
Let $\{f_n\}_{n\geq1}\subset\overline{\mathcal{H}}$ satisfy $f_n\rightarrow f$ as $n\rightarrow\infty$. 
We need to show $\chi(f_n)\rightarrow\chi(f)$ as $n\rightarrow\infty$.  

For each $v\in\Tfp$, we have $\psi_{f_n,v}^{-1}\rightarrow\psi_{f,v}^{-1}$ uniformly on $\overline{\mathbb{D}}$ by Proposition \ref{continuity of Ufv}. Now let $v\in\Tfn$.  
For each $t\in T_v$, when constructing $h_{f_n,v,t}$ and $h_{f,v,t}$ in the proof of Lemma \ref{hfvt}, 
we require that $\zeta_{f_n,v,t}=\zeta_{f,v,t}$ and $\widetilde\zeta_{f_n,v,t}=\widetilde\zeta_{f,v,t}$ for $n$ large enough. 
Similar to Lemma \ref{uniform convergence of extension}, we can check that $g_{f_n,v}\rightarrow g_{f,v}$ uniformly on compact subsets of $\mathbb{C}$.  
Applying Lemma \ref{stability of external rays} to the external markings $\Gamma_{f_n}$ and $\Gamma_f$, we see that 
$(\mathbf{g}_{f_n},K_{f_n},Z_{f_n},\Gamma_{f_n})$ converges to $(\mathbf{g}_f,K_f,Z_f,\Gamma_f)$. 
By Proposition \ref{continuity-straightening-LM}, we have $\mathbf{P}_{f_n}\rightarrow\mathbf{P}_f$. 

By Remark \ref{continuity-conjugacy}, we can modify the conjugacies $(\psi_{f_n,v})_{v\in|T(f_0)|}$ and  $(\psi_{f,v})_{v\in|T(f_0)|}$ so that $\psi_{f_n,v}\rightarrow \psi_{f,v}$ uniformly on compact subsets of $\mathbb{C}$ for each $v\in|T(f_0)|$. 
It follows that $\psi_{f_n,v}(c_{v,k}(f_n))\rightarrow\psi_{f,v}(c_{v,k}(f))$ for each $(v,k)\in I_0$. 
This completes the proof. 
\end{proof}

\begin{proof}
[Proof of Theorem \ref{top-boundary}]
Let $f_0$ be the center of $\mathcal{H}$. 
Given an external marking for $f_0$, 
we have a straightening map $\chi:\Hbar\rightarrow \modelf$. 
By Lemmas \ref{chi is injective} and \ref{chi is continuous}, the map $\chi$ is injective and continuous. 
Because $\Cdim(\mathcal H) = \Cdim(\widehat{\mathcal{P}}^{T(f_0)}_0)$ (see Lemma \ref{basic facts of H}), the restriction $\chi|_{\mathcal{H}}:\mathcal{H}\rightarrow \widehat{\mathcal{P}}^{T(f_0)}_0$ is an open map. 
In particular, $\chi(\mathcal{H})$ is an open subset of $\modelf^\circ$. 
On the other hand, $\chi(\mathcal{H})$ is closed in $\modelf^\circ$ by $\chi(\partial\mathcal{H}) \subset \partial\modelf$. 
By the proof of Proposition \ref{fiberwise connectedness locus C(P0T)}, we see that  $\modelf^\circ$ 
is homeomorphic to $\mathbb{D}^{\Cdim(\mathcal H)}$, hence connected. 
This implies $\chi(\mathcal{H})= \modelf^\circ$. 
It follows that $\chi(\Hbar) = \overline{\chi(\mathcal{H})} = \modelf$.
Thus $\chi$ is a homeomorphism. 
By Proposition \ref{fiberwise connectedness locus C(P0T)}, the boundary $\partial\mathcal{H}$ is homeomorphic to the sphere $S^{2\Cdim(\mathcal H)-1}$. 
\end{proof}

\section{Transfer equalities}
\label{section transfer}
This section is the first step to prove Theorem \ref{hd-boundary}.
In order to get a relation of the Hausdorff dimensions  between the parameter set and the dynamical set, we first need to establish some transfer equalities.

For $z\in \mathbb{C}^n$, let $\|z\|$ denote $\|z\|_\infty=\max\{|z_1|,\dots,|z_n|\}$. 
Let $\mathbb D^n(z,r)=\{\zeta\in \mathbb C^n{;~} \|\zeta-z\|<r\}$ be the \emph{polydisk} centered at $z$ of radius $r$. 
Set $\mathbb D^n(r)=\mathbb D^n(0,r)$ and $\mathbb D^n=\mathbb D^n(1)$. 
If $n=1$, we set $\mathbb{D}(\cdot)=\mathbb{D}^1(\cdot)$. 
For $z\in A\subset \mathbb{C}^n$, the connected component of $A$ containing $z$ will be denoted by $\Comp_z(A)$. 
For two subsets $A$ and $B$ of $\mathbb{C}^n$, the notation $A\Subset B$ means $\overline{A}\subset B^\circ$.

Our transfer equalities read as follows:

\begin{proposition} 
[Transfer equalities]
\label{HD-trans} 
Let $X$ be a subset of $\mathbb C$ with $x_0\in X$, and let $\Omega$ be a domain in $\mathbb{C}^n$ with $z_0\in\Omega$. 
Let $h: \Omega\times X\rightarrow \mathbb C$ be a holomorphic motion with basepoint $z_0$ (i.e. $h(z_0,\cdot)=\id$),  and $v: \Omega\rightarrow\mathbb C$ be a holomorphic function such that $v(z_0)=x_0$. 

Let $\mathcal X=\{z\in\Omega{;~} v(z)\in h(z, X)\}$. For $\delta,r>0$, let $\mathcal{X}_\delta=\mathcal{X}\cap\mathbb{D}^n(z_0,\delta)$  and $X_r=X\cap \mathbb{D}(x_0,r)$. 
%With the same notations and assumptions as Proposition \ref{HD-trans-in}, for $\delta>0$, let $\mathcal{X}%_\delta=\mathcal{X}\cap\mathbb{D}^n(z_0,\delta)$.
Assume the transversality $dv|_{z=z_0}\neq dh(z,x_0)|_{z=z_0}$ holds. 
Then
\begin{align}
\label{HD-trans1}
\lim_{\delta\rightarrow 0}{\Hdim}(\mathcal X_\delta) &= 2(n-1)+\lim_{r\rightarrow 0}{\Hdim}(X_r), \\
\label{HD-trans2}
\lim_{\delta\rightarrow 0}{\Hdim}(\Comp_{z_0}(\mathcal X_\delta)) &= 2(n-1)+\lim_{r\rightarrow 0}{\Hdim}(\Comp_{x_0}(X_r)). 
\end{align}
%If $X$ is homogenuous, then 
%$$\lim_{r\rightarrow 0^+}{\Hdim}(Q_v^r)= 2(n-1)+{\Hdim}(X).$$
\end{proposition}

Proposition \ref{HD-trans} is a sharp version of Tan Lei's transfer inequality \cite[Proposition 4.1]{Tan}: 
$${\Hdim}(\mathcal X)\geq 2(n-1)+\lim_{r\rightarrow 0}{\Hdim}(X_r).$$
This inequality is useful to estimate some parameter sets with large Hausdorff dimension, but it is not enough for our purpose.  Our transfer equalities allow us to find  explicit Hausdorff dimension formulas for some parameter sets, and are expected to find more applications in other situations. 

Before proving the transfer equalities, let us discuss 
the H\"older continuity. 
% the H\"older continuity of holomorphic motions. 

%\begin{proposition} 
%[Tan Lei's transfer inequality]\label{HD-trans-in} 
%Let $X$ be a subset of $\mathbb C$ with $x_0\in X$, and let $\Omega$ be a domain in $\mathbb{C}^n$ with %$z_0\in\Omega$. 
%Let $h: \Omega\times X\rightarrow \mathbb C$ be a holomorphic motion with basepoint $z_0\in \Omega$ (i.e. $h(z_0,\cdot)=\id$),  and $v: \Omega\rightarrow\mathbb C$ be a holomorphic function such that $v(z_0)=x_0$ and $dv|_{z=z_0}\neq dh(z,x_0)|_{z=z_0}$ (transversality). 

%Let $\mathcal X=\{z\in\Omega; v(z)\in h(z, X)\}$. 
%For $r>0$, let $X_r=X\cap \mathbb{D}(x_0,r)$. 
%Then $${\Hdim}(\mathcal X)\geq 2(n-1)+\lim_{r\rightarrow 0}{\Hdim}(X_r).$$
%\end{proposition}

%Tan Lei's transfer  In order to get equalities, we need to consider subsets of $\mathcal{X}$. Our transfer %equalities read as follows: 

\subsection{H\"older continuity}

\begin{lemma}
[Schwarz's lemma]
\label{Hol-contin of holomorphic function}
Let $f:\mathbb{D}^n\rightarrow \mathbb{D}$ be a holomorphic function. 
Given $\delta\in(0,1)$. 
Then for any $z_1,z_2\in\mathbb{D}^n(\delta)$, 
we have $$|f(z_1)-f(z_2)|\leq \frac{2}{1-\delta}\|z_1-z_2\|.$$
\end{lemma}

\begin{proof}
First, we show the Schwarz lemma for unit polydisk: for a holomorphic function $g:\mathbb{D}^n\rightarrow \mathbb{D}$ with $g(0)=0$, we have $|g(z)|\leq \|z\|$ for any $z\in \mathbb{D}^n$. 
Fix $z\in \mathbb{D}^n\setminus\{0\}$, we have $z\in\{a\frac{z}{\|z\|};a\in\mathbb{D}\}\subset\mathbb{D}^n$. 
Define $h:\mathbb{D}\rightarrow\mathbb{D}$ by $h(a)=g(a\frac{z}{\|z\|})$. 
By the classical Schwarz lemma, we have $|g(z)|=|h(\|z\|)|\leq \|z\|$. 

% Let $\mathbb{D}(z,r)=\{\zeta\in\mathbb{C}^n; \|\zeta-z\|_2\leq r\}$ be the \emph{ball centered at $z$ with radius $r$}. 
Fix $z_1,z_2\in\mathbb{D}^n(\delta)$. 
Choose an integer $m$ so that $m>\frac{\|z_1-z_2\|}{1-\delta}$. 
For $0\leq k\leq m$, let $w_k=z_1+\frac{k}{m}(z_2-z_1)$. 
For $0\leq k\leq m-1$, since $\|w_{k+1}-w_k\|=\frac{\|z_1-z_2\|}{m}<1-\delta$, applying the Schwarz lemma to $f:\mathbb{D}^n(w_k,1-\delta)\rightarrow \mathbb{D}(f(w_k),2)$, we have $\frac{|f(w_{k+1})-f(w_k)|}{2}\leq \frac{\|w_{k+1}-w_k\|}{1-\delta}$. 
Therefore $$|f(z_1)-f(z_2)|\leq \sum_{k=0}^{m-1}|f(w_{k+1})-f(w_k)|\leq \sum_{k=0}^{m-1} \frac{2\|w_{k+1}-w_k\|}{1-\delta}=\frac{2\|z_1-z_2\|}{1-\delta}.$$
This completes the proof. 
\end{proof}

\begin{lemma}
[Mori's theorem]
\label{Hol-contin of qcm}
% Mori's theorem: Hol-contin of qcm
Let $f:\mathbb{C}\rightarrow\mathbb{C}$ be a $K$-quasiconformal mapping normalized by $f(0)=0,f(1)=1$.  Let $r>0$. 
Then there are constants $C_1=C_1(K,r),C_2=C_2(K,r)>0$ 
so that \begin{equation}
\label{bi-Hol of qcm}
C_1|z_1-z_2|^K\leq |f(z_1)-f(z_2)|\leq C_2|z_1-z_2|^{1/K} 
\end{equation}
for any $|z_1|,|z_2|\leq r$. 
\footnote{
For examples, the map $f(z) = z|z|^{\frac{1}{K}-1}$ and its inverse $f^{-1}(z)=z|z|^{K-1}$. 
}
\end{lemma}

\begin{proof}
For the existence of $C_2(K,r)$, we refer to \cite[Chapter III, C]{Ahlfors-qcm06}. 

By the right inequality of (\ref{bi-Hol of qcm}), if $|z|\leq r$, then $|f(z)|\leq C_2(K,r) r^{1/K}=:r'$. 
Note that $f^{-1}$ is a $K$-quasiconformal mapping with $f^{-1}(0)=0,f^{-1}(1)=1$. 
For any $|z_1|,|z_2|\leq r$,  we have $|f(z_1)|,|f(z_2)|\leq r'$, so $$|z_1-z_2|=|f^{-1}(f(z_1))-f^{-1}(f(z_2))|\leq C_2(K,r')|f(z_1)-f(z_2)|^{1/K}.$$
Let $C_1 = 1/(C_2(K,r')^K)$. The left inequality of (\ref{bi-Hol of qcm}) follows. 
\end{proof}

\begin{lemma}
[S{\l}odkowski's theorem]
\label{dilatation of holomorphic motion}
Let $X$ be a nonempty subset of $\mathbb C$, and let $h: \mathbb{D}^n\times X\rightarrow \mathbb C$ be a holomorphic motion with basepoint $0\in \mathbb D^n$.  Given $z\in\mathbb{D}^n$. 
Then the map $h(z,\cdot):X\rightarrow \mathbb C$ can be extended to a $\frac{1+\|z\|}{1-\|z\|}$-quasiconformal mapping of $\mathbb{C}$. 
\end{lemma}

\begin{proof}
Since $h(0,\cdot)$ can be extended to $\id_{\mathbb{C}}$, we assume $z\neq0$ in the following. 
Then we have $z\in\{a\frac{z}{\|z\|};a\in\mathbb{D}\}\subset\mathbb{D}^n$. 
Consider the holomorphic motion $g:\mathbb{D}\times X\rightarrow \mathbb{C}$ defined by $g(a,x)=h(a\frac{z}{\|z\|},x)$. 
By S{\l}odkowski's theorem, the map $h(z,\cdot)=g(\|z\|,\cdot)$ can be extended to a $\frac{1+\|z\|}{1-\|z\|}$-quasiconformal mapping of $\mathbb{C}$. 
\end{proof}

\begin{lemma}
\label{Hol-contin of holomorphic motion}
Let $X$ be a nonempty bounded subset of $\mathbb{C}$, let $h: \mathbb{D}^n\times X\rightarrow \mathbb C$ be a holomorphic motion with basepoint $0\in \mathbb D^n$. Given $\delta\in(0,1)$, and let $K = \frac{1+\delta}{1-\delta}$. Define $H:\mathbb{D}^{n}(\delta)\times X \rightarrow \mathbb{D}^{n}(\delta)\times \mathbb{C}$ by $H(z,x)=(z,h(z,x))$. 
% $$H:\begin{cases} \mathbb{D}^{n}(\delta)\times X \rightarrow \mathbb{D}^{n}(\delta)\times \mathbb{C},\\(z,x)\mapsto(z,h(z,x)).\end{cases}$$
Then $H$ is bi-H\"older with exponent $1/K$: for any 
$w_1,w_2\in \mathbb{D}^{n}(\delta)\times X$, we have 
\begin{equation}
\label{bi-Hol of H}
C_1\|w_1-w_2\|^K\leq \|H(w_1)-H(w_2)\|\leq C_2\|w_1-w_2\|^{1/K}, 
% C_3|(z_1,x_1)-(z_2,x_2)|^K\leq |H(z_1,x_1)-H(z_2,x_2)|\leq C_4|(z_1,x_1)-(z_2,x_2)|^{1/K}.
\end{equation}
where $C_1$ and $C_2$ are positive constants dependent on $h$ and $\delta$. 
\end{lemma}

\begin{proof}
If $X$ is a singleton, then $h$ represents a holomorphic function, 
so (\ref{bi-Hol of H}) follows from Lemma \ref{Hol-contin of holomorphic function}. 
Consider $X=\{0\}$ and $h\not\equiv0$. 
Multiplying $h$ with a large number, we know that $C_1$ and $C_2$ are dependent on $h$. 

Now we assume $X$ contains at least two points. 
By an affine transformation of $X$, we may further assume $0,1\in X$. 
By the $\lambda$-lemma \cite{McM}, we may assume $X$ is closed, thus compact. 
Fix $(z_1,x_1),(z_2,x_2)\in \mathbb{D}^{n}(\delta)\times X$. 
Then 
\begin{align*}
% &|H(z_1,x_1)-H(z_2,x_2)|=\max\{|z_1-z_2|, |h(z_1,x_1)-h(z_2,x_2)|\}, \\
&|h(z_1,x_1)-h(z_2,x_2)|\geq-|h(z_1,x_1)-h(z_2,x_1)|+|h(z_2,x_1)-h(z_2,x_2)|,\\
&|h(z_1,x_1)-h(z_2,x_2)|\leq |h(z_1,x_1)-h(z_2,x_1)|+|h(z_2,x_1)-h(z_2,x_2)|.
\end{align*}
By Lemma \ref{Hol-contin of holomorphic function}, we have 
$$|h(z_1,x_1)-h(z_2,x_1)|\leq c_0\|z_1-z_2\|.$$
We remark that the positive constants $c_k$'s here and in the following discussion, can be chosen independent of $(z_1,x_1)$ and $(z_2,x_2)$. 
% where $R = \sup\{|h(w)|;w\in \mathbb{D}^{n}(\delta)\times X\}<\infty$. 
By Lemma \ref{dilatation of holomorphic motion}, the map $h(z_2,\cdot):X\rightarrow \mathbb C$ can be extended to a $K$-quasiconformal mapping of $\mathbb{C}$. 
Because 
\begin{align*}
0 <\inf\{|h(z,1)-h(z,0)|;z\in \mathbb{D}^{n}(\delta)\} 
  \leq \sup\{|h(z,1)-h(z,0)|;z\in \mathbb{D}^{n}(\delta)\}<\infty, 
\end{align*}
% \begin{split}  \end{split}
applying Lemma \ref{Hol-contin of qcm} to $\frac{h(z_2,x)-h(z_2,0)}{h(z_2,1)-h(z_2,0)}$, 
we have $$c_1|x_1-x_2|^K\leq |h(z_2,x_1)-h(z_2,x_2)| \leq c_2|x_1-x_2|^{1/K}.$$ 
It follows that 
$$-c_0\|z_1-z_2\|+c_1|x_1-x_2|^K\leq |h(z_1,x_1)-h(z_2,x_2)|\leq c_0\|z_1-z_2\|+c_2|x_1-x_2|^{1/K}.$$
By the facts 
\begin{align*}
& {\left(\frac{\|z_1-z_2\|}{2}\right)}^K\leq \frac{\|z_1-z_2\|}{2}\leq {\left(\frac{\|z_1-z_2\|}{2}\right)}^{1/K},\\
& \frac{1}{2}\max\{a,a+b\}\leq \max\{a,b\}\leq 2\max\{a,b-a\}, \  a,b\geq0, 
\end{align*}
we have 
\begin{align*}
            \|(z_1,x_1)-(z_2,x_2)\|^{K}
&=          \max\{\|z_1-z_2\|^K,|x_1-x_2|^K\}\\
&\leq   c_3 \max\{c_0\|z_1-z_2\|,c_1|x_1-x_2|^K\}\\
&\leq  2c_3 \max\{c_0\|z_1-z_2\|,-c_0\|z_1-z_2\|+c_1|x_1-x_2|^K\}\\
&\leq   c_4 \max\{\|z_1-z_2\|,|h(z_1,x_1)-h(z_2,x_2)|\}\\
&=      c_4 \|H(z_1,x_1)-H(z_2,x_2)\|\\
&\leq   c_5 \max\{\|z_1-z_2\|,\|z_1-z_2\|+|x_1-x_2|^{1/K}\}\\
&\leq  2c_5 \max\{\|z_1-z_2\|,|x_1-x_2|^{1/K}\}\\
&\leq   c_6 \max\{\|z_1-z_2\|^{1/K},|x_1-x_2|^{1/K}\}\\
&=      c_6 \|(z_1,x_1)-(z_2,x_2)\|^{1/K}. 
\end{align*}
This completes the proof. 
\end{proof}

\subsection{Proof of Proposition \ref{HD-trans}}
\begin{proof}
[Proof of Proposition \ref{HD-trans}]
{\bf Step 1.} 
By translations, we assume $z_0=0$ and $x_0=0$. 
Replacing $v$, $h$ by $v(z)-h(z,0)$, $h(z, x)-h(z,0)$ if necessary, we may assume $h(z,0)\equiv 0$. 
Then $v(0)=0$ and $dv|_{z=0}\neq 0$.

For $z\in\mathbb{C}^n$, we let $z=(z',z_n)\in\mathbb{C}^{n-1}\times\mathbb{C}$. 
Shrinking $\Omega$ and then making a linear change of coordinates for $v(\cdot)$ and $h(\cdot,x)$ simultaneously, we assume $v(z)=z_n+o(\|z\|)$. 
This will not affect the transfer equalities since we focus on $\mathcal{X}_\delta$ with $\delta$ small. 
Shrink $\Omega$ again if necessary and choose $\delta_0>0$ such that 
$\Omega\rightarrow \mathbb{D}^n(\delta_0)$, $z\mapsto (z',v(z))$ is biholomorphic.  
Define $\varphi:\Omega\rightarrow \mathbb{D}^n(2)$, $z\mapsto m(z',v(z))$, where $m=2/\delta_0$. 
Consider the coordinate transformation $\zeta=\varphi(z)$. 
Replacing $\Omega$, $X$, $h$, $v$ by 
$$\mathbb{D}^n(2), \quad m X, \quad 
m h(\varphi^{-1}(\zeta), y/m), \quad m v(\varphi^{-1}(\zeta))=\zeta_n,$$ 
we may further assume $v(z)=z_n$. 
All assumptions are stated as follows: 
\begin{align*}
&\Omega=\mathbb{D}^{n}(2), \quad 0\in X, \\
&h(0,\cdot)=\id, \quad h(z,0)\equiv 0, \quad v(z)=z_n. 
\end{align*}
We also assume $n\geq2$. 
If $n=1$, the proof can be simplified partly. 
%% if n=1, the transversality condition $dv|_{z=z_0} \neq dh(z,x_0)|_{z=z_0}$ can be replaced by  $v(z) \not\equiv h(z,x_0)$

% Count zeros by Rouch\'e's theorem. 
{\bf Step 2.} 
For $r>0$, let
%In the following, we view $v(\cdot)$ and  $h(\cdot,y)$ as funtions in $(\mathbf{z}, w)$ in the polydisk $\mathbb D^{d-2}(\epsilon)$.
%Let 
$$b(r)=\sup\big\{|h(z, x)|{;~} z\in \mathbb D^{n}, x\in  X_r \big\}.$$
By the $\lambda$-lemma, we can extend $h$ to a holomorphic motion on $\Omega\times \overline{X}$, so $h$ is uniformly continuous on $\mathbb D^{n}\times X_r$. 
Then by the assumption $h(z,0)\equiv 0$, we have $b(r)\rightarrow 0$ as $r\rightarrow 0$.

Given $\delta\in(0,1)$. 
Choose $r>0$ so that $b(r)<\delta$. 

Fix a triple $(z', \lambda,x)\in  \mathbb D^{n-1} \times \mathbb D(1/\delta)\times X_r$. 
Consider the equation in $z_n\in \mathbb D(\delta)$:
\begin{equation}
\label{equation of zn}
z_n-h((z', \lambda z_n), x)=0.
\end{equation}
Note that when $|z_n|=\delta$, we have $(z',\lambda z_n)\in\mathbb{D}^n$ and
$$|z_n|=\delta> b(r)\geq |h((z', \lambda z_n), x)|.$$
%Note that when 
%$|\lambda|=R$, $|v(\mathbf{z}, w)|>|h((\mathbf{z}, \lambda w), y)|$.
By Rouch\'e's theorem, the equation (\ref{equation of zn}) has exactly 
one solution for $z_n$ in the disk $\mathbb D(\delta)$, 
which will be denoted by $w((z', \lambda),x)$.

Let us check that $w:(\mathbb D^{n-1}\times \mathbb D(1/\delta))\times X_r
\rightarrow\mathbb{D}(\delta)$ is a holomorphic motion with basepoint $(0',0)$. 
\begin{itemize}
\item Let $(z',\lambda)=(0',0)$ and $x\in X_r$. Then  $w((0',0),x)=h((0',0),x)=x$. 

\item Fix $(z',\lambda)\in \mathbb D^{n-1}\times \mathbb D(1/\delta)$ and let $x_1,x_2\in X_r$. 
Suppose $w((z',\lambda),x_1)=w((z',\lambda),x_2)$, and denote it by $y$.  
Then $y=h((z',\lambda y),x_1)=h((z',\lambda y),x_2)$. 
Because $h((z',\lambda y),\cdot)$ is injective, we have $x_1=x_2$. 
Therefore $w((z',\lambda),\cdot)$ is injective. 

\item Fix $x\in X_r$.  Then $w(\cdot,x)$ is holomorphic by the following generalized Cauchy's integral formula: 
$$w((z',\lambda),x)=\frac{1}{2\pi i}\int_{|\zeta|=\delta} \zeta\frac{\partial(\zeta-h((z',\lambda\zeta),x))/\partial\zeta}{\zeta-h((z',\lambda\zeta),x)}d\zeta.$$ 
\end{itemize}

{\bf Step 3.} 
As the equation (\ref{equation of zn}) shows, we are concerned with $\lambda=1$. 
If $(z',x)\in  \mathbb D^{n-1}(\delta) \times X_r$, then $(z',w((z',1),x))\in \mathcal{X}_\delta$. 
Define 
$$H:\begin{cases} \mathbb D^{n-1}(\delta)\times X_r \rightarrow \mathcal{X}_\delta,\\
(z',x)\mapsto
(z',w((z',1),x)).
\end{cases}$$
Then $H$ is injective, so $H^{-1}$ is well-defined on $H(\mathbb D^{n-1}(\delta)\times X_r)$. 
We claim that $H$ is bi-H\"older with exponent $1/{K_\delta^2}$ (i.e. $H$ and $H^{-1}$ are H\"older continous with exponent $1/{K_\delta^2}$), where $K_{\delta}=(1+\delta)/(1-\delta)$.
%$$K_{\delta}=\frac{1+\delta}{1-\delta}.$$

This implies the distortion of the Hausdorff dimension
\begin{equation} 
\label{distortion of H-dim}
\frac{1}{K_\delta^2}\leq \frac{{\Hdim}(H(\mathbb D^{n-1}(\delta)\times X_r))}{{\Hdim}(\mathbb D^{n-1}(\delta)\times X_r)}\leq K_\delta^2.
\end{equation}

Now we show the claim. 
For $(z',\lambda)\in \mathbb D^{n-1}\times \mathbb D(1/\delta)$, let 
$X^{z',\lambda}_r=w((z', \lambda),X_r)$. 
In particular, we have $X^{0',0}_r = X_r$. 
Let $w_1:\mathbb{D}(1/\delta)\times X_r\rightarrow \mathbb{C}$ be the holomorphic motion defined by $w_1(\lambda,x)=w((0',\lambda),x)$. 
By S{\l}odkowski's theorem, the map $\psi:=w_1(1,\cdot):X_r\rightarrow X_r^{0',1}$ can be extended to a $K_\delta$-quasiconformal mapping of $\mathbb{C}$. 
Thus $\psi$ is bi-H\"older with exponent $1/K_\delta$ by Lemma \ref{Hol-contin of qcm}. 
Define $$H_1:\begin{cases}
\mathbb D^{n-1}(\delta)\times X_r\rightarrow \mathbb D^{n-1}(\delta)\times X_r^{0',1},\\
(z',x)\mapsto (z',\psi(x)). 
\end{cases}$$
Then $H_1$ is bi-H\"older with exponent $1/K_\delta$. 
Consider the holomorphic motion $w_2:\mathbb{D}^{n-1}\times X_r^{0',1}\rightarrow \mathbb{C}$ defined by $w_2(z',x)=w((z',1),\psi^{-1}(x))$, and then define 
$$H_2:\begin{cases}
\mathbb D^{n-1}(\delta)\times X_r^{0',1}\rightarrow \mathbb D^{n-1}(\delta)\times \mathbb{C},\\
(z',x)\mapsto (z',w_2(z',x)). 
\end{cases}$$
By Lemma \ref{Hol-contin of holomorphic motion}, the map $H_2$ is bi-H\"older with expoent $1/K_\delta$. 
Then the claim follows from $H=H_2\circ H_1$.

{\bf Step 4.} 
Recall that we have given $\delta\in(0,1)$ and chosen $r>0$ so that $b(r)<\delta$. 
Note that $H(\mathbb D^{n-1}(\delta)\times X_r)\subset\mathcal{X}_\delta$. 
In order to compare Hausdorff dimension, we need to find $\varepsilon>0$ so that $\mathcal{X}_\varepsilon\subset H(\mathbb D^{n-1}(\delta)\times X_r)$. 

For $\varepsilon\in(0,1)$, define 
$$\rho(\varepsilon)=\inf\big\{|h(z, x)|{;~} z\in \mathbb{D}^{n}(\varepsilon), x\in  X\setminus\mathbb{D}(r) \big\},$$
where we set $\inf\emptyset=+\infty$. 
Regarding $X$ as a subset of $\widehat{\mathbb{C}}$, 
by the $\lambda$-lemma again, we can extend $h: \Omega\times X\rightarrow \widehat{\mathbb C}$ to a holomorphic motion on $\Omega\times \overline{X}$, so $h$ is uniformly continuous on $\mathbb D^{n}\times X\Subset\Omega\times\widehat{\mathbb{C}}$. 
Since $h(0,\cdot)=\id$, we have $\liminf_{\varepsilon\rightarrow0}\rho(\varepsilon)\geq r$. 

Choose $\varepsilon\in(0,\min\{\delta,r/2\})$ so that $\rho(\varepsilon)>r/2$. 
We claim that $$\mathcal{X}_\varepsilon\subset H(\mathbb D^{n-1}(\delta)\times X_r).$$

To prove the claim, let $z=(z',z_n)\in\mathcal{X}_\varepsilon$. 
This means $z\in\mathbb{D}^n(\varepsilon)$ and there is an $x\in X$ such that $z_n=h(z,x)$. 
Then $|h(z,x)|=|z_n|<\varepsilon<r/2<\rho(\varepsilon)$. 
According to the definition of $\rho(\varepsilon)$, since  $z\in\mathbb{D}^n(\varepsilon)$, $x\in X$ and $|h(z,x)|<\rho(\varepsilon)$, we have $x\in X_r$. 
By $\|z'\|<\varepsilon<\delta<1$, $x\in X_r$ and $|z_n|<\varepsilon<\delta$, 
we have $z_n=w((z',1),x)$ (see the equation (\ref{equation of zn})).  
Since $\|z'\|<\delta$ and $x\in X_r$, 
we have $H(z',x)=(z',w((z',1),x))=(z',z_n)=z$, so $z\in H(\mathbb D^{n-1}(\delta)\times X_r)$. 
This shows the claim. 

Combining the claim with $H(\mathbb D^{n-1}(\delta)\times X_r)\subset\mathcal{X}_\delta$, we have 
\begin{equation}
\label{two-side inclusion}
\mathcal{X}_\varepsilon\subset H(\mathbb D^{n-1}(\delta)\times X_r)\subset\mathcal{X}_\delta.
\end{equation}
By a dimension formula for products (see \cite[Corollary 7.4]{Falconer}), we have 
\begin{equation}
\label{dimension formula for products}
\begin{split}
\Hdim(\mathbb D^{n-1}(\delta)\times X_r) 
&= \Hdim(\mathbb D^{n-1}(\delta))+ \Hdim(X_r)\\
&= 2(n-1) + \Hdim(X_r). 
\end{split}
\end{equation}
By (\ref{distortion of H-dim}), (\ref{two-side inclusion}) and (\ref{dimension formula for products}), we have 
$$\Hdim(\mathcal X_{\varepsilon})  /  {K_\delta^2}
\leq  2(n-1)+ \Hdim(X_r)
\leq \Hdim(\mathcal X_{\delta}) {K_\delta^2}.$$
Letting $\delta\rightarrow 0,r\rightarrow 0$, 
% we do not assume $r<\delta$ first. 
which forces $K_\delta\rightarrow1,\varepsilon\rightarrow 0$, 
gives (\ref{HD-trans1}).

{\bf Step 5.}  
Finally, we show the transfer equality (\ref{HD-trans2}). 
Because the restriction of $H$ on $\mathbb D^{n-1}(\delta)\times \Comp_0(X_r)$ is also bi-H\"older with exponent $1/{K_\delta^2}$, 
we have 
\begin{equation} 
\label{distortion of H-dim2}
\frac{1}{K_\delta^2}
\leq \frac{\Hdim(H(\mathbb D^{n-1}(\delta)\times \Comp_0(X_r)))}
{\Hdim(\mathbb D^{n-1}(\delta)\times \Comp_0(X_r))}
\leq K_\delta^2.
\end{equation}
Since $H(0)=0$ and $H:\mathbb{D}^{n-1}(\delta)\times X_r\rightarrow H(\mathbb{D}^{n-1}(\delta)\times X_r)$ is a homeomorphism, we have 
\begin{equation} 
\label{Comp0-Xr}
\begin{split}
  H(\mathbb D^{n-1}(\delta)\times \Comp_0(X_r))
&= H(\Comp_0( \mathbb{D}^{n-1}(\delta)\times X_r ))\\
&= \Comp_0( H(\mathbb{D}^{n-1}(\delta)\times X_r) ). 
\end{split}
\end{equation}
Taking connected components containing $0$ in (\ref{two-side inclusion}) gives 
\begin{equation}
\label{two-side inclusion 2}
        \Comp_0(  \mathcal{X}_\varepsilon  )
\subset \Comp_0(  H(\mathbb D^{n-1}(\delta)\times X_r)  )  
\subset \Comp_0(  \mathcal{X}_\delta  ).
\end{equation}
By \cite[Corollary 7.4]{Falconer} again, we have 
\begin{equation}
\label{dimension formula for products 2}
\Hdim(\mathbb D^{n-1}(\delta)\times \Comp_0(X_r)) = 2(n-1)+ \Hdim(\Comp_0(X_r)). 
\end{equation}
By (\ref{distortion of H-dim2}), (\ref{Comp0-Xr}), (\ref{two-side inclusion 2}) and (\ref{dimension formula for products 2}), 
we have 
\begin{align*}
            \Hdim(\Comp_0( \mathcal{X}_\varepsilon )) / {K_\delta^2}
&\leq 2(n-1)+\Hdim(\Comp_0( X_r ))\\
&\leq        \Hdim(\Comp_0( \mathcal{X}_\delta )) {K_\delta^2}.
\end{align*}
Letting $\delta\rightarrow 0,r\rightarrow 0$, 
% we do not assume $r<\delta$ first. 
which forces $K_\delta\rightarrow1,\varepsilon\rightarrow 0$, 
gives (\ref{HD-trans2}). 
The proof is completed.
\end{proof}

\section{Transversality}
\label{transver}

This section is the second step to prove Theorem \ref{hd-boundary}.
In order to use the transfer equality, we need to verify the transversality assumption in Proposition \ref{HD-trans}. 
The aim of this section is to prove the transversality result for some maps on the boundary of a capture hyperbolic component. 
The idea of the proof is inspired by van Strien \cite{vStrien}. 

Let $\mathcal{H}\subset\mathcal{F}$ be a capture hyperbolic component with center $f_0$. 
Recall from \S\ref{section behavior of crit-orbits} that each $v\in\Tfp$ induces a continuous function $\zeta_v: \mathcal{F}\rightarrow\mathbb{C}$ so that $\zeta_v(f)$ is a superattracting periodic point of $f\in\mathcal{F}$ and $\zeta_v(f_0)=v$. 
Recall also that $U_{f,v} = U_{f}(\zeta_v(f))$ and $A(f)=\bigcup_{v\in\Tfp} U_{f,v}$. 

Let $f\in\Hbar\setminus\Sigma(\mathcal{F})$.  
Here $\Sigma(\mathcal{F})$ is the singular locus of the algebraic set $\mathcal{F}$ (see Appendix \ref{appendix smooth points}). 
The assumption $f\notin\Sigma(\mathcal{F})$ means $f$ is a smooth point of $\mathcal{F}$,  
which allows us to consider holomorphic motions parameterized by a neighborhood $\mathcal{N}\subset\mathcal{F}$ of $f$. 
\footnote{
This is a technical assumption. 
For $\mathcal{F}$ defined by (\ref{superattracting relations}), 
we have no specific example of $\mathcal{F}$ so that $\Sigma(\mathcal{F})$ is nonempty.
}  

Furthermore, suppose $\crit(f)\cap \partial A(f)=\emptyset$. 
By Proposition \ref{behavior of critical orbits} and Ma\~n\'e's theorem \cite{Mane,Shishikura-Tan}, 
the restriction $f|_{\partial A(f)}$ is expanding. 
Thus there is a neighborhood $\mathcal{N}\subset \mathcal{F}$ of $f$ and a holomorphic motion $h:\mathcal{N}\times\partial A(f)\rightarrow \mathbb{C}$ with basepoint $f$ (i.e. $h(f,\cdot)=\id$) such that $h(g,\partial A(f)) = \partial A(g)$ and $h(g,f(z)) = g(h(g,z))$ for all $g\in\mathcal{N}$ and $z\in\partial A(f)$ 
 (compare \cite[\S4.1]{McM}). 
By the density of periodic cycles in $\partial A(f)$, such  holomorphic motion is unique. 
It follows that for $g\in \mathcal N$ and $v\in \Tfp$, there is a B\"ottcher coordinate $B_{g,v}:U_{g,v}\rightarrow\mathbb{D}$ such that 
$B_{g,\sigma(v)}(g(z))=B_{g,v}(z)^{\delta(v)}$, which is holomorphic with respect to $g$.  
When $z\in U_{f,v}$ for some $v\in\Tfp$, 
we extend $h$ by $h(g,z) = B_{g,v}^{-1}\circ B_{f,v}(z)$. 
Then we have a holomorphic motion $$h:\mathcal{N}\times\overline{A(f)}\rightarrow\mathbb{C}.$$ 

Recall that $I_0 = \{(v,k){;~} v\in|T(f_0)|, 1\leq k<\delta(v)\}$. 
As in \S\ref{section straightening map}, 
for $f\in\widehat{\mathcal{P}}^d$, we use the representative 
$$\mathbf{c}(f) = (c_{v,k}(f))_{(v,k)\in I_0}$$ so that $c_{v,k}(f_0) = v$ for every $(v,k)\in I_0$. 
Let $$I = \{(v,k){;~} v\in\Tfn, 1\leq k<\delta(v)\}\subset I_0$$ 
be the \emph{index set of free critical points}. 
Then $\#(I) =\Cdim (\mathcal H) = d-1-\ell$. 
Recall that $r_v\geq0$ is the minimal integer such that $\sigma^{r_v}(v)\in\Tfp$ whenever $v\in|T(f_0)|$. 

\begin{definition}
[Free critical relation]
\label{defi-free-crit-rela}
Let $f\in\overline{\mathcal{H}}$. 
We will say $f$ has a \emph{free critical relation} (with respect to $\mathcal{H}$) if $c_{v',k'}(f) = f^j(c_{v,k}(f))$ for some $(v,k)\neq(v',k')\in I$ and $0\leq j<r_{v}$, 
which is said to be \emph{internal} or \emph{external} according to whether $v'=\sigma^j(v)$ or $v'\neq\sigma^j(v)$. 
\end{definition} 

Since $c_{v',k'}(f)$ is not periodic, the integer $j$ is determined by $(v,k)$ and $(v',k')$. 
So all free critical relations of $f$ form a subset of $I\times I$. 
% It is possible that a given map $f$ can have different kinds of free critical relations with respect to different capture hyperbolic components. 
The kinds of a free critical relation of a polynomial may be different 
with respect to different capture hyperbolic components;  
see Example \ref{example-FreeCriticalRelation}. 

\begin{proposition} 
[Transversality]
\label{transversality-prop} 
Given $f\in\Hbar\setminus\Sigma(\mathcal{F})$ satisfying $\crit(f)\cap \partial A(f)=\emptyset$. 
Let $\mathcal N\subset\mathcal{F}$ be a neighborhood of $f$ so that the holomorphic motion $h: \mathcal N\times \overline{A(f)}\rightarrow \mathbb C$ is well-defined.  
Define a holomorphic map $G = (G_{v,k})_{(v,k)\in I}:\mathcal N\rightarrow \mathbb{C}^{\Cdim(\mathcal{H})}$ by 
\begin{align*}
G_{v,k}(g) = g^{r_v}(c_{v,k}(g))-h(g,f^{r_v}(c_{v,k}(f))). 
\end{align*}
\begin{enumerate}
\item If $f$ has no free critical relation, then $G$ is a local diffeomorphism near $f$. 
\item If $f$ has an internal free critical relation, then $G$ is not a local diffeomorphism near $f$. 
\end{enumerate}
\end{proposition}

\subsection{Internal free critical relation} 
The internal free critical relation of polynomials in $\overline{\mathcal{H}}$ corresponds to the free critical relation of generalized polynomials in $\widehat{\mathcal{P}}_0^{T(f_0)}$ defined as follows. 

For $\mathbf{x} = (\mathbf{c}_v,a_v)_{v\in|T(f_0)|}\in \widehat{\mathcal{P}}_0^{T(f_0)}$, let 
$f_v=f_{\mathbf{c}_v,a_v}$ for each $v\in|T(f_0)|$. 
For each $(v,k)\in I$, let $$w^0_{v,k}=c_{v,k}, 
~w^j_{v,k}=f_{\sigma^{j-1}(v)}\circ\cdots\circ f_v(c_{v,k}),  ~1\leq j\leq r_v.$$
We will say that $\mathbf{x}$ has a \emph{free critical relation} if $c_{v',k'} = w^j_{v,k}$ and $v' = \sigma^j(v)$ for some $(v,k)\neq(v',k')\in I$ and $0\leq j<r_{v}$. 
Define 
\begin{equation}
\label{crit-val-Tf0}
W:  \begin{cases} 
\widehat{\mathcal{P}}_0^{T(f_0)}\rightarrow \mathbb{C}^{\Cdim (\mathcal H)}, \\
\mathbf{x} \mapsto   (w^{r_v}_{v,k})_{(v,k)\in I}.  
\end{cases}
\end{equation}
As shown in Lemma \ref{proper-holo-W}, the map $W$ is a proper holomorphic map.  

We regard $c_{v,\delta(v)-1} = -(c_{v,1}+\cdots+c_{v,\delta(v)-2})$ as a function of $(c_{v,1},\dots,c_{v,\delta(v)-2})$ for each $v\in\Tfn$ with $\delta(v)\geq3$. 
Let $J_W$ denote the Jacobian matrix of $W$ with respect to 
$$(c_{v,1},\dots,c_{v,\delta(v)-2},a_v)_{v\in\Tfn,\ \delta(v)\geq2},$$
where the vector $(c_{v,1},\dots,c_{v,\delta(v)-2},a_v)$ becomes $(a_v)$ if $\delta(v)=2$. 

\begin{proposition}
\label{prop-JW}
The Jacobian determinant of $W$ is given by 
\begin{equation}
\label{JW}
\begin{split}
\det(J_W) &= C 
\prod_{
\begin{subarray}{c}
(v',k')\neq(v,k)\in I,\\ 
0\leq j<r_v, v'=\sigma^j(v)
\end{subarray}
}
(c_{v',k'}-w^j_{v,k})\\
&= C' 
\prod_{
\begin{subarray}{c}
v\in\Tfn,\\
1\leq k_1< k_2\leq \delta(v)-1
\end{subarray}
} {\left(c_{v,k_1}-c_{v,k_2}\right)}^2
\prod_{
\begin{subarray}{c}
v\in\Tfn,\\
1\leq k<\delta(v),\\
1\leq j<r_v,\\
1\leq k'<\delta(\sigma^j(v))
\end{subarray}
} (c_{\sigma^j(v),k'}-w^{j}_{v,k}), 
\end{split}
\end{equation}
where $C$ and $C'$ are nonzero constants depending on $T(f_0)$. 
\footnote{
The empty product is $1$ by convention. 
}
In particular, the map $W$ is locally one-to-one near $\mathbf{x}$ if and only if $\mathbf{x}$ has no free critical relation. 
\end{proposition}

\begin{example}
\label{example-W-1}
Let $T = (|T|, \sigma, \delta)$ satisfy $|T| = \{ v_0,v_1 \}$, $\sigma(v_0) = v_0$, $\sigma(v_1) = v_0$, $\delta(v_0) = 2$ and $\delta(v_1) =3$. 
Then we can identify $\widehat{\mathcal{P}}_0^{T}$ with $\widehat{\mathcal{P}}^3$. 
For $(\mathbf{c},a)\in \widehat{\mathcal{P}}^3$, we have 
$$ f_{\mathbf{c},a}(z) = z^3 + 3 c_1  c_2 z + a = z^3 - 3 c_1 ^2 z + a.$$
It follows that $W = (f_{\mathbf{c},a}(c_1), f_{\mathbf{c},a}(c_2)) = (-2 c_1^3+a, 2 c_1^3+a)$ and 
$$\renewcommand\arraystretch{1.5}
\det(J_W) = \left|\begin{array}{cc}
\frac{\partial f_{\mathbf{c},a}(c_1)}{\partial c_1} & \frac{\partial f_{\mathbf{c},a}(c_1)}{\partial a}\\
\frac{\partial f_{\mathbf{c},a}(c_2)}{\partial c_1} & \frac{\partial f_{\mathbf{c},a}(c_2)}{\partial a}
\end{array}\right|
=\left|\begin{array}{cc}
-6 c_1^2 & 1\\
6 c_1^2 & 1
\end{array}
\right| = -12 c_1^2 = -3(c_1-c_2)^2.$$
Then $\crit(W) = \{(0,0,a){;~} a\in\mathbb{C}\}$, and $\deg(W,\mathbf{x}) = 3 $ for any $\mathbf{x} \in \crit(W)$. 
\end{example}

\begin{example}
\label{example-W-2}
Let $T = (|T|, \sigma, \delta)$ satisfy $|T| = \{0,1,2 \}$, $\sigma(0) = 0$, $\sigma(1) = 0$, $\sigma(2) = 1$ and $\delta(0) = \delta(1) = \delta(2)=2$. 
For $(a_1, a_2) \in \widehat{\mathcal{P}}_0^{T}$, we have $f_1(z)=z^2+a_1$ and $f_2(z)=z^2+a_2$. 
It follows from $c_1 = c_2 =0$ that $W = (f_1(c_1), f_1\circ f_2(c_2)) = (a_1, a_2^2+a_1)$ and 
$$\renewcommand\arraystretch{1.5}
\det(J_W) = \left|\begin{array}{cc}
\frac{\partial f_1(c_1)}{\partial a_1} & \frac{\partial f_1(c_1)}{\partial a_2}\\
\frac{\partial f_1\circ f_2(c_2)}{\partial a_1} & \frac{\partial f_1\circ f_2(c_2)}{\partial a_2}
\end{array}\right|
=\left|\begin{array}{cc}
1 & 0\\
1 & 2 a_2
\end{array}
\right| = 2 a_2 = -2(c_1-f_2(c_2)).$$
Then $\crit(W) = \{(a_1,0){;~} a_1\in\mathbb{C}\}$, and $\deg(W,\mathbf{x}) = 2 $ for any $\mathbf{x}\in \crit(W)$. 
\end{example}

Before proving Proposition \ref{prop-JW}, let us consider a simple case. 

\begin{lemma}
\label{lemma-JV}
Let $\delta\geq3$, and let $V:\widehat{\mathcal{P}}^\delta \rightarrow \mathbb{C}^{\delta-1}$ be the \emph{critical value map} defined by $V(\mathbf{c},a) = (f_{\mathbf{c},a}(c_k))_{1\leq k\leq \delta-1}$. 
We regard $c_{\delta-1} = -(c_{1}+\cdots+c_{\delta-2})$ as a function of $(c_{1},\dots,c_{\delta-2})$. 
Let $J_{V}$ denote the Jacobian matrix of $V$ with respect to $(c_{1},\dots,c_{\delta-2},a)$. 
Then we have 
$$\det(J_V) = C\prod_{1\leq k_1< k_2\leq \delta-1} (c_{k_1}-c_{k_2})^2,$$
where $C=C(\delta)$ is a nonzero constant.
\footnote{
As comparison, let $J_\pi$ denote the Jacobian matrix of the projection 
$\pi:\widehat{\mathcal{P}}^\delta\rightarrow\mathcal{P}^\delta$, 
$(\mathbf{c},a)\mapsto f_{\mathbf{c},a}$ with respect to $(c_{1},\dots,c_{\delta-2},a)$. 
Then we have $\det(J_\pi) = C\prod_{1\leq k_1< k_2\leq \delta-1} (c_{k_1}-c_{k_2})$. 
}
\end{lemma}

\begin{proof}
Note that $f_{\mathbf{c},a}$ takes the form $f_{\mathbf{c},a}(z)=\sum_{j=0}^{\delta-1}a_{j}z^{\delta-j}+a$, where $a_{j}$ is a $j$-th symmetric function of $\mathbf{c}$. 
For each $1\leq k\leq \delta-1$ and $1\leq n\leq \delta-2$, since $f'_{\mathbf{c},a}(c_k)=0$, we have 
\begin{align*}
\frac{\partial f_{\mathbf{c},a}(c_k)}{\partial c_n} 
&= \left.\frac{\partial f_{\mathbf{c},a}(z)}{\partial c_n}\right|_{z=c_k} 
+ \left.\frac{\partial f_{\mathbf{c},a}(z)}{\partial c_{\delta-1}}\right|_{z=c_k}
\frac{\partial c_{\delta-1}}{\partial c_n}
+ \left.\frac{\partial f_{\mathbf{c},a}(z)}{\partial z}\right|_{z=c_k}
\frac{\partial c_k}{\partial c_n}
\\
&= \left.\frac{\partial f_{\mathbf{c},a}(z)}{\partial c_n}\right|_{z=c_k} 
- \left.\frac{\partial f_{\mathbf{c},a}(z)}{\partial c_{\delta-1}}\right|_{z=c_k},  \\
\frac{\partial f_{\mathbf{c},a}(c_k)}{\partial a} 
&= 1, 
\end{align*}
where we regard $f_{\mathbf{c},a}(c_k)$ as a function of $(c_1,\dots,c_{\delta-2},a)$, and regard $f_{\mathbf{c},a}(z)$ as a function of $(c_1,\dots,c_{\delta-1},a,z)$. 
Let $1\leq k_1<k_2\leq \delta-1$. 
If $c_{k_1}=c_{k_2}$, then the $k_1$-th row and the $k_2$-th row of $J_V$ are same.
It follows that $c_{k_1}-c_{k_2}\mid \det(J_V)$. 
If we exchange $c_{k_1}$ and $c_{k_2}$, when $k_2<\delta-1$, the $k_1$-th row and the $k_2$-th row of $J_V$ exchange, and the $k_1$-th column and the $k_2$-th column of $J_V$ also exchange, so $\det(J_V)$ does not change;  
when $k_2=\delta-1$, after subtracting the $k_1$-th column from the $n$-th column for every $1\leq n\leq \delta-2$ with $n\neq k_1$, the $k_1$-th row and the $(\delta-1)$-th row of $J_V$ exchange, and the $k_1$-th column changes sign, 
so $\det(J_V)$ does not change. 
This implies $(c_{k_1}-c_{k_2})^2\mid \det(J_V)$. 
As multivariate polynomials of $(c_1,\dots,c_{\delta-2})$, 
we have $\deg{\left(\frac{\partial f_{\mathbf{c},a}(c_k)}{\partial c_n}\right)}\leq \delta-1$ for each $1\leq k\leq \delta-1$ and $1\leq n\leq \delta-2$, so  
$$\deg(\det(J_V))\leq (\delta-1)(\delta-2) = \deg \Bigg(\prod_{1\leq k_1< k_2\leq \delta-1} {\left(c_{k_1}-c_{k_2}\right)}^2\Bigg).$$
We remark that $c_{k_1}-c_{\delta-1} = c_{k_1}+c_1+\cdots+c_{\delta-2}$. 
It follows that $\det(J_V)$ and $\prod_{1\leq k_1< k_2\leq \delta-1} {\left(c_{k_1}-c_{k_2}\right)}^2$ are mutually divisible. 
\end{proof}

\begin{proof}
[Proof of Proposition \ref{prop-JW}]
Let $v_1,\dots,v_{n_1}$ denote the vertices $v\in\Tfn$ with $r_v=1$; let $v_{n_1+1},\dots,v_{n_2}$ denote the vertices $v\in\Tfn$ with $r_v=2$; $\dots$
Finally, we obtain an order $(\Tfn,\prec)$ of $\Tfn$.  
For each $v\in\Tfn$ with $\delta(v)\geq 2$, let $J_v$ denote the Jacobian matrix of $(w^{r_v}_{v,1},\dots, w^{r_v}_{v,\delta(v)-1})$ with respect to $(c_{v,1},\dots,c_{v,\delta(v)-2},a_v)$.  
According to the choice of the order $(\Tfn,\prec)$, we see that 
$J_W$ is a block lower triangular matrix with diagonal $(J_v)_{v\in\Tfn,\ \delta(v)\geq2}$. 
It follows that 
$\det(J_W) = \prod_{v\in\Tfn,\ \delta(v)\geq2}\det(J_v)$. 
To prove (\ref{JW}), we just need to show that 
\begin{equation}
\label{Jv}
\det(J_v) = C_v 
\prod_{1\leq k_1< k_2\leq \delta(v)-1} {\left(c_{v,k_1}-c_{v,k_2}\right)}^2
\prod_{
\begin{subarray}{c}
1\leq k<\delta(v),\\
1\leq j<r_v,\\
1\leq k'<\delta(\sigma^j(v))
\end{subarray}
} (c_{\sigma^j(v),k'}-w^{j}_{v,k}), 
\end{equation}
where $C_v$ is a nonzero constant. 

If $\delta(v)=2$ and $r_v=1$, then $J_v=1$. 
If $\delta(v)\geq3$ and $r_v=1$, then (\ref{Jv}) follows from Lemma \ref{lemma-JV}. 
In the following, we assume $r_v\geq 2$. 
Let $F = f_{\sigma^{r_v - 1}(v)}\circ\cdots\circ f_{\sigma(v)}$. 

Assume $\delta(v)\geq3$. 
Let $1\leq k\leq \delta(v)-1$.
Since $F$ is independent of $(\mathbf{c}_v,a_v)$, we have 
\begin{align*}
   \frac{\partial w^{r_v}_{v,k}}{\partial c_{v,n}} 
&= \frac{\partial F(w_{v,k}^1)}{\partial c_{v,n}} 
 = F'(w_{v,k}^1) \frac{\partial w_{v,k}^1}{\partial c_{v,n}},~1\leq n\leq \delta(v)-2,\\
   \frac{\partial w^{r_v}_{v,k}}{\partial a_v} 
&= \frac{\partial F(w_{v,k}^1)}{\partial a_v} 
 = F'(w_{v,k}^1) \frac{\partial w_{v,k}^1}{\partial a_v}.
\end{align*} 
Therefore the $k$-th row of $J_v$ is 
\begin{equation}
\label{k-row}
F'(f_v(c_{v,k}))
{{\left(\frac{\partial f_v(c_{v,k})}{\partial c_{v,1}},\dots,
\frac{\partial f_v(c_{v,k})}{\partial c_{v,\delta(v)-2}},\frac{\partial f_v(c_{v,k})}{\partial a_{v}}\right)}},  
\end{equation}
where 
\begin{equation}
\label{fvcvk}
F'(f_v(c_{v,k})) = \prod_{j=1}^{r_v-1} f_{\sigma^j(v)}'(w^j_{v,k}) 
= \prod_{j=1}^{r_v-1} \Bigg(\delta(\sigma^j(v))
\prod_{1\leq k'<\delta(\sigma^j(v))} (w^j_{v,k} - c_{\sigma^j(v),k'})\Bigg).  
\end{equation}
Then (\ref{Jv}) follows from (\ref{k-row}), (\ref{fvcvk}) and Lemma \ref{lemma-JV}. 

Now assume $\delta(v) = 2$. 
Then $$J_v = \frac{\partial w^{r_v}_{v,1}}{\partial a_v} 
= F'(f_v(c_{v,1})) \frac{\partial f_v(c_{v,1})}{\partial a_{v}} = F'(f_v(c_{v,1})). $$
Letting $k=1$ in (\ref{fvcvk}) gives (\ref{Jv}). 
This proves (\ref{JW}). 

By \cite[Theorem 15.1.8]{Rudin}, the map $W$ is locally one-to-one near $\mathbf{x}$ if and only if $\det(J_W(\mathbf{x}))\neq 0$. 
It follows from (\ref{JW}) that $\det(J_W(\mathbf{x}))\neq 0$ if and only if  $\mathbf{x}$ has no free critical relation. 
\end{proof}

By Proposition \ref{continuity of Ufv}, for $f\in\overline{\mathcal{H}}$ and $(v,k)\in I$, 
we have $f^{r_v}(c_{v,k}(f))\in \overline{U_{f,\sigma^{r_v}(v)}}$. 
Define 
\begin{equation}
\label{crit-val-H}
\widetilde W:\begin{cases} 
\Hbar \rightarrow \overline{\mathbb{D}}^{\Cdim (\mathcal H)}, \\ 
f \mapsto (B_{f,\sigma^{r_v}(v)}(f^{r_v}(c_{v,k}(f))))_{(v,k)\in I}. 
\end{cases}
\end{equation}
Recall that $\chi:\Hbar\rightarrow \modelf$ is the straightening map defined by (\ref{straightening map}). 
Then $\widetilde W = W\circ\chi$. That is, the diagram 
$$\xymatrix{
  \Hbar   \ar[rd]_-{\widetilde W} \ar[r]^-{\chi}
& \modelf \ar@{^(->}[r]
& \widehat{\mathcal{P}}_0^{T(f_0)}  \ar[d]^{W} \\
& \overline{\mathbb{D}}^{\Cdim (\mathcal H)} \ar@{^(->}[r]
& \mathbb{C}^{\Cdim (\mathcal H) }
}$$ is commutative. 

As an application of Proposition \ref{prop-JW}, we have the following lemma, 
which will be used in the proof of Proposition \ref{transversality-prop}. 

\begin{lemma}
\label{tilde-W-local-homeo}
Let $f\in\overline{\mathcal{H}}$. 
Then $f$ has no internal free critical relation if and only if $\widetilde W$ is locally one-to-one near $f$.
\end{lemma}

\begin{proof}
Since the straightening map $\chi$ preserves the dynamics on the small filled Julia sets, 
the polynomial $f$ has no internal free critical relation if and only if $\mathbf{x}:=\chi(f)$ has no free critical relation. 
That is, the map $W$ is locally one-to-one near $\mathbf{x}$ by Proposition \ref{prop-JW}. 
Note that $W:\widehat{\mathcal{P}}_0^{T(f_0)}\rightarrow \mathbb{C}^{\Cdim (\mathcal H)}$ is a proper holomorphic map satisfying 
$W^{-1}\big({\overline{\mathbb{D}}}^{\Cdim (\mathcal H)}\big) = \modelf$. 
Let $W_0$ denote the restriction $W:\modelf \rightarrow \overline{\mathbb{D}}^{\Cdim (\mathcal H)}$. 

\vspace{6pt}
{\bf Claim.}
{\it If the restriction $W_0$ is locally one-to-one near $\mathbf{x}$, 
then $W$ is locally one-to-one near $\mathbf{x}$. }
\vspace{6pt}

Let $\mathcal{N}\subset \mathbb{C}^{\Cdim (\mathcal H)}$ be an open ball centered at $W(\mathbf{x})$, 
and let $\mathcal{M}$ be the connected component of $W^{-1}(\mathcal{N})$ containing $\mathbf{x}$. 
Then $W:\mathcal{M}\rightarrow\mathcal{N}$ is a proper holomorphic map. 
Since $W_0$ is locally one-to-one near $\mathbf{x}$, by choosing $\mathcal{N}$ small enough, we get a homeomorphism 
$$W:\mathcal{M} \cap \modelf \rightarrow 
\mathcal{N}\cap\overline{\mathbb{D}}^{\Cdim (\mathcal H)}.$$
Let $\mathbf{y}\in\big(\mathcal{N} \cap {\overline{\mathbb{D}}}^{\Cdim (\mathcal H)}\big)\setminus W(\crit(W))$. 
By the mapping properties of proper holomorphic maps (see, e.g. \cite[\S15.1]{Rudin}), we have 
$$\deg(W:\mathcal{M}\rightarrow\mathcal{N})=\#(\mathcal{M}\cap W^{-1}(\mathbf{y}))=1.$$
It follows that $W:\mathcal{M}\rightarrow\mathcal{N}$ is one-to-one. 
\vspace{6pt}

The claim implies that $W$ is locally one-to-one near $\mathbf{x}$ if and only if $W_0$ is locally one-to-one near $\mathbf{x}$. 
Because $\widetilde W = W\circ\chi = W_0\circ\chi$, the map $W_0$ is locally one-to-one near $\mathbf{x}$ if and only if $\widetilde W$ is locally one-to-one near $f$. 
The proof is completed. 
\end{proof}

\subsection{External free critical relation} 
For a neighborhood $\mathcal{N}\subset\mathcal{F}$ of $f\in\partial\mathcal{H}$, we define $$\mathcal{N}^* 
= \big\{g\in\mathcal{N}{;~} g^{r_v}(c_{v,k}(g))\in \overline{U_{g,\sigma^{r_v}(v)}},~\forall~(v,k)
\in I\big\}.$$ By Proposition \ref{continuity of Ufv}, we have $\mathcal{N}\cap\overline{\mathcal{H}} 
\subset\mathcal{N}^*$. On the other hand, we have the following result, which is useful in the proofs 
of Proposition \ref{transversality-prop} and Theorem \ref{hd-boundary}. 

\begin{proposition}
\label{capture-separation}
Given $f\in \partial \mathcal{H}\setminus\Sigma(\mathcal{F})$ satisfying $\crit(f)\cap \partial A(f) 
= \emptyset$. If $f$ has no external free critical relation, then there is a neighborhood $\mathcal{N} 
\subset\mathcal{F}$ of $f$ so that $$\mathcal{N}^*\subset\overline{\mathcal{H}}.$$
\end{proposition}

Combining $\mathcal{N}^*\subset\overline{\mathcal{H}}$ and $\mathcal{N}\cap\overline{\mathcal{H}} 
\subset\mathcal{N}^*$, we have $\mathcal{N}\cap\overline{\mathcal{H}}=\mathcal{N}^*$. We remark that 
if $f$ has an external free critical relation, then the neighborhood $\mathcal{N}$ in Proposition 
\ref{capture-separation} may not exist; see Example \ref{example-FreeCriticalRelation}. 

Recall from (\ref{Kfv}) that $K_{f,v}$ is well-defined for all $f\in\Hbar$ and all $v\in|T(f_0)|$. 

\begin{lemma}
\label{non-repeating orbit}
Let $f\in\Hbar$, $v\in\Tfn$ and $x\in \partial K_{f,v}$. Then $x,f(x),\dots, f^{r_v}(x)$ are pairwise 
different. 
\end{lemma}

\begin{proof}
For each $0\leq j\leq r_v$, we have $f^j(x)\in \partial K_{f,\sigma^j(v)}$. Let $u=\sigma^{r_v-1}(v)$. 
Then $u\in\Tfn$. We first show that $f_0$ has no periodic point in $\partial U_{f_0}(u)$. 

Suppose $f_0$ has a periodic point $y\in \partial U_{f_0}(u)$; we will find a contradiction. Replacing 
$f_0$ by an appropriate iterate, we may assume $y$ and $f_0(u)$ are fixed points of $f_0$. Then $y\in 
\partial U_{f_0}(u)\cap \partial U_{f_0}(f_0(u))$. Since $f_0(U_{f_0}(u)) = U_{f_0}(f_0(u)) = 
f_0(U_{f_0}(f_0(u)))$, we have $y\in\crit(f_0)$. This contradicts $\crit(f_0)\cap J(f_0)=\emptyset$. 

Because $f_0$ has no periodic point in $\partial U_{f_0}(u)$, the angle $\theta^-_u(t)$ is not periodic 
under $m_d$ for any $t\in\RZ$ (see \S\ref{subsection def-straightening-map} for $\theta^-_u$). It 
follows that $f$ has no periodic point in $\partial K_{f,u}$. In particular, the point $f^{r_v-1}(x) 
\in \partial K_{f,u}$ is not periodic under $f$. This implies $x,f(x), \dots, f^{r_v-1}(x), f^{r_v}(x)$ 
are pairwise different. 
\end{proof}

\begin{proof}
[Proof of Proposition \ref{capture-separation}] 
Let $g\in\mathcal{N}^*$. 
We can shrink $\mathcal{N}$ if necessary. 
To prove $g\in\Hbar$, we will first show $\lambda_\mathbb{Q}(f_0)\subset\lambda_\mathbb{Q}(g)$ (see Claims 1, 2 and 3 below).  
This allows us to define $(K_{g,v})_{v\in|T(f_0)|}$ as (\ref{Kfv}). 
After fixing an external marking for $f_0$, the polynomial $g$ induces a marked lemniscate map $(\mathbf{g}_g,K_g,Z_g,\Gamma_g)$ over $T(f_0)$ as (\ref{(gf,Kf,Zf,Gammaf)}). 
By Proposition \ref{straightening of marked LM}, this marked lemniscate map 
and the critical marking $\mathbf{c}(g)$ determine a generalized polynomial $\mathbf{P}\in \modelf$. 
Recall that $\chi:\Hbar \rightarrow \modelf$ is the straightening map defined by (\ref{straightening map}). 
Let $\tilde g = \chi^{-1}(\mathbf{P})$. 
Similar to the proof of Lemma \ref{chi is injective}, we will show $g = \tilde g$. 
This implies $g\in \Hbar$. 

Recall that $|T^\infty(f_0)|$ is the grand orbit of $\crit(f_0)$ under $f_0$. 
For each $v\in|T^\infty(f_0)|$, we define $\sigma(v)=f_0(v)$ and $\delta(v)=\deg(f_0,v)$. 
Choose a collection $(\theta_v)_{v\in|T^\infty(f_0)|}$ of angles in $\RZ$ so that $\theta_v = \theta^-_{f_0,v}(0)$ for each $v\in\Tfp$,  
and $R_{f_0}(\theta_v)$ lands at a point in $\partial U_{f_0}(v)$ and $\theta_{\sigma(v)}\equiv d\cdot\theta_v\modZ$ for each $v\in|T^\infty(f_0)|$.   
For each $v\in|T^\infty(f_0)|$, the maps $\theta^-_v,\theta^+_v:\RZ\rightarrow\RZ$ are well-defined as in \S\ref{subsection def-straightening-map}, and satisfy (\ref{d-delta}).

Let $\mathcal N\subset\mathcal{F}$ be a neighborhood of $f$ so that the holomorphic motion $h:\mathcal N\times \overline{A(f)}\rightarrow \mathbb C$ is well-defined.

\vspace{6pt}
{\bf Claim 1.}
{\it Shrinking $\mathcal{N}$ if necessary, the external rays $R_g(\theta^-_{v}(t))$ and $R_g(\theta^+_{v}(t))$ land at a common point in $\partial U_{g,v}$ for any $g\in\mathcal{N}$, $v\in\Tfp$ and $t\in\RZ$. (Using the holomorphic motion.)}
\vspace{6pt}

Fix $v\in \Tfp$. 
Recall from Lemma \ref{invariant external angles} that $\theta_{f,v}^-=\theta_v^-$ and $\theta_{f,v}^+=\theta_v^+$. 
Define 
\begin{align*}
X_v &= \overline{U_{f,v}}\bigsqcup \bigsqcup_{t\in\RZ} \Big(R_f(\theta^-_{v}(t))\cup R_f(\theta^+_{v}(t))\Big),\\ 
T_v &= \big\{t\in\RZ{;~} \text{$\theta^-_{v}(t)\neq \theta^+_{v}(t)$ and 
$\crit(f) \cap S_{f}(\theta^-_{v}(t),\theta^+_{v}(t))\neq\emptyset$} \big\}. 
\end{align*}
By the assumption $\crit(f)\cap \partial A(f)=\emptyset$, we have 
$$T_v= \big\{t\in\RZ{;~} \text{$\theta^-_{v}(t)\neq \theta^+_{v}(t)$ and 
$\crit(f_0) \cap S_{f_0}(\theta^-_{v}(t), \theta^+_{v}(t))\neq\emptyset$} \big\}.$$
It follows from Lemma \ref{sectors attaching to U} that  $T_v\subset\QZ$. 
Note that $T_v$ is a finite set. 
By the stability of external rays (see Lemma \ref{stability of external rays 0}) and the continuity of $h$, we can shrink $\mathcal{N}$ so that  
whenever $g\in\mathcal{N}$, 
\begin{itemize}
\item the dynamical rays $R_{g,v}(t) = B^{-1}_{g,v}([0,1) e^{2\pi i t})$, $R_g(\theta^-_{v}(t))$ and $R_g(\theta^+_{v}(t))$ land at the same pre-repelling point $B_{g,v}^{-1}(e^{2\pi i t})$ for any $t\in T_v$; 

\item $\{c_{u,k}(g){;~}(u,k)\in I\}\subset \bigcup_{t\in T_v} S_g(\theta^-_{v}(t), \theta^+_{v}(t))$.  
\end{itemize}
Let $g\in\mathcal{N}$ and let $t\in\RZ$. 
According to the locations of critical points, we have  $$\crit(g)\cap \bigcup_{n\in\mathbb{N}} g^n(R_g(\theta^-_v(t)))=\emptyset,$$ so the external ray $R_g(\theta^-_v(t))$  does not bifurcate. 
Now we can extend $B_g$ holomorphically to a neighborhood of the external ray $R_g(\theta^-_v(t))$, and the external ray $R_g(\theta^+_v(t))$ similarly. 
When $z\in R_f(\theta^-_{v}(t))\cup R_f(\theta^+_{v}(t))$, we extend $h$ by $h(g,z) = B_g^{-1}\circ B_f(z)$. 
Then we have a holomorphic motion 
\begin{equation}
\label{extension-h}
h:\mathcal{N}\times \bigcup_{v\in\Tfp} X_v\rightarrow\mathbb{C},
\end{equation} which is continuous by the $\lambda$-lemma \cite{McM}. 
Therefore Claim 1 holds.

\vspace{6pt}
{\bf Claim 2.}
{\it Shrinking $\mathcal{N}$ if necessary, the external rays  $R_g(\theta^-_{v}(t))$ and $R_g(\theta^+_{v}(t))$ land at the same point 
for any $g\in\mathcal{N}^*$, $v\in\Tfn$ and $t\in\RZ$.  (Using the assumption that $f$ has no external free critical relation.)}
\vspace{6pt}

Note that $\Tfn$ is finite and $\RZ$ is compact. 
To prove Claim 2, we just need to show: 

\vspace{6pt}
{\bf Subclaim.} 
{\it Given $v\in\Tfn$ and $x\in\partial K_{f,v}$. 
Shrinking $\mathcal{N}$ if necessary, there is an $\varepsilon>0$ so that $R_g(\theta^-_{v}(t))$ and $R_g(\theta^+_{v}(t))$ land at the same point 
for any $g\in\mathcal{N}^*$ and $t\in T_{v,x}^{\varepsilon}$, where $T_{v,x}^{\varepsilon}$ is the $\varepsilon$-neighborhood of 
$$T_{v,x}=\big\{t\in\RZ{;~} \text{$R_f(\theta^-_v(t))$ lands at $x$}\big\}.$$ }
\vspace{-10pt}

By Lemma \ref{non-repeating orbit}, the points $x,f(x),\dots, f^{r_v}(x)$ are pairwise different. 
By Claim 1, the subclaim is true for the pair $(\sigma^{r_v}(v), f^{r_v}(x))$. 
By induction, we assume the subclaim is true for the pairs $(\sigma^j(v), f^j(x))$ with $1\leq j\leq r_v$.

\vspace{6pt}
{\bf Case 1.} 
{\it $x\notin\crit(f)$. }
\vspace{6pt}

Shrinking $\mathcal{N}$ if necessary, there is an $r>0$ so that $g|_{\mathbb{D}(x,r)}$ is univalent for any $g\in \mathcal{N}$. 
Furthermore, there is an $\varepsilon>0$ so that the landing points of $R_g(\theta^-_{v}(t))$ and $R_g(\theta^+_{v}(t))$ are in $\mathbb{D}(x,r/2)$ for any $g\in\mathcal{N}^*$ and $t\in T_{v,x}^{\varepsilon}$. 
Then the subclaim follows from the induction and the pullback of $g|_{\mathbb{D}(x,r)}$.

\vspace{6pt}
{\bf Case 2.} 
{\it $x\in\crit(f)$.}
\vspace{6pt}

Let 
\begin{equation}
\label{Ix}
I(x) = \big\{(w,k)\in I{;~} c_{w,k}(f)=x\big\}.
\end{equation}
% by $\{(v_m(x),k_m(x))\}_{m=1}^{M(x)}$. 
Since $f$ has no external free critical relation, for any $(w,k),(w',k')\in I(x)$, we have $w=w'$, which will be denoted by $v(x)$. 
(If $f$ has an external free critical relation, then there may exist $(w,k),(w',k')\in I(x)$ so that $w\neq w'$; see Example \ref{example-FreeCriticalRelation}.) 
We remark that both $v(x)=v$ and $v(x)\neq v$ are possible. 
The following discussion will focus on $u:= v(x)$. 

% $\sigma^{r_{v(x)}}(v(x))$  $f^{r_{v(x)}}(x)$

Let $0=j_0<j_1<\cdots<j_N<r_v$ be all integers so that $f^{j_0}(x),f^{j_1}(x),\dots,f^{j_N}(x) $ $\in$ $\crit(f)$. 
Then $r_{u}\geq j_N+1$. 
Let $t_0\in\RZ$ be the internal angle of $f^{r_{u}}(x)$ with respect to $U_{f,\sigma^{r_{u}}(u)}$. 
For $\varepsilon_0>0$ small enough, we define an arc 
\begin{align*}
\alpha_f =~ &B_{f,\sigma^{r_u}(u)}^{-1}\Big( (1-\varepsilon_0) e^{2\pi i[t_0-\varepsilon_0,t_0+\varepsilon_0]} \Big) \\
&\cup B_{f,\sigma^{r_u}(u)}^{-1}\Big( [1-\varepsilon_0, 1] e^{2\pi i(t_0 - \varepsilon_0)} \Big)
\cup B_{f,\sigma^{r_u}(u)}^{-1}\Big( [1-\varepsilon_0, 1] e^{2\pi i(t_0 + \varepsilon_0)} \Big)  \\
&\cup B_f^{-1}\Big( (1,1+\varepsilon_0] e^{2\pi i \theta_u^-(t_0-\varepsilon_0)} \Big)
 \cup B_f^{-1}\Big( (1,1+\varepsilon_0] e^{2\pi i \theta_u^-(t_0+\varepsilon_0)} \Big).
\end{align*}
Choose an arc $\beta_f\subset \mathbb{C}\setminus \overline{U_{f,\sigma^{r_{u}}(u)}}$ connecting the end points of $\alpha_f$ so that they are disjoint elsewhere 
and the bounded connected component of $\mathbb{C}\setminus(\alpha_f\cup\beta_f)$ contains $f^{r_{u}}(x)$. 
Let $D_{f,N+1}$ denote this connected component. 
For $0\leq n\leq N$, let $D_{f,n}$ be the connected component of $f^{-(r_u-j_n)}(D_{f,N+1})$ containing $f^{j_n}(x)$. 
By decreasing $\varepsilon_0$ and choosing $\beta_f$ properly, the diameter of $D_{f,N+1}$ can be arbitrarily small. 
By Lemma \ref{non-repeating orbit} again, the points $x,f(x),\dots, f^{r_u}(x)$ are pairwise different. 
Now fix $\varepsilon_0>0$ so that 
\begin{itemize}
\item $\deg(f^{r_u}|_{D_{f,0}}) = \deg(f^{r_u},x)$ and $\crit(f)\cap\partial D_{f,n}=\emptyset$ for any $0\leq n\leq N$; 
\item the closures of $D_{f,0},D_{f,1},\dots,D_{f,N+1}$ are pairwise disjoint. 
\end{itemize}

For $g\in \mathcal{N}$, we define $\alpha_g = h(g,\alpha_f)$, where $h$ is given by (\ref{extension-h}). 
Similarly, we choose an arc $\beta_g\subset \mathbb{C}\setminus \overline{U_{g,\sigma^{r_{u}}(u)}}$ connecting the end points of $\alpha_g$, 
and let $D_{g,N+1}$ be the bounded connected component of $\mathbb{C}\setminus(\alpha_g\cup\beta_g)$. 
By shrinking $\mathcal{N}$ and choosing $\beta_g$ properly, the arc $\alpha_g$ is close to $\alpha_f$, 
and the arc $\beta_g$ is close to $\beta_f$ in Hausdorff topology. 
Set $X_{g,N+1} = D_{g,N+1} \cap \overline{U_{g,\sigma^{r_u}(u)}}$. 
For $0\leq n\leq N$, let $D_{g,n}$ be the connected component of $g^{-(r_u-j_n)}(D_{g,N+1})$ close to $D_{f,n}$, 
and let $X_{g,n} = g^{-(r_u-j_n)}(X_{g,N+1}) \cap D_{g,n}$. 
See Figure \ref{figure Dgn}. 
All in all, we can shrink $\mathcal{N}$ so that the perturbations of the boundaries $\partial D_{g,0}, \dots, \partial D_{g,N+1}$, the sets $X_{g,0}, \dots, X_{g,N+1}$ and 
the critical marking $\mathbf{c}(g)$ in Hausdorff topology are much less than $\min\{{\diam}(\partial D_{f,n}){;~}0\leq n\leq N+1\}$.  
% the sizes of $D_{f,0}, \dots, D_{f,N+1}$. 

\begin{figure}[ht]
\centering
\includegraphics{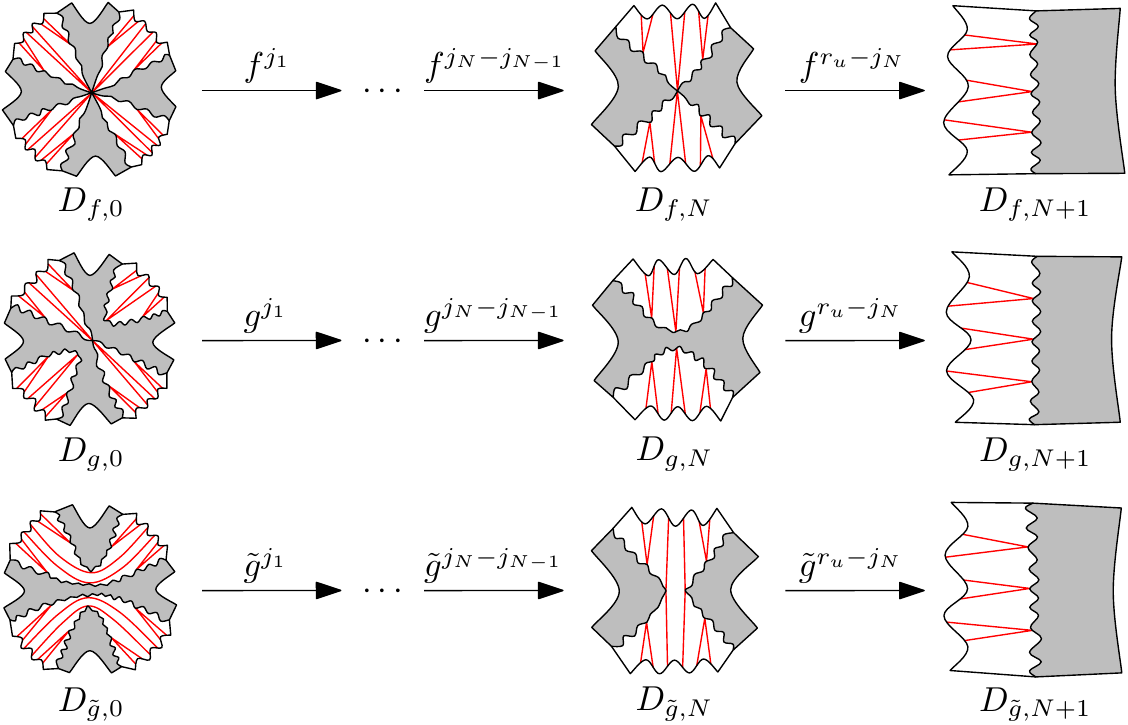}
\caption{A sketch of $D_{g,n}$, with $X_{g,n}$ colored gray and external rays colored red. 
The second row illustrates $g\in\mathcal{N}^*$. 
The third row illustrates the possible situation for some $\tilde g\in \mathcal{N}\setminus\mathcal{N}^*$.}
\label{figure Dgn}
\end{figure}

Let $(w,k)\in I(f^{j_n}(x))$ for some $0\leq n\leq N$; see (\ref{Ix}) for $I(\cdot)$. 
Since $f$ has no external free critical relation, 
we have $w = v(f^{j_n}(x)) = \sigma^{j_n}(v(x)) = \sigma^{j_n}(u)$. 
It follows that $r_w = r_u - j_n$. 
Let $g\in\mathcal{N}^*$. Since 
$$g^{r_u - j_n}(c_{w,k}(g)) = g^{r_w}(c_{w,k}(g))\in\overline{U_{g,\sigma^{r_w}(w)}} = \overline{U_{g,\sigma^{r_u}(u)}},$$
we can shrink $\mathcal{N}$ so that $g^{r_u - j_n}(c_{w,k}(g))\in X_{g,N+1}$. 
By induction on $n=N,N-1,\dots,0$, we see that $X_{g,n}$ is connected (compare Lemma \ref{connectedness-crit}), 
and $g^{r_u - j_n}$ is univalent on each connected component of $D_{g,n}\setminus X_{g,n}$. 

Since $f$ has no external free critical relation, we have $X_{f,n} = D_{f,n} \cap K_{f,\sigma^{j_n}(u)}$ for each $0\leq n\leq N$. 
It follows that $\#(T_{u,x}) = \deg(f^{r_u},x)$. 
Each angle in $T_{u,x}$ corresponds to a connected component of $D_{f,0}\setminus X_{f,0}$; see Figure \ref{figure Dgn}. 
Recall that $t_0$ is the internal angle of $f^{r_{u}}(x)$ with respect to $U_{f,\sigma^{r_{u}}(u)}$. 
By Claim 1, there is an $\varepsilon>0$ so that 
$R_g(\theta^-_{\sigma^{r_u}(u)}(t))$ and $R_g(\theta^+_{\sigma^{r_u}(u)}(t))$ land at a common point in $\partial X_{g,N+1}$ 
for any $g\in \mathcal{N}$ and $t\in (t_0-\varepsilon,t_0+\varepsilon)$. 
Since $g^{r_u}$ is univalent on each connected component of $D_{g,0}\setminus X_{g,0}$, we can decrease $\varepsilon$ if necessary 
so that $R_g(\theta^-_{u}(t))$ and $R_g(\theta^+_{u}(t))$ land at a common point in $\partial X_{g,0}$ 
for any $g\in\mathcal{N}^*$ and $t\in T_{u,x}^\varepsilon$. 
% For $\tilde g\in \mathcal{N}\setminus\mathcal{N}^*$, the set $X_{\tilde g,n}$ may be disconnected, 
% and $\tilde g^{r_u - j_n}$ may be not univalent on some connected component of $D_{\tilde g,n}\setminus X_{\tilde g,n}$. 
% (As comparison, for any $\varepsilon>0$, there may exist $\tilde g\in \mathcal{N}\setminus\mathcal{N}^*$ and $t\in T_{u,x}^\varepsilon$ such that 
% $R_{\tilde g}(\theta^-_{u}(t))$ and $R_{\tilde g}(\theta^+_{u}(t))$ land at two different points in $\partial X_{g,0}$; see Figure \ref{figure Dgn}.) 

If $u=v$, the subclaim is clear. 

Now assume $u\neq v$. 
By (\ref{Kfv}), we have $K_{f,u}\cap K_{f,v}=\{x\}$, 
and there is a unique $\tau\in T_{u,x}$ such that $K_{f,v}\subset \overline{S_{f}(\theta^-_u(\tau),\theta^+_u(\tau))}$.  
% Since $S_{f_0}(\theta^-_u(\tau),\theta^+_u(\tau))$ is a maximal sector attaching to $U_{f_0}(u)$, we have $\theta^-_u(\tau),\theta^+_u(\tau)\in\QZ$ by Lemma \ref{sectors attaching to U}. 
For $g\in \mathcal{N}^*$, let $Y_g$ denote the connected component of $D_{g,0}\setminus X_{g,0}$ containing the tail of the external ray $R_g(\theta_u^-(\tau))$. 
Note that $g$ is univalent on $Y_g$. 
% we can shrink $\mathcal{N}$ if necessary and choose $r>0$ so that $g^{r_v}|_{\overline{S_g} \cap \mathbb{D}(r(S_g),r)}$ is injective for any $g\in \mathcal{N}^*$. 
The subclaim follows from the induction and the pullback of $g|_{Y_g}$ (compare Case 1). 

% We remark that $\tau\in T_{u,x}$, $\#(T_{u,x})>1$ and $\#( T_{v,x})=1$. 

% For each $0\leq n\leq N$, we have $${\left\{s_{w,k}{;~}(w,k)\in I(f^{j_n}(x))\right\}} \subset  X_{g,n} \subset  \overline{S_g(d^{j_n}\cdot\theta^+_u(\tau), d^{j_n}\cdot\theta^-_u(\tau))}.$$

This shows the subclaim and Claim 2.

\vspace{6pt}
{\bf Claim 3.}
{\it The external rays $R_g(\theta^-_{v}(t))$ and $R_g(\theta^+_{v}(t))$ land at the same point 
for any $g\in \mathcal{N}^*$, $v\in|T^\infty(f_0)|\setminus|T(f_0)|$ and $t\in\RZ$. (Making a combinatorial discussion of rational lamination.)}
\vspace{6pt}

%Let $g\in\mathcal{C}(\Pd)$. 
%If $\bigcup_{v\in|T(f_0)|} \Theta(f_0,v)\subset \lambda_\mathbb{Q}(g)$, then $\lambda_\mathbb{Q}(f_0)\subset\lambda_\mathbb{Q}(g)$. 
%Using the \emph{pinched disk model} (see \cite{DCompa}) of rational laminations

Given $v\in|T^\infty(f_0)|\setminus|T(f_0)|$ with $\sigma(v)\in|T(f_0)|$. 
Let $u\in \sigma^{-1}(\sigma(v))\setminus\{v\}$. 
For the allowable open arc $(v,u)\subset K(f_0)$, 
we have $(v,u)\cap \crit(f_0)\neq \emptyset$ by Lemma \ref{allowable arcs in K(f0)}(\ref{allowable arcs in K(f0)-1}). 
Choose $c_u\in (v,u)\cap \crit(f_0)\subset |T(f_0)|$. 
Then there is a unique $t_u\in\RZ$ such that $u\in S_{f_0}(\theta^-_{c_u}(t_u), \theta^+_{c_u}(t_u))$. 

Let $g\in \mathcal{N}^*$. 
By Claims 1 and 2, the external rays $R_g(\theta^-_{c_u}(t_u))$ and $R_g(\theta^+_{c_u}(t_u))$ land at the same point. 
Define 
$$X_g = \bigcup_{u\in \sigma^{-1}(\sigma(v))\setminus\{v\}} \overline{S_g(\theta^-_{c_u}(t_u), \theta^+_{c_u}(t_u))}.$$
Let $t\in\RZ$ satisfy $\theta^-_v(t)\neq \theta^+_v(t)$. 
By (\ref{d-delta}) and $v\notin\crit(f_0)$, we have 
\begin{align*}
d\cdot\theta^-_v(t)&\equiv\theta^-_{\sigma(v)}(\delta(v)\cdot t) = \theta^-_{\sigma(v)}(t) \modZ,\\ 
d\cdot\theta^+_v(t)&\equiv\theta^+_{\sigma(v)}(\delta(v)\cdot t) = \theta^+_{\sigma(v)}(t) \modZ. 
\end{align*}
It follows that 
\begin{align}
\label{inverse of external rays-1}
& g^{-1}(R_g(\theta^-_{\sigma(v)}(t)))\setminus X_g = R_g(\theta^-_{v}(t)), \\
\label{inverse of external rays-2}
& g^{-1}(R_g(\theta^+_{\sigma(v)}(t)))\setminus X_g = R_g(\theta^+_{v}(t)). 
\end{align}
By Claims 1 and 2 again, the external rays $R_g(\theta^-_{\sigma(v)}(t))$ and $R_g(\theta^+_{\sigma(v)}(t))$ land at the same point. 
Let $x$ and $x'$ denote the landing points of $R_g(\theta^-_{v}(t))$ and $R_g(\theta^+_{v}(t))$ respectively. 

If $x\notin X_g$, then there is an angle $\theta\in m_d^{-1}(\theta^+_{\sigma(v)}(t))$ such that $R_g(\theta)$ lands at $x$. 
It follows from (\ref{inverse of external rays-2}) that $\theta=\theta^+_{v}(t)$, hence $x=x'$. 

Now assume $x\in X_g$. 
Let $Y$ denote the connected component of $X_g$ containing $x$. 
Then $\partial Y\cap K(g)=\{x\}$ and 
$$
\#\big\{\theta\in m_d^{-1}(\theta^-_{\sigma(v)}(t)){;~} R_g(\theta)\subset Y\big\}
= \#\big\{\theta\in m_d^{-1}(\theta^+_{\sigma(v)}(t)){;~} R_g(\theta)\subset Y\big\}. 
$$
By (\ref{inverse of external rays-1}) and $x\in Y$, we have 
$$\#\big\{\theta\in m_d^{-1}(\theta^-_{\sigma(v)}(t)){;~} x_g(\theta)\in Y\big\}
= \#\big\{\theta\in m_d^{-1}(\theta^-_{\sigma(v)}(t)){;~} R_g(\theta)\subset Y\big\}+1, 
$$
where $x_g(\theta)$ is the landing point of $R_g(\theta)$. 
It follows that 
\begin{align*}
\#\big\{\theta\in m_d^{-1}(\theta^+_{\sigma(v)}(t)){;~} x_g(\theta)\in Y\big\}
& = \#\big\{\theta\in m_d^{-1}(\theta^-_{\sigma(v)}(t)){;~} x_g(\theta)\in Y\big\}\\
& = \#\big\{\theta\in m_d^{-1}(\theta^+_{\sigma(v)}(t)){;~} R_g(\theta)\subset Y\big\}+1. 
\end{align*}
Combining this with (\ref{inverse of external rays-2}), we have $x'\in Y$. 
Then $\partial Y\cap K(g)=\{x\}$ gives $x=x'$. 

Thus Claim 3 is true for $v\in|T^\infty(f_0)|\setminus|T(f_0)|$ with $\sigma(v)\in|T(f_0)|$. 
By induction, Claim 3 is true for any $v\in|T^\infty(f_0)|\setminus|T(f_0)|$.

\vspace{6pt}

Let $g\in\mathcal{N}^*$. 
By Claims 1 and 2, we can define $(K_{g,v})_{v\in|T(f_0)|}$ by (\ref{Kfv}). 
Using the external marking $(\theta_v)_{v\in |T(f_0)|}$ for $f_0$, the polynomial $g$ 
induces a marked lemniscate map $(\mathbf{g}_g,K_g,Z_g,\Gamma_g)$ over $T(f_0)$ as (\ref{(gf,Kf,Zf,Gammaf)}). 
By Proposition \ref{straightening of marked LM}, 
there exists a unique $\mathbf{P}\in \mathcal{C}\big(\mathcal{P}_0^{T(f_0)}\big)$ 
such that $(\mathbf{g}_g,K_g,Z_g,\Gamma_g)$ is conjugate to $(\mathbf{P},K(\mathbf{P}),Z',\Gamma')$, where $(Z',\Gamma')$ is the standard marking of $(\mathbf{P},K(\mathbf{P}))$. 
Furthermore, we fix a conjugacy $(\psi_{g,v})_{v\in|T(f_0)|}$ from $(\mathbf{g}_g,K_g,Z_g,\Gamma_g)$ to $(\mathbf{P},K(\mathbf{P}),Z',\Gamma')$. 
Let $(v,k)\in I$. If $c_{v,k}(f)\in K_{f,v}^\circ$, we can shrink $\mathcal{N}$ so that $c_{v,k}(g)\in K_{g,v}^\circ$. 
If $c_{v,k}(f)\in \partial K_{f,v}$, by the discussion of Case 2 in Claim 2, we have $c_{v,k}(g)\in K_{g,v}$. 
Moreover, for each $v\in\Tfn$ with $\delta(v)\geq2$, the ordered list $(c_{v,k}(g))_{1\leq k<\delta(v)}$ is a critical marking of $g_{g,v}$. 
It follows that $\mathbf{c} :=(\psi_{g,v}(c_{v,k}(g)))_{(v,k)\in I_0}$ is a critical marking of $\mathbf{P}$, 
i.e. $(\mathbf{P}, \mathbf{c})\in\modelf$. 
Recall that $\chi:\Hbar \rightarrow \modelf$ is the straightening map defined by (\ref{straightening map}). 
Let $\tilde g = \chi^{-1}(\mathbf{P}, \mathbf{c})$. 
By Corollary \ref{inclusion relation of real laminations}, we have $\lambda_\mathbb{R}(f_0)\subset\lambda_\mathbb{R}(\tilde g)$. 
To prove $g=\tilde g$, we will first show $\lambda_\mathbb{R}(f_0)\subset\lambda_\mathbb{R}(g)$. 

Let $$\Theta=\big\{(\theta^-_v(t),\theta^+_v(t)){;~}
\text{$v\in|T^\infty(f_0)|$, $t\in\RZ$ and $\theta^-_v(t)\neq\theta^+_v(t)$}\big\}.$$ 
By Lemma \ref{sectors attaching to U}, we have $\Theta\subset\lambda_{\mathbb{Q}}(f_0)$. 
By Proposition \ref{two sequences of external angles 2}, we have $\langle\Theta\rangle_{\mathbb{Q}}=\lambda_\mathbb{Q}(f_0)$. 
By Claims 1, 2 and 3, we have $\Theta\subset\lambda_\mathbb{Q}(g)$. 
Taking $\langle\cdot\rangle_\mathbb{Q}$ on both sides of $\Theta\subset\lambda_\mathbb{Q}(g)$ gives  $\lambda_\mathbb{Q}(f_0)\subset\lambda_\mathbb{Q}(g)$. 
By the behavior of critical orbits, the polynomial $g$ is semi-hyperbolic, so the real lamination $\lambda_{\mathbb{R}}(g)$ of $g$ makes sense. 
Taking $\langle\cdot\rangle_\mathbb{R}$ on both sides of $\lambda_\mathbb{Q}(f_0)\subset\lambda_\mathbb{R}(g)$ gives $\lambda_\mathbb{R}(f_0)\subset\lambda_\mathbb{R}(g)$ by Lemma \ref{connection between rational lamination and real lamination}. 

Similar to Lemma \ref{corresponding of external rays}, for any $v\in|T(f_0)|$ and $t\in\RZ$, 
the external ray $R_{\mathbf{P},v}(t)$ lands at $\psi_{g,v}(x)$, where $x$ is the landing point of $R_g(\theta_v^-(t))$. 
Note that $\lambda_\mathbb{R}(f_0)\subset\lambda_\mathbb{R}(g)$ and $\lambda_\mathbb{R}(f_0)\subset\lambda_\mathbb{R}(\tilde g)$. 
Repeating the discussion in the proof of Lemma \ref{chi is injective}, we see that $\lambda_\mathbb{R}(g) = \lambda_\mathbb{R}(\tilde g)$. 
By Lemma \ref{continuous extension on the closure of infinity basin}, 
the map $B_{\tilde g}^{-1}\circ B_g:U_{g,\infty}\rightarrow U_{\tilde g,\infty}$ extends to 
a homeomorphism $\varphi:\overline{U_{g,\infty}}\rightarrow\overline{U_{\tilde g,\infty}}$. 
Furthermore, we extend $\varphi$ to $\Phi:\mathbb{C}\rightarrow\mathbb{C}$ by 
$$\Phi(z) = \begin{cases}
\varphi(z), &z\in \overline{U_{g,\infty}}, \\
\psi_{\tilde g,v}^{-1}\circ\psi_{g,v}(z), &\text{$z\in K_{g,v}$ for some $v\in|T^\infty(f_0)|$.}
\end{cases}$$
We remark that the notations $K_{g,v}$ and $\psi_{g,v}$ are still valid for $v\in|T^\infty(f_0)|$. 
Then $\Phi$ is well-defined and $\Phi\circ g = \tilde g\circ\Phi$. 
Since $J(g)$ and $J(\tilde g)$ are conformally removable, the map $\Phi$ is a conformal conjugacy between $g$ and $\tilde g$. 
It follows that $\Phi=\id$. 
Therefore $(\mathbf{c}(g), g(0)) =  (\mathbf{c}(\tilde g) ,\tilde g(0))$, i.e. $g=\tilde g \in\Hbar$. 
This proves $\mathcal{N}^*\subset \Hbar$. 
\end{proof}

\subsection{Isolated zero}
\begin{lemma} 
\label{alg-uni}
Let $\Omega\subset \mathbb{C}^n$ be a domain, and let $F=(F_1, \dots, F_n): \Omega\rightarrow 
\mathbb{C}^n$ be a holomorphic map. Assume that $F$ has an isolated zero at $z_0\in \Omega$. Then 
there is a univalent map $$\phi: \mathbb{D}\rightarrow\{z\in \Omega{;~} F_2(z)=\cdots=F_n(z)=0\}$$ 
with $\phi(0)=z_0$.
\end{lemma}

For more details of isolated zeros, see \cite[Appendix B]{Mil-Sing}. 
The univalent map $\phi$ may not be locally surjective, and its image $\phi(\mathbb{D})$ may not 
be smooth at $z_0$; see the following examples. 

\begin{example}
Consider $F=(F_1, F_2): \mathbb{C}^2 \rightarrow \mathbb{C}^2$ defined by $F_1(z_1, z_2) = z_1 + z_2$ 
and $F_2(z_1, z_2) = z_1 z_2$. Clearly, $(z_1,z_2) = (0,0)$ is an isolated zero of $F$. Note that 
$F_2^{-1}(0)=\{(z_1, z_2){;~} z_1 z_2=0\}=(\{0\}\times \mathbb{C})\cup (\mathbb{C}\times \{0\})$, 
which can be thought as the gluing of two copies of $\mathbb{C}$ at $(0,0)$. A holomorphic map $\phi: 
\mathbb{D}\rightarrow F_2^{-1}(0)$ with $\phi(0)=(0,0)$ necessarily takes values in one copy of 
$\mathbb{C}$, hence not locally surjective. 
\end{example}

\begin{example}
Consider $F=(F_1, F_2): \mathbb{C}^2\rightarrow \mathbb{C}^2$ defined by $F_1(z_1, z_2) = z_1$ and 
$F_2(z_1, z_2) = z_1^3 - z_2^2$. Then $F_2^{-1}(0)=\{(z_1, z_2){;~} z_1^3 - z_2^2 = 0\}$ has a cusp 
at $(0,0)$. We can take $\phi: \mathbb{C}\rightarrow F_2^{-1}(0)$, $\zeta\mapsto (\zeta^2, \zeta^3)$. 
\end{example}

\begin{proof}
[Sketch of proof of Lemma \ref{alg-uni}]
Let $X = \bigcap_{2\leq k\leq n}F_k^{-1}(0)$. Since $z_0$ is an isolated zero of $F$, 
we see that $X$ is a complex analytic set with $\Cdim
(X)=1$. Let $(Y,z_0)$ be an irreducible component of the germ $(X,z_0)$ of complex analytic set. By 
Puiseux's theorem \cite[Chapter 1, \S6]{Chirka}, we see that  $(Y,z_0)$ is the univalent image of 
$(\mathbb{D},0)$. 
\end{proof}

\begin{lemma}
\label{local diffeomorphism}
Let $\Omega\subset \mathbb{C}^n$ be a domain, and let $F = (F_1, \dots, F_n): \Omega \rightarrow 
\mathbb{C}^n$ be a holomorphic map with an isolated zero at $z_0\in \Omega$. For each $1 \leq k \leq 
n$, let $X_k = \bigcap_{j\neq k} F_j^{-1}(0)$, and choose a univalent map $\phi_k:\mathbb{D} 
\rightarrow X_k$ so that $\phi_k(0)=z_0$ (see Lemma \ref{alg-uni}). Then $F$ is a local diffeomorphism 
near $z_0$ if and only if $(F_k\circ\phi_k)'(0)\neq 0$ for any $1\leq k\leq n$. 
\end{lemma}

\begin{proof}
$\Rightarrow$. 
Shrinking $\Omega$ if necessary, we can assume $F:\Omega\rightarrow F(\Omega)$ is 
biholomorphic. Then for each $1\leq k\leq n$, the set $X_k$ is a smooth complex curve through $z_0$. It 
follows that $\phi_k:\mathbb{D}\rightarrow \phi_k(\mathbb{D})$ and $F_k: X_k\rightarrow F_k(X_k)$ are 
biholomorphic. So $(F_k\circ\phi_k)'(0)\neq 0$. 

$\Leftarrow$. 
For $1\leq j, k\leq n$, we have $(F_j\circ\phi_k)'(0) = \sum_{m=1}^n\frac{\partial F_j}{\partial z_m}(z_0) \phi_{k,m}'(0)$, 
where $\phi_k = (\phi_{k,1}, \dots, \phi_{k,n})$.  
If $j\neq k$, then $F_j\circ\phi_k\equiv 0$, hence $(F_j\circ\phi_k)'(0)=0$. 
It follows that $$
\bigg(\frac{\partial F_j}{\partial z_m}(z_0)\bigg)_{j,m}
\bigg( \phi_{k,m}'(0) \bigg)_{m,k} = 
{\left(\begin{array}{ccc}
(F_1\circ\phi_1)'(0) & &\\
&\ddots&\\
&&(F_n\circ\phi_n)'(0)
\end{array}\right)}.
$$
By the assumption $(F_k\circ\phi_k)'(0)\neq 0$ for any $1\leq k\leq n$, 
the Jacobian matrix ${\left(\frac{\partial F_j}{\partial z_m}(z_0)\right)}_{j,m}$ is invertible. 
Therefore $F$ is a local diffeomorphism near $z_0$. 
\end{proof}

\subsection{Distortion theorems of holomorphic motions}
For a path $\gamma:[0,1]\rightarrow\mathbb{C}$ and a point $z\notin\gamma([0,1])$, let $\Delta(\gamma, 
z)$ denote the  angular increment along the path $\gamma$ around $z$. For example, if we let $\gamma(t) 
= e^{i\pi t}, t\in[0,1]$, then $\Delta(\gamma,-i) = \pi/2$ and $\Delta(\gamma,\sqrt{3}i) = -\pi/3$. 

\begin{lemma} 
[Argument distortion]
\label{argument-distortion} 
Let $f:\mathbb{C}\rightarrow\mathbb{C}$ be a $K$-quasiconformal mapping. Suppose $z_0,z_1,z_2$ are 
successively three points on a straight line. Then there is a constant $C=C(K)>0$ such that $$|\Delta
(f\circ\gamma,f(z_0))|\leq C\log\frac{z_2-z_0}{z_1-z_0},$$ where $\gamma(t) = (1-t)z_1+t z_2, t\in 
[0,1]$ is the path from $z_1$ to $z_2$. 
\footnote{
For example, consider $f(z)=e^{i\log|z|}z = z^{1+i/2}\bar{z}^{i/2}$.
}
\end{lemma}

See van Strien \cite[Lemma 4.1]{vStrien} for a proof. The following two lemmas are gathered from 
\cite{vStrien} as well, which will be used in the proof of Proposition \ref{transversality-prop}. 

\begin{lemma}
\label{distor-holo-motion}
Let $h:\mathbb{D}^n\times \mathbb{C}\rightarrow \mathbb C$ be a holomorphic motion with basepoint 
$0\in \mathbb{D}^n$, and let $z_0\in\mathbb{C}$. Define $u_z(\lambda) = h(\lambda,z)-h(\lambda,z_0)$. 
Given $\delta\in(0,1)$. Then there are constants $r\in(0,1)$ and $C>0$ such that $$\left| 
\frac{ \partial u_z(\lambda) / \partial \lambda_j}{u_z(\lambda)}\right|\leq C\log\frac{1}{|z-z_0|}$$
for any $1\leq j\leq n$, $z\in \mathbb{D}^*(z_0,r)=\{\zeta{;~} 0<|\zeta-z_0|<r\}$ and $\lambda\in\mathbb{D}^n(\delta)$. 
\footnote{
For example, consider $n=1$, $h(\lambda,z) = z|z|^{\frac{2\lambda}{1-\lambda}} 
= ze^{\frac{2\lambda}{1-\lambda}\log|z|}$ and $z_0=0$. 
Then $\left|\frac{u_z'(\lambda)}{u_z(\lambda)}\right|= \frac{2}{|1-\lambda|^2}\log\frac{1}{|z|}$ for any $0< |z|<1$ and $|\lambda|<1$. 
}
\end{lemma}

\begin{proof}
Let $K = \frac{1+\delta}{1-\delta}$. 
By S{\l}odkowski's theorem (Lemma \ref{dilatation of holomorphic motion}), for all $\lambda\in \mathbb{D}^n(\delta)$, the map $h(\lambda, \cdot)$ is a $K$-quasiconformal mapping. 
By Lemma \ref{Hol-contin of qcm}, there are constants $C_1,C_2$ depending on $K$ such that  
\begin{equation}
\label{estimate of u}
C_1|z-z_0|^{K}\leq |u_z(\lambda)| \leq C_2|z-z_0|^{1/K} 
\end{equation}
for any $z\in \mathbb{D}(z_0,1)$ and $\lambda\in \mathbb{D}^n(\delta)$. 
Here we remark that although $C_1,C_2$ also depend on the normalization of $h(\lambda, \cdot)$, they can be chosen uniformly for all $\lambda\in \mathbb{D}^n(\delta)$.

When $z\in\mathbb{C}\setminus\{z_0\}$, the function $u_z:\mathbb{D}^n\rightarrow\mathbb{C}\setminus\{0\}$ has a holomorphic logarithm $$v_z(\lambda)=\log(u_z(\lambda))$$ 
with $\operatorname{Im}(v_z(0)) = \arg(z-z_0)\in [0, 2\pi)$, since $\mathbb{D}^n$ is simply connected. 
Now let $z\in \mathbb{D}^*(z_0,1)$ and $\lambda\in \mathbb{D}^n(\delta)$. 
By (\ref{estimate of u}), we have 
\begin{equation}
\label{estimate of Re v}
\log(C_1) + K \log|z-z_0|\leq  \operatorname{Re}(v_z(\lambda))\leq \log(C_2)+\frac{1}{K}\log|z-z_0|.
\end{equation}
Let $z' = z_0+\frac{z-z_0}{|z-z_0|}$. 
Then $z_0,z,z'$ are successively three points on a straight line. 
Since the map $w\mapsto u_w(\lambda)$ is $K$-quasiconformal, 
applying Lemma \ref{argument-distortion} to the triple $(z_0,z,z')$, 
there is a constant $C_3$ depending on $K$ such that 
$|\Delta(u_{\gamma(\cdot)}(\lambda),0)| \leq C_3 \log\frac{1}{|z-z_0|}$, 
where $\gamma(t)=(1-t)z+tz',t\in[0,1]$ is the path from $z$ to $z'$. 
Note that $|z'-z_0|=1$ and $\lambda\in \mathbb{D}^n(\delta)\Subset\mathbb{D}^n$. By the continuity of $h$, there is a constant $C_4>0$ independent of $z'$ and $\lambda$ so that 
$|\operatorname{Im}(v_{z'}(\lambda))| \leq C_4$. 
It follows that 
\begin{equation}
\label{estimate of Im v}
|\operatorname{Im}(v_z(\lambda))|
=|\operatorname{Im}(v_{z'}(\lambda)) - \Delta(u_{\gamma(\cdot)}(\lambda),0)|
\leq C_4 + C_3\log\frac{1}{|z-z_0|}. 
\end{equation}
Combining (\ref{estimate of Re v}) with (\ref{estimate of Im v}), there are constants $r\in(0,1)$ 
and $C_5>0$ so that $|v_z(\lambda)| \leq C_5\log\frac{1}{|z-z_0|}$ for any $z\in\mathbb{D}^*(z_0,r)$ 
and $\lambda\in \mathbb{D}^n(\delta)$. Similarly, we can decrease $r$ if necessary and choose a constant 
$C_6>0$ so that $|v_z(\lambda)| \leq C_6\log\frac{1}{|z-z_0|}$ for any $z\in\mathbb{D}^*(z_0,r)$ and 
$\lambda\in \mathbb{D}^n{\left(\frac{1+\delta}{2}\right)}$. By Cauchy's integral formula, for any 
$z\in\mathbb{D}^*(z_0,r)$ and $\lambda\in \mathbb{D}^n(\delta)$, we have 
\begin{align*}
\left|\frac{ \partial u_z / \partial \lambda_j }{u_z}\right|
&= \left| \frac{\partial v_z}{\partial \lambda_j}\right|= \left| \frac{1}{(2\pi i)^n} 
\int_{|\zeta_n-\lambda_n|=\frac{1-\delta}{2}} \cdots \int_{|\zeta_1-\lambda_1|=\frac{1-\delta}{2}}
 \frac{v_z(\zeta) d\zeta_1 \cdots d\zeta_n}{(\zeta_j-\lambda_j) \prod_{k=1}^n (\zeta_k-\lambda_k)}  \right| \\
&\leq \frac{2}{1-\delta}C_6\log\frac{1}{|z-z_0|}. 
\end{align*}
% One can choose $w\in X(\varepsilon)$ arbitrarily close to  $f^{r_v}(c_{v,k})$ to make $\operatorname{Re}(v_x(0))\rightarrow -\infty$. 
Taking $C=\frac{2}{1-\delta}C_6$ completes the proof. 
\end{proof}

% For Lemma \ref{distor-holo-motion}, the case $n=1$ will be used in the proof of Lemma \ref{simple zeros}; 
% the case $n\geq1$ will be used in the proof of Lemma \ref{openness of transversality}. 

\begin{lemma}
\label{simple zeros}
For a holomorphic motion $h:\mathbb{D}\times \mathbb{C}\rightarrow \mathbb C$ with basepoint $0\in \mathbb D$ 
and a holomorphic function $a:\mathbb{D} \rightarrow \mathbb C$, 
we define $\psi_z(\lambda) = a(\lambda)-h(\lambda, z)$. 
Suppose $\lambda = 0$ is an isolated zero of $\psi_{z_0}$ with multiplicity $m\geq1$ for some $z_0\in\mathbb{C}$. 
Then for any $\delta>0$ small enough, there is a constant $r=r(\delta)>0$ such that $\psi_z$ has $m$ simple zeros in $\mathbb{D}^*(\delta)$ 
and no other zero in $\mathbb{D}(\delta)$ for any $z\in \mathbb{D}^*(z_0,r)$. 
\end{lemma}

\begin{proof}
Choose $\delta_0\in(0,1)$ so that $\psi_{z_0}(\lambda)\neq 0$ for any $0<|\lambda|<\delta_0$. 
Let $K=\frac{1+\delta_0}{1-\delta_0}$. 
Define $$u_z(\lambda) = h(\lambda,z)-h(\lambda,z_0).$$
By the same argument as (\ref{estimate of u}), there is a constant $C_1$ depending on $K$ such that 
\begin{equation}
\label{upper bound of uzlambda}
|u_z(\lambda)| \leq C_1|z-z_0|^{1/K} 
\end{equation}
for any $z\in \mathbb{D}(z_0,1)$ and $\lambda\in \mathbb{D}(\delta_0)$. 
% Here we remark that although $C_1$ also depends on the normalization of $h(\lambda, \cdot)$, it can be chosen uniformly for all $\lambda\in \mathbb{D}(\delta_0)$. 
Let $\delta\in(0,\delta_0)$. 
Choose $r = r(\delta)\in(0,1)$ so that $C_1 r^{1/K}<\min_{|\xi|=\delta}|\psi_{z_0}(\xi)|$. 
When $z\in \mathbb{D}(z_0,r)$ and $|\lambda|=\delta$, we have 
\begin{align*} 
|\psi_z(\lambda)- \psi_{z_0}(\lambda)|
=    |u_z(\lambda)|
\leq C_1|z-z_0|^{1/K}
<    C_1 r^{1/K} 
<    \min_{|\xi|=\delta}|\psi_{z_0}(\xi)|
\leq |\psi_{z_0}(\lambda)|.
\end{align*}
By Rouch\'e's theorem, the function $\psi_z$ has $m$ zeros (counted with multiplicity) in $\mathbb D(\delta)$ for any $z\in \mathbb{D}(z_0,r)$. 
If $z\neq z_0$, then $\psi_z(0) = z_0 - z\neq0$, so $\lambda=0$ is not a zero of $\psi_z$. 

\vspace{6 pt}
{\bf Claim.}
{\it When $\delta$ is small enough, all zeros of $\psi_z$ in $\mathbb{D}^*(\delta)$ are simple for any $z\in \mathbb{D}^*(z_0,r)$.}
\vspace{6 pt}

In the following, we always assume $\delta$ is small enough, and let $C_2,C_3,\dots$ denote suitable positive constants. 
Write $\psi_{z_0}(\lambda) = \lambda^m \varphi(\lambda)$ for some holomorphic function $\varphi$. 
It follows that $\psi_z(\lambda) = -u_z(\lambda) + \lambda^m\varphi(\lambda)$. 
If the claim is not true for some $z\in \mathbb{D}^*(z_0,r)$, there is  $\lambda_0\in\mathbb{D}^*(\delta)$ so that $\psi_z(\lambda_0)=0$ and $\psi'_z(\lambda_0)=0$. 
That is, 
\begin{align*}
% \label{multiple root}
&u_z(\lambda_0)=\lambda_0^m \varphi(\lambda_0),\\ &u'_z(\lambda_0)=m\lambda_0^{m-1} \varphi(\lambda_0)+\lambda_0^m\varphi'(\lambda_0). 
\end{align*}
% Since $z\neq z_0$, we have $u_z(\lambda)\neq0$ for any $\lambda\in\mathbb{D}$, so $u_z(\lambda_0)\neq0$ and $\lambda_0\neq0$. 
It follows that 
\begin{equation*}
\label{1 over lambda0}
\left|\frac{u'_z(\lambda_0)}{u_z(\lambda_0)}\right|
=\left|\frac{m}{\lambda_0}+\frac{\varphi'(\lambda_0)}{\varphi(\lambda_0)}\right|
\geq \frac{C_2}{|\lambda_0|}. 
\end{equation*}
On the other hand, by Lemma \ref{distor-holo-motion} and  (\ref{upper bound of uzlambda}), we have 
\begin{align*}
\left|\frac{u_z'(\lambda_0)}{u_z(\lambda_0)}\right|
&\leq C_3\log\frac{1}{|z-z_0|}
\leq C_3 K \log\frac{C_1}{|u_z(\lambda_0)|}
= C_3 K \log\frac{C_1}{|\lambda_0^m \varphi(\lambda_0)|}
\\
&=  C_3 K{\left(m\log\frac{1}{|\lambda_0|} + \log\frac{C_1}{|\varphi(\lambda_0)|}\right)}
\leq C_4 \log\frac{1}{|\lambda_0|}.
\end{align*}
Therefore $\delta\log(\delta)\leq |\lambda_0|\log|\lambda_0|\leq -C_2/C_4<0$, where $\delta\in(0,1/e]$. 
This is impossible because $\delta\log(\delta)\rightarrow 0$ as $\delta\rightarrow0$, and the constants $C_k$'s do not depend on $\delta$. 
This proves the claim. 
\end{proof}

%\begin{lemma}
%\label{openness of transversality}
%
%\end{lemma}

\subsection{Proof of Proposition \ref{transversality-prop}}
\begin{proof}
[Proof of Proposition \ref{transversality-prop}]
For $g\in\mathcal{F}$ and $(v,k)\in I$, let $w_{v,k}(g) = g^{r_v}(c_{v,k}(g))$ be the postcritical value of $c_{v,k}(g)$. 
Then $G_{v,k}(g) = w_{v,k}(g) - h(g,w_{v,k}(f))$ if $g\in\mathcal{N}$, and 
$\widetilde W_{v,k}(g) = B_{g,\sigma^{r_v}(v)}(w_{v,k}(g))$ if $g\in \Hbar$. 
See (\ref{crit-val-H}) for $\widetilde W = (\widetilde W_{v,k})_{(v,k)\in I}$. 
Recall that the map $W = (W_{v,k})_{(v,k)\in I}$ defined by (\ref{crit-val-Tf0}) and the straightening map $\chi:\Hbar\rightarrow 
\modelf$ defined by (\ref{straightening map}) satisfy $\widetilde W = W\circ\chi$. 

(1) 
First, we show that $f$ is an isolated zero of $G$. 
If not, then there is a non-repeating sequence $\{g_n\}_{n\geq1}\subset G^{-1}(0)$ such that $g_n\rightarrow f$ as $n\rightarrow\infty$. 
Note that $w_{v,k}(g_n) = h(g_n, w_{v,k}(f)) \in\overline{U_{g_n,\sigma^{r_v}(v)}}$ for any $(v,k)\in I$. 
Since $f$ has no external free critical relation, when $n$ is large enough, we have $g_n\in \Hbar$ by Proposition \ref{capture-separation}, hence $\widetilde W(g_n)=\widetilde W(f)$. 
% By Lemma \ref{tilde-W-local-homeo}, 
Because $W$ is a proper holomorphic map (see Lemma \ref{proper-holo-W}) and $\widetilde W = W\circ\chi$, the set $\widetilde W^{-1}(\widetilde W(f))$ is finite. 
It follows that $g_n=f$ for $n$ large enough. 
This is a contradiction. 

By Lemma \ref{alg-uni}, for each $(v,k)\in I$, there is a univalent map  
$$\phi_{v,k}:\mathbb D\rightarrow\bigcap_{(v',k')\neq(v,k)}G_{v',k'}^{-1}(0)$$
with $\phi_{v,k}(0)= f$. 
By Lemma \ref{local diffeomorphism}, to prove $G$ is a local diffeomorphism near $f$, 
we just need to show $(G_{v,k}\circ \phi_{v,k})'(0)\neq 0$ for any $(v,k)\in I$. 
% Assume $(G_{v,k}\circ \phi_{v,k})'(0) = 0$ for some $(v,k)\in I$; we will find a contradiction. 

Fix $(v,k)\in I$. 
Let $g_\lambda = \phi_{v,k}(\lambda)$. 
Clearly $g_0 = f$. 
By S{\l}odkowski's theorem, the map $h(g_{\lambda}, z)$ can be extended to a holomorphic motion defined on $\mathbb{D}\times\mathbb{C}$. 
Define $$\psi_z(\lambda) = w_{v,k}(g_{\lambda})-h(g_{\lambda}, z). $$
Clearly $\lambda=0$ is an isolated zero of $$\psi_{w_{v,k}(f)}(\lambda) = w_{v,k}(g_\lambda)-h(g_{\lambda}, w_{v,k}(f)) 
= G_{v,k}\circ \phi_{v,k}(\lambda).$$ 
Let $m\geq 1$ denote its multiplicity. 
By Lemma \ref{simple zeros}, for any $\varepsilon>0$ small enough, there is a constant $r=r(\varepsilon)>0$ such that $\psi_z$ has 
$m$ simple zeros $\lambda_1,\dots, \lambda_m\in \mathbb{D}^*(\varepsilon)$ for any $z\in \mathbb{D}^*(w_{v,k}(f), r)\cap \overline{U_{f,\sigma^{r_v}(v)}}$. 
Then for $1\leq j\leq m$, we have 
$$\begin{cases}
w_{v',k'}(g_{\lambda_j}) = h(g_{\lambda_j}, w_{v',k'}(f)) \in\overline{U_{g_{\lambda_j},\sigma^{r_{v'}}(v')}}, ~ (v',k')\neq (v,k), \\
w_{v,k}(g_{\lambda_j})   = h(g_{\lambda_j}, z)            \in\overline{U_{g_{\lambda_j},\sigma^{r_v}(v)}}. 
\end{cases}$$
Since $f$ has no external free critical relation, when $\varepsilon$ is small enough, we have $g_{\lambda_j} \in\Hbar$ by Proposition \ref{capture-separation}. 
It follows from 
$$\begin{cases}
\widetilde W_{v',k'}(g_{\lambda_j}) = \widetilde W_{v',k'}(f), ~ (v',k')\neq (v,k), \\
\widetilde W_{v,k}  (g_{\lambda_j}) = B_{f,\sigma^{r_v}(v)}(z)
\end{cases}$$
that $\widetilde W(g_{\lambda_1})=\cdots=\widetilde W(g_{\lambda_m})$.  
Since $f$ has no internal free critical relation, 
the map $\widetilde W$ is locally one-to-one near $f$ by Lemma \ref{tilde-W-local-homeo}. 
Since $\lambda_1,\dots, \lambda_m\in \mathbb{D}^*(\varepsilon)$, when $\varepsilon$ is small enough, 
we have $g_{\lambda_1} = \cdots=g_{\lambda_m}$, 
i.e. $\phi_{v,k}(\lambda_1) = \cdots = \phi_{v,k}(\lambda_m)$. 
Since $\phi_{v,k}$ is univalent, we have $m=1$, i.e. $(G_{v,k}\circ \phi_{v,k})'(0)\neq 0$. 

(2) 
Because $W$ is a proper holomorphic map (see Lemma \ref{proper-holo-W}), the point $\chi(f)$ is an isolated zero of $W-W(\chi(f))$ (i.e. $W-\widetilde W(f)$). 
By Lemma \ref{alg-uni}, for each $(v,k)\in I$, there is a univalent map  
$$\Phi_{v,k}:\mathbb D\rightarrow\bigcap_{(v',k')\neq(v,k)}W_{v',k'}^{-1}( \widetilde W_{v',k'}(f) )$$
with $\Phi_{v,k}(0)= \chi(f)$. 
Clearly $W_{v,k}\circ \Phi_{v,k}(\zeta)\not\equiv \widetilde W_{v,k}(f)$ for $\zeta\in\mathbb{D}$. 
Since $f$ has an internal free critical relation, the generalized polynomial $\chi(f)$ has a free critical relation. 
By Proposition \ref{prop-JW}, the proper holomorphic map $W$ is not a local diffeomorphism near $\chi(f)$. 
By Lemma \ref{local diffeomorphism}, we can choose $(v,k)\in I$ so that 
\begin{equation}
\label{zero-deriva-vk}
(W_{v,k}\circ \Phi_{v,k})'(0)=0. 
\end{equation}

Assume $G$ is a local diffeomorphism near $f$; we will find a contradiction. 
% We can define $\phi_{v,k}(\lambda)$ and $\psi_z(\lambda)$ as (\ref{phi-vk}) and (\ref{psi-z}) respectively. 
By the assumption on $G$, shrinking $\mathcal{N}$ if necessary, we can choose a biholomorphic map $$\phi_{v,k}:\mathbb D\rightarrow \bigcap_{(v',k')\neq(v,k)}G_{v',k'}^{-1}(0)$$ so that $\phi_{v,k}(0)= f$. 
Let $g_\lambda = \phi_{v,k}(\lambda)$ and define $$\psi_z(\lambda) = w_{v,k}(g_{\lambda})-h(g_{\lambda}, z)$$ as before. 
By the assumption on $G$ again, we have $(G_{v,k}\circ \phi_{v,k})'(0)\neq 0$ (see Lemma \ref{local diffeomorphism}). 
It follows that $\lambda=0$ is a simple zero of $$\psi_{w_{v,k}(f)}(\lambda) = w_{v,k}(g_\lambda)-h(g_{\lambda}, w_{v,k}(f)) 
= G_{v,k}\circ \phi_{v,k}(\lambda).$$ 
By Lemma \ref{simple zeros}, to deduce a contradiction, we just need to show for any given $\varepsilon\in(0,1)$, 
there is a constant $r>0$ such that $\psi_z$ has two distinct zeros in $\mathbb{D}(\varepsilon)$ for any $z\in \mathbb{D}^*(w_{v,k}(f), r)\cap \overline{U_{f,\sigma^{r_v}(v)}}$.

For $\zeta\in \Phi_{v,k}^{-1}(\chi(\mathcal{N}\cap\Hbar))$, 
it follows from $$\widetilde W_{v',k'}( \chi^{-1}( \Phi_{v,k}(\zeta)) )= W_{v',k'}( \Phi_{v,k}(\zeta)) = \widetilde W_{v',k'}(f),~(v',k')\neq (v,k)$$
that $\chi^{-1}( \Phi_{v,k}(\zeta))\in \bigcap_{(v',k')\neq(v,k)}G_{v',k'}^{-1}(0) = \phi_{v,k}(\mathbb{D})$. 
This shows $$\chi^{-1}\Big( \Phi_{v,k}(\mathbb{D})\cap\chi(\mathcal{N}\cap\Hbar) \Big) \subset \phi_{v,k}(\mathbb{D}).$$
Fix $\varepsilon\in(0,1)$. 
By the continuity of $\Phi_{v,k}$ and $\chi^{-1}$, there is an $\varepsilon'\in(0,1)$ such that 
$$\chi^{-1}\Big( \Phi_{v,k}(\mathbb{D}(\varepsilon'))\cap\chi(\Hbar) \Big) \subset \phi_{v,k}(\mathbb{D}(\varepsilon)).$$
According to $W_{v,k}\circ \Phi_{v,k}\not\equiv \widetilde W_{v,k}(f)$ and $(W_{v,k}\circ \Phi_{v,k})'(0)=0$ (see (\ref{zero-deriva-vk})), there is a constant $r>0$ 
such that the equation $$W_{v,k}\circ \Phi_{v,k}(\zeta) = B_{f,\sigma^{r_v}(v)}(z)$$ of $\zeta$ 
has two distinct solutions $\zeta_1,\zeta_2\in\mathbb{D}(\varepsilon')$ for any $z\in \mathbb{D}^*(w_{v,k}(f), r)\cap \overline{U_{f,\sigma^{r_v}(v)}}$. 
For $j=1,2$, it follows from $$\begin{cases}
W_{v',k'}( \Phi_{v,k}(\zeta_j)) =  \widetilde W_{v',k'}(f) \in \overline{\mathbb{D}}, ~ (v',k')\neq (v,k), \\
W_{v,k}( \Phi_{v,k}(\zeta_j))   = B_{f,\sigma^{r_v}(v)}(z) \in \overline{\mathbb{D}}
\end{cases}$$
that $\Phi_{v,k}(\zeta_j)\in\modelf = \chi(\Hbar)$. 
Now let $\lambda_j = \phi_{v,k}^{-1}(\chi^{-1}(\Phi_{v,k}(\zeta_j)))\in \mathbb{D}(\varepsilon)$, i.e. $g_{\lambda_j} = \chi^{-1}(\Phi_{v,k}(\zeta_j))$. 
Then 
%\begin{align*}
%w_{v,k}(g_{\lambda_j}) 
%&= B_{g_{\lambda_j},\sigma^{r_v}(v)}^{-1}( \widetilde W_{v,k}( \chi^{-1}(\Phi_{v,k}(\zeta_j)) ))\\
%&= B_{g_{\lambda_j},\sigma^{r_v}(v)}^{-1}( W_{v,k}(\Phi_{v,k}(\zeta_j)) )\\
%&= B_{g_{\lambda_j},\sigma^{r_v}(v)}^{-1}( B_{f,\sigma^{r_v}(v)}(z) )\\
%&= h(g_{\lambda_j},z).  
%\end{align*}
\begin{align*}
w_{v,k}(g_{\lambda_j}) 
&= B_{g_{\lambda_j},\sigma^{r_v}(v)}^{-1}( \widetilde W_{v,k}( \chi^{-1}(\Phi_{v,k}(\zeta_j)) ))
= B_{g_{\lambda_j},\sigma^{r_v}(v)}^{-1}( W_{v,k}(\Phi_{v,k}(\zeta_j)) )\\
&= B_{g_{\lambda_j},\sigma^{r_v}(v)}^{-1}( B_{f,\sigma^{r_v}(v)}(z) )
= h(g_{\lambda_j},z).  
\end{align*}
Therefore $\psi_z(\lambda_1) = \psi_z(\lambda_2) = 0$. 
Since $\Phi_{v,k}$, $\chi^{-1}$ and $\phi_{v,k}^{-1}$ are injective, we have $\lambda_1\neq \lambda_2$. 
So $\lambda_1$ and $\lambda_2$ are two distinct zeros of $\psi_z$ in $\mathbb{D}(\varepsilon)$. 
This completes the proof. 
\end{proof}

\subsection{Examples}
To illustrate these results, we present (without proofs) several examples of free critical relations. 

\begin{example}
\label{example-FreeCriticalRelation}
Consider $\mathcal{F}=\big\{f\in\widehat{\mathcal{P}}^5{;~} f(c_1(f))=c_1(f), \ f(c_2(f))=c_2(f)\big\}$. 
By Example \ref{example of F}, we have $\Sigma(\mathcal{F})=\emptyset$. 
Let $f_*\in\mathcal{F}$ satisfy $f_*^3(c_3(f_*)) = f_*^2(c_3(f_*))$ and 
$c_3(f_*)=c_4(f_*)$ as in Figure \ref{figure fstar}.  
\footnote{
$f_* = (\mathbf{c}(f_*),f_*(0)) \approx (-0.71909970+0.28778226i,~    -0.79091944-0.43455062i,~\\     0.75500957+0.07338418i,~     0.75500957+0.07338418i,~     0.44279049+0.31992373i)$.

$f_1 \approx (-0.6957+0.2889i,~    -0.7758-0.4464i,~     0.7434-0.0889i,~     0.7281+0.2463i,~\\     0.3830+0.3179i)$.

$f_2 \approx (-0.7301+0.2652i,~    -0.8101-0.4424i,~     0.9004+0.1774i,~     0.6399-0.0002i,~\\     0.4763+0.3878i)$.

$f_3 \approx (-0.7400+0.3071i,~    -0.7929-0.4103i,~     0.6140+0.1160i,~     0.9189-0.0128i,~\\     0.5016+0.2576i)$.

$f_4 \approx (-0.7244+0.2922i,~    -0.7908-0.4286i,~     0.7576+0.0682i,~     0.7576+0.0682i,~\\     0.4617+0.3021i)$.

$f_5 \approx (-0.7195+0.2823i,~    -0.7948-0.4372i,~     0.7571+0.0775i,~     0.7571+0.0775i,~\\     0.4450+0.3405i)$.

$g_1 \approx (-0.71785439+0.28759365i,~    -0.79025587-0.43535917i,~     0.75196467+0.03761302i,~\\     0.75614558+0.11015250i,~     0.43929268+0.32056723i)$.

$g_2 \approx (-0.71952246+0.28717907i,~    -0.79151609-0.43468631i,~     0.77968300+0.08691225i,~\\     0.73135555+0.06059499i,~     0.44405294+0.32160325i)$.

$g_3 \approx (-0.71942183+0.28854670i,~    -0.79069082-0.43389820i,~     0.73053354+0.08859262i,~\\     0.77957911+0.05675889i,~     0.44364010+0.31771958i)$.

$g_4 \approx (-0.71909970+0.28778226i,~    -0.79091481-0.43455287i,~     0.75877977+0.07829465i,~\\     0.75123474+0.06847596i,~     0.44279444+0.31989123i)$.

$g_5 \approx (-0.71909970+0.28778226i,~    -0.79092259-0.43454701i,~     0.75109910+0.07789205i,~\\     0.75892319+0.06887270i,~     0.44279896+0.31995299i)$.
}

\begin{figure}[ht]
\centering
\includegraphics{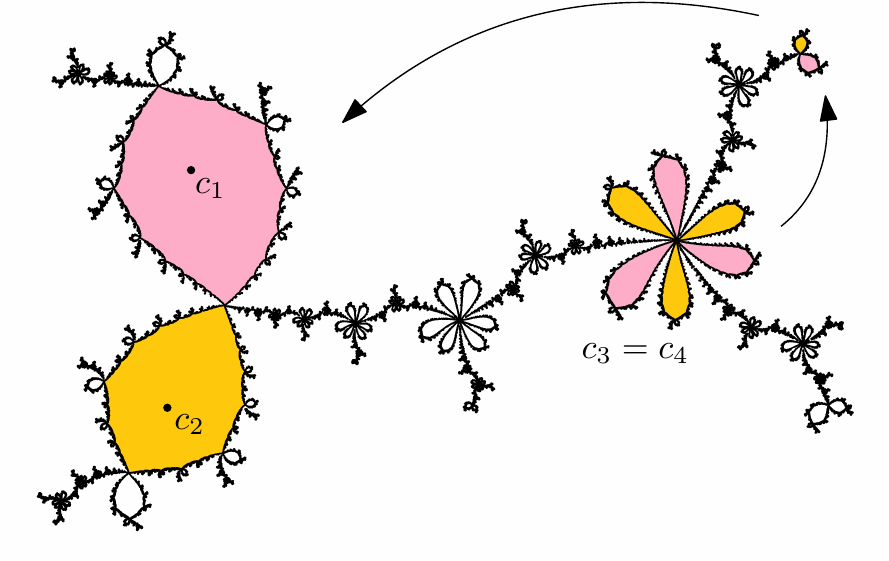}
\caption{The Julia set of $f_*$.}
\label{figure fstar}
\end{figure}

\begin{figure}[htbp]
\centering
\includegraphics{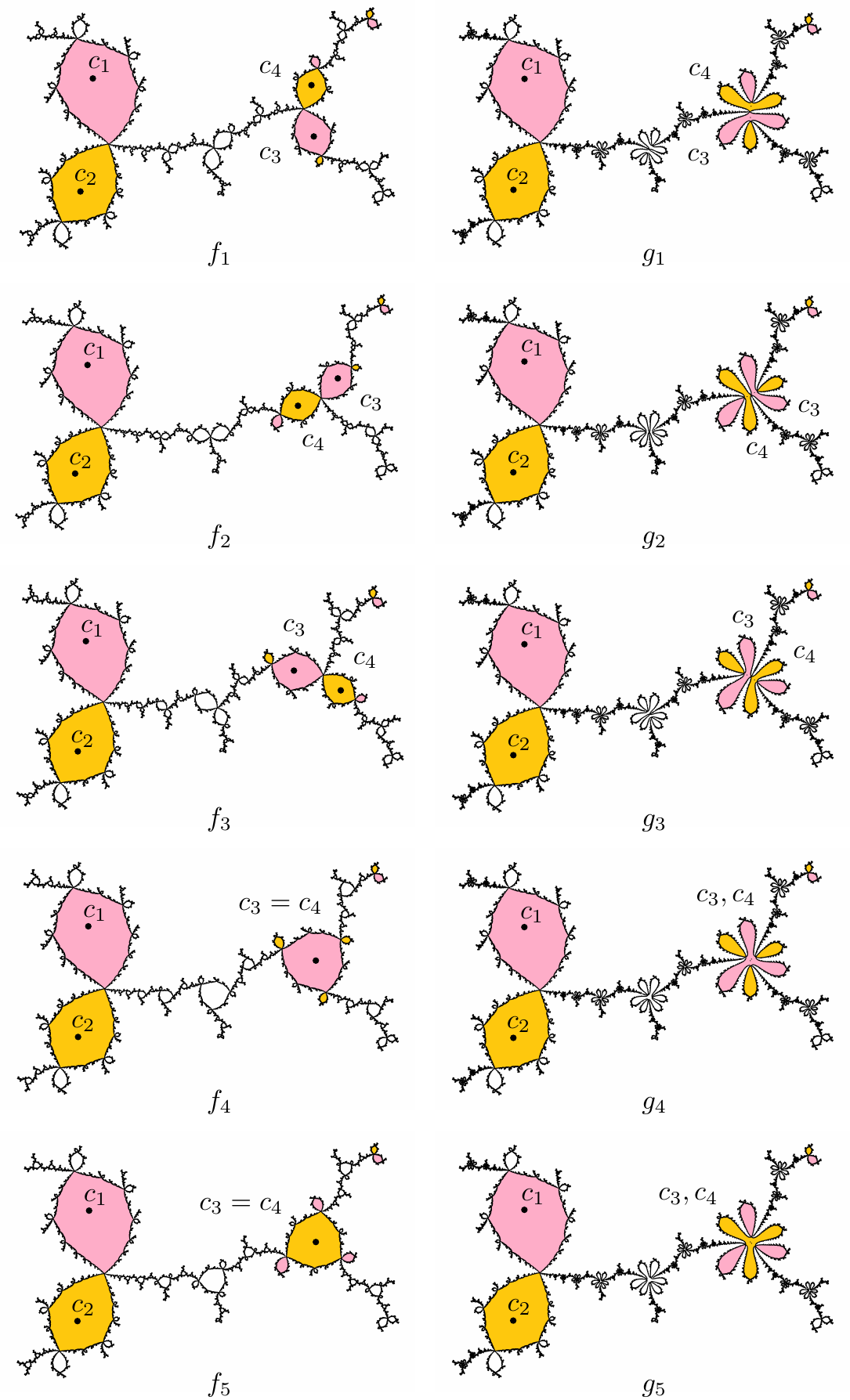}
\caption{Perturbations of $f_*$.}
\label{figure perturbations of fstar}
\end{figure}

For $k=1,2,3$, let $f_k\in\mathcal{F}$ satisfy $f_k^2(c_3(f_k))=c_1(f_k)$ and  
$f_k^2(c_4(f_k))=c_2(f_k)$ as in Figure \ref{figure perturbations of fstar}. 
For $k=4,5$, let $f_k\in\mathcal{F}$ satisfy $f_k^2(c_3(f_k))=c_{k-3}(f_k)$ and  
$c_3(f_k)=c_4(f_k)$ as in Figure \ref{figure perturbations of fstar}.  
For $1\leq k\leq 5$, let $\mathcal{H}_k\subset\mathcal{F}$ denote the capture hyperbolic component with center $f_k$. 
Then $f_*\in\partial\mathcal{H}_k$. 
Choose $g_k\in\mathcal{H}_k$ close to $f_*$ as in Figure \ref{figure perturbations of fstar}. 
For $k=1,2,3$, if we exchange the free critical points $c_3(f_k)$ and $c_4(f_k)$, we will get another capture hyperbolic component in $\mathcal{F}$, 
whose boundary contains $f_*$. 

The free critical relation $c_3(f_*)=c_4(f_*)$ of $f_*$ is external with respect to $\mathcal{H}_1$. 
For any neighborhood $\mathcal{N}\subset\mathcal{F}$ of $f_*$, 
since 
$${\left\{f\in\mathcal{N}{;~} 
f^2(c_3(f))\in \overline{U_f(c_1(f))}, \ f^2(c_4(f))\in \overline{U_f(c_2(f))}\right\}}
\cap\mathcal{H}_2\neq\emptyset,$$
we have 
$${\left\{f\in\mathcal{N}{;~} 
f^2(c_3(f))\in \overline{U_f(c_1(f))}, \ f^2(c_4(f))\in \overline{U_f(c_2(f))}\right\}}
\not\subset\overline{\mathcal{H}_1}.$$

The free critical relation $c_3(f_*)=c_4(f_*)$ of $f_*$ is internal with respect to $\mathcal{H}_4$. 
But $f_*$ has no external free critical relation with respect to $\mathcal{H}_4$. 
By Proposition \ref{capture-separation}, there is a neighborhood $\mathcal{N}\subset\mathcal{F}$ of $f_*$ so that 
$${\left\{f\in\mathcal{N}{;~} 
f^2(c_3(f)), f^2(c_4(f))\in \overline{U_f(c_1(f))}\right\}}
\subset\overline{\mathcal{H}_4}.$$
\end{example}

\begin{example}
Let \begin{align*}
\mathcal{F} &= \big\{(\mathbf{c},a)\in\mathbb{C}^4{;~} c_1+c_2+c_3=0, \ f_{\mathbf{c},a}(c_1)=c_1\big\}\\
&= \big\{(c_1,c_2,-c_1-c_2, c_1^4 - 2 c_1^3 c_2 - 2 c_1^2 c_2^2 + c_1) {;~} c_1,c_2\in\mathbb{C}  \big\}\cong \mathbb{C}^2. 
\end{align*}
Let $f = (-\sqrt[3]{2}, \sqrt[3]{2}, 0, \sqrt[3]{2})\in\mathcal{F}$. 
Then $f(z) = z^4 - \sqrt[3]{32}z^2 + \sqrt[3]{2}$, $f(\sqrt[3]{2}) = -\sqrt[3]{2}$ and $f(0) = \sqrt[3]{2}$. 
See Figure \ref{figure QuarticAirplane}. 
Define a holomorphic map $G = (G_2,G_3):\mathcal{F} \rightarrow \mathbb{C}^2$ by 
$$G_2(g)=g(c_2(g)) - c_1(g), \quad G_3(g) = g^2(c_3(g)) - c_1(g).$$
Note that each $g = (\mathbf{c},a)\in\mathcal{F}$ can be represented by $(c_1,c_2)\in\mathbb{C}^2$. 
Let $J_G$ denote the Jacobian matrix of $G$ with respect to $(c_1,c_2)$. 
A direct calculation shows 
$$\renewcommand\arraystretch{1.5}
J_G(f) = {\left(\begin{array}{cc}
\frac{\partial G_2}{\partial c_1} & \frac{\partial G_2}{\partial c_2}\\
\frac{\partial G_3}{\partial c_1} & \frac{\partial G_3}{\partial c_2}
\end{array}\right)}_{g=f}
={{\left(\begin{array}{cc}
-16 & -16\\
-16 & -16
\end{array}
\right)}}.$$
So $G$ is not a local diffeomorphism near $f$. 

\begin{figure}[ht]
\centering
\includegraphics{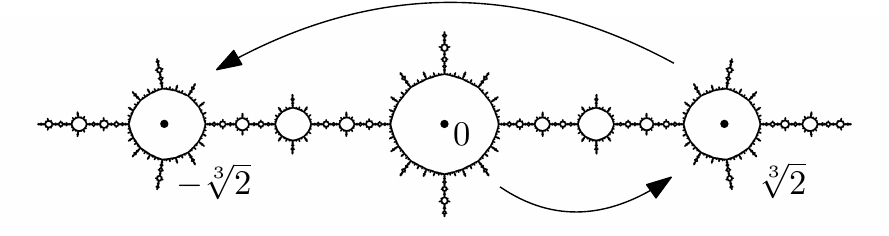}
\caption{The Julia set of $f(z) = z^4 - \sqrt[3]{32}z^2 + \sqrt[3]{2}$.}
\label{figure QuarticAirplane}
\end{figure}

Let $\mathcal{H}\subset\mathcal{F}$ denote the capture hyperbolic component with center $f$. 
Then $f(c_3(f)) = c_2(f)$ (i.e. $f(0) = \sqrt[3]{2}$) is an internal free critical relation
of $f$ with respect to $\mathcal{H}$. 
Again, we see that $G$ is not a local diffeomorphism near $f$ by Proposition \ref{transversality-prop}. 
\end{example}

\section{Continuity of Hausdorff dimension}
\label{continuity-hd}
This section is the third step to prove Theorem \ref{hd-boundary}. 
Let $\mathcal{H}\subset\mathcal{F}$ be a capture hyperbolic component with center $f_0$. 
As is known that for all $f\in \Hbar$ and all $v\in\Tfp$, the Fatou component $U_{f,v}$ is well-defined. 

\begin{proposition} 
\label{continuous-hdim} 
For any  $v\in\Tfp$, the Hausdorff dimension $\Hdim(\partial U_{f,v})$ is continuous with respect to $f\in\Hbar$.
\end{proposition}

Given $f\in\Hbar$ and $v\in\Tfp$. 
Suppose $\crit(f)\cap  \bigcup_{n\geq0}\partial U_{f,\sigma^{n}(v)}=\emptyset$. 
Then the Fatou component $U_{g,v}$ can be defined for $g$ in an open neighborhood $\mathcal{N}(f)\subset\widehat{\mathcal{P}}^d$ of $f$. 
% see Lemma \ref{expanding-nbh1}  
Furthermore, the boundary $\partial U_{g,v}$ admits a holomorphic motion for $g\in \mathcal{N}(f)$. 
According to a result of Ruelle \cite{Ruelle}, the function
$$  
\begin{cases} \mathcal{N}(f) \rightarrow  \mathbb{R}, \\
g \mapsto   \Hdim(\partial U_{g,v}) 
\end{cases}$$
is real analytic. 
Thus $\lim_{g\rightarrow f}\Hdim(\partial U_{g,v})=\Hdim(\partial U_{f,v})$. 
 
For general $f\in\Hbar$, there may be critical points in $\bigcup_{n\geq0}\partial U_{f,\sigma^{n}(v)}$ (see Example \ref{example F2}). 
The method of holomorphic motion fails. 
%However, we can introduce conformal measures for $\partial U_{f,v}$ to transfer the convergence of Hausdorff dimension into the convergence of conformal dimension. 
However, we can transfer the convergence of Hausdorff dimension ($\Hdim$) into the convergence of conformal dimension ($\dconf$) by the following identity:  

\begin{proposition}
\label{equalities for Hdim of partial-U}
Let $f$ be a semi-hyperbolic polynomial of degree $d\geq2$, 
and suppose $U$ is a bounded periodic Fatou component of $f$. 
Then 
$$\hypdim(\partial U)=\Hdim(\partial U)=\Hdim(\Con(\partial U))=\dconf(\partial U) = \dPoin(\partial U).$$
\end{proposition}

% After introducing these quantities for $\partial U$ from \S\S\ref{subsection Poincare sequence} to \ref{subsection hyperbolic set}, 
% we will prove Proposition \ref{equalities for Hdim of partial-U} in the end of \S\ref{subsection hyperbolic set}. 
The notations will be defined later in this section. 
Since $f$ is semi-hyperbolic, it is well-known that the corresponding quantities for $J(f)$ agree (see the survey \cite{U03}). 
Proposition \ref{equalities for Hdim of partial-U} might be known to experts, but we did not find the result in literature. For this reason, we will give the poof in this section.

% But unlike the Julia set $J(f)$, the boundary $\partial U$ may not be totally invariant (i.e. $f^{-1}(\partial U) = \partial U$ may not hold), which is a new difficulty. 

To prove the convergence of conformal dimension, we need the following result, 
which has an independent interest in polynomial dynamics. 

% The following result proves a uniform expansion property along external rays under perturbation. 
%It has an independent interest in polynomial dynamics, and plays an important role in the proof of 
% Proposition \ref{continuous-hdim}.

\begin{proposition}
[Uniform expansion on external rays]
\label{expanding on external rays}
Let $d\geq2$ be an integer, and let $\theta\in\QZ$ have preperiod $q\geq1$ and period $p\geq 1$ 
under $m_d: \RZ\rightarrow \RZ$, $t\mapsto dt$. 
\footnote{
The \emph{preperiod} of $\theta\in\QZ$ under $m_d$ is defined to be the minimal integer $q\geq0$ so that $d^q\theta$ is periodic under $m_d$. 
}
That is, $d^{p+q} \theta\equiv d^q\theta \modZ$.
For $f\in\mathcal{C}(\mathcal{P}^d)$,
let $$g_f=f^{-q}\circ f^p\circ f^q: \overline{R_{f}(\theta)}\rightarrow \overline{R_{f}(\theta)}$$
be the lift of $f^p:\overline{R_{f}(d^q\theta)}\rightarrow \overline{R_{f}(d^q\theta)}$.
Given $h\in\mathcal{C}(\mathcal{P}^d)$ and
suppose $R_h(\theta)$ lands at a pre-repelling point.
Then there exists a neighborhood $\mathcal{N}\subset \mathcal{P}^d$ of $h$ and a constant $n>0$ such that
$$\big|(g_f^n)'(z)\big|\geq 2$$
for every $f\in\mathcal{C}(\mathcal{P}^d)\cap\mathcal{N}$ and every $z\in\overline{R_{f}(\theta)}$.
\end{proposition}

The proof of Proposition \ref{expanding on external rays} relies on various distortion estimates for quasicircles, and will be given in \S\ref{distortion-thm}. 
After that, we will complete the proof of Proposition \ref{continuous-hdim} in \S\ref{continuity-h-d}. 
 
Now let us introduce some notations and the Koebe distortion theorem. 
Let $\dist(x,A)=\inf\{|x-y| {;~}  y\in A\}$ be the distance between a point $x$ and a subset $A$ of $\mathbb{C}$. 
Let ${\diam}(A)=\sup\{|x-y| {;~}  x,y\in A\}$ be the diameter of $A\subset\mathbb{C}$. 
For $x\in A$, recall that $\Comp_x(A)$ is the connected component of $A$ containing $x$. 
For $z\in\mathbb{C}$ and $r>0$, let $\mathbb{D}(z,r) = \{w\in\mathbb{C}{;~}|w-z|<r\}$. 
For two subsets $A$ and $B$ of $\mathbb{C}$, the notation $A\Subset B$ means $\overline{A}\subset B^\circ$.

\begin{lemma}
[Koebe distortion theorem]
\label{Koebe-distortion}
Define $\CK(\varepsilon)=(1+\varepsilon)^4/(1-\varepsilon)^4$ for $0<\varepsilon<1$. 
Then for every univalent function $g:\mathbb{D}(z,r)\rightarrow\mathbb{C}$, we have 
$$\mathbb{D}(g(z),{|g'(z)|\varepsilon r}/{\CK(\varepsilon)})\subset g(\mathbb{D}(z,\varepsilon r))
\subset \mathbb{D}(g(z),{|g'(z)|\varepsilon r}\cdot{\CK(\varepsilon)}),$$
and $|g'(z_1)|/|g'(z_2)|\leq \CK(\varepsilon)$ for every $z_1,z_2\in \mathbb{D}(z,\varepsilon r)$. 
\end{lemma}

From \S\S\ref{subsection Poincare sequence} to \ref{subsection hyperbolic set}, we fix $f$ and $U$ as in Proposition \ref{equalities for Hdim of partial-U}. 
Replacing $f$ by a suitable iterate $f^p$, we always assume $f(U)=U$.

\subsection{Poincar\'e sequence}
\label{subsection Poincare sequence}
% Given $f\in\piHbar$, $v\in\Tfp$ and let $U=U_{f,v}$. 
% Without loss of generality, assume $f(U)=U$. 

In this part, we will consider the Poincar\'e sequence for $\partial U$. 
This is an analogy of the Poincar\'e sequence for rational maps; see, e.g. \cite{Przy-Poin}. 

\begin{definition}
[Poincar\'e sequence for $\partial U$]
We call a point $x\in\partial U$ \emph{admissible} if it does not belong to $\overline{\bigcup_{j\geq1}f^j(\crit(f)\cap\partial U)}$.  
% it does not belong to $\bigcup_{j\geq1}f^j(\crit(f)\cap\partial U)$.  
For an admissible point $x\in\partial U$ and $t\geq0$, we define the \emph{Poincar\'e sequence} by 
\begin{equation}
\label{Poin-sequ}
P_n(x,t)= \sum_{y\in f^{-n}(x)\cap\partial U}|(f^n)'(y)|^{-t}, ~n\in\mathbb{N}.
\end{equation}
\end{definition}

Because $f$ is semi-hyperbolic, the set $$\overline{\bigcup_{j\geq1}f^j(\crit(f)\cap\partial U)}
= \bigcup_{c\in \crit(f)\cap\partial U}\overline{\{f^j(c){;~}j\geq1\}}$$
is totally disconnected (or empty). 
Therefore the set of all admissible points in $\partial U$ is open and dense in $\partial U$.

\begin{lemma}
\label{admis-pt-part-U}
Let $x\in \partial U$ be  admissible. 
Then there exists an $r>0$, a $\lambda>1$ and an $M>0$ 
such that for every $n\geq1$ and every $y\in f^{-n}(x)\cap \partial U$, 
we have \begin{align}
\label{disjoint with crit}
&\crit(f)\cap\overline{\Comp_y f^{-n}(\mathbb{D}(x,r))}=\emptyset, \\
&\diam(\Comp_y f^{-n}(\mathbb{D}(x,r)))\leq M\lambda^{-n}. \notag
\end{align}
\end{lemma}

\begin{proof}
% By Corollary \ref{semi-hyperbolicity}, the polynomial $f$ is semi-hyperbolic. 
By Lemma \ref{properties semi-hyper}, there exists an $r>0$,  a $\lambda>1$ and an $M>0$ such that for every $n\geq1$ 
and every connected component $V$ of $f^{-n}(\mathbb{D}(x,r))$, we have ${\diam}(V)\leq M\lambda^{-n}$. 
% By Proposition \ref{behavior of critical orbits}, the set  $\bigcup_{j\geq1}f^j(\crit(f)\cap\partial U)$ of non-admissible points in $\partial U$ is finite. 
Decrease $r$ if necessary so that 
$\overline{\mathbb{D}(x,r)} \cap \overline{\bigcup_{j\geq1}f^j(\crit(f) \cap \overline{U})} = \emptyset$. 
Choose $N>0$ so that $M\lambda^{-N}<\min_{c\in\crit(f)\setminus\partial U}\dist(c,\partial U)$. 
Then for every $n\geq N$ and every $y\in f^{-n}(x)\cap \partial U$, the property (\ref{disjoint with crit}) holds. 
Decrease $r$ if necessary so that (\ref{disjoint with crit}) holds for every $1\leq n<N$ and every $y\in f^{-n}(x)\cap \partial U$. 
This completes the proof. 
\end{proof}

\begin{lemma}
\label{tree-pressure}
For an admissible point $x\in\partial U$ and $t\geq0$, define 
\begin{equation}
\label{Cauchy-sequ}
C(x,t) = \limsup_{n\rightarrow\infty}\frac{1}{n}\log P_n(x,t).
\end{equation}
Then $C(x,t)$ has the following properties. 
\begin{enumerate}
\item There are $\Lambda>\lambda>1$ so that $\log \delta-t\log \Lambda \leq C(x,t)\leq \log \delta-t\log \lambda$, where $\delta = \deg(f|_U)$. 
\item $C(x,t)$ is a Lipschitz continuous function of $t\geq0$. 
\item $C(x,t)$ is strictly decreasing with respect to $t\geq0$. 
\item $C(x,t)$ is independent of the admissible point $x$. 
\end{enumerate}
\end{lemma}

\begin{proof}
Given an admissible point $x\in\partial U$. 
Choose $r>0$, $\lambda>1$ and $M>0$ as in Lemma \ref{admis-pt-part-U}. 
By the Koebe distortion theorem (Lemma \ref{Koebe-distortion}), for every $n\geq1$ and every $y\in f^{-n}(x)\cap\partial U$, we have $|(f^n)'(y)|\geq \frac{r}{\CK(1/2)}M^{-1} \lambda^n=:a \lambda^n$, since ${\diam}(\Comp_y f^{-n}(\mathbb{D}(x,r)))\leq M\lambda^{-n}$. 
Let $\Lambda= \max_{z\in \partial U} |f'(z)|$. 
Then the first property follows from  $\delta^n(\Lambda^n)^{-t}\leq P_n(x,t)\leq \delta^n( a \lambda^n)^{-t}$. 
For $\varepsilon>0$, we have $(\Lambda^n)^{-\varepsilon}\leq |(f^n)'(y)|^{-\varepsilon}\leq (a \lambda^n)^{-\varepsilon}$. 
It follows that $\varepsilon \log\lambda\leq C(x,t)-C(x,t+\varepsilon)\leq  \varepsilon \log\Lambda$, since $|(f^n)'(y)|^{-(t+\varepsilon)} = |(f^n)'(y)|^{-t}\cdot |(f^n)'(y)|^{-\varepsilon}$. 
% $$\frac{1}{n}\log P_n(x,t) -\frac{1}{n}\log P_n(x,t+\varepsilon)
%= \frac{1}{n}\log\frac{P_n(x,t)}{P_n(x,t+\varepsilon)}
% \in(\varepsilon(\log\lambda+(\log a)/n), \varepsilon \log\Lambda).$$
This implies the second and the third properties. 

For $x\in U\setminus\overline{\bigcup_{j\geq1}f^j(\crit(f)\cap U)}$ and $t\geq0$, we can similarly define the Poincar\'e sequence $\{P_n(x,t)\}_{n\in\mathbb{N}}$ and the corresponding limit $C(x,t)$ by taking $y\in f^{-n}(x)\cap U$. 
Now fix $t\geq0$. 
By the Koebe distortion theorem, the function $C(x,t)$ of $x$ is locally constant on $\overline{U}\setminus \overline{\bigcup_{j\geq1}f^j(\crit(f)\cap\overline{U})}$, hence  constant. 
This shows (4). 
\end{proof}

Given an admissible point $x\in\partial U$. 
According to Lemma \ref{tree-pressure}, 
we have $C(x,0)=\log(\deg(f|_U))>0$ and $\lim_{t\rightarrow\infty}C(x,t)=-\infty$. 
Moreover, the continuity and the monotonicity imply that 
there is a unique zero of $C(x,t)$, independent of the choice of $x$. 
This zero is called the \emph{Poincar\'e exponent} of $\partial U$, 
and denoted by $\dPoin(\partial U)$.

\subsection{Conformal measure}
% Given $f\in\piHbar$, $v\in\Tfp$ and let $U=U_{f,v}$. 
% Without loss of generality, assume $f(U)=U$.  

\begin{definition}
[Conformal measure for $\partial U$]
Let $t\geq0$. 
A \emph{$t$-conformal measure} for $\partial U$ is a Borel probability measure $\mu$ supported on $\partial U$ satisfying
\begin{equation}
\label{transition formula for mu} 
\mu(f(E))=\int_E|f'|^t d\mu 
\end{equation}
for every Borel set $E\subset\partial U$ such that $f|_E$ is injective. 
(If $f'(z)=0$ and $t=0$, we set $0^0=0$.)
\end{definition}
%It is clear that if $\mu$ is  $t_1$-conformal as well as $t_2$-conformal. Then $t_1=t_2$.
%Hence $t$ is uniquely determined by $\mu$. 
%We call $t$ the \emph{exponent} of $\mu$. 

The \emph{conformal dimension} of $\partial U$ is defined to be 
$$\dconf(\partial U)=\inf\{t\geq0 {;~}  \text{$\exists$ a $t$-conformal measure for $\partial U$} \}.$$

\begin{remark}
\label{Borel subsets of the plane}
Let $\mu$ be a $t$-conformal measure for $\partial U$. 
% Note that $f(\partial U)=\partial U$ and $f^{-1}(\partial U)\neq\partial U$. 
Note that $f(\partial U)=\partial U$ holds, but $f^{-1}(\partial U) = \partial U$ may not hold. 
Let $E_0 = f^{-1}(\partial U)\setminus \partial U$. 
Let $E$ be a Borel subset of $\mathbb{C}$ with $f|_E$ injective. 
Then we have 
$$\mu(f(E\setminus E_0)) = \mu(f(E\cap \partial U))=\int_{E\cap \partial U}|f'|^t d\mu=\int_{E}|f'|^t d\mu.$$
It follows that 
\begin{equation}
\label{transition inequality for mu} 
%\mu(f(E))\geq\int_E|f'|^t d\mu, 
\mu(f(E)) = \mu(f(E\setminus E_0))+\mu(f(E\cap E_0)) % =\int_{E}|f'|^t d\mu + \mu(f(E\cap E_0))
\geq \int_{E}|f'|^t d\mu. 
\end{equation}
% Moreover, if $E\cap(f^{-1}(\partial U)\setminus \partial U)=\emptyset$, then (\ref{transition formula for mu}) holds. 
\end{remark}

It is worth observing that
\begin{fact}  
\label{conf-md} 
$\dconf(\partial U)>0$ and  $\dconf(\partial U)$-conformal measure for $\partial U$ exists.
\end{fact}

To prove the fact, first note that there is no $0$-conformal measure for  $\partial U$,  since $\deg(f|_{\partial U})\geq2$.
%To see the fact, we first show the existence of a $\dconf(\partial U)$-conformal measure.
By the Patterson-Sullivan construction of a conformal measure for $\partial U$ (see the proof of Lemma \ref{Patterson-Sullivan construction}), the set $\{t>0 {;~}  \text{$\exists$ a $t$-conformal measure for $\partial U$} \}$ is nonempty. 
By the definition of $\dconf(\partial U)$, 
we can choose a $t_n$-conformal measure $\mu_n$ for $\partial U$ ($n\in\mathbb{N}$) such that $t_n\rightarrow \dconf(\partial U)$ as $n\rightarrow\infty$. 
By passing to a subsequence, we may assume $\mu_n\rightarrow\mu$ in \emph{weak topology}. 
As the usual conformal measures of a rational map, we can check that $\mu$ is a $\dconf(\partial U)$-conformal measure for $\partial U$ (see Appendix \ref{appendix weak topology on measures}). 
%Combining this with the Patterson-Sullivan construction of a conformal measure for $\partial U$ (Lemma %\ref{Patterson-Sullivan construction}), the infimum $\dconf(\partial U)$ is achieved.  
Since there is no $0$-conformal measure for  $\partial U$, 
we have $\dconf(\partial U)>0$. %This proves the fact.

\begin{lemma}
\label{Patterson-Sullivan construction}
$\dconf(\partial U)\leq \dPoin(\partial U)$. 
\end{lemma}

\begin{proof}
We will recall the Patterson-Sullivan construction of a conformal measure for $\partial U$; 
see, e.g. \cite[Theorem 4.1]{McM2000}, \cite[Theorem 3]{Sul1983}. 

Let $x$ be an admissible point in $\partial U$. 
For $t\geq0$, consider the \emph{Poincar\'e series} defined by 
$$P(x,t)=\sum_{n=0}^\infty P_n(x,t).$$ 
By Cauchy's root test, if $t<\dPoin(\partial U)$, then $P(x,t)=\infty$ since $C(x,t)>0$; if $t>\dPoin(\partial U)$, then $P(x,t)<\infty$ since $C(x,t)<0$. 
The construction breaks into two cases, depending on $P(x,\dPoin(\partial U))=\infty$ or $P(x,\dPoin(\partial U))<\infty$. 

% {\bf Case 1:} $P(x,\dPoin(\partial U))=\infty$. 
\vspace{6pt}
{\bf Case 1.}
{\it $P(x,\dPoin(\partial U))=\infty$.}
\vspace{6pt}

For $t>\dPoin(\partial U)$, define an atomic probability measure 
\begin{equation*}
\mu_t=\frac{1}{P(x,t)}  \sum_{n=0}^\infty\sum_{y\in f^{-n}(x)\cap \partial U}|(f^{n})'(y)|^{-t} \delta_y,
\end{equation*}
where $\delta_y$ is the Dirac measure at $y$. 
By checking singletons, for any Borel set $E\subset\partial U$ with $f|_E$ injective, we have 
\begin{equation}
\label{estimation 1 of mu}
\mu_t(f(E))=\begin{cases}
\int_E|f'|^t d\mu_t-|f'(x)|^t/P(x,t), & x\in E,\\
\int_E|f'|^t d\mu_t, & x\notin E. 
\end{cases}
\end{equation}

% {\bf Case 2:} $P(x,\dPoin(\partial U))<\infty$. 
\vspace{6pt}
{\bf Case 2.}
{\it $P(x,\dPoin(\partial U))<\infty$.}
\vspace{6pt}

By \cite[Lemma 3.1]{Patter}, 
there is a homeomorphism $h:[0,\infty)\rightarrow[0,\infty)$ such that 
\begin{itemize}
\item $\widetilde{P}(x,t)=\infty$ when $0\leq t\leq\dPoin(\partial U)$ and $\widetilde{P}(x,t)<\infty$ when $t>\dPoin(\partial U)$, 
where $$\widetilde{P}(x,t) = 
\sum_{n=0}^\infty\sum_{y\in f^{-n}(x)\cap \partial U}  h(|(f^{n})'(y)|) \cdot |(f^{n})'(y)|^{-t};$$ 
\item if $\varepsilon>0$ is given, there is an $a_0>0$ such that for all $a>a_0$ and $b>0$, 
\begin{equation}
\label{estimation of slowly oscillating function}
\min\{b^\varepsilon,1\}\leq\frac{h(b a)}{h(a)}\leq\max\{b^\varepsilon,1\}.
\end{equation}
\end{itemize}
For $t>\dPoin(\partial U)$, define
\begin{equation*}
\mu_t=\frac{1}{\widetilde{P}(x,t)} \sum_{n=0}^\infty\sum_{y\in f^{-n}(x)\cap \partial U} 
h(|(f^{n})'(y)|) \cdot |(f^{n})'(y)|^{-t} \delta_y. 
\end{equation*}
Let $E$ be a Borel subset of $\partial U$ with $f|_E$ injective and $\overline{E}\cap\crit(f)=\emptyset$. 
Then by (\ref{estimation of slowly oscillating function}) 
and $\widetilde{P}(x,t)\rightarrow\infty$ as $t\searrow\dPoin(\partial U)$, 
for any $\varepsilon>0$, there is a $t_0>\dPoin(\partial U)$ so that for any $t\in(\dPoin(\partial U),t_0)$, 
\begin{equation}
\label{estimation 2 of mu}
\int_E |f'|^t d\mu_t -\varepsilon \leq \mu_t(f(E)) \leq \int_E |f'|^t d\mu_t +\varepsilon. 
\end{equation}

In either case, choose a sequence $\{t_j\}_{j\geq1}$ such that 
$t_j\searrow\dPoin(\partial U)$ and $\mu_{t_j}\rightarrow\mu$ in weak topology as $j\rightarrow\infty$. 
Then $\mu$ is supported on $\partial U$. 
According to (\ref{estimation 1 of mu}) (resp. (\ref{estimation 2 of mu})), 
by an explicit computation 
% (see, e.g. Lemma \ref{weak limit of conformal measure}), 
(similar to the proof of Lemma \ref{weak limit of conformal measure}), 
the measure $\mu$ is a $\dPoin(\partial U)$-conformal measure for $\partial U$.
It follows that $\dconf(\partial U)\leq \dPoin(\partial U)$.  
\end{proof}

\subsection{Conical points}
\label{subsection conical points}
% Given $f\in\piHbar$, $v\in\Tfp$ and let $U=U_{f,v}$. 
% Without loss of generality, assume $f(U)=U$.  

In this subsection, we will introduce the conical points and give some properties. 
See, e.g. \cite[\S1]{DMNU}, \cite[\S2]{McM2000}. 

\begin{definition}
[Conical point]
For $\rho>0$, a point $x\in J(f)$ is called \emph{$\rho$-conical} (or \emph{$\rho$-radial}) for $f$ 
if there exists a sequence of integers $n_k\rightarrow\infty$ such that 
each $f^{n_k}$ is injective on 
$\Comp_x f^{-n_k}(\mathbb{D}(f^{n_k}(x),\rho))$. 
\end{definition}

We denote the set of $\rho$-conical points by $\Con(f,\rho)$, 
and the set of conical points by $\Con(f)=\bigcup_{\rho>0}\Con(f,\rho)$. 
Let $\Con(\partial U, \rho)=\Con(f,\rho)\cap \partial U$ 
and $\Con(\partial U)=\Con(f)\cap \partial U$. 

% (Corollary \ref{semi-hyperbolicity})
Since $f$ is semi-hyperbolic, by \cite[Proposition 6.1]{U94}, we have 
$$\Con(f) = J(f) \setminus \bigcup_{n\geq0} f^{-n}(\crit(f)).$$
It follows that 
\begin{equation}
\label{conical set formula}
\Con(\partial U) = \partial U \setminus \bigcup_{n\geq0} f^{-n}(\crit(f) ).
\end{equation}

\begin{lemma}
\label{a sequence of radii}
Given $\rho>0$, a $t$-conformal measure $\mu$ for $\partial U$ and $x\in\Con(\partial U,\rho)$. 
Then there exists a constant $C=C(\rho,\mu)>1$ and a sequence of positive numbers $r_k=r_k(\rho,x)\rightarrow0$ such that 
\begin{equation}
\label{estimate of mu-measure}
C^{-1}r_k^t \leq {\mu(\mathbb{D}(x,r_k))}\leq C r_k^t. 
\end{equation}
\end{lemma}

\begin{proof}
If $0<\rho_1<\rho_2$, then $\Con(\partial U,\rho_1)\supset\Con(\partial U,\rho_2)$.  
So we can assume $\Con(\partial U,\rho)\neq\emptyset$ 
and $\rho<{\diam}(\partial U)/10$ without loss of generality. 
Let $x\in\Con(\partial U,\rho)$. 
Then there exists a sequence of integers $n_k\rightarrow\infty$ such that each $f^{n_k}$ is injective on 
$$D_k(\rho) := \Comp_x f^{-n_k}(\mathbb{D}(f^{n_k}(x),\rho)).$$
Applying the Koebe distortion theorem (Lemma \ref{Koebe-distortion}) to $g=f^{- n_k}:\mathbb{D}(f^{n_k}(x),\rho)\rightarrow D_k(\rho)$, we have 
$$ D_k(\eta\rho)\subset \mathbb{D}(x,r_k)\subset D_k(\rho/2),$$
where $\eta\in(0,1/2)$ satisfies $\eta\cdot \CK(\eta)=1/(2\CK(1/2))$, 
and $$r_k=\frac{\rho}{2 \CK(1/2) \cdot |(f^{n_k})'(x)| }.$$
By the topological exactness of $f|_{\partial U}$ 
(for any open arc $\gamma\subset\partial U$, there is an $n>0$ such that $f^{n}(\gamma)=\partial U$), 
we have $\lim_{k\rightarrow\infty}\dist(x,\partial D_k(\rho))=0$. 
Therefore $\lim_{k\rightarrow\infty}|(f^{n_k})'(x)|=\infty$ and $\lim_{k\rightarrow\infty} r_k=0$. 

By (\ref{transition inequality for mu}) and the Koebe distortion theorem, we have 
\begin{align*}
1&\geq \mu(f^{n_k}(\mathbb{D}(x,r_k)))\geq\int_{\mathbb{D}(x,r_k)}|(f^{n_k})'|^t d\mu\\
&\geq \mu(\mathbb{D}(x,r_k)){\left(\frac{|(f^{n_k})'(x)|}{\CK(1/2)}\right)}^t
=\frac{\mu(\mathbb{D}(x,r_k))}{r_k^t} \frac{\rho^t}{2^t \CK(1/2)^{2t}}.
\end{align*}
On the other hand, we need to give a lower bound of $\mu(\mathbb{D}(x,r_k))$. 
Let $\phi:\partial U\rightarrow\partial\mathbb{D}$ be a topological conjugacy such that $\phi(f(z))=\phi(z)^{\deg(f|_{\partial U})}$. 
By the uniform continuity of $\phi^{-1}$ on $\partial {\mathbb{D}}$, for every open arc $\gamma\subset\partial U$ with ${\diam}(\gamma)\geq\eta\rho$, 
there is an $n_\rho>0$ independent of $\gamma$ such that $f^{n_\rho}(\gamma)=\partial U$. 
By (\ref{transition formula for mu}), for every Borel set $E\subset\partial U$, we have $\mu(f(E))\leq\int_E|f'|^t d\mu$. 
It follows that 
\begin{align*}
m(\rho,\mu):=\inf\{\mu(\gamma) {;~} \text{$\gamma\subset\partial U$ is an open arc with ${\diam}(\gamma)\geq\eta\rho$}\}>0.
\end{align*}
By the Koebe distortion theorem and (\ref{transition formula for mu}), we have 
\begin{align*}
\frac{\mu(\mathbb{D}(x,r_k))}{r_k^t} \frac{\rho^t}{2^t}
&=\mu(\mathbb{D}(x,r_k))\cdot( |(f^{n_k})'(x)|\cdot \CK(1/2) )^t\\
&\geq  \int_{D_k(\eta\rho)\cap\partial U} |(f^{n_k})'|^t d\mu
=  \mu(f^{n_k}( D_k(\eta\rho)\cap\partial U ))\\
&\geq  \mu( \Comp_{f^{n_k}(x)}(\mathbb{D}(f^{n_k}(x),\eta\rho)\cap\partial U) )
\geq m(\rho,\mu). 
\end{align*}
Letting $C=\max{\left\{ \frac{2^t \CK(1/2)^{2t}}{\rho^t} , \frac{\rho^t}{2^t m(\rho,\mu)} \right\}}$ gives (\ref{estimate of mu-measure}). 
\end{proof}

For a subset $E$ of $\mathbb{C}$, we define $\chi_E:\mathbb{C}\rightarrow\{0,1\}$ by $\chi_E(x)=1$ if $x\in E$ and $\chi_E(x)=0$ if $x\notin E$. 
The following inequality is an analogue of \cite[Corollary 2.4]{McM2000}. 

\begin{corollary}
\label{upper bound of Hausdorff dimension}
$\Hdim(\Con(\partial U))\leq\dconf(\partial U)$. 
\end{corollary}

\begin{proof}
Given $\rho>0$ with $\Con(\partial U,\rho)\neq\emptyset$ and a $t$-conformal measure $\mu$ for $\partial U$. 
We claim that $\Hdim(\Con(\partial U,\rho))\leq t$. 

Let $\varepsilon>0$, and let $C=C(\rho,\mu)$ as in Lemma \ref{a sequence of radii}. 
Then for every $x\in\Con(\partial U,\rho)$, 
there is an $r_x\in(0,\varepsilon]$ so that $ r_x^t/C \leq \mu(\mathbb{D}(x,r_x))$. 
By the Vitali covering lemma (see, e.g. \cite[\S2]{Mattila}), 
there exists a countable subset $E$ of $\Con(\partial U,\rho)$ such that 
$$\sum_{x\in E}\chi_{\mathbb{D}(x,r_x)}\leq1, \quad \chi_{\Con(\partial U,\rho)}\leq\sum_{x\in E}\chi_{\mathbb{D}(x,5r_x)}.$$
It follows that 
$$\sum_{x\in E}(\diam \mathbb{D}(x,5r_x))^t= \sum_{x\in E} (10r_x)^t
\leq 10^t C \sum_{x\in E}\mu(\mathbb{D}(x,r_x))\leq 10^t C,$$
and $\{\mathbb{D}(x,5r_x) {;~}  x\in E\}$ is a $10\varepsilon$-cover of $\Con(\partial U,\rho)$. 
Letting $\varepsilon\rightarrow0$ gives the claim. 

Since $\Con(\partial U)=\bigcup_{n\geq1}\Con(\partial U,1/n)$, we have $\Hdim(\Con(\partial U))\leq t$.  
It follows that $\Hdim(\Con(\partial U))\leq \dconf(\partial U)$, 
since the pair $(\mu,t)$ is arbitrary.
%since $t>\dconf(\partial U)$ is arbitrary.
%because $\mu$ was an arbitrary conformal measure for $\partial U$. 
\end{proof}

%For $j\geq1$, inductively let $\mathscr{B}_j$ be the collection of all balls $\mathbb{D}(x,s)$ 
%centered at $x\in\Con(\partial U,\rho)$, with raidus $s\in(0,\varepsilon]$, 
%satisfying (\ref{esimate of mu-measure}) and disjoint from $\bigcup_{\ell=1}^{j-1} \mathbb{D}(x_\ell,s_\ell)$; 
%let $\mathbb{D}(x_j,s_j)$ be a ball in $\mathscr{B}_j$ such that $$s_j>\frac{\sup\{s {;~}  \mathbb{D}(x,s)\in\mathscr{B}_j\}}{2}.$$ 
%Finally, we have $\{\mathbb{D}(x_j,s_j)\}_{1\leq j< N_{\varepsilon}}$, where $2\leq N_{\varepsilon}\leq\infty$. 
%
%We claim that $\Con(\partial U,\rho)\subset\bigcup_{j=1}^{N_\varepsilon-1}\mathbb{D}(x_j,5s_j)$. 
%Let $x\in \Con(\partial U,\rho)\setminus \bigcup_{j=1}^{N_\varepsilon-1}\mathbb{D}(x_j,s_j)$. 
%By Lemma \ref{a sequence of radii}, there is an  $s\leq\varepsilon$ such that $\mathbb{D}(x,s)$ satisfies (\ref{esimate of mu-measure}). 
%Then $\mathbb{D}(x,s)$ must meet one $\mathbb{D}(x_j,s_j)$. 
%Choose the minimum $j$ such that $\mathbb{D}(x,s)\cap \mathbb{D}(x_j,s_j)\neq\emptyset$. 
%Then $\mathbb{D}(x,s)\in\mathscr{B}_j$ and $s_j>s/2$. 
%Therefore $x\in \mathbb{D}(x,s)\subset \mathbb{D}(x_j,5s_j)$. 
%This shows the claim. 

\begin{corollary}
\label{mu supported on conical set}
Let $\mu$ be a $t$-conformal measure for $\partial U$ with $\mu(\Con(\partial U))>0$.  
Then $t = \dconf(\partial U)$. 
\end{corollary}

\begin{proof} By Fact \ref{conf-md},  let $\mu_0$ be a $t_0$-conformal measure for $\partial U$, where $t_0=\dconf(\partial U)$. 
Clearly $t\geq t_0$. 
% Assume $t>\dconf(\partial U)$; we will find a contradiction.
Since $\Con(\partial U)=\bigcup_{n\geq1}\Con(\partial U,1/n)$ and $\mu(\Con(\partial U))>0$, there is a $\rho>0$ such that $\mu(\Con(\partial U,\rho))>0$ (one may verify that $\Con(\partial U,\rho)$ is  Borel). 
 
Let $\varepsilon>0$, and let $C=\max\{C(\rho,\mu_0),C(\rho,\mu)\}$ as in Lemma \ref{a sequence of radii}.  
Then for every $x\in\Con(\partial U,\rho)$, 
there is an $r_x\in(0,\varepsilon]$ so that $\mu(\mathbb{D}(x,r_x))\leq C r_x^t$ and $ r_x^{t_0}/C \leq \mu_0(\mathbb{D}(x,r_x))$. 
By the Besicovitch covering theorem (see, e.g. \cite[\S2]{Mattila}), 
there exists a countable subset $E$ of $\Con(\partial U,\rho)$ and a constant $M>0$ such that 
$$\chi_{\Con(\partial U,\rho)}\leq\sum_{x\in E}\chi_{\mathbb{D}(x,r_x)}\leq M.$$
It follows that 
\begin{align*}
\mu(\Con(\partial U,\rho))
& \leq \sum_{x\in E}\mu(\mathbb{D}(x,r_x))
  \leq C\sum_{x\in E} r_x^t\\
& \leq C\varepsilon^{t-t_0}\sum_{x\in E} r_x^{t_0}
  \leq C^2\varepsilon^{t-t_0}\sum_{x\in E} \mu_0(\mathbb{D}(x,r_x))\\
& \leq C^2\varepsilon^{t-t_0} M\mu_0\Bigg(\bigcup_{x\in E}\mathbb{D}(x,r_x)\Bigg)
  \leq C^2\varepsilon^{t-t_0} M. 
\end{align*}
Letting $\varepsilon\rightarrow 0$ gives $t=t_0 = \dconf(\partial U)$, because $\mu(\Con(\partial U,\rho))>0$. 
\end{proof}

\subsection{Hyperbolic set}
\label{subsection hyperbolic set}
% Given $f\in\piHbar$, $v\in\Tfp$ and let $U=U_{f,v}$. 
% Without loss of generality, assume $f(U)=U$.  

We say a compact set $X\subset J(f)$ is \emph{hyperbolic} if 
$f(X)\subset X$ and there exists an $n\geq1$ such that $|(f^{n})'(x)|>1$ for all $x\in X$. 
The \emph{hyperbolic Hausdorff dimension} of $\partial U$ is defined to be 
$$\hypdim(\partial U)=\sup\{\Hdim(X) {;~} \text{$X$ is a hyperbolic subset of $\partial U$}\}.$$

Let $N\geq2$, and let $B,D_1,\dots, D_N$ be Jordan domains in $\mathbb{C}$ so that $\overline{D_1},\dots,\overline{D_N}$ are pairwise disjoint, and $\bigcup_{n=1}^N D_n\Subset B$. 
Let $F :\bigcup_{n=1}^N D_n\rightarrow B$ be a holomorphic map so that $F$ maps $D_n$ conformally onto $B$ for each $1\leq n\leq N$. 
Let $X = \bigcap_{k=0}^\infty F^{-k}(B)$. 
Then $f:X\rightarrow X$ is topologically conjugate to the shift map on $\{1,2,\dots, N\}^{\mathbb{N}}$. 
Let $\Comp(A)$ be the collection of all connected components of $A\subset\mathbb{C}$. 
For $t\geq0$, the \emph{topological pressure} is defined by 
\begin{equation}
\label{topological-pressure}
P(F|_X,-t\log|F'|) = \lim_{k\rightarrow \infty}\frac{1}{k}
\log \Bigg( \sum_{D\in\Comp(F^{-k}(B))} \max_{z\in D\cap X}|(F^k)'(z)|^{-t} \Bigg). 
\end{equation}
See, e.g. \cite[\S3.3]{Zinsmeister}, \cite[\S3.9]{dede}. 
By Bowen's formula, the topological pressure $P(F|_X,-t\log|F'|)$ is a continuous and strictly decreasing function of $t$, with a unique zero $\Hdim(X)$; see, e.g. \cite[Theorem 5.12]{Zinsmeister}, \cite[Theorem 3.9.11]{dede}.

\begin{lemma}
\label{Poin-expo vs hyp-dim}
$\dPoin(\partial U) \leq \hypdim(\partial U)$. 
\end{lemma}

\begin{proof}
The proof is inspired by \cite{Przy05}. 
See also \cite[Proposition 2.1]{PRLS-pressures}. 

Fix an admissible point $x\in\partial U$. 
Choose $r>0$, $\lambda>1$ and $M>0$ as in Lemma \ref{admis-pt-part-U}. 
For $\rho\in(0,r]$, let $\gamma(\rho)$ denote the open arc  $\Comp_x(\partial U\cap \mathbb{D}(x,\rho))$. 
Let $r' = \dist(x,\partial U\setminus \gamma(r))\in(0,r]$. 
Choose $m\geq1$ so that $f^m(\gamma(r'/4))\supset \partial U$. 
For every $y\in f^{-n}(x)\cap\partial U$ ($n\geq1$), we find $y'\in f^{-m}(y) \cap \gamma(r'/4)$. 
Let $D_{y'} = \Comp_{y'} f^{-(m+n)}(\mathbb{D}(x,r))$. 
Note that ${\diam}(D_{y'})\leq M\lambda^{-(m+n)}$. 
Choose $N\geq1$ so that $M\lambda^{-(m+N)}<r'/4$. 
Then $D_{y'}\subset \mathbb{D}(x,r'/2)\subset \mathbb{D}(x,r/2)$ for any $n\geq N$. 
In the following, we always consider $n\geq N$. 
Define 
$$F_n = f^{m+n}:\bigcup_{y\in f^{-n}(x)\cap\partial U} D_{y'}\rightarrow \mathbb{D}(x,r).$$
Let $X_n = \bigcap_{k=0}^\infty F_n^{-k}(\mathbb{D}(x,r))$. 
Then $X_n$ is a hyperbolic subset of $J(f^{m+n}) = J(f)$ with respect to $f^{m+n}$. 
By Bowen's formula as before, the topological pressure $P(F_n|_{X_n},-t\log|F'_n|)$ is a cotinuous and strictly decreasing function of $t\geq0$, with a unique zero $\Hdim(X_n)$. 

Note that $\partial U\cap D_{y'}\subset\partial U\cap \mathbb{D}(x,r'/2)\subset \gamma(r)$. 
The compact set $X_n$ is equal to the attractors of 
the following two \emph{iterated function systems} \cite{Falconer}:  
\begin{align*}
&\big\{f^{-(m+n)}: \overline{\mathbb{D}(x,r)}\rightarrow \overline{D_{y'}}
 \big\}_{y\in f^{-n}(x)\cap\partial U};\\
&\big\{f^{-(m+n)}: \overline{\gamma(r)}\rightarrow \overline{\Comp_{y'}(\partial U\cap D_{y'})} 
 \big\}_{y\in f^{-n}(x)\cap\partial U}.
\end{align*}
It follows that $X_n\subset \partial U$. 

Note that $F_n^{-1}(\mathbb{D}(x,r))\subset \mathbb{D}(x,r/2)$. 
By (\ref{topological-pressure}) and the Koebe distortion theorem (Lemma \ref{Koebe-distortion}), we have 
\begin{align*}
P(F_n|_{X_n},-t\log|F'_n|)
&=    \lim_{k\rightarrow \infty}\frac{1}{k}\log \Bigg( \sum_{D\in\Comp(F_n^{-k}(\mathbb{D}(x,r)))} \max_{z\in D\cap X_n}|(F_n^k)'(z)|^{-t}\Bigg)\\
&\geq \lim_{k\rightarrow \infty}\frac{1}{k}\log \left( \Bigg( \sum_{y\in f^{-n}(x)\cap\partial U}{\left(\CK(1/2)\cdot|F'_n(y')|\right)}^{-t} \Bigg)^k \right)\\
&=    \log \Bigg(\CK(1/2)^{-t}\sum_{y\in f^{-n}(x)\cap\partial U}|(f^{m+n})'(y')|^{-t} \Bigg)\\
&\geq \log \Bigg(\CK(1/2)^{-t} \Lambda^{-mt}\sum_{y\in f^{-n}(x)\cap\partial U}|(f^{n})'(y)|^{-t} \Bigg), 
\end{align*}
where $\Lambda = \max_{z\in \partial U} |f'(z)|$. 
%Let $X'_n := \bigcup_{j=0}^{m+n-1}f^j(X_n)$, and let $C(t) = \CK(1/2)^{-t} M^{-mt}$. Then we have 
%\begin{align*}
%P(f|_{X'_n},-t\log|f'|)
%&\geq \frac{1}{m+n} P(F_n|_{X_n},-t\log|F'_n|)\\
%&\geq \frac{1}{m+n} {{\left(\log(C(t))+\log\sum_{y\in f^{-n}(x)\cap\partial U}|(f^{n})'(y)|^{-t}\right)}}.
%\end{align*}

Let $X'_n = \bigcup_{j=0}^{m+n-1}f^j(X_n)$. 
Then $X'_n$ is a hyperbolic subset of $\partial U$, with Hausdorff dimension $\Hdim(X'_n)=\Hdim(X_n)$. 
Let $t_0=\hypdim(\partial U)$. 
Recall that $P_n(x,t_0)$ is the Poincar\'e sequence defined by (\ref{Poin-sequ}), 
and $C(x,t_0)$ is defined by (\ref{Cauchy-sequ}). 
Since $t_0\geq \Hdim(X_n)$, we have 
$$\frac{1}{n} \log \Big(\CK(1/2)^{-t_0} \Lambda^{-mt_0}P_n(x,t_0) \Big)
\leq \frac{1}{n}P(F_n|_{X_n},-t_0\log|F'_n|)
\leq 0.$$ 
Letting $n\rightarrow\infty$ gives $C(x,t_0)\leq 0$. 
By Lemma \ref{tree-pressure}, we have $\dPoin(\partial U)\leq t_0 = \hypdim(\partial U)$. 
\end{proof}

\begin{proof}
[Proof of Proposition \ref{equalities for Hdim of partial-U}]
As before, we assume $f(U)=U$. 
Clearly $\hypdim(\partial U)\leq\Hdim(\partial U)$. 
By (\ref{conical set formula}), we know that $\partial U\setminus\Con(\partial U)$ is countable, 
so $\Hdim(\partial U)=\Hdim(\Con(\partial U))$. 
By Corollary \ref{upper bound of Hausdorff dimension}, we have $\Hdim(\Con(\partial U))\leq\dconf(\partial U)$. 
By Lemma \ref{Patterson-Sullivan construction}, we have 
$\dconf(\partial U)\leq \dPoin(\partial U)$. 
By Lemma \ref{Poin-expo vs hyp-dim}, we have $\dPoin(\partial U) \leq \hypdim(\partial U)$. 
Therefore all these quantities agree. 
\end{proof}

%\begin{lemma}
%\label{dimension rigidity}
%Let $f\in\Hbar$, and let $U$ be a bounded Fatou component of $f$. 
%Then we have $$1<\Hdim(\partial U)<2.$$
%\end{lemma}
%
%\begin{proof}
%Without loss of generality, we assume $f(U)=U$. 
%Note that $f$ is monic and centered, but $f(z)\neq z^d$. 
%By a result of Przytycki \cite{Przy06}, we have $\hypdim(\partial U)>1$.
%It follows that $\Hdim(\partial U)\geq \hypdim(\partial U)>1$. 
%(In fact, we have $\Hdim(\partial U)=\hypdim(\partial U)$ by Proposition \ref{equalities for Hdim of partial-U}.)
%Since $f$ is semi-hyperbolic (Corollary \ref{semi-hyperbolicity}), the boundary of $U$ is a quasicircle (Lemma \ref{properties semi-hyper}), hence $\Hdim(\partial U)<2$. 
%\end{proof}
%

\subsection{Uniform expansion on external rays} 
% \subsection{Distortion theorems} 
\label{distortion-thm} 
The aim of this subsection is to prove Proposition \ref{expanding on external rays}. 
Before that, let us introduce several distortion theorems for quasicircles. 

Let $\gamma$ be a Jordan curve in $\mathbb{C}$, and let $1\leq M<\infty$. 
We say $\gamma$ is a \emph{quasicircle} with turning $M$ if for $z_1,z_2\in \gamma$ with $z_1\neq z_2$ we have 
$$\min_{j=1,2} \diam(\gamma_j)\leq M|z_1-z_2|,$$
where $\gamma_1,\gamma_2$ are the two components of $\gamma\setminus\{z_1,z_2\}$. 
% Before proving Proposition \ref{continuous-hdim}, 

Let $K$ be a connected and compact subset of $\mathbb{C}$ containing at least two points. 
For any $z_1,z_2\in K$, define the \emph{turning} of $K$ about $z_1$ and $z_2$ by 
$$\Delta(K;z_1,z_2)=\frac{{\diam}(K)}{|z_1-z_2|}.$$
It is easy to see that $1\leq \Delta(K;z_1,z_2)\leq\infty$, 
and $\Delta(K;z_1,z_2)=\infty$ if and only if $z_1=z_2$. 

For a Riemann surface $A$ biholomorphic to the standard annulus $\{z\in\mathbb{C}{;~}1<|z|<R\}$ for some $R\in(1,\infty)$, the \emph{modulus} of $A$ is defined by $$\operatorname{mod}(A)=\frac{\log R}{2\pi}.$$

\begin{lemma}
[Turning distortion {\cite[Lemma 6.1.2]{QWY}}]
\label{turning distortion}
Let $g:V_1\rightarrow V_2$ be a proper holomorphic map of degree $d$ between Jordan domains. 
Let $U_2\Subset V_2$ be a Jordan domain satisfying $\operatorname{mod}(V_2\setminus\overline{U_2})\geq m>0$, 
and let $U_1$ be a connected component of $g^{-1}(U_2)$. 
Then there is a constant $C(d,m)>0$ such that
$$ \Delta(K;z_1,z_2) \leq C(d,m) \Delta(g(K);g(z_1),g(z_2))$$
for any connected and compact subset $K$ of $U_1$ containing at least two points and any $z_1,z_2\in K$. 
\end{lemma}

\begin{corollary}
\label{preimage of quasidisk}
Let $f$ be a polynomial of degree $d\geq1$, let $D$ be a Jordan domain and let $\Omega$ be a connected component of $f^{-1}(D)$. 
Suppose $f(\crit(f))\cap D=\emptyset$ and $\partial D$ is a quasicircle with turning $M$. 
Then there is a constant $L(d,M)\geq1$ such that $\partial\Omega$ is a quasicircle with turning $L(d,M)$. 
\end{corollary}

\begin{proof}
Choose $R>0$ large enough such that $f(\crit(f))\cup\overline{D}\subset \mathbb{D}(0,R)$. 
Let $U_2=\mathbb{D}(0,R)$ and $V_2=\mathbb{D}(0,e^{2\pi}R)$. 
Then $\operatorname{mod}(V_2\setminus\overline{U_2})=1$. 
Let $U_1=f^{-1}(U_2)$ and $V_1=f^{-1}(V_2)$. 
Then $U_1$ and $V_1$ are Jordan domains, since $f(\crit(f))\subset U_2$. 
It follows from $\overline{D}\subset U_2$ that $\overline{\Omega}\subset U_1$. 
Let $z_1,z_2\in\partial \Omega$ with $z_1\neq z_2$, and let $\gamma_1,\gamma_2$ denote the two components of $\partial\Omega\setminus\{z_1,z_2\}$. 
Note that $f:\overline{\Omega}\rightarrow\overline{D}$ is a homeomorphism, because $f(\crit(f))\cap D=\emptyset$. 
By Lemma \ref{turning distortion}, we have 
$$\min_{j=1,2}\Delta(\gamma_j;z_1,z_2)\leq C(d,1)\min_{j=1,2}\Delta(f(\gamma_j);f(z_1),f(z_2))
\leq C(d,1)M.$$
Let $L(d,M)=C(d,1)M$. 
Then $\partial\Omega$ is a quasicircle with turning $L(d,M)$. 
\end{proof}

For $r_0>0$ and $0<t_0\leq\pi/2$, let $$S(r_0,t_0)=\{re^{i t} {;~} 0<r<r_0,-t_0<t<t_0\}$$ be a \emph{sector}. 
The following result provides a derivative distortion for quasidisks (quasicircles), 
which is an analogy of \cite[Theorem 5.2]{Pom}.
% an immediate consequence of 
% Note that a quasidisk must be a John domain. 

\begin{lemma}
[Derivative distortion]
\label{derivative distortion of quasidisks}
Let $r_0>0$, and let $f$ map $S(r_0,\pi/4)$ conformally onto a Jordan domain $\Omega$. 
Suppose $\partial \Omega$ is a quasicircle with turning $M_1$ and $\dist(f(r_0/2),\partial \Omega)\geq M_2 \diam(\Omega)$. 
Then there is a constant $C=C(M_1,M_2)>0$ such that for $0< r_1\leq r_2\leq r_0/2$,
\begin{equation}
\label{deri-distor-0}
\left|\frac{f'(r_2)}{f'(r_1)}\right|\geq C{\left(\frac{r_2}{r_1}\right)}^{C-1}.
\end{equation}
\end{lemma}

\begin{proof}
Choose $t_0\in (0, \pi/4)$ so that $4\sin(t_0/2) = \sin(\pi/4)$. 
Fix $r\in (0, r_0/2]$. 
% and let $d_f(r)=\dist(f(r), \partial \Omega)$. 
By \cite[Theorem 5.2, (i) $\Rightarrow$ (ii)]{Pom}, there is a constant $C_1=C_1(M_1,M_2)>0$ such that 
$$\diam f(S(r, \pi/4))\leq C_1 \dist(f(r), \partial \Omega).$$
Define $$\varphi(r)=\int_0^r \rho|f'(\rho)|^2 d\rho.$$
For any $\rho e^{it}\in S(r, t_0)$, 
by the Koebe distortion theorem (Lemma \ref{Koebe-distortion}),
we have $|f'(\rho e^{it})|\geq |f'(\rho)|/\CK(1/2)$, 
hence
\begin{equation}
\label{area1}
\area f(S(r, t_0))=\int_{-t_0}^{t_0} \int_{0}^{r} \rho |f'(\rho e^{it})|^2 d\rho d t \geq \frac{2 t_0}{\CK(1/2)^2} \varphi(r).
\end{equation}
Applying the Koebe quarter theorem to $g=f^{-1}:\mathbb{D}(f(r), \dist(f(r), \partial \Omega))\rightarrow S(r_0, \pi/4)$, we have $\dist(f(r), \partial \Omega)\leq   4 |f'(r)|\dist(r, \partial S(r_0, \pi/4))= 2\sqrt{2} r |f'(r)|$.
It follows that 
\begin{equation}
\label{area2}
\begin{split}
\area f(S(r, t_0))&
\leq \pi (\diam f(S(r, t_0)))^2
\leq \pi (\diam f(S(r, \pi/4)))^2\\
&\leq \pi C_1^2 \dist(f(r), \partial \Omega)^2\leq 8\pi C_1^2 r^2 |f'(r)|^2.
\end{split}
\end{equation}
Then (\ref{area1}) and (\ref{area2}) give $\varphi(r)\leq C_2 r^2 |f'(r)|^2=C_2 r \varphi'(r) $. 
On the other hand, for any $\rho\in (r(1-\sin t_0),r)$, by the Koebe distortion theorem again, we have $|f'(\rho)|\geq |f'(r)|/\CK(1/2)$, hence
$$\varphi(r)\geq\int_{r(1-\sin t_0)}^r  \rho|f'(\rho)|^2 d\rho\geq C_3 r^2|f'(r)|^2.$$
For $0< r_1\leq r_2\leq r_0/2$, we have 
$$\log \frac{\varphi(r_2)}{\varphi(r_1)}=\int_{r_1}^{r_2} \frac{\varphi'(\rho)}{\varphi(\rho)} d\rho \geq \int_{r_1}^{r_2} \frac{1}{C_2 \rho} d\rho=\frac{1}{C_2} \log\frac{r_2}{r_1}.$$
It follows that 
$${\left( \frac{r_2}{r_1} \right)}^{1/C_2}
\leq \frac{\varphi(r_2)}{\varphi(r_1)}
\leq \frac{C_2 r_2^2 |f'(r_2)|^2}
          {C_3 r_1^2 |f'(r_1)|^2} 
~\Longleftrightarrow~
     \left|\frac{f'(r_2)}{f'(r_1)}\right|
\geq \sqrt{\frac{C_3}{C_2}} 
  {\left(\frac{r_2}{r_1} \right)}^{1/(2C_2)-1}.$$
Taking $C=\min\{\sqrt{C_3/C_2}, 1/(2 C_2)\}$ completes the proof. 
\end{proof}

An arc $\gamma\subset \mathbb{C}$ is a \emph{quasiarc} if there is a constant $C>0$ such that 
for $0\leq x\leq y\leq z\leq 1$, 
$$|\xi(x)-\xi(z)|\geq C|\xi(x)-\xi(y)|,$$
where $\xi:[0,1]\rightarrow\gamma$ is a parameterization of $\gamma$. 
Let $\gamma$ be a Jordan curve in $\mathbb{C}$. 
It is easy to check that $\gamma$ is a quasicircle if and only if there are finitely many subarcs $\{\gamma_j\}$ of $\gamma$ such that:  
each $\gamma_j$ is a quasiarc; every $x\in\gamma$ is contained in some $\gamma_j$ and $x$ is not an end point of $\gamma_j$. 

The following result provides a criterion for quasiarcs. 

\begin{lemma}
[{\cite[Theorem 2.12]{BHOT}}]
\label{criterion for quasiarcs}
Let $\gamma_1,\gamma_2\subset\mathbb{C}$ be two arcs with a common endpoint $0$ and disjoint elsewhere. 
Suppose $\gamma_1\cup\gamma_2$ is smooth except possibly at $0$. 
Let $\lambda\in\mathbb{C}$ with $|\lambda|>1$. 
Furthermore, suppose $\lambda\gamma_1\supset\gamma_1$ and $\lambda\gamma_2\supset\gamma_2$. 
Then $\gamma_1\cup\gamma_2$ is a quasiarc. 
\end{lemma}

\begin{proof}
Let $\gamma=\gamma_1\cup\gamma_2$, and let $x_0,y_0$ be the end points of $\gamma$ so that $x_0\in\gamma_1,y_0\in\gamma_2$. 
For any $a,b\in\gamma$ with $a\neq b$, let $[a,b]_{\gamma}$ denote the subarc of $\gamma$ with end points $a,b$, and let 
$[a,b)_{\gamma}=[a,b]_{\gamma}\setminus\{b\}$. 
For $n\in\mathbb{N}$, let $x_n=x_0/\lambda^n$ and $y_n=y_0/\lambda^n$. 

Let $a,b,c\in\gamma$ so that $x_0,a,b,c,y_0$ are successively five points in $\gamma$. 
For convenience, we assume $a,b,c$ are pairwise different. 
Note that there exists a unique $n\in\mathbb{N}$ so that $\{\lambda^n a, \lambda^n b, \lambda^n c\}\subset\gamma$; $\lambda^n a\in[x_0,x_1)_\gamma$ or $\lambda^n c\in[y_0,y_1)_\gamma$. 
We just need to consider $a\in[x_0,x_1)_\gamma$ or $c\in[y_0,y_1)_\gamma$. 
If $a\in[x_0,x_1)_\gamma$ and $c\in[x_2,y_0]_\gamma$, then $$\frac{|a-c|}{|a-b|}
\geq\frac{\min\{|x-y|{;~} x\in[x_0,x_1]_\gamma, \ y\in[x_2,y_0]_\gamma\}}{{\diam}(\gamma)}=:M_1.$$ 
Similarly, we obtain a constant $M_2$ if $c\in[y_0,y_1)_\gamma$ and $a\in[x_0,y_2]_\gamma$. 
If $a\in[x_0,x_1)_\gamma$ and $c\in[x_0,x_2]_\gamma$, by the smoothness of the open arc  $\lambda\cdot([x_0,x_4]_\gamma\setminus\{x_0,x_4\})$, the subarc $[x_0,x_2]_\gamma$ is a quasiarc, so there is a constant  $M_3>0$ independent of $a,b,c$ such that $\frac{|a-c|}{|a-b|}\geq M_3$.
Similarly, we obtain a constant $M_4$ if $c\in[y_0,y_1)_\gamma$ and $a\in[y_0,y_2]_\gamma$. 
Finally, let $M=\min_{1\leq k\leq 4}M_k$. 
\end{proof}

\begin{proof}
[Proof of Proposition \ref{expanding on external rays}]
Without loss of generality, we assume $\theta=1/d$. Then $g_f=f^{-1}\circ f\circ f: \overline{R_{f}(\theta)}\rightarrow \overline{R_{f}(\theta)}$
is the lift of $f:\overline{R_{f}(0)}\rightarrow \overline{R_{f}(0)}$. 
We need to find a neighborhood $\mathcal{N}\subset \mathcal{P}^d$ of $h$ and a constant $n>0$ such that
$$\big|(g_f^n)'(z)\big|\geq 2$$
for every $f\in\mathcal{C}(\mathcal{P}^d)\cap\mathcal{N}$ and every $z\in\overline{R_{f}(\theta)}$.
Let $C_1,C_2,\dots$ denote suitable positive constants. 

Let $y_h$ denote the landing point of $R_h(0)$. 
Then $y_h$ is repelling. 
By the linearization of $h$ near $y_h$, we can choose $R>0$ such that $h$ is injective on $\overline{\mathbb{D}(y_h,R)}$, 
and $\mathbb{D}(y_h,R)\Subset h(\mathbb{D}(y_h,R))$. 
For $r>0$, let $S(r)=S(r,\pi/4)$, and let $S_h(r) = (B_h^{-1}\circ \exp)(S(r))$ be a \emph{log-B\"ottcher sector} with vertex $y_h$. 
(For more applications of the log-B\"ottcher sector, see \cite[\S8]{BF}.) 
Fix $r\in(0,1)$ small enough such that $S_h(r)\subset \mathbb{D}(y_h,R)$. 
Then $h: S_h(r/d)\rightarrow S_h(r)$ is conformally conjugate to 
$$\widetilde m_d:  
\begin{cases} 
S(r/d)\rightarrow S(r), \\
z\mapsto d\cdot z. 
\end{cases}$$
By the linearization of $h|_{\mathbb{D}(y_h,R)}$ and Lemma \ref{criterion for quasiarcs}, 
the boundary of $S_h(r)$ is a quasicircle. 

For $k\in\mathbb{N}$, let $A_k=\exp( S(r/d^k)\setminus S(r/d^{k+1}) )$. 
Following the idea of Douady and Hubbard \cite[Proposition 8.5]{Orsay} (see also \cite[Lemma 3.1]{Wang21}), 
since $y_h$ is repelling, there is an open ball $\mathcal{N}\subset\mathcal{P}^d$ centered at $h$ so that for every $f\in\mathcal{N}$, 
\begin{itemize}
\item after perturbation in $\mathcal{N}$, the repelling fixed point $y_h$ of $h$ becomes a repelling fixed point $y_f$ of $f$, 
with a common linearization neighborhood $\mathbb{D}(y_h,R)$; 

\item the inverse $B_f^{-1}$ of the B\"ottcher coordinate is well-defined and univalent on $\{z\in\mathbb{C} {;~}  |z|>\rho\}$, 
where $\rho=(1+\inf_{z\in A_0}|z|)/2>1$. 
\end{itemize}
Now we can define
$$\phi:
\begin{cases} 
\mathcal{N}\times S_h(r)\rightarrow\mathbb{C}, \\
(f,z)\mapsto (f|_{\mathbb{D}(y_h,R)})^{-k}\circ B_f^{-1}\circ B_h\circ h^k(z), ~z\in B_h^{-1}(A_k). 
\end{cases}$$
Then $\phi$ is a holomorphic motion. 
For $f\in\mathcal{N}$, let $S_f(r)=\phi(f,S_h(r))$ be a log-B\"ottcher sector with vertex $y_f$. 
It is easy to extend $\phi$ to a holomorphic motion defined on $\mathcal{N}\times \overline{S_h(r)}$ 
(or by the $\lambda$-lemma \cite{McM}). 
Since $\partial S_h(r)$ is a quasicircle, shrinking $\mathcal{N}$ if necessary, 
the collection $\{\partial S_f(r)\}_{f\in\mathcal{N}}$ consists of uniform quasicircles. 
That is, there is a constant $C_1\geq 1$ so that $\partial S_f(r)$ is a quasicircle with turning $C_1$ for every $f\in\mathcal{N}$. 

From now on, let us consider $g_f$ for $f\in\mathcal{C}(\mathcal{P}^d)$. 
Let $x_f$ denote the landing point of $R_f(\theta)$. 
Since $g_f$ can extend holomorphically to a neighborhood of $\overline{R_{f}(\theta)}$, 
the derivative of $g_f$ along $\overline{R_{f}(\theta)}$ makes sense. 
For example, we have 
\begin{equation}
\label{derivative at xf}
\big|g'_f(x_f)\big|=\big|f'(y_f)\big|^{1/\deg(f,x_f)},
\end{equation} 
where $y_f=f(x_f)$ is the landing point of $R_f(0)$. 

For $f\in\mathcal{C}(\mathcal{P}^d)\cap\mathcal{N}$, let $V_f(r)$ be the connected component of $f^{-1}(S_f(r))$ 
containing the tail of $R_f(\theta)$. 
Then $g_f$ has a holomorphic extension on $V_f(r/d)$. 
Let $\varphi_f=\log\circ B_f\circ f:V_f(r)\rightarrow S(r)$ and let $\psi_f=\varphi_f^{-1}$, where we take $\log(1)=0$.  
Then $g_f(z)=\psi_f\circ \widetilde m_d\circ \psi_f^{-1}(z)$ for $z\in V_f(r/d)$. 
By Corollary \ref{preimage of quasidisk}, there is a constant $C_2\geq 1$ so that 
the collection $\{\partial V_f(r)\}_{f\in\mathcal{C}(\mathcal{P}^d)\cap\mathcal{N}}$ consists of quasicircles with turning $C_2$. 
It is easy to see that there is a constant $C_3>0$ such that $\dist(\psi_f(r/2),\partial V_f(r))\geq C_3\diam(V_f(r))$ 
for every $f\in\mathcal{C}(\mathcal{P}^d)\cap\mathcal{N}$. 
Now applying Lemma \ref{derivative distortion of quasidisks} to $\psi_f: S(r)\rightarrow V_f(r)$, 
there is a constant $C_4>0$ such that for 
every $f\in\mathcal{C}(\mathcal{P}^d)\cap\mathcal{N}$, every $k>0$ and every $z\in R_f(\theta)\cap V_f(r/d^{k+1})$, we have 
\begin{align*}
\big|(g^k_f)'(z)\big| = \big|(\psi_f\circ \widetilde m_d^k \circ \psi_f^{-1})'(z)\big|
 = d^k \left|\frac{\psi'_f(d^k\psi_f^{-1}(z))}{\psi'_f(\psi_f^{-1}(z))}\right|
 \geq d^k C_4 (d^k)^{C_4-1}
 = C_4 d^{k C_4}, 
\end{align*}
where $0<\psi_f^{-1}(z)\leq d^k\psi_f^{-1}(z)\leq r/d\leq r/2 $. 

Let $f\in\mathcal{C}(\mathcal{P}^d)\cap\mathcal{N}$. 
We will divide $\overline{R_f(\theta)}$ into several parts. 
Shrinking $\mathcal{N}$ if necessary, there is a constant $C_5>1$ such that $\big|f'(y_f)\big|\geq C_5>1$. 
By (\ref{derivative at xf}), we have $\big|g'_f(x_f)\big|\geq C_5^{1/d}>1$. 
So we can choose $n_1>0$ such that $\big|(g^{n_1}_f)'(x_f)\big|\geq 2$ and $C_4 d^{n_1 C_4}\geq 2$. 
Let $s_1=r/d^{n_1+2}$, which is the $G_f$-value of the only point in $R_f(\theta)\cap \partial V_f(r/d^{n_1+1})$. 
Then 
$$\inf{\left\{\big|(g^{n_1}_f)'(z)\big|{;~} f\in\mathcal{C}(\mathcal{P}^d)\cap\mathcal{N}, 
\ z\in \overline{R_f(\theta)}, 
\ 0\leq G_f(z)\leq s_1\right\}}\geq2.$$
By the behavior of $g_f$ near $\infty$, we can choose $s_2\gg s_1$ so that 
$$\inf{\left\{\big|(g^{n_1}_f)'(z)\big|{;~} f\in\mathcal{C}(\mathcal{P}^d)\cap\mathcal{N}, 
\ z\in \overline{R_f(\theta)}, 
\ G_f(z)\geq s_2\right\}}\geq2.$$
Shrinking $\mathcal{N}$ if necessary, we have 
$$C_6 := \inf{\left\{\big|(g^{n_1}_f)'(z)\big|{;~} f\in\mathcal{C}(\mathcal{P}^d)\cap\mathcal{N}, 
\ z\in \overline{R_f(\theta)}, 
\ s_1\leq G_f(z)\leq s_2\right\}}>0.$$
Note that $G_f(g_f(z))=d\cdot G_f(z)$. 
Choose $n_2>0$ such that $d^{n_1 n_2}s_1>s_2$, and choose $n_3>0$ such that $2^{n_3}C_6^{n_2}\geq 2$. 
Let $n=n_1(n_2+n_3)$. 
Then $\big|(g_f^n)'(z)\big|\geq 2$ for every $f\in\mathcal{C}(\mathcal{P}^d)\cap\mathcal{N}$ and every $z\in\overline{R_{f}(\theta)}$. 
This completes the proof. 
\end{proof}

\subsection{Proof of Proposition \ref{continuous-hdim}} 
\label{continuity-h-d} 

\begin{proof}
[Proof of Proposition \ref{continuous-hdim}] 
By Corollary \ref{semi-hyperbolicity}, every map in $\Hbar$ is semi-hyperbolic. 
Given $v\in\Tfp$. 
Without loss of generality, we assume $\sigma(v)=v$. 
Let $\{f_n\}_{n\geq1}\subset\Hbar$ satisfy $\lim_{n\rightarrow\infty} f_n = h$.

{\bf Step 1.}   By Fact \ref{conf-md}, for each $n\geq1$, choose an $s_n$-conformal measure $\mu_n$ for $\partial U_{f_n,v}$, 
where $s_n=\dconf(\partial U_{f_n,v})$. 
After passing to a subsequence, we can assume $\mu_n\rightarrow\mu$ in weak topology and $s_n\rightarrow s$ as $n\rightarrow\infty$. 
% By Proposition \ref{equalities for Hdim of partial-U}, we have $1\leq s_n\leq2$, so $1\leq s\leq 2$.
By Proposition \ref{continuity of Ufv}, we see that  $\partial U_{f_n,v}\rightarrow \partial U_{h,v}$ in Hausdorff topology as $n\rightarrow\infty$, 
so $\mu$ is supported on $\partial U_{h,v}$. 
By an explicit computation 
% (see, e.g. Lemma \ref{weak limit of conformal measure}), 
(similar to the proof of Lemma \ref{weak limit of conformal measure}), 
the measure $\mu$ is an $s$-conformal measure for $\partial U_{h,v}$. 

\vspace{6 pt}
{\bf Claim.} 
{\it Suppose $\{f_n\}_{n\geq1}\subset \mathcal{H}$. 
Then for any $c\in\crit(h)\cap\partial U_{h,v}$ (if any), we have $\mu(\{c\})=0$.}
\vspace{6 pt}

First, we consider the special case $\{f_n\}_{n\geq1}\subset \mathcal{H}$. 
By the claim and (\ref{conical set formula}), we have  $\mu(\Con(\partial U_{h,v}))=1$. 
It follows that $s=\dconf(\partial U_{h,v})$ by Corollary \ref{mu supported on conical set}. 
According to Proposition \ref{equalities for Hdim of partial-U}, 
$$\lim_{n\rightarrow\infty}\Hdim(\partial U_{f_n,v})=\lim_{n\rightarrow\infty}s_n 
= s =\dconf(\partial U_{h,v}) = \Hdim(\partial U_{h,v}).$$

For general $\{f_n\}_{n\geq1}\subset\Hbar$, 
assume $s\neq\Hdim(\partial U_{h,v})$; we will find a contradiction. 
By the above case, for each $n\geq1$, there is a polynomial  $\widetilde{f}_n\in \mathcal{H}$ so that $\|\widetilde{f}_n-f_n\|_2<1/n$ and 
\begin{equation}
\label{f in H}
\big|\Hdim(\partial U_{\widetilde{f}_n,v})-\Hdim(\partial U_{f_n,v})\big|<\big|\Hdim(\partial U_{h,v})-s\big|/2, 
\end{equation}
where $\|\cdot\|_2$ is the $2$-norm of the $(d-1)$-dimensional complex affine space $\widehat{\mathcal{P}}^d$. 
Then $\{\widetilde{f}_n\}_{n\geq1} \subset  \mathcal{H} \rightarrow h$, so 
$\Hdim(\partial U_{\widetilde{f}_n,v})\rightarrow\Hdim(\partial U_{h,v})$ as $n\rightarrow\infty$ by the above case again. 
It follows that (\ref{f in H}) does not hold when $n$ is large enough; a contradiction. 
This shows the continuity of $\Hdim(\partial U_{f,v})$ with respect to $f\in\Hbar$.

% proof of the claim
{\bf Step 2.}  
To complete the proof, we show the claim (compare \cite[Theorem 11.2]{McM2000}). 
Given $c\in\crit(h)\cap\partial U_{h,v}$. 
Then $c$ is pre-repelling by Proposition \ref{behavior of critical orbits}. 
Replacing $h$ by an iterate, we can assume $h(c)$ is a fixed point of $h$. 
Let $t$ be the internal angle of $c$ with respect to $U_{h,v}$. 
Set $\theta=\theta^-_{v}(t)$ and $\theta'=\theta^+_{v}(t)$ 
(see Lemma \ref{invariant external angles} for $\theta^-_v$ and $\theta^+_v$). 
Let $x_f$ denote the same landing point of $R_{f}(\theta)$ and $R_{f}(\theta')$ for $f\in\Hbar$. 
In particular, $x_h=c$. 

% define g_f
To continue our discussion, we associate each $f\in\Hbar$ with the sets $L_f$ and $Y_f$. 
Let $L_f = R_{f}(\theta)\cup\{x_f\}\cup R_{f}(\theta')$ and let 
$$g_f=f^{-1}\circ f\circ f: L_f \rightarrow L_f$$
be the lift of the restriction of $f$ on $f(L_f)$. 
% extend g_f
Let $t_1,t_2\in\QZ$ so that $t_1,t,t_2$ are in positive cyclic order, and let $\rho>1$. 
Let $Y_f$ be the connected component of 
\begin{align*}
\mathbb{C}\setminus \Big(&
L_f\cup  \overline{
R_{f,v}(t_1)\cup R_f(\theta^-_v(t_1))  \cup
R_{f,v}(t_2)\cup R_f(\theta^-_v(t_2))  } \\
&\cup B_f^{-1}(\rho\cdot \partial\mathbb{D}) \cup
B_{f,v}^{-1}(1/\rho\cdot \partial\mathbb{D})
\Big)
\end{align*}
containing $\{z\in R_{f,v}(t){;~} |B_{f,v}(z)| > 1/\rho\}$ 
(in the notations of \S\ref{subsection internal rays}).  
%If $t_1$ and $t_2$ are close to $t$, then $g_{f_0}$ admits a univalent extension on $D_{f_0}$, and then $g_f$ has a univalent extension on $D_f$ for every $f\in\piHbar$, 
%by the stability of dynamical rays (see Lemmas \ref{stability of external rays} and \ref{stability of internal rays}). 
% In the following, we always assume $t_1$ and $t_2$ are close to $t$ enough so that $g_f$ has a univalent extension on $D_f$ for every $f\in\piHbar$. 
See Figure \ref{figure Yf}.  
The constants $t_1,t_2,\rho$ are to be determined, and so is  $Y_f$. 

\begin{figure}[ht]
\centering
\includegraphics{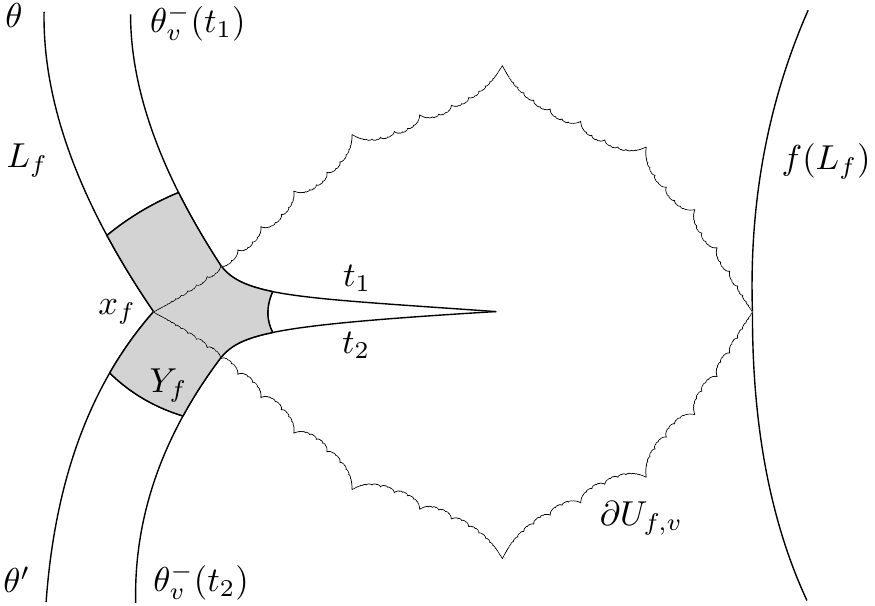}
\caption{A sketch of $Y_f$.}
\label{figure Yf}
\end{figure}

By Proposition \ref{expanding on external rays}, there exists a neighborhood $\mathcal{N}\subset \widehat{\mathcal{P}}^d$ of $h$ and a constant $m>0$ such that  
\begin{equation}
\label{derivative estimate on Lf}
\text{$\big|(g_f^m)'(z)\big|\geq 4$ for $f\in\Hbar\cap\mathcal{N}$ and $z\in L_f$.}
\end{equation}
% using of the maximum modulus principle
In particular, $\big|(g_h^m)'(c)\big|\geq 4$. 
Therefore we can choose $R>0$ so that: 
$g_h^m$ is injective on $\overline{\mathbb{D}(c,R)}$; 
$\mathbb{D}(c,R)\Subset g_h(\mathbb{D}(c,R))\Subset\cdots\Subset g_h^m(\mathbb{D}(c,R))$; 
and $\big|(g_h^m)'(z)\big|\geq3$ for every $z\in \mathbb{D}(c,R)$. 
Now fix $t_1,t_2$ close to $t$ and $\rho$ close to $1$ so that 
\begin{itemize}
\item $Y_h\Subset \mathbb{D}(c,R)$; 
\item $f$ is univalent on $Y_f$ and $f(Y_f)$ respectively for any $f\in\Hbar\cap\mathcal{N}$ (by the stability of dynamical rays: Lemmas \ref{stability of external rays} and \ref{stability of internal rays}); 
\item  $g_f^m$ has a univalent extension on a neighborhood of $\overline{Y_f}$ for any $f\in\Hbar\cap\mathcal{N}$. 
\end{itemize}
% By the stability of dynamical rays (see Lemmas \ref{stability of external rays} and \ref{stability of internal rays}), for any $f\in\piHbar$, the restriction of $f$ on $Y_f$ is univalent, thus $g_f$ has a univalent extension on a neighborhood of $\overline{Y_f}$. 
Since $\big|(g_h^m)'(z)\big|\geq3$ for every $z\in \overline{Y_h}$, shrinking $\mathcal{N}$ if necessary, we have 
\begin{equation}
\label{derivative estimate on the boundary of Yf}
\text{$\big|(g_f^m)'(z)\big|\geq 2$ for $f\in\Hbar\cap\mathcal{N}$ and $z\in \overline{(\partial Y_f) \setminus L_f}$,}
\end{equation}
where $\overline{(\partial Y_f) \setminus L_f}$ is close to $\overline{(\partial Y_h) \setminus L_h}$ in Hausdorff topology. 
Combining (\ref{derivative estimate on Lf}) and (\ref{derivative estimate on the boundary of Yf}), 
by the maximum modulus principle, 
\begin{equation}
\label{derivative estimate on Yf}
\text{$\big|(g_f^m)'(z)\big|\geq 2$ for $f\in\Hbar\cap\mathcal{N}$ and $z\in \overline{Y_f}$.}
\end{equation}

Let $f\in\Hbar\cap\mathcal{N}$. 
Then $$Y_f\cap \partial U_{f,v} = B^{-1}_{f,v}(e^{2\pi i(t_1,t_2)})\setminus \{x_f\}.$$ 
By Fact \ref{conf-md}, let $\mu_f$ be an $s_f$-conformal measure for $\partial U_{f,v}$, where $s_f=\dconf(\partial U_{f,v})$. 
For $k\geq0$, let $Y_f^{k}=g_f^{-m k}(Y_f)$. 
Since $f$ is univalent on $Y_f$ and $f(Y_f)$ respectively, 
by (\ref{transition formula for mu}) and $g_f=f^{-1}\circ f\circ f$, we have 
\begin{equation}
\label{transition formula for gf} 
\mu_f(Y_f)=\int_{Y_f^{k}}\big|(g^{m k}_f)'\big|^{s_f} d\mu_f. \end{equation}
Note that $s_f = \Hdim(\partial U_{f,v})\geq1$. 
Then by (\ref{derivative estimate on Yf}) and (\ref{transition formula for gf}), 
$$\mu_f(Y_f)\geq 2^k \mu_f(Y_f^k).$$

{\bf Step 3.} 
By the assumption $\{f_n\}_{n\geq1}\subset \mathcal{H}$ in the claim and (\ref{transition formula for mu}), for each $n\geq1$, 
we have $\mu_n(\{x_{f_n}\})=\mu_n(\{f_n(x_{f_n})\})=0$. 
(Since $h(c)$ is a repelling fixed point of $h$, we have $\mu(\{h(c)\})=0$, 
but we can not deduce $\mu(\{c\})=0$ by $\mu(\{h(c)\})=|h'(c)|^s \mu(\{c\})$.) 
Without loss of generality, we can assume $\{f_n\}_{n\geq1}\subset\mathcal{N}$. 

Given $k\geq 0$. 
Choose $r_k>0$ so that $\mathbb{D}(c,r_k)\Subset Y_h^k\cup\overline{{{S}}_h(\theta,\theta')}$. 
By the stability of dynamical rays and equipotential curves, there exists an $N_k>0$ so that for every $n> N_k$, 
we have $\mathbb{D}(c,r_k)\Subset Y_{f_n}^k\cup\overline{{{S}}_{f_n}(\theta,\theta')}$, 
hence $\mathbb{D}(c,r_k)\cap \partial U_{f_n,v}\subset Y_{f_n}^k\cup\{x_{f_n}\}$. 
For every $n>N_k$, 
because $\mu_n(\{x_{f_n}\})=0$ and $\mu_n(Y_{f_n})\leq 1$, we have 
$$\mu_n(\mathbb{D}(c,r_k)) \leq \mu_n(Y_{f_n}^k)+\mu_n(\{x_{f_n}\}) \leq 2^{-k}\mu_n(Y_{f_n})+\mu_n(\{x_{f_n}\})\leq 2^{-k}.$$
It follows that $\mu(\{c\})\leq \mu(\mathbb{D}(c,r_k))\leq \liminf_{n\rightarrow\infty}\mu_n(\mathbb{D}(c,r_k))\leq 2^{-k}$ (see (W4) in Appendix \ref{appendix weak topology on measures}). 
Letting $k\rightarrow\infty$ gives $\mu(\{c\})=0$. 
This proves the claim. 
\end{proof}

\section{The Hausdorff dimension formula}
\label{section the hd-formula}

Let $\mathcal{H}\subset\mathcal{F}$ be a capture hyperbolic component with center $f_0$. 
With the same notations as \S\ref{transver}, for $f \in\partial \mathcal{H}$, let 
\begin{align*}
& I^J(f) = \big\{(v,k)\in I{;~} c_{v,k}(f)\in J(f)\big\} \subset I, \\
& A^J(f) = \bigcup_{(v,k)\in I^J(f)} U_{f,\sigma^{r_v}(v)}\subset A(f). 
\end{align*}
Here $A^J(f)$ is the union of the bounded attracting Fatou components of $f$ associated with the free critical points in the Julia set $J(f)$.  
% For each $v\in\Tfn$, according to the straightening map (\ref{straightening map}), we see that $c_{v,k}\in J(f)$ for some $1\leq k<\delta(v)$ if and only if  $(\{v\}\times \crit(P_{f,v}))\cap J(\mathbf{P}_f)\neq\emptyset$, 
% thus $A^J(f)$ is independent of the critical marking $\mathbf{c}$ when $(f,\mathbf{c})\in\partial \mathcal{H}$. 
% the union of attracting periodic Fatou components of $f$
% the union of bounded immediate attracting basins of $f$
% For $(f,\mathbf{c})\in\partial\mathcal{H}$ in Example \ref{example of gf}, we have $V(f)=\{-1\}$ and  $A^J(f)=U_{f,-1} = U_f(c_1)$. 
The aim of this section is to prove Theorem \ref{hd-boundary}: 
$${\Hdim}(\partial \mathcal H)= 2 \Cdim (\mathcal H)-2+\max_{f\in\partial\mathcal{H}}{\Hdim}(\partial A^J(f)).$$

Combining Proposition \ref{continuity of Ufv} and Proposition \ref{continuous-hdim}, 
$\Hdim(\partial A^J(f))$ is upper semi-continuous with respect to $f\in\partial \mathcal H$, hence the maximum in the identity is achieved. 
See also Lemma \ref{taking-maximum}.

\subsection{Proof of the Hausdorff dimension formula} 
\label{proof-hdf}

To prove the Hausdorff dimension formula, we first study the properties of some subsets of $\partial\mathcal{H}$.  

Let $\partial_T\mathcal{H}\subset \partial \mathcal H$ consist of all $f\in\partial \mathcal H$ such that (see Proposition \ref{transversality-prop})
\begin{itemize}
\item $f\notin\Sigma(\mathcal{F})$; 
\item $\crit(f)\cap \partial A(f)=\emptyset$; and 
\item $f$ has no free critical relation.
\end{itemize}
% The following lemma implies most maps in $\partial \mathcal{H}$ have transversality. 

\begin{lemma} 
\label{partial-T-H}
The set $\partial_T\mathcal H$ is open and dense in $\partial\mathcal H$, 
with Hausdorff dimension $\Hdim(\partial_T\mathcal H) = \Hdim(\partial\mathcal H)$. 
\end{lemma}

\begin{proof} 
Let 
\begin{align*}
% &X_1 = \Sigma(\mathcal{F})\cap \partial\mathcal{H},\\
&X_1 = \big\{f\in\partial\mathcal{H}{;~} \crit(f)\cap \partial A(f)\neq\emptyset\big\}, \\
&X_2 = \big\{f\in\partial\mathcal{H}{;~} \text{$f$ has at least one free critical relation}\big\}.
\end{align*}
By Proposition \ref{continuity of Ufv}, the boundary $\partial A(f)$ is 
continuous with respect to $f\in\Hbar$ in Hausdorff topology, hence $X_1$ is closed. 
It is clear that $\Sigma(\mathcal{F})$ and $X_2$ are closed. 
So $\partial_T{\mathcal{H}} = \partial\mathcal{H}\setminus(\Sigma(\mathcal{F})\cup X_1\cup X_2)$ is open in $\partial\mathcal{H}$. 

% By Proposition \ref{behavior of critical orbits} and the stability of external rays (Lemma \ref{stability of external rays}), 
By Proposition \ref{behavior of critical orbits} and Lemma \ref{invariant external angles}, 
the set $X_1$ is contained in the union $Y_1$ of finitely many algebraic subsets of $\mathcal{F}$: 
each is defined by the critical relation of the form
$f^{q+p}(c_{v,k}(f))=f^{q}(c_{v,k}(f))$ for some $p,q\geq 1$ and some $(v,k)\in I$. 
According to Definition \ref{defi-free-crit-rela}, 
the set $X_2$ is contained in the union $Y_2$ of finitely many algebraic subsets of $\mathcal{F}$: 
each is defined by the free critical relation of the form 
$c_{v',k'}(f) = f^j(c_{v,k}(f))$ for some $(v,k)\neq(v',k')\in I$ and some $0\leq j<r_{v}$. 
Now let $Y=\Sigma(\mathcal{F})\cup Y_1\cup Y_2$. 
Then $Y$ is an algebraic subset of $\mathcal{F}$ and $\partial\mathcal{H}\setminus\partial_T\mathcal{H} \subset Y\cap\partial\mathcal{H}$.  
% Note that these two sets may be not equal due to the form of Y1. 

As the proof of Lemma \ref{sing-in-part-H}, 
let $\mathcal{F}_0$ be the irreducible component of $\mathcal{F}$ containing $\mathcal{H}$. 
Since $\mathcal{H}$ is not contained in $Y$, 
we have $Y\cap\mathcal{F}_0\subsetneqq \mathcal{F}_0$.
Combining this with the irreducibility of $\mathcal{F}_0$, we have 
$\Cdim(Y\cap\mathcal{F}_0)
<\Cdim(\mathcal{F}_0) = \Cdim(\mathcal{H})$. 
It follows that  
\begin{align*}
\Hdim(\partial\mathcal{H}\setminus\partial_T\mathcal{H})
& \leq \Hdim(Y\cap\partial\mathcal{H})
\leq \Hdim(Y\cap\mathcal{F}_0)\\
& = 2\Cdim(Y\cap\mathcal{F}_0)
\leq 2\Cdim(\mathcal{H})-2.
\end{align*}
Note that for a subset of a Euclidean space, 
its Hausdorff dimension is not less than its topological dimension \cite[Theorem VII 2]{Hurewicz-Wallman}. 
% (which is a topological invariant). 
Since $\partial\mathcal{H}$ is homeomorphic to the sphere $S^{2\Cdim (\mathcal H)-1}$ 
(Theorem \ref{top-boundary}), 
the set $\partial_T{\mathcal{H}}$ is dense in $\partial\mathcal{H}$; 
otherwise, its complement  $\partial\mathcal{H}\setminus\partial_T\mathcal{H}$ includes a subset with Hausdorff dimension at least $2\Cdim(\mathcal{H})-1$, 
which contradicts $\Hdim(\partial\mathcal{H} \setminus \partial_T\mathcal{H}) \leq 2\Cdim(\mathcal{H})-2$. 
Similarly, we have  $\Hdim(\partial_T\mathcal H) = \Hdim(\partial\mathcal H)$. 
The proof is completed. 
\end{proof}

For each $(v,k)\in I$, we define 
$$\partial_{v,k}\mathcal{H}=\big\{f\in \partial\mathcal{H}{;~} I^J(f)=\{(v,k)\}\big\}.$$

\begin{lemma}
\label{partial-vk-H}
For any $f\in\partial\mathcal{H}$, we have $f\in \bigcap_{(v,k)\in I^J(f)}\overline{\partial_{v,k} \mathcal H}$. 
\end{lemma}

\begin{proof} 
By the proof of Theorem \ref{top-boundary}, the straightening map $\chi:\Hbar\rightarrow \modelf$ defined by (\ref{straightening map}) is a homeomorphism. 
Recall that $W: \widehat{\mathcal{P}}_0^{T(f_0)}\rightarrow \mathbb{C}^{\Cdim (\mathcal H)}$ 
and $\widetilde W: \Hbar \rightarrow {\overline{\mathbb{D}}}^{\Cdim (\mathcal H)}$ are maps defined by (\ref{crit-val-Tf0}) and (\ref{crit-val-H}) respectively. 
Then the result follows from $\widetilde W = W\circ\chi$ and Lemma \ref{partial-vk-model}. 
\end{proof}

\begin{lemma}
\label{taking-maximum}
We have the following equality 
$$\max_{f\in \partial \mathcal H}{\Hdim}(\partial A^J(f)) = 
\max_{(v,k)\in I} \max_{f\in\overline{\partial_{v,k}\mathcal{H}}} 
{\Hdim}(\partial U_{f,\sigma^{r_v}(v)}).$$
\end{lemma}

\begin{proof}
By the continuity of Hausdorff dimension (Proposition \ref{continuous-hdim}), there is an index $(v_*,k_*)\in I$ and a map $f_*\in\overline{\partial_{v_*,k_*}\mathcal{H}}$ such that 
\begin{equation}
\label{taking-maximum-eq1}
\max_{(v,k)\in I} \max_{f\in\overline{\partial_{v,k}\mathcal{H}}} 
{\Hdim}(\partial U_{f,\sigma^{r_v}(v)})
= \Hdim(\partial U_{f_*,\sigma^{r_{v_*}}(v_*)}). 
\end{equation}
Since $f_*\in\overline{\partial_{v_*,k_*}\mathcal{H}}$, we have $(v_*,k_*)\in I^J(f_*)$, hence 
\begin{equation}
\label{taking-maximum-eq2}
\Hdim(\partial U_{f_*,\sigma^{r_{v_*}}(v_*)})
\leq \max_{(v,k)\in I^J(f_*)} {\Hdim}(\partial U_{f_*,\sigma^{r_v}(v)})
= {\Hdim}(\partial A^J(f_*)).
\end{equation}
For any $f\in \partial\mathcal{H}$, we have 
\begin{align*}
{\Hdim}(\partial A^J(f))
&= \max_{(v,k)\in I^J(f)} {\Hdim}(\partial U_{f,\sigma^{r_v}(v)})\\
&\leq \max_{(v,k)\in I^J(f)} \max_{g \in\overline{\partial_{v,k}\mathcal{H}}} 
{\Hdim}(\partial U_{g,\sigma^{r_v}(v)}) 
\quad(\text{by Lemma \ref{partial-vk-H}})\\
&\leq \max_{(v,k)\in I} \max_{g \in\overline{\partial_{v,k}\mathcal{H}}} 
{\Hdim}(\partial U_{g,\sigma^{r_v}(v)})\\
&\leq {\Hdim}(\partial A^J(f_*))
\quad(\text{by (\ref{taking-maximum-eq1}) and (\ref{taking-maximum-eq2})}). 
\end{align*}
It follows that 
$$
\sup_{f \in \partial \mathcal H}{\Hdim}(\partial A^J(f)) 
= \max_{(v,k)\in I} \max_{f \in\overline{\partial_{v,k}\mathcal{H}}} 
{\Hdim}(\partial U_{f,\sigma^{r_v}(v)})
= {\Hdim}(\partial A^J(f_*)). 
$$
Hence the supremum in the left side can be attained at $f_*$, and can be replaced by maximum.
\end{proof}

\begin{proof} 
[Proof of Theorem \ref{hd-boundary}]
Let $f\in\partial_T\mathcal{H}$, and let $\mathcal N_r(f)\subset\mathcal{F}$ be an $r$-neighborhood of $f$ (in a suitable metric for parameters) so that 
\begin{itemize} 
% \item the collection $(U_{g,v})_{g\in \mathcal N_r(f), v\in\Tfp}$ is well-defined;
\item the holomorphic motion $h: \mathcal N_r(f)\times \overline{A(f)}\rightarrow \mathbb C$ is well-defined; 
\item $g^{r_v}(c_{v,k}(g))\in U_{g,\sigma^{r_v}(v)}$ for any $g \in\mathcal N_r(f)$ and $(v,k)\in I\setminus I^J(f)$. 
\end{itemize}  
For each $(v,k)\in I^J(f)$, let 
$$\mathcal{X}_r^{v,k}(f) = \big\{g \in\mathcal N_r(f){;~} 
g^{r_v}(c_{v,k}(g))\in\partial U_{g,\sigma^{r_v}(v)} \big\}.$$
By Proposition \ref{continuity of Ufv}, we have 
\begin{equation}
\label{AJf-leq-1}
\mathcal N_r(f)\cap\partial\mathcal{H}
\subset\bigcup_{(v,k)\in I^J(f)} \mathcal{X}_r^{v,k}(f).
\end{equation}

Fix $\varepsilon>0$. 
Let $(v,k)\in I^J(f)$. 
Since $f \in\partial_T\mathcal{H}$, by the transversality (Proposition \ref{transversality-prop}), we know that
$$d(g^{r_v}(c_{v,k}(g))-h(g,f^{r_v}(c_{v,k}(f))))|_{g = f}\neq 0.$$
% Using Proposition \ref{HD-trans} for
Applying Proposition \ref{HD-trans} to the holomorphic motion $h: \mathcal N_r(f)\times \partial U_{f,\sigma^{r_v}(v)}\rightarrow \mathbb C$, there is an $r(f)>0$ independent of $(v,k)$ so that 
\begin{equation}
\label{AJf-leq-2}
{\Hdim}{\left(\mathcal{X}_{r(f)}^{v,k}(f)\right)}
\leq 2\Cdim (\mathcal H)-2+{\Hdim}(\partial U_{f,\sigma^{r_v}(v)}) + \varepsilon.
\end{equation}
Here we use implicitly the homogeneity of $\partial U_{f,\sigma^{r_v}(v)}$: for any $x\in \partial U_{f,\sigma^{r_v}(v)}$ and any $\rho>0$, we have ${\Hdim}(\mathbb{D}(x,\rho)\cap \partial U_{f,\sigma^{r_v}(v)})={\Hdim}(\partial U_{f,\sigma^{r_v}(v)})$.
By (\ref{AJf-leq-1}) and (\ref{AJf-leq-2}), we have 
\begin{align*} 
{\Hdim}(\mathcal N_{r(f)}(f)\cap\partial\mathcal{H})
&\leq \max_{(v,k)\in I^J(f)} {\Hdim}{\big(\mathcal{X}_{r(f)}^{v,k}(f)\big)}\\
&\leq 2\Cdim (\mathcal H)-2+\max_{(v,k)\in I^J(f)}{\Hdim}(\partial U_{f,\sigma^{r_v}(v)}) + \varepsilon\\
&= 2\Cdim (\mathcal H)-2+{\Hdim}(\partial A^J(f)) + \varepsilon.
\end{align*}
Since $\partial_T \mathcal H$ is an open subset of  $\partial\mathcal H$ (Lemma \ref{partial-T-H}), it is the union of a sequence of increasing compact sets. 
Note that $\big\{\mathcal N_{r(f)}(f){;~} f\in  \partial_T \mathcal H\big\}$ is an open covering of $\partial_T \mathcal H$. 
There is a countable covering 
$\big\{\mathcal N_{r(f_n)}(f_n){;~} n\geq 1\big\}\subset \big\{\mathcal N_{r(f)}(f){;~} f \in  \partial_T \mathcal H\big\}$ of $\partial_T \mathcal H$.  
By the property of Hausdorff dimension, we have 
\begin{align*} 
{\Hdim}(\partial_T\mathcal H)
&\leq {\Hdim}{\bigg(\bigcup_{n\geq1} (\mathcal N_{r(f_n)}(f_n)\cap \partial\mathcal{H})\bigg)}\\
&=\sup_{n\geq 1} \Hdim (\mathcal N_{r(f_n)}(f_n)\cap \partial\mathcal{H})\\
&\leq  2\Cdim (\mathcal H)-2+ \sup_{n\geq 1} {\Hdim}(\partial A^J(f_n)) + \varepsilon\\
&\leq  2\Cdim (\mathcal H)-2+\max_{f \in \partial \mathcal H} {\Hdim}(\partial A^J(f)) + \varepsilon.
\end{align*}
By Lemma \ref{partial-T-H}, letting $\varepsilon\rightarrow0$ gives 
$$\Hdim(\partial\mathcal{H}) 
= \Hdim(\partial_T\mathcal{H})
\leq 2\Cdim (\mathcal H)-2+\max_{f \in \partial \mathcal H} {\Hdim}(\partial A^J(f)).$$
This proves one part of the Hausdorff dimension formula.  

To prove the other part, let $(v,k)\in I$ and let $f \in \partial_{v,k}\mathcal{H} \cap \partial_T\mathcal{H}$. 
By Proposition \ref{capture-separation}, we can choose $r>0$ small enough so that
$$\mathcal{X}_r^{v,k}(f)\subset\partial\mathcal{H}.$$
By Propositions \ref{HD-trans} and \ref{transversality-prop} again, we have 
$$ {\Hdim}(\partial \mathcal H)
\geq {\Hdim}(\mathcal{X}_r^{v,k}(f))
\geq  2\Cdim (\mathcal H)-2+{\Hdim}(\partial U_{f,\sigma^{r_v}(v)}).$$
It is clear that $\partial_{v,k}\mathcal H$ is open in $\partial\mathcal{H}$. 
By Lemma \ref{partial-T-H}, we have $\overline{\partial_{v,k}\mathcal{H} \cap \partial_T\mathcal{H}}
=\overline{\partial_{v,k}\mathcal{H}}$. 
Then by Proposition \ref{continuous-hdim} and Lemma \ref{taking-maximum}, we have 
\begin{align*}
{\Hdim}(\partial \mathcal H)
&\geq 2\Cdim (\mathcal H)-2+
\max_{(v,k)\in I} \sup_{f
\in\partial_{v,k}\mathcal{H}\cap\partial_T\mathcal{H}} 
{\Hdim}(\partial U_{f,\sigma^{r_v}(v)})\\
&= 2\Cdim (\mathcal H)-2+
\max_{(v,k)\in I} \max_{f
\in \overline{\partial_{v,k}\mathcal{H} \cap \partial_T\mathcal{H}}} 
{\Hdim}(\partial U_{f,\sigma^{r_v}(v)})\\
&= 2\Cdim (\mathcal H)-2+
\max_{(v,k)\in I} \max_{f
\in\overline{\partial_{v,k}\mathcal{H}}} 
{\Hdim}(\partial U_{f,\sigma^{r_v}(v)})\\
&= 2\Cdim (\mathcal H)-2+ 
\max_{f \in \partial \mathcal H}
{\Hdim}(\partial A^J(f)). 
\end{align*}
This completes the proof. 
\end{proof}

\subsection{A concluding example}

Let $2\leq m\leq d-1$ and $\mathcal F_m$ be the family of polynomials 
$$f(z)=a_m z^m + \cdots + a_{d-1} z^{d-1} + z^d, ~(a_m,\dots, a_{d-1})\in \mathbb C^{d-m}.$$
For each $f\in\mathcal{F}_m$, let $U_f(0)$ denote the immediate attracting basin of $0$, and let $A_f(0)$ denote the whole attracting basin of $0$. 
A \emph{capture hyperbolic component} in $\mathcal F_m$ is a hyperbolic component with a representative 
polynomial $f$ satisfying $\deg(f|_{U_f(0)})=m$ and $\crit(f)\subset A_f(0)$. 

% With the same proof as the case $\mathcal{F}\subset\widehat{\mathcal{P}}^d$, we can establish 

\begin{corollary} 
Let $\mathcal H$ be a capture hyperbolic component in $\mathcal F_m$. 
Then $\partial\mathcal H$ is homeomorphic to the sphere $S^{2d-2m-1}$, 
with Hausdorff dimension 
$${\Hdim}(\partial \mathcal H)= 2(d-m-1)+\max_{f\in \partial \mathcal H}{\Hdim}(\partial U_f(0)).$$
\end{corollary}

\begin{proof}
[Sketch of proof]
Similar to the proof of Theorem \ref{top-boundary}, we can define a straightening map $\chi:\Hbar\rightarrow \mathcal{C}(\mathcal{P}_0^T)$, 
and show it is a homeomorphism. 
Then $\partial\mathcal{H}$ is homeomorphic to $\partial \mathcal{C}(\mathcal{P}_0^T)$, 
which is homeomorphic to the sphere $S^{2d-2m-1}$ by Remark \ref{remark on C(P0T)}. 

For each $\mathbf{c} = (c_1,\dots,c_{d-m})\in\mathbb C^{d-m}$, let 
$$f_{\mathbf{c}}(z)=d\int_{0}^{z}
\zeta^{m-1}(\zeta-c_{1})\cdots(\zeta-c_{d-m})d\zeta.$$
Define $\pi: \mathbb C^{d-m}\rightarrow \mathcal F_m$ by $\pi(\mathbf{c})=f_{\mathbf{c}}$. 
Let $\widetilde{\mathcal{H}}$ be a connected component of $\pi^{-1}(\mathcal{H})$. 
With the same proof as the case $\mathcal{F}\subset\widehat{\mathcal{P}}^d$, the boundary  $\partial \widetilde{\mathcal{H}}$ is homeomorphic to the sphere $S^{2d-2m-1}$, 
with Hausdorff dimension 
$${\Hdim}(\partial \widetilde{\mathcal{H}})= 2(d-m-1)+\max_{\mathbf{c}\in \partial\widetilde{\mathcal{H}}}{\Hdim}(\partial U_{f_{\mathbf{c}}}(0)).$$  
Let $\crit(\pi)$ denote the critical set of the proper holomorphic map $\pi$. 
By \cite[\S15.1]{Rudin}, we have $\Hdim(\crit(\pi))\leq 2(d-m-1)$. 
Then $\pi(\partial \widetilde{\mathcal{H}})=\partial{\mathcal{H}}$ gives $\Hdim(\partial \widetilde{\mathcal{H}}) = \Hdim(\partial{\mathcal{H}})$. 
The Hausdorff dimension formula of $\partial\mathcal{H}$ follows from that of $\partial \widetilde{\mathcal{H}}$. 
\end{proof}

\addtocontents{toc}{\protect\contentsline{section}{Appendices}{\pageref*{appendix smooth points}}{appendix.A}}
\addtocontents{toc}{\protect\setcounter{tocdepth}{0}}

\appendix
\section{Smooth points and singular points} 
\label{appendix smooth points}

% irreducible component 
%
%An algebraic set $X$ is \emph{reducible} if there exist 
%algebraic sets $X_1,X_2\subsetneqq X$ such that $X = X_1 \cup X_2$. 
%Otherwise $X$ is \emph{irreducible}. 
%% Any algebraic set $X$ in $\mathbb{C}^d$ is a finite union of 
%% irreducible algebraic sets (see \cite[Theorem 1.4]{Shafarevich}). 
%An \emph{irreducible component} of an algebraic set $X$ is an irreducible algebraic set $X_1\subset X$ such that if $X_2$ is an irreducible algebraic set satisfying $X_1\subset X_2\subset X$, then $X_1=X_2$. 
%Any algebraic set $X$ in $\mathbb{C}^d$ is a finite union of its irreducible components (see \cite[\S3.1]{Shafarevich}). 

% smooth, singular 
Let $\mathbb{C}[z_1,\dots,z_n]$ be a polynomial ring in $n$-letters. 
An \emph{algebraic set} in $\mathbb{C}^n$ is the set of all common zeros of 
a finite number of polynomials in $\mathbb{C}[z_1,\dots,z_n]$. 
Let $X$ be a nonempty algebraic set in $\mathbb{C}^n$. 

A point $x\in X$ is called \emph{smooth}, of dimension $k$, 
if there is an open neighborhood $\mathcal{N}$ of $x$ such that 
$\mathcal{N}\cap X$ is a complex submanifold of $\mathbb{C}^n$ of dimension $k$. 
(The degenerate case $k=0$ corresponds to that $x$ is an isolated point in $X$.) 
The \emph{dimension} $\Cdim(X)$ of $X$ is the largest $k$ such that $X$ has a smooth point of dimension $k$. 
We set $\Cdim(\emptyset)=0$. 

A point $x\in X$ is called \emph{singular}, if it is not smooth. 
The \emph{singular locus} $\Sigma(X)$ is defined to be the set of all singular points of $X$. 
Note that $\Sigma(X)$ is also an algebraic set in $\mathbb{C}^n$. 
If $\Cdim(X)\geq1$, then $\Cdim(\Sigma(X))<\Cdim(X)$. 
So we can inductively divide $\Sigma(X)$ 
into a finite disjoint union of complex submanifolds of $\mathbb{C}^n$ with dimension less than 
$\Cdim(X)$ (see \cite[Theorem 2.4]{Mil-Sing}). 
Therefore $\Hdim(X)=2\Cdim(X)$. 

For more details of algebraic sets, we refer to 
\cite[\S2]{Mil-Sing}, \cite[\S{II}]{Lang}.

% If $X$ is nonempty, then $\Sigma(X)\subsetneqq X$ is also an algebraic set in 
% $\mathbb{C}^n$ (see \cite[Lemma 2.2]{Mil-Sing}). 

%\begin{lemma}
%Let $X, Y$ be nonempty algebraic sets in $\mathbb{C}^n$. 
%Assume $Y$ is irreducible. Then we have: 
%\begin{enumerate}
%\item $\Cdim(\Sigma(X))<\Cdim(X)$;
%\item $Y\setminus \Sigma(Y)$ is connected;
%\item if $X\subset Y$ and $\Cdim(X)=\Cdim(Y)$, then $X=Y$. 
%\end{enumerate}
%\end{lemma}

Given $d\geq 3$. 
Let $1\leq\ell\leq d-2$ and $n_1,\dots,n_\ell\geq1$. 
Recall that % $\mathcal{F}$ is an algebraic set in $\mathbb{C}^d$ defined by 
\begin{align*}
\mathcal{F}&=\big\{f\in \widehat{\mathcal{P}}^d{;~}
f^{n_j}(c_j(f))=c_j(f), ~\forall~1\leq j\leq\ell  \big\}\\
&=\big\{(f,\mathbf{c})\in \widehat{\mathcal{P}}^d{;~}
f^{n_j}(c_j)=c_j, ~\forall~1\leq j\leq\ell  \big\}\\ 
&\cong \big\{(\mathbf{c},a)\in \mathbb{C}^d{;~}
\text{$c_1+\cdots+c_{d-1}=0$ 
and $f_{\mathbf{c},a}^{n_j}(c_j)=c_j, ~\forall~ 1\leq j\leq\ell$}  \big\}.
\end{align*}
% $\sum_{k=1}^{d-1}c_k=0$ 
So we can regard $\mathcal{F}$ as an algebraic set in $\mathbb{C}^d$. 

\begin{example}
\label{example of F}
Let $$\mathcal{F}=\big\{(\mathbf{c},a)\in \mathbb{C}^5{;~}c_1+c_2+c_3+c_4=0,
\  f_{\mathbf{c},a}(c_1)=c_1, 
\  f_{\mathbf{c},a}(c_2)=c_2\big\}.$$ 
If $c_1=c_2$, then $c_1,c_2,a$ can be represented by  $(c_3,c_4)\in\mathbb{C}^2$ as follows: 
\begin{align*}
&c_1=c_2=-(c_3+c_4)/2;\\
&a = a - f_{\mathbf{c},a}(c_1) + c_1 = (c_3 + c_4)(3 c_3^4 + 32 c_3^3 c_4 + 58 c_3^2 c_4^2 + 32 c_3 c_4^3 + 3 c_4^4 - 48)/96. 
\end{align*}
If $c_1\neq c_2$, then $c_4 = -(c_1+c_2+c_3)$ and 
$\frac{(f_{\mathbf{c},a}(c_1)-c_1)-(f_{\mathbf{c},a}(c_2)-c_2)}{c_1-c_2}=0$ give 
$$G(c_1,c_2,c_3) := (c_1-c_2)^2 (5 c_3^2 + 5 (c_1+c_2) c_3 -(4c_1^2 + 7c_1c_2+4c_2^2)) - 6 = 0. $$
A direct computation shows that the equation ${\left(\frac{\partial G}{\partial c_1},\frac{\partial G}{\partial c_2},\frac{\partial G}{\partial c_3}\right)} = (0,0,0)$ has no solution with $c_1\neq c_2$. 
It follows that $\mathcal{F} = \{(\mathbf{c},a)\in\mathcal{F}{;~}c_1 = c_2\}
\sqcup \{(\mathbf{c},a)\in\mathcal{F}{;~}c_1\neq c_2\}$ has no singular point, i.e. $\Sigma(\mathcal{F})=\emptyset$. 
\end{example}
%  To show $G$ is irreducible, qudratic equation, solve(G,c3) 
% Since $G$ is irreducible, $\mathcal{F}$ has two connected components $\{(\mathbf{c},a)\in\mathcal{F}{;~}c_1=c_2\}$ and $\{(\mathbf{c},a)\in\mathcal{F}{;~}c_1\neq c_2\}$. 
% 

\begin{lemma}
\label{basic facts of H}
Let $\mathcal{H}\subset \mathcal F$ be a capture hyperbolic component.  
Then $\mathcal{H}$ is a complex submanifold of $\mathbb{C}^d$ with dimension $\Cdim(\mathcal{H}) = d-1-\ell$.   
\end{lemma}

\begin{proof}
There is a unique bounded hyperbolic component $\widetilde{\mathcal{H}}$ 
of $\widehat{\mathcal{P}}^d$ containing $\mathcal{H}$. 
By \cite{MilHyperCompo}, there is a unique $f_0\in\widetilde{\mathcal{H}}$ 
so that $f_0$ is postcritically finite. 
Let $\mathcal{H}_0$ denote the hyperbolic component of $\widehat{\mathcal{P}}^{T(f_0)}$ containing $0$. 
By \cite[\S8]{MilHyperCompo}, there is a biholomorphic map $\chi: \widetilde{\mathcal{H}} \rightarrow \mathcal{H}_0$, which 
preserves the dynamics on bounded Fatou components of $f\in \widetilde{\mathcal{H}}$.  
Note that $\widehat{\mathcal{P}}_0^{T(f_0)}$ is a complex affine space of dimension $d-1-\ell$. 
The capture hyperbolic component $\mathcal{H} = \chi^{-1}(\mathcal{H}_0\cap \widehat{\mathcal{P}}_0^{T(f_0)})$ 
is a complex submanifold of $\mathbb{C}^d$ with the same dimension. 
% which will be denoted by $\Cdim(\mathcal{H}) = d-1-\ell$. 
% Furthermore, the hyperbolic components  $\mathcal{H}_0\big(\widehat{\mathcal{P}}_0^{T(f_0)}\big)$ and $\mathcal{H}$ are topological cells of dimension $2\Cdim(\mathcal{H})$. 
\end{proof}

\begin{lemma}
\label{sing-in-part-H}
Let $\mathcal{H}\subset \mathcal F$ be a capture hyperbolic component. 
Then 
$$\Hdim(\Sigma(\mathcal{F})\cap\partial\mathcal{H})
\leq 2(\Cdim(\mathcal{H})-1).$$
\end{lemma}

\begin{proof}
Because $\mathcal{H}$ is open in $\mathcal{F}$ and 
it is a complex submanifold of $\mathbb{C}^d$ (Lemma \ref{basic facts of H}),  
there exists a unique \emph{irreducible component} $\mathcal{F}_0$ 
of $\mathcal{F}$ containing $\mathcal{H}$ 
(see, e.g. \cite[\S{II}.2, Theorem 3]{Lang}). 
Since $\Sigma(\mathcal{F})\cap\mathcal{H}=\emptyset$,  
we have $\Sigma(\mathcal{F})\cap\mathcal{F}_0\subsetneqq \mathcal{F}_0$.
Combining this with the irreducibility of $\mathcal{F}_0$, we have 
$\Cdim(\Sigma(\mathcal{F})\cap\mathcal{F}_0)
<\Cdim(\mathcal{F}_0) = \Cdim(\mathcal{H})$ 
(see, e.g. \cite[\S{II}.3, Corollary 1]{Lang}). 
It follows that \begin{align*}\Hdim(\Sigma(\mathcal{F})\cap\partial\mathcal{H})&\leq 
\Hdim(\Sigma(\mathcal{F})\cap\mathcal{F}_0)
= 2\Cdim(\Sigma(\mathcal{F})\cap\mathcal{F}_0)\\
&\leq 2(\Cdim(\mathcal{H})-1).\end{align*}
The proof is completed. 
\end{proof}

\section{Weak topology on measures}
\label{appendix weak topology on measures}

The $\sigma$-algebra of Borel subsets of $\widehat{\mathbb{C}}$ will be denoted by $\mathscr{B}(\widehat{\mathbb{C}})$. 
We shall denote by $\mathscr{M}(\widehat{\mathbb{C}})$ the collection of all probability measures 
defined on the measurable space $(\widehat{\mathbb{C}},\mathscr{B}(\widehat{\mathbb{C}}))$. 
Then $\mathscr{M}(\widehat{\mathbb{C}})$ is a compact metric space as a topological space with the \emph{weak topology}. 
For the definition and more facts of the weak topology in the space of measures, we refer to \cite[\S2.6]{Partha}. 
Let $\mathscr{C}(\widehat{\mathbb{C}})$ be the space of all real valued continuous functions on $\widehat{\mathbb{C}}$. 
For $\mu_n,\mu\in \mathscr{M}(\widehat{\mathbb{C}})$ ($n\in\mathbb{N}$), the following equivalent conditions are useful in computations. 
\begin{enumerate}[\indent (W1)]
\item $\mu_n\rightarrow\mu$ in the weak topology;
\item For all $g\in \mathscr{C}(\widehat{\mathbb{C}})$, $\lim_{n\rightarrow\infty}\int_{\widehat{\mathbb{C}}} g d\mu_n=\int_{\widehat{\mathbb{C}}} g d\mu$;
\item For every closed set $K$, $\limsup_{n\rightarrow\infty}\mu_n(K)\leq\mu(K)$; 
\item For every open set $G$, $\liminf_{n\rightarrow\infty}\mu_n(G)\geq\mu(G)$; 
\item For every $A\in\mathscr{B}(\widehat{\mathbb{C}})$ with $\mu(\partial A)=0$, $\lim_{n\rightarrow\infty}\mu_n(A)=\mu(A)$. 
\end{enumerate}

% Given $f\in\piHbar$, $v\in\Tfp$ and let $U=U_{f,v}$. 
% Without loss of generality, assume $f(U)=U$.  
The following lemma is an analogue of the closeness of conformal measures for rational maps. 
We give a proof as that of \cite[Lemma]{Przy05}. 

\begin{lemma}
\label{weak limit of conformal measure}
Let $f$ be a semi-hyperbolic polynomial of degree $d\geq2$, and suppose $U$ is a bounded invariant Fatou component of $f$. 
Suppose $\mu_n$ is a $t_n$-conformal measure for $\partial U$ ($n\in\mathbb{N}$) such that 
$t_n\rightarrow t$ and $\mu_n\rightarrow\mu$ in weak topology as $n\rightarrow\infty$. 
Then $\mu$ is a $t$-conformal measure for $\partial U$. 
\end{lemma}

\begin{proof}
Let $E$ be a Borel subset of $\partial U$ such that $f|_E$ is injective. 
By decomposing $E$ into finitely many critical points and 
countably many Borel sets with positive distance to $\crit(f)$, 
we just need to consider the following two cases.  
In either case, we will check that $\mu(f(E))=\int_{E}|f'|^t d\mu$. 

\vspace{6pt}
{\bf Case 1.}
{\it 
The set $E$ is contained in a small disk $D$ with the following properties: 
the restriction $f:\overline{D}\rightarrow f(\overline{D})$ is a homeomorphism; 
$\overline D\cap \crit(f)=\emptyset$; 
and $\overline{D}\cap (f^{-1}(\overline{U})\setminus \overline{U})=\emptyset$. 
}
\vspace{6pt}

By the regularity of measures, for every $\varepsilon>0$, there exists a compact set $K$ and an open set $G$ such that 
$K\subset E\subset G\subset D$, $\mu(G)-\mu(K)<\varepsilon$ and $\mu(f(G))-\mu(f(K))<\varepsilon$. 
Choose a continuous function $\chi:\widehat{\mathbb{C}} \rightarrow [0,1]$ such that $\chi$ is $1$ on $K$ and $0$ on $\widehat{\mathbb{C}}\setminus G$. 
Since $\partial \chi^{-1}([s,1]), \partial \chi^{-1}((s,1])\subset \chi^{-1}(s)$, 
there exist $0<s_1<s_2<1$ such that for $\widetilde G:=\chi^{-1}((s_1,1])$ and $\widetilde K:=\chi^{-1}([s_2,1])$, 
$$\mu(\partial\widetilde K)=\mu(\partial\widetilde G)=\mu(\partial f(\widetilde K))=\mu(\partial f(\widetilde G))=0.$$
Then by (W5), we have $\lim_{n\rightarrow\infty}\mu_n(f(\widetilde K))=\mu(f(\widetilde K))$. 
At the same time, it follows from $\overline{D}\cap (f^{-1}(\overline{U})\setminus \overline{U})=\emptyset$ 
that $\mu_n(f(\widetilde K))=\int_{\widetilde K} |f'|^{t_n} d\mu_n$  
(see Remark \ref{Borel subsets of the plane}).  

Consider another continuous function $\widetilde\chi:\widehat{\mathbb{C}} \rightarrow [0,1]$ such that 
$\widetilde\chi$ is $1$ on $\widetilde K$ and $0$ on $\widehat{\mathbb{C}} \setminus\widetilde G$. 
Choose $M>0$ such that $|f'(z)|^{t_n},|f'(z)|^t\leq M$ for any $z\in \overline{D}$ and $n\in\mathbb{N}$. 
Because $\mu(\partial(\widetilde G\setminus\widetilde K))=\mu(\partial\widetilde G\cup \partial\widetilde K)=0$, 
by (W5) again, when $n$ is large enough, we have $\mu_n(\widetilde G\setminus\widetilde K)\leq2\mu(\widetilde G\setminus\widetilde K)$, hence 
\begin{align*}
\left| \int_{\widetilde K} |f'|^{t_n} d\mu_n - \int_{\mathbb{C}} \widetilde\chi|f'|^{t_n} d\mu_n \right|
& =    \int_{\widetilde G\setminus\widetilde K} \widetilde\chi|f'|^{t_n} d\mu_n
  \leq \mu_n(\widetilde G\setminus\widetilde K)M\\
& \leq 2\mu(\widetilde G\setminus\widetilde K)M
  <    2\varepsilon M.
\end{align*}
Since $\overline{D}\cap\crit(f)=\emptyset$, the function $\widetilde\chi|f'|^{t_n}$ 
converges to $\widetilde\chi|f'|^{t}$ uniformly on $\overline{D}$ (even if $t=0$). 
By (W2), we have $\lim_{n\rightarrow\infty}\int_{\mathbb{C}} \widetilde\chi|f'|^t d\mu_n=\int_{\mathbb{C}} \widetilde\chi|f'|^t d\mu$. 
For any $\varepsilon'>0$, when $n$ is large enough, 
\begin{align*}
&{\left|\mu(f(E))-\int_E |f'|^t d\mu\right|}\\
\leq~ & {\left|\mu(f(E))-\mu(f(\widetilde K))\right|}
+ {\left|\mu(f(\widetilde K))-\mu_n(f(\widetilde K))\right|}\\
      &+ {\left| \int_{\widetilde K} |f'|^{t_n} d\mu_n - \int_{\mathbb{C}} \widetilde\chi|f'|^{t_n} d\mu_n \right|}
+ {\left| \int_{D} \widetilde\chi|f'|^{t_n} - \widetilde\chi|f'|^{t} d\mu_n \right|}\\
      &+ {\left| \int_{\mathbb{C}} \widetilde\chi|f'|^{t} d\mu_n   - \int_{\mathbb{C}} \widetilde\chi|f'|^{t} d\mu   \right|}
+ {\left| \int_{\mathbb{C}} \widetilde\chi|f'|^{t} d\mu     - \int_E |f'|^{t} d\mu \right|} \\
\leq~ & \varepsilon + \varepsilon' + 2\varepsilon M + \varepsilon' + \varepsilon' + \varepsilon M. 
\end{align*}
Letting $\varepsilon,\varepsilon'\rightarrow0$ gives $\mu(f(E))=\int_E |f'|^t d\mu$. 

\vspace{6pt}
{\bf Case 2.}
{\it $E=\{c\}$, where $c$ is a critical point in $\partial U$ (if it exists).}
\vspace{6pt}

Clearly $\mu(\{f(c)\})=\int_{\{c\}}|f'|^t d\mu$ is equivalent to $\mu(\{f(c)\})=0$. 

Firstly, assume $t\neq0$. 
Note that $f|_{\overline{U}}$ is locally injective at any point in $\partial U$. 
There are $a,b>0$ such that for $r>0$ small enough and $n$ large enough, 
\begin{align*}
\mu_n(\mathbb{D}(f(c),r))=\int_{\Comp_c f^{-1}( \mathbb{D}(f(c),r) )} |f'|^{t_n} d\mu_n\leq b r^a,
\end{align*}
where $\Comp_c$ means the connected component containing $c$. 
Then $\mu(\mathbb{D}(f(c),r))\leq b r^a$ by (W4). 
Letting $r\rightarrow0$ gives $\mu(\{f(c)\})=0$. 

Secondly, assume $t=0$. 
Without loss of generality, we assume there is no critical point in 
$S:=\partial U \cap \bigcup_{j\geq1}f^{-j}(c)$; otherwise, we consider a critical point in $S$ first. 
By Case 1 and the assumption $t=0$, every point in $S\cup\{c\}$ has the same $\mu$-measure, thus $\mu(\{c\})=0$. 
For every $\varepsilon>0$, there is an $r_1>0$ such that $\mu(\overline{\mathbb{D}(c,r_1)})<\varepsilon$. 
Then $\mu_n(\overline{\mathbb{D}(c,r_1)})<\varepsilon$ for $n$ large enough by (W3). 
For $r_2>0$ small enough and $n$ large enough, 
\begin{align*}
\mu_n(\mathbb{D}(f(c),r_2))=\int_{\Comp_c f^{-1}( \mathbb{D}(f(c),r_2) )} |f'|^{t_n} d\mu_n\leq \mu_n(\mathbb{D}(c,r_1)) M\leq \varepsilon M,
\end{align*}
where $M$ is a positive constant. 
By (W4) again, $\mu(\mathbb{D}(f(c),r_2))\leq \varepsilon M$. 
Letting $\varepsilon \rightarrow0$ gives $\mu(\{f(c)\})=0$. 
This completes the proof. 
\end{proof}

%\begin{remark}
%Since $\mu$ and $\mu_n$ ($n\in\mathbb{N}$) in Lemma \ref{weak limit of conformal measure} are supported on $\partial U$, 
%we can also consider weak topology of $\mathscr{M}(\partial U)$, which is the collection of all probability measures 
%defined on the Borel measurable space $(\partial U,\mathscr{B}(\partial U))$. 
%Then the proof can be simplified a little. 
%However, we encounter probability measures supported on $\partial U_{f,v}$ for different $f\in\Hbar$ in the proof of Proposition \ref{continuous-hdim}. 
%The above proof is more general. 
%\end{remark}

\section{Uniform quasicircles}
\label{appendix u-qc}
Let $\alpha$ be an arc in $\mathbb{C}$, and let $z_1,z_2$ denote its end points. 
The turning of $\alpha$ about $z_1,z_2$ is defined by 
$\Delta(\alpha)=\frac{{\diam}(\alpha)}{|z_1-z_2|}$. 
We say a Jordan curve $\gamma$ in $\mathbb{C}$ a quasicircle with turning $M<\infty$ if 
$\min_{j=1,2} \Delta(\gamma_j) \leq M$ for any $z_1,z_2\in \gamma$ with $z_1\neq z_2$, 
where $\gamma_1,\gamma_2$ are the closures of the connected components of $\gamma\setminus\{z_1,z_2\}$. 
A family of quasicircles is called \emph{uniform}, if these quasicircles have a common finite turning. 
By \cite[Chapter IV, E]{Ahlfors-qcm06}, 
% this is equivalent to that the quasiconformal mappings defining these quasicircles have uniformly bounded dilatations. 
this is equivalent to that these quasicircles are the images of the unit circle 
under a family of quasiconformal homeomorphisms of $\mathbb{C}$, 
whose dilatations are bounded by a common constant smaller than $1$. 
Based on the turning distortion lemma (see \cite[Lemma 6.1.2]{QWY} or Lemma \ref{turning distortion}), we will prove: 

\begin{proposition} 
\label{capture-qc} 
Let $\mathcal{H}\subset\mathcal{F}$ be a capture hyperbolic component with center $f_0$. 
Then the boundaries $\{\partial U_{f,v}{;~} f\in \Hbar, v\in\Tfp\}$ are uniform quasicircles. 
\end{proposition}

\begin{proof} 
Given $v\in\Tfp$ and $f\in\Hbar$. 
Since $\Tfp$ is finite and $\Hbar$ is compact, it suffices to find a neighborhood $\mathcal{N}\subset\Hbar$ of $f$ 
so that $\{\partial U_{g,v}{;~} g\in \mathcal{N}\}$ are uniform quasicircles. 

By Lemma \ref{an iterate of f0}, replacing $f_0$ by an iterate of $f_0$ (every polynomial in $\Hbar$ need to be iterated the same number of times), 
we can make the following assumptions without loss of generality:  
every point in $\Tfp$ is fixed by $f_0$; 
for any $u\in\Tfp$ and any sector $S$ attaching to $U_{f_0}(u)$, if ${{{r}}}(S)$ is periodic, 
then it is fixed by $f_0$. 

Recall from Lemma \ref{invariant external angles} that 
the functions $\theta^-_{g,v}$ and $\theta^+_{g,v}$ are independent of 
$g\in\Hbar$, which are denoted by $\theta^-_v$ and $\theta^+_v$ respectively.
Let $$A = \bigcup_{n\geq1}m^n \big\{t\in\RZ{;~} \text{$\theta^-_{v}(t)\neq \theta^+_{v}(t)$ and  
$\crit(f_0) \cap S_{f_0}(\theta^-_{v}(t), \theta^+_{v}(t)) \neq\emptyset$} \big\},$$ 
where $m = m_{\delta(v)} :\RZ\rightarrow \RZ$, $t\mapsto \delta(v)t$. 
By Lemma \ref{sectors attaching to U}, we have $A\subset\QZ$. 
Furthermore, the set $A$ is finite. 
Let $A_{\mathrm{p}}$ consist of all periodic points in $A$ under $m$. 
By the assumptions on $f_0$, every angle in $A_{\mathrm{p}}$ is fixed by $m$. 
For each $t\in m^{-1}(A)$, let $r_t\geq 0$ be the minimal integer so that $m^{r_t}(t)\in A_{\mathrm{p}}$. 

For $g\in\Hbar$ and $t\in\RZ$, let $x_{g,t} = B_{g,v}^{-1}(e^{2\pi i t}) \in \partial U_{g,v} $. 
Define a homeomorphism $q: \RZ\rightarrow \partial \mathbb{D}$ by $q(t) = e^{2\pi i t}$. 
For $g$ in a small neighborhood $\mathcal{N}\subset\Hbar$ of $f$ and $t\in\RZ$, we will construct two Jordan domains $D^1_{g,t}$ and $D^2_{g,t}$ 
with the following properties. 
\begin{enumerate}[\indent (P1)]
\item We have $x_{g,t}\in D^1_{g,t} \Subset D^2_{g,t}$, and there is an open arc $X_t\subset \RZ$ independent of $g\in\mathcal{N}$ so that 
 $t\in  X_t \subset q^{-1}( B_{g,v}( D^1_{g,t}\cap \partial U_{g,v}) )$. 

\item For every $n\geq1$ and every connected component $V$ of
$g^{-n}(D^2_{g,t})$ intersecting with $\partial U_{g,v}$, 
$$\deg(g^n: V\rightarrow D^2_{g,t}) \leq d_0:=\max_{s\in m^{-1}(A)} \deg(f^{r_s}, x_{f,s}).$$

\item There is a constant $m_t>0$ independent of $g\in\mathcal{N}$ so that $\operatorname{mod}(D^2_{g,t}\setminus \overline{D^1_{g,t}}) \geq m_t$. 
\end{enumerate}
The construction will be divided into two cases according to whether $t\notin A$ or $t\in A$. 

Firstly, assume $t\notin A$. 
Choose $t_1, t'_1, t_2, t'_2\in\QZ$ so that $t_2,t_1,t,t'_1,t'_2$ are in positive cyclic order and $[t_2,t'_2]\cap A = \emptyset$. 
For $g\in \Hbar$ and $j=1,2$, let $D^j_{g,t}$ be the connected component of 
\begin{align*}
\mathbb{C}\setminus \Big(& \overline{
R_{g,v}(t_j)\cup R_g(\theta^-_v(t_j))  \cup
R_{g,v}(t'_j)\cup R_g(\theta^-_v(t'_j))  } \\
&\cup B_g^{-1}((1+j) \partial\mathbb{D}) \cup
B_{g,v}^{-1} ({\textstyle\frac{1}{1+j}} \partial\mathbb{D} )
\Big)
\end{align*}
containing $x_{g,t}$. 
Let $X_t = (t_1,t'_1)$. 
Then (P1) holds for all $g\in\Hbar$. 
If $c_j(f_0) \in S_{f_0}(\theta_v^-(s),\theta_v^+(s))$ for some $1\leq j\leq d-1$ and some $s\in m^{-1}(A)$, 
then $c_j(g)\in \overline{S_{g}(\theta_v^-(s),\theta_v^+(s))}$ for any $g\in \Hbar$ 
by the stability of external rays (see Lemmas \ref{key lemma} and \ref{stability of external rays}). 
Thus for every $n\geq1$ and every connected component $V$ of
$g^{-n}(D^2_{g,t})$ intersecting with $\partial U_{g,v}$, we have $\deg(g^n: V\rightarrow D^2_{g,t}) = 1$. 
This implies (P2) holds for all $g\in\Hbar$. 
By the stability of dynamical rays and equipotential curves, the modulus $\operatorname{mod}(D^2_{h,t}\setminus \overline{D^1_{h,t}})$ is continuous with respect to $h\in\Hbar$.  
Taking $m_t=\min_{h\in \Hbar} \operatorname{mod}(D^2_{h,t}\setminus \overline{D^1_{h,t}})$ gives (P3). 

Secondly, assume $t\in A$. 
Since $\Hbar$ is compact and $A_{\mathrm{p}}$ is finite, there is a $\rho>0$ so that 
$\mathbb{D}(x_{g,s},\rho) = \{ z\in\mathbb{C}{;~} |z-x_{g,s}|<\rho \}$ is a linearization neighborhood of $g$ at the repelling fixed point $x_{g,s}$ for any $g\in\Hbar$ and $s\in A_{\mathrm{p}}$. 
For each $s\in m^{-1}(A)$, we have $g^{r_s}(x_{g,s}) = x_{g,m^{r_s}(s)}$, so we can define 
$E^j_{g,s}$ to be the connected component of $g^{-r_s}( \mathbb{D}(x_{g, m^{r_s}(s)}, \frac{j}{2}\rho) )$ containing $x_{g,s}$, where $j=1,2$. 
% $E_{g,t} = \Comp_{x_{f,t}} f^{-r_t}( \mathbb{D}(x_{f, m^{r_t}(t)}, \rho) )$
We can decrease $\rho$ so that 
\begin{itemize}
\item the closures of $E^2_{f,s}$ for $s\in m^{-1}(A)$ are pairwise disjoint; 
\item if $s\in m^{-1}(A)\setminus A_{\mathrm{p}}$, then 
$\deg( f: E^2_{f,s}\rightarrow  E^2_{f,m(s)} ) = \deg( f, x_{f,s} )$ and $\crit(f)\cap \partial E^2_{f,s} = \emptyset$; 
\item if $s\in m^{-1}(A)\setminus A$, then $\overline{E^2_{f,s}}\subset D^2_{f,s}$. 
\end{itemize}
Shrinking $\mathcal{N}$ if necessary, for any $g\in\mathcal{N}$, we have: 
\begin{itemize}
\item the closures of $E^2_{g,s}$ for $s\in m^{-1}(A)$ are pairwise disjoint; 
\item if $c_j(f) = x_{f,s}$ for some $1\leq j\leq d-1$ and some $s\in m^{-1}(A)$, then $c_j(g) \in E^1_{g,s}$; 
if $c_j(f) \notin \{ x_{f,s} {;~} s\in m^{-1}(A) \}$ for some $1\leq j\leq d-1$, then $c_j(g)\notin \bigcup_{s\in m^{-1}(A)}\overline{E^2_{g,s}}$; 
\item if $s\in m^{-1}(A)\setminus A$, then $\overline{E^2_{g,s}}\subset D^2_{g,s}$. 
\end{itemize}
For $g\in\mathcal{N}$ and $j=1,2$, let $D^j_{g,t} = E^j_{g,t}$. 
Then (P1) follows from Proposition \ref{continuity of Ufv}. 
By the locations of $c_j(g)$ for $1\leq j\leq d-1$, we have $\deg(g:E^2_{g,s}\rightarrow E^2_{g,m(s)})=\deg(f,x_{f,s})$ for each $s\in m^{-1}(A)\setminus A_{\mathrm{p}}$. 
Combining this with $\overline{E^2_{g,s}}\subset D^2_{g,s}$ for $s\in m^{-1}(A)\setminus A$ gives (P2). 
Since $g^{r_t}: E^2_{g,t}\setminus \overline{E^1_{g,t}} \rightarrow E^2_{g,m^{r_t}(t)}\setminus \overline{E^1_{g,m^{r_t}(t)}}$ is a covering map of degree $\deg(f^{r_t},x_{f,t})$, 
$${\operatorname{mod}}{\left(E^2_{g,t}\setminus \overline{E^1_{g,t}}\right)} 
= \frac{{\operatorname{mod}}{\left(E^2_{g,m^{r_t}(t)}\setminus \overline{E^1_{g,m^{r_t}(t)}}\right)}}{\deg(f^{r_t},x_{f,t})}
= \frac{\log(2)}{2\pi \deg(f^{r_t},x_{f,t})}.$$
Letting $m_t = \frac{\log(2)}{2\pi \deg(f^{r_t},x_{f,t})}$ gives (P3). 
This completes the construction.

Since $\RZ$ is compact, we can extract a finite subcover $\{X_{t_k}\}_{k=1}^N$ of $\{X_t\}_{t\in\RZ}$. 
Let $\varepsilon\in(0,\frac{1}{2})$ be a Lebesgue number of $\{X_{t_k}\}_{k=1}^N$. 
That is, every subset of $\RZ$ having diameter less than $\varepsilon$ is contained in some member of $\{X_{t_k}\}_{k=1}^N$. 
Recall from Proposition \ref{continuity of Ufv} that the function $$\Psi_v:\begin{cases} \Hbar\times\overline{\mathbb{D}} 
\rightarrow \mathbb{C}, \\ (g,z)\mapsto B_{g,v}^{-1}(z) \end{cases}$$ is continuous. 
Recall also that $\operatorname{frac}:\RZ\rightarrow[0,1)$ is the fractional part function which 
maps each point of $\RZ$ to its unique representative in $[0,1)$. 
Let $$M_0 = \frac{\max\big\{{\diam}(\partial U_{g,v}){;~} g\in\Hbar\big\}}
{\min\big\{|\Psi_v(g, q(t_2)) - \Psi_v(g, q(t_1))|{;~} 
\text{$g\in\Hbar$ and $\operatorname{frac}(t_2-t_1) \in [\frac{\varepsilon}{\delta(v)}, \frac{1}{2}]$} \big\}}.$$
Then $M_0<\infty$.

Let $t_1,t_2\in\RZ$ with $\operatorname{frac}(t_2-t_1) \in (0, \frac{1}{2}]$.
If $\operatorname{frac}(t_2-t_1) \in [\frac{\varepsilon}{\delta(v)}, \frac{1}{2}]$, then 
$$\Delta(L_g(t_1,t_2)) \leq M_0$$
for any $g\in\Hbar$, where $L_g(t_1,t_2) = \Psi_v( g, e^{2\pi i[t_1,t_2]} )$. 
Now assume $\operatorname{frac}(t_2-t_1) \in (0,\frac{\varepsilon}{\delta(v)})$. 
Let $n\geq1$ be the minimal integer so that 
$$\operatorname{frac}(m^n(t_2)-m^n(t_1)) \in \Big[\frac{\varepsilon}{\delta(v)}, \varepsilon\Big).$$
By the choice of $\varepsilon$, we can choose $1\leq k\leq N$ so that $[m^n(t_1),m^n(t_2)] \subset X_{t_k}$. 
Then (P1) gives $$g^n( L_g(t_1,t_2) ) = L_g(m^n(t_1),m^n(t_2)) \subset D^1_{g,t_k}$$ for any $g\in\mathcal{N}$. 
For $j=1,2$, let $V^j_g$ denote the connected component of $g^{-n}(D^j_{g,t_k})$ containing $L_g(t_1,t_2)$. 
Then (P2) gives $\deg(g^n:V^2_g\rightarrow D^2_{g,t_k})\leq d_0$, and (P3) gives $\operatorname{mod}(D^2_{g,t_k}\setminus\overline{D^1_{g,t_k}})\geq m_0$, 
where $m_0 = \min_{1\leq j\leq N} m_{t_j}$. 
Applying Lemma \ref{turning distortion} to 
$g^n: (V^2_g,V^1_g)\rightarrow (D^2_{g,t_k}, D^1_{g,t_k})$, there is a constant $C(d_0,m_0)>0$ such that 
\begin{align*}
\Delta(L_g(t_1,t_2)) &\leq C(d_0,m_0) \Delta( g^n( L_g(t_1,t_2) ) ) \\
&= C(d_0,m_0) \Delta( L_g(m^n(t_1),m^n(t_2)) )\\
&\leq C(d_0,m_0) M_0.
\end{align*}
Thus for any  $g\in \mathcal{N}$ and any $t_1,t_2\in\RZ$ with $\operatorname{frac}(t_2-t_1) \in (0, \frac{1}{2}]$, 
the turning $\Delta(L_g(t_1,t_2))$ is bounded by $\max\{M_0, C(d_0,m_0) M_0\}$. 
This completes the proof. 
\end{proof}

\section{Uniform convergence of univalent functions}
\label{appendix uniform-convergence}
For a full nondegenerate continuum $K\subset\mathbb C$, by the Riemann mapping theorem, 
there is a unique conformal isomorphism $\phi_K:\mathbb{C}\setminus\overline{\mathbb{D}}
\rightarrow \mathbb{C}\setminus K$ such that $\lim_{z\rightarrow\infty} \phi_K(z)/z >0$. 
Let $K$ and $K_n$ ($n\in\mathbb{N}$) be full nondegenerate continua in $\mathbb C$ 
so that $K_n\rightarrow K$ in Hausdorff topology. 
By the Carath\'eodory kernel convergence theorem, the map $\phi_{K_n}$ converges to 
$\phi_K$ uniformly on compact subsets of $\mathbb{C}\setminus\overline{\mathbb{D}}$. 

In general, the convergence may not be uniform on $\{1<|z|\leq 2\}$, 
% the whole domain $\mathbb C\setminus \overline{\mathbb D}$, 
even if $\partial K$ and $\partial K_n$ ($n\in\mathbb{N}$) are locally connected. 
One example is the following: 
$$K_n=\{x+i\sin(n\pi x)\in \mathbb C{;~} -1\leq x\leq 1\}, \quad K=[-1,1]\times [-1,1].$$
At the same time, the convergence may not be uniform on $\{2\leq|z|<\infty\}$. 
For example, consider $\phi_{K_n}(z) = (1+\frac{1}{n})z$. 
However, what we really care about is the uniform convergence on $\{1<|z|\leq 2\}$. 
 
In order to get the uniform convergence, we need to pose some conditions to control the boundary behavior 
of the sequence of Riemann mappings. 

Let $\phi:\mathbb{C}\setminus\overline{\mathbb{D}} \rightarrow \mathbb{C}$ be a univalent function 
with $\phi(z)\rightarrow\infty$ as $z\rightarrow\infty$. 
For $t_1\neq t_2\in\RZ$, we define 
\begin{equation}
\label{Iphi-bigcap}
I_\phi^{t_1,t_2} = \bigcap_{r>1}  \overline{\phi(S_r^{t_1,t_2})}, 
\end{equation}
where $S_r^{t_1,t_2}=\{\rho e^{2\pi it} {;~} \rho\in(1,r), t\in (t_1,t_2)\}$. 
It is easy to see that 
\begin{equation}
\label{Iphi-sqcup}
\partial \phi(S_r^{t_1,t_2}) = I_\phi^{t_1,t_2}\sqcup \phi((1,r)e^{2\pi i t_1})\sqcup \phi(r e^{2\pi i [t_1,t_2]})\sqcup\phi((1,r)e^{2\pi i t_2}).
\end{equation}
If $\phi$ can be extended to $\mathbb{C}\setminus\mathbb{D}$ continuously, then $I_\phi^{t_1,t_2} = \phi(e^{2\pi i [t_1,t_2]})$. 
Recall that ${\diam}(X) = \max\{|x_1-x_2|{;~} x_1,x_2\in X\}$ is the diameter of the nonempty compact set $X$ in $\mathbb{C}$. 

\begin{proposition}
\label{non-oscillation condition}
Let $\{\phi_n\}_{n\in\mathbb{N}}$ be a sequence of univalent functions on $\mathbb{C}\setminus\overline{\mathbb{D}}$ 
with $\phi_n(z)\rightarrow\infty$ as $z\rightarrow\infty$. 
Assume $\phi_n$ converges to a nonconstant function $\phi$ uniformly on compact subsets of $\mathbb{C}\setminus\overline{\mathbb{D}}$. 
Then the following two conditions are equivalent. 
\begin{enumerate}
\item The function $\phi$ can be extended to $\mathbb{C}\setminus\mathbb{D}$ continuously, and $\phi_n$ converges to $\phi$ uniformly on $\{1<|z|\leq 2\}$. 
\item $\displaystyle \lim_{m\rightarrow\infty}~ \max_{0\leq j\leq m-1}\limsup_{n\rightarrow \infty}
~{\diam} \Big( I_{\phi_n}^{\frac{j}{m},\frac{j+1}{m}} \Big) = 0.$ 
\end{enumerate}
\end{proposition}

To prove Proposition \ref{non-oscillation condition}, we will use 
Lemma \ref{converge uniformly on some external rays}
and a basic property of univalent functions as follows. 

Define a distance function between any point $x$ and any nonempty compact set $Y$ in $\mathbb{C}$ by 
$d(x,Y) = \min_{y\in Y} |x-y|$. 
Define a distance function between any two nonempty compact sets $X$ and $Y$ in $\mathbb{C}$ by 
$d(X,Y) = \max_{x\in X} d(x,Y)$. 
Then $d(X,Y)\leq \varepsilon$ means $X$ is contained in the $\varepsilon$-neighborhood of $Y$, i.e. 
$$X\subset Y_\varepsilon:=\{z\in\mathbb{C}{;~} d(z,Y)\leq \varepsilon\}.$$

\begin{lemma}
\label{lower semi-continuity of boundaries}
Let $\{\phi_n\}_{n\in\mathbb{N}}$ be a sequence of uniformly bounded univalent functions on $\mathbb{D}$. 
Assume $\phi_n$ converges to a nonconstant function $\phi$ uniformly on compact subsets of $\mathbb{D}$. 
Then $$\lim_{n\rightarrow\infty} d(\partial\phi(\mathbb{D}),\partial\phi_n(\mathbb{D})) =0.$$
\end{lemma}

\begin{proof}
If not, we can choose $\varepsilon>0$ and pass to a subsequence such that 
$d(\partial\phi(\mathbb{D}),$ $\partial\phi_n(\mathbb{D}))> 2\varepsilon$ for any $n\in\mathbb{N}$. 
That is, $\partial\phi(\mathbb{D})\not\subset (\partial\phi_n(\mathbb{D}))_{2\varepsilon}$. 
Choose $x_n\in \partial\phi(\mathbb{D})\setminus (\partial\phi_n(\mathbb{D}))_{2\varepsilon}$. 
By passing to a further subsequence, we may assume $x_n$ tends to $x$. 
Then $x\in \partial\phi(\mathbb{D})$ and 
$\mathbb{D}(x,\varepsilon)\cap \partial\phi_n(\mathbb{D})=\emptyset$ for $n$ large enough. 

Take $y\in \mathbb{D}(x,\frac{\varepsilon}{2})\cap \phi(\mathbb{D})$. 
Then $\phi_n(\phi^{-1}(y)) \in \mathbb{D}(x,\varepsilon)$ for $n$ large enough. 
Combining this with $\mathbb{D}(x,\varepsilon)\cap \partial\phi_n(\mathbb{D})=\emptyset$, 
we have $\mathbb{D}(x,\varepsilon) \subset \phi_n(\mathbb{D})$ for $n$ large enough. 
By passing to a further subsequence, we may assume $\phi_n^{-1}|_{\mathbb{D}(x,\varepsilon)}$ converges to 
$\psi$ uniformly on compact subsets of $\mathbb{D}(x,\varepsilon)$. 

In the following, we always assume $n$ is large enough so that $\phi_n(\phi^{-1}(y)) \in \mathbb{D}(x,\frac{\varepsilon}{2})$. 
Let $ M = \sup_{n\in\mathbb{N}} |(\phi_n)'(\phi^{-1}(y))|< \infty$. 
Then $|(\phi_n^{-1})'(\phi_n(\phi^{-1}(y)))|\geq \frac{1}{M}$. 
By the Koebe distortion theorem (Lemma \ref{Koebe-distortion}), we have 
$|(\phi_n^{-1})'(x)|\geq \frac{1}{\CK(1/2) M}>0$. 
Thus $\psi$ is nonconstant. 
It follows that $\psi(\mathbb{D}(x,\varepsilon))\subset\mathbb{D}$. 
Let $r=\frac{1-|\psi(x)|}{2}$. 
% When $n$ is large enough, we have $\phi_n^{-1}(x)\in \mathbb{D}(\psi(x),r)$. 
Because $\phi_n|_{\mathbb{D}(\psi(x),r)}\rightarrow \phi|_{\mathbb{D}(\psi(x),r)}$  
uniformly and $\phi_n^{-1}(x)\rightarrow \psi(x)$, we have 
$$|\phi_n(\phi_n^{-1}(x))- \phi(\psi(x))|\leq 
|\phi_n(\phi_n^{-1}(x))- \phi(\phi_n^{-1}(x))|+|\phi(\phi_n^{-1}(x))- \phi(\psi(x))| \rightarrow0.$$
This implies $x = \phi(\psi(x))\in\phi(\mathbb{D})$, which contradicts $x\in \partial\phi(\mathbb{D})$. 
% we have $\phi\circ \psi = \id_{\mathbb{D}(x,\varepsilon)}$. 
% similar to the proof the Carath\'eodory kernel convergence theorem
\end{proof}

\begin{proof} 
[Proof of Proposition \ref{non-oscillation condition}]
It is clear that  $(1)\Rightarrow(2)$. 
We now show $(2)\Rightarrow(1)$. 

By Lemma \ref{converge uniformly on some external rays}, 
we get a full-measure subset $E$ of $\RZ$, 
and a subsequence $\{\phi_{n_k}\}_{k\in\mathbb{N}}$ of $\{\phi_n\}_{n\in\mathbb{N}}$, 
such that $\phi_{n_k}$ converges to $\phi$ uniformly on $[1,2]e^{2\pi i t}$ for any $t\in E$.

Given $\varepsilon >0$. Choose $m\geq 2$ so that $$\max_{0\leq j\leq m-1} \limsup_{k\rightarrow \infty} ~ 
{\diam} \Big( I_{\phi_{n_k}}^{\frac{j}{m},\frac{j+1}{m}} \Big) \leq \varepsilon.$$ Choose 
$\theta_j\in E\cap(\frac{j}{m},\frac{j+1}{m})$ for $j=0,\dots,m-1$, and set $\theta_m = \theta_0$. Note that $\phi_{n_k}\rightarrow \phi$ uniformly on compact subsets of $S_2
^{\theta_j, \theta_{j+1}}$. By Lemma \ref{lower semi-continuity of boundaries}, we have 
$$\lim_{k\rightarrow\infty} d \Big( \partial \phi(S_2^{\theta_j,\theta_{j+1}}), \partial \phi_{n_k}
(S_2^{\theta_j,\theta_{j+1}}) \Big)=0.$$ Then by $\theta_j,\theta_{j+1}
\in E$ and (\ref{Iphi-sqcup}), we have $$\lim_{k\rightarrow \infty} d \Big( I_\phi^{\theta_j, \theta_{j+1}}, 
I_{\phi_{n_k}}^{\theta_j, \theta_{j+1}} \Big) =0.$$ Thus $${\diam} \Big( I_\phi^{\theta_j, 
\theta_{j+1}} \Big) \leq \liminf_{k\rightarrow \infty} ~ {\diam} \Big( I_{\phi_{n_k}}
^{\theta_j,\theta_{j+1}} \Big) \leq \liminf_{k\rightarrow \infty} ~ {\diam} \Big( I_{
\phi_{n_k}}^{\frac{j}{m},\frac{j+2}{m}} \Big) \leq 2\varepsilon.$$ By (\ref{Iphi-bigcap}), there is an 
$r\in(1,2)$ sufficiently close to $1$ such that $$
{\diam} \Big( \partial \phi(S_r^{\theta_j,\theta_{j+1}}) \Big) =
{\diam} \Big( \overline{ \phi(S_r^{\theta_j,\theta_{j+1}}) } \Big)
\leq {\diam} \Big( I_\phi^{\theta_j,\theta_{j+1}} \Big) + \varepsilon \leq 3\varepsilon$$ for any $j=0,\dots,m-1$. 
This implies $\phi$ can be extended to $\mathbb{C}\setminus\mathbb{D}$ continuously. 

By the choice of $m$, there is an $N\geq 0$ independent of $j$ such that 
$${\diam} \Big( I_{\phi_{n_k}}^{\theta_j,\theta_{j+1}} \Big) \leq {\diam} \Big( I_{\phi_{n_k}}^{\frac{j}{m},\frac{j+2}{m}} \Big)\leq 2\varepsilon$$ 
for any $k\geq N$. 
By $\theta_j,\theta_{j+1}\in E$ and (\ref{Iphi-sqcup}) again, there is an $N'\geq N$ independent of $j$ such that 
$${\diam} \Big( \partial \phi_{n_k}(S_r^{\theta_j,\theta_{j+1}}) \setminus I_{\phi_{n_k}}
^{\theta_j,\theta_{j+1}} \Big)
\leq {\diam} \Big( \partial \phi(S_r^{\theta_j,\theta_{j+1}}) \setminus I_\phi
^{\theta_j,\theta_{j+1}} \Big) + \varepsilon\leq 4\varepsilon$$ 
for any $k\geq N'$. 
It follows that 
\begin{align*}
      {\diam} \Big( \phi_{n_k}(S_r^{\theta_j,\theta_{j+1}}) \Big)
&=    {\diam} \Big( \partial \phi_{n_k}(S_r^{\theta_j,\theta_{j+1}}) \Big) \\
&\leq {\diam} \Big( I_{\phi_{n_k}}^{\theta_j,\theta_{j+1}} \Big) 
+     {\diam} \Big( \partial \phi_{n_k}(S_r^{\theta_j,\theta_{j+1}}) 
                    \setminus I_{\phi_{n_k}}^{\theta_j,\theta_{j+1}} \Big)\\
&\leq 2\varepsilon + 4\varepsilon = 6\varepsilon
\end{align*}
for any $k\geq N'$. 
Choose $N''\geq N'$ such that $|\phi_{n_k}(z)-\phi(z)|\leq \varepsilon$ for any $r\leq |z|\leq 2$ and any $k\geq N''$. 
Then for any $z\in \overline{S_r^{\theta_j,\theta_{j+1}}} \setminus \partial\mathbb{D}$ and any $k\geq N''$, we have 
\begin{align*} 
&|\phi_{n_k}(z)- \phi(z)|\\
\leq~& |\phi_{n_k}(z)-\phi_{n_k}(r e^{2\pi i \theta_j})|
     + |\phi_{n_k}(r e^{2\pi i \theta_j})-\phi(r e^{2\pi i \theta_j})| 
     + |\phi(r e^{2\pi i \theta_j})-\phi(z)| \\
\leq~& {\diam} \Big( \phi_{n_k}(S_r^{\theta_j,\theta_{j+1}}) \Big) + \varepsilon + 
       {\diam} \Big( \phi(S_r^{\theta_j,\theta_{j+1}}) \Big) \\
\leq~& 6\varepsilon+\varepsilon+3\varepsilon=10\varepsilon.
\end{align*}
Therefore $|\phi_{n_k}(z)-\phi(z)|\leq 10\varepsilon$ for any $1< |z|\leq 2$ and any $k\geq N''$. 
Thus $\phi_{n_k}$ converges to $\phi$ uniformly on $\{1<|z|\leq 2\}$. 

The above discussion shows any subsequence of $\{\phi_n\}_{n\in\mathbb{N}}$ has a subsequence converging to $\phi$
uniformly on $\{1<|z|\leq 2\}$, hence $\phi_n$ converges to $\phi$ uniformly on $\{1<|z|\leq 2\}$. 
\end{proof}

\addtocontents{toc}{\protect\setcounter{tocdepth}{1}}
%\bibliographystyle{plain} 
%\bibliography{references}

\addtocontents{toc}{\protect\setcounter{tocdepth}{2}} %This line is for the depth of bookmarks
\end{document}